\pdfoutput=1
\documentclass[10pt]{amsart} 
\newif\ifpersonal
\newif\ifpersonalsection
\usepackage{amsmath,amsthm,amssymb,mathrsfs,mathtools,bm,centernot} 
\usepackage{microtype,lmodern} 
\usepackage{enumitem,comment,braket,xspace,tikz-cd,adjustbox} 
\usepackage[T1]{fontenc} 
\usepackage[centering,vscale=0.7,hscale=0.7]{geometry}
\usepackage[hidelinks,linktoc=all]{hyperref}
\usepackage[capitalize]{cleveref}
\allowdisplaybreaks \usepackage{graphicx}

\numberwithin{equation}{section}
\theoremstyle{plain}
\newtheorem{theorem}[equation]{Theorem}
\newtheorem{lemma}[equation]{Lemma}
\newtheorem*{lemma*}{Lemma}
\newtheorem{question}[equation]{Question}
\newtheorem{claim}[equation]{Claim}
\newtheorem*{claim*}{Claim}
\newtheorem{conjecture}[equation]{Conjecture}
\newtheorem{proposition}[equation]{Proposition}
\newtheorem*{proposition*}{Proposition}
\newtheorem{corollary}[equation]{Corollary}
\theoremstyle{definition}
\newtheorem{definition}[equation]{Definition}
\newtheorem{definition-theorem}[equation]{Definition-Theorem}
\newtheorem{definition-lemma}[equation]{Definition-Lemma}
\newtheorem{definition-remark}[equation]{Definition-Remark}
\newtheorem{definition-proposition}[equation]{Definition-Proposition}
\newtheorem{definition-corollary}[equation]{Definition-Corollary}
\newtheorem{construction}[equation]{Construction}
\newtheorem{assumption}[equation]{Assumption}
\newtheorem{assumptions}[equation]{Assumptions}
\newtheorem{notation}[equation]{Notation}
\newtheorem{example}[equation]{Example}
\newtheorem{remark}[equation]{Remark}

\ifpersonal
\newcommand{\todo}[1]{\textcolor{red}{(Todo: #1)}}
\newcommand{\personal}[1]{\textcolor[rgb]{0,0,1}{(Personal: #1)}}
\newcommand{\discussion}[1]{\textcolor{violet}{(Discussion: #1)}}
\else
\newcommand{\todo}[1]{\ignorespaces}
\newcommand{\personal}[1]{\ignorespaces}
\newcommand{\discussion}[1]{\ignorespaces}
\fi

\providecommand{\abs}[1]{\lvert#1\rvert}

\newcommand{\bbA}{\mathbb A}

\newcommand{\bbB}{\mathbb B}

\newcommand{\bbD}{\mathbb D}
\newcommand{\bbE}{\mathbb E}
\newcommand{\bbF}{\mathbb F}
\newcommand{\bbG}{\mathbb G}
\newcommand{\bbH}{\mathbb H}
\newcommand{\bbk}{\mathbb F}

\newcommand{\bbL}{\mathbb L}
\newcommand{\bbM}{\mathbb M}
\newcommand{\obbM}{\widebar{\bbM}}
\newcommand{\bbN}{\mathbb N}
\newcommand{\bbP}{\mathbb P}

\newcommand{\bbQ}{\mathbb Q}
\newcommand{\bbR}{\mathbb R}
\newcommand{\bbT}{\mathbb T}

\newcommand{\bbU}{\mathbb U}

\newcommand{\bbW}{\mathbb W}
\newcommand{\tbbW}{\widetilde{\bbW}}

\newcommand{\bbV}{\mathbb V}

\newcommand{\bbX}{\mathbb X} 
\newcommand{\bbY}{\mathbb Y}
\newcommand{\tbbY}{\widetilde{\bbY}}
\newcommand{\obbY}{\widebar{\bbY}}
\newcommand{\bbZ}{\mathbb Z}
\newcommand{\tbbD}{\widetilde{\bbD}}

\newcommand{\fA}{\mathfrak A}
\newcommand{\fB}{\mathfrak B}
\newcommand{\fC}{\mathfrak C}
\newcommand{\fD}{\mathfrak D}

\newcommand{\fF}{\mathfrak F}
\newcommand{\fG}{\mathfrak G}
\newcommand{\fH}{\mathfrak H}

\newcommand{\fM}{\mathfrak M}
\newcommand{\fP}{\mathfrak P}

\newcommand{\fS}{\mathfrak S}
\newcommand{\fT}{\mathfrak T}

\newcommand{\fV}{\mathfrak V}
\newcommand{\fW}{\mathfrak W}
\newcommand{\fc}{\mathfrak c}
\newcommand{\fX}{\mathfrak X}
\newcommand{\fY}{\mathfrak Y}
\newcommand{\fZ}{\mathfrak Z}

\newcommand{\ofM}{\widebar{\fM}}

\newcommand{\cC}{\mathcal C}
\newcommand{\cD}{\mathcal D}

\newcommand{\cE}{\mathcal E}

\newcommand{\cH}{\mathcal H}
\newcommand{\cG}{\mathcal G}

\newcommand{\cJ}{\mathcal J}

\newcommand{\cL}{\mathcal L}

\newcommand{\cO}{\mathcal O}
\newcommand{\cP}{\mathcal P}

\newcommand{\cT}{\mathcal T}

\newcommand{\cV}{\mathcal V}
\newcommand{\cW}{\mathcal W}

\newcommand{\cX}{\mathcal X}
\newcommand{\cY}{\mathcal Y}
\newcommand{\cZ}{\mathcal Z}
\newcommand{\tcZ}{\tilde{\cZ}}
\newcommand{\ocY}{\widebar{\cY}}

\newcommand{\bA}{\mathbf A}
 
\newcommand{\bD}{\mathbf D}

\newcommand{\bP}{\mathbf P}

\newcommand{\bx}{\mathbf x}

\newcommand{\scrH}{\mathscr{H}}

\makeatletter
\let\save@mathaccent\mathaccent
\newcommand*\if@single[3]{%
	\setbox0\hbox{${\mathaccent"0362{#1}}^H$}%
	\setbox2\hbox{${\mathaccent"0362{\kern0pt#1}}^H$}%
	\ifdim\ht0=\ht2 #3\else #2\fi
}
\newcommand*\rel@kern[1]{\kern#1\dimexpr\macc@kerna}
\newcommand*\widebar[1]{\@ifnextchar^{{\wide@bar{#1}{0}}}{\wide@bar{#1}{1}}}
\newcommand*\wide@bar[2]{\if@single{#1}{\wide@bar@{#1}{#2}{1}}{\wide@bar@{#1}{#2}{2}}}
\newcommand*\wide@bar@[3]{%
	\begingroup
	\def\mathaccent##1##2{%
		\let\mathaccent\save@mathaccent
		\if#32 \let\macc@nucleus\first@char \fi
		\setbox\z@\hbox{$\macc@style{\macc@nucleus}_{}$}%
		\setbox\tw@\hbox{$\macc@style{\macc@nucleus}{}_{}$}%
		\dimen@\wd\tw@
		\advance\dimen@-\wd\z@
		\divide\dimen@ 3
		\@tempdima\wd\tw@
		\advance\@tempdima-\scriptspace
		\divide\@tempdima 10
		\advance\dimen@-\@tempdima
		\ifdim\dimen@>\z@ \dimen@0pt\fi
		\rel@kern{0.6}\kern-\dimen@
		\if#31
		\overline{\rel@kern{-0.6}\kern\dimen@\macc@nucleus\rel@kern{0.4}\kern\dimen@}%
		\advance\dimen@0.4\dimexpr\macc@kerna
		\let\final@kern#2%
		\ifdim\dimen@<\z@ \let\final@kern1\fi
		\if\final@kern1 \kern-\dimen@\fi
		\else
		\overline{\rel@kern{-0.6}\kern\dimen@#1}%
		\fi
	}%
	\macc@depth\@ne
	\let\math@bgroup\@empty \let\math@egroup\macc@set@skewchar
	\mathsurround\z@ \frozen@everymath{\mathgroup\macc@group\relax}%
	\macc@set@skewchar\relax
	\let\mathaccentV\macc@nested@a
	\if#31
	\macc@nested@a\relax111{#1}%
	\else
	\def\gobble@till@marker##1\endmarker{}%
	\futurelet\first@char\gobble@till@marker#1\endmarker
	\ifcat\noexpand\first@char A\else
	\def\first@char{}%
	\fi
	\macc@nested@a\relax111{\first@char}%
	\fi
	\endgroup
}
\makeatother

\newcommand{\oA}{\widebar A}
\newcommand{\oB}{\widebar B}
\newcommand{\oC}{\widebar C}

\newcommand{\oE}{\widebar E}

\newcommand{\oK}{\widebar K}

\newcommand{\oM}{\widebar M}

\newcommand{\oP}{\widebar P}

\newcommand{\oQ}{\widebar Q}
\newcommand{\oR}{\widebar R}
\newcommand{\oS}{\widebar S}

\newcommand{\oV}{\widebar V}

\newcommand{\oGamma}{\widebar\Gamma}

\newcommand{\oSigma}{\widebar\Sigma}

\newcommand{\Ex}{\operatorname{Ex}}

\newcommand{\hD}{\widehat D}

\newcommand{\hR}{\widehat R}

\newcommand{\tA}{\widetilde A}

\newcommand{\tW}{\widetilde W}
\newcommand{\tX}{\widetilde X}
\newcommand{\tY}{\widetilde Y}
\newcommand{\tZ}{\widetilde Z}

\newcommand{\tbeta}{\widetilde\beta}

\newcommand{\longto}{\longrightarrow}

\newcommand{\tf}{\widetilde f}

\newcommand{\tk}{\widetilde k}

\newcommand{\tSigma}{\widetilde\Sigma}

\newcommand{\oa}{\widebar a}
\newcommand{\ob}{\widebar b}

\newcommand{\os}{\widebar s}

\newcommand{\ow}{\widebar w}

\newcommand{\obbR}{\widebar{\bbR}}

\newcommand{\ofC}{\widebar{\fC}}

\newcommand{\ff}{\mathfrak f}
\newcommand{\fh}{\mathfrak h} 
\newcommand{\fm}{\mathfrak m}

\newcommand {\fd} {\mathfrak{d}}

\DeclareFontFamily{U}{BOONDOX-calo}{\skewchar\font=45 }
\DeclareFontShape{U}{BOONDOX-calo}{m}{n}{<-> s*[1.05] BOONDOX-r-calo}{}
\DeclareFontShape{U}{BOONDOX-calo}{b}{n}{<-> s*[1.05] BOONDOX-b-calo}{}
\DeclareMathAlphabet{\mathcalboondox}{U}{BOONDOX-calo}{m}{n}

\newcommand{\Bnd}{\mathrm{Bd}}

\newcommand{\Cont}{\mathrm{Cont}}

\newcommand{\DC}{\mathrm{DC}}

\newcommand{\ISk}{\mathrm{ISk}}
\newcommand{\NT}{\mathrm{NT}}

\newcommand{\SP}{\mathrm{SP}}

\newcommand{\an}{\mathrm{an}}
\newcommand{\gen}{\mathrm{gen}}

\newcommand{\alg}{\mathrm{alg}}

\newcommand{\dom}{\mathrm{dom}}
\newcommand{\ess}{\mathrm{ess}}

\newcommand{\ev}{\mathrm{ev}}

\newcommand{\gp}{\mathrm{gp}}

\newcommand{\sot}{\mathrm{sot}}
\newcommand{\smot}{\mathrm{smot}}

\newcommand{\pr}{\mathrm{pr}}
\newcommand{\red}{\mathrm{red}}
\newcommand{\sd}{\mathrm{sd}}

\newcommand{\sm}{\mathrm{sm}}

\newcommand{\tr}{\mathrm{tr}}
\newcommand{\tri}{trmp}

\newcommand{\trop}{\mathrm{trop}}
\newcommand{\trunc}{\mathrm{trunc}}

\newcommand{\hsigma}{\hat{\sigma}}
\newcommand{\hosigma}{\hat{\overline{\sigma}}}
\newcommand{\osigma}{\overline{\sigma}}
\newcommand{\whscrO}{\widehat{\mathscr{O}}}
\newcommand{\scrY}{\mathscr{Y}}

\newcommand{\inv}{^{-1}}
\newcommand{\kanal}{$k$-analytic\xspace}
\newcommand{\kc}{k^\circ}

\newcommand{\bb}{\mathbb}
\newcommand{\al}{\text{al}}
\newcommand{\fBl}{\widehat{\text{Bl}}}
\newcommand{\whTV}{\widehat{\text{TV}}}

\newcommand{\codim}{\text{codim}}
\newcommand{\Bl}{\text{Bl}}

\newcommand{\Prim}{\text{Prim}}

\newcommand{\whbbX}{\widehat{\bbX}}
\newcommand{\whbbD}{\widehat{\bbD}}
\newcommand{\whbbL}{\widehat{\bbL}}
\newcommand{\whbbY}{\widehat{\mathbb{Y}}}
\newcommand{\whSigma}{\widehat{\Sigma}}
\newcommand{\whscrI}{\widehat{\mathscr{I}}}
\newcommand{\rec}{\text{rec}}

\DeclareMathOperator{\CH}{CH}

\DeclareMathOperator{\Cox}{Cox}

\DeclareMathOperator{\Hom}{Hom}

\DeclareMathOperator{\gr}{gr}

\DeclareMathOperator{\link}{link}
\DeclareMathOperator{\Map}{Map}
\DeclareMathOperator{\MoriFan}{MoriFan}
\DeclareMathOperator{\Mov}{Mov}
\DeclareMathOperator{\Ample}{Ample}
\DeclareMathOperator{\MovFan}{MovFan}
\DeclareMathOperator{\MovSec}{MovSec}
\DeclareMathOperator{\NE}{NE}
\DeclareMathOperator{\NB}{NB}
\DeclareMathOperator{\Nef}{Nef}
\DeclareMathOperator{\Pic}{Pic}
\DeclareMathOperator{\In}{In}

\DeclareMathOperator{\Proj}{Proj}

\DeclareMathOperator{\Sk}{Sk}

\DeclareMathOperator{\oSk}{\overline{Sk}}

\DeclareMathOperator{\Sec}{Sec}

\DeclareMathOperator{\Spec}{Spec}
\DeclareMathOperator{\Sing}{Sing}

\DeclareMathOperator{\Spf}{Spf}
\DeclareMathOperator{\Sp}{Sp}

\DeclareMathOperator{\interior}{int}

\DeclareMathOperator{\toric}{tor}
\DeclareMathOperator{\TV}{TV}
\DeclareMathOperator{\cTV}{\mathcal TV}
\newcommand{\ocT}{\widebar{\cT}}

\DeclareMathOperator{\Wall}{Wall}

\DeclareMathOperator*{\ord}{ord}

\DeclareMathOperator{\val}{val}

\DeclareMathOperator{\skinterior}{\interior}
\newcommand{\Cohh}{\mathrm{Coh}^\heartsuit}

\renewenvironment{abstract}{%
\quotation \small \textbf{\textit{\abstractname.}} 
}{\endquotation}

\begin{document}

\title{Log Calabi-Yau mirror symmetry and non-archimedean disks}
\author{Sean Keel}
\address{Sean Keel, Department of Mathematics, 1 University Station C1200, Austin, TX 78712-0257, USA}
\email{keel@math.utexas.edu}
\author{Logan White}
\address{Logan White, Department of Mathematics, 1 University Station C1200, Austin, TX 78712-0257, USA}
\email{keel@math.utexas.edu}
\author{Tony Yue Yu}
\address{Tony Yue YU, Department of Mathematics, M/C 253-37, Caltech, 1200 E.\ California Blvd., Pasadena, CA 91125, USA}
\email{yuyuetony@gmail.com}

\maketitle

\begin{abstract}
  We give a uniform construction that includes
  the mirror algebra to a smooth log Calabi-Yau variety with maximal boundary, proper
  over an affine variety, or to a one parameter maximal compact Calabi-Yau degeneration, 
  as the spectrum of a commutative associative algebra with a canonical basis, whose structure constants are counts of non-archimedean analytic disks.
  More generally, we study the enumeration of non-archimedean analytic curves with boundaries, associated to a given transverse spine in the essential skeleton of the log Calabi-Yau variety.
  The moduli spaces of such curves are infinite dimensional. In order to obtain finite counts, we impose a boundary regularity condition so that the curves can be analytically continued into tori, that are unrelated to the given log Calabi-Yau variety. We prove the properness of the resulting moduli spaces, and show that the mirror algebra is a finitely generated commutative associative algebra, giving rise to a mirror family of log Calabi-Yau varieties.  
\end{abstract}


\setcounter{tocdepth}{1} 
\tableofcontents

\section{Introduction and statements of main results} \label{sec:introduction}
Polarized (log) Calabi-Yau varieties are conjectured to have 
\emph{theta functions}, a canonical basis of sections, indexed by integer points in the essential skeleton of the mirror variety
(see \cite[Conjecture 0.6]{Gross_Mirror_symmetry_for_log_Calabi-Yau_surfaces_I_published}, \cite{LZ23}).
This suggests a construction of the mirror variety: make the vector space with this basis into an
algebra by giving a multiplication rule, and then take the spectrum of
this algebra. Giving the multiplication rule 
is equivalent to defining {\it structure constants}, see \cref{eq:sc}.
Ours will be counts (in general virtual, but in many
important cases, naive) of $k$-analytic disks. 
We will give the precise definition shortly. We think its simple natural form is our main contribution.

We begin with a formalism that will allow us to simultaneously
treat affine log CYs and maximal degenerations of compact CYs. Throughout the paper $k$ is
an algebraically closed field of characteristic zero.

\begin{assumption} \label{ass:basicsetup}
  Let $(\bbY,\bbD)$ be a formal snc pair (see \cref{def:SNCPairs})
  with $\bbY \to \obbY$ a projective morphism to
  an affine formal
scheme topologically formally of finite type (see \cref{def:SpfScheme})  
whose underlying scheme $\obbY_s$ is a point, $p$.
  We assume $f^{-1}(p)$ is geometrically connected. 
 We assume
\begin{equation} \label{eq:kpd}
  K_{\bbY} + \bbD = \bbW
\end{equation} 
  for an effective Cartier divisor $\bbW$ with support
contained in $|\bbD|$. We assume there exists a relative minimal model, see \cref{ass:SYZ}.
Let $(Y,D) \coloneqq (\bbY,\bbD)_{\eta}$ be the Berkovich generic fibre, viewing $\bbY$ as a special formal scheme
over the trivially valued ground field $k$, see \cref{sec:BTS}.  Let
$U \coloneqq Y \setminus D$.
We assume $U$ has a {\it proper exhaustion}: i.e.
$\bbY$ has rational functions $f_1,\dots,f_m$ inducing regular $k$-analytic
functions on $U$ such that 
$(|f_1|,\dots,|f_n|): U \to \bbR^n_{\geq 0}$ is proper.

Finally, we assume $U$ has {\it maximal boundary} (we recall the definition below).
\end{assumption}

Our main interest will be in the following three special cases:

\begin{example}[Mirrors to smooth affine log CYs with maximal boundary] \label{ex:affine} Let $\bbU \subset \bbY$ be an snc compactification of smooth
  affine log CY variety over $k$, 
  and $\bbD \coloneqq \bbY \setminus U$. We assume $\bbD^{\ess} \subset \bbD$ (those irreducible components where the volume form
  of $\bbU$ has a (necessarily simple) pole). We assume $\bbY$ is projective (and take $\obbY = \obbY_s$ to
  be a point. In particular $\bbY$ is a smooth projective variety, in this case the formal schemes are ordinary schemes).
  In this case $U$ contains no complete curves, 
  our structure constants are naive counts, and our main result gives
  an (algebraic)  mirror family of affine log
  CYs over $\TV(\Nef(\bbY))$ (the affine $\bbT_{\Pic(\bbY)} \coloneqq \Pic(\bbY) \otimes \bbG_m$
  toric variety for the nef cone), each fibre endowed with a canonical basis of regular functions.
  See \cref{thm:as} and \cref{thm:lcsings}.
\end{example}

\begin{example}[Mirrors to maximal compact CY degenerations]
  \label{ex:CYD} $\obbY = \Sp(k[[t]])$, and $p: \bbY' \to \obbY$ is a
  (formal) maximal projective Calabi-Yau degeneration. We take $\bbY \to \bbY'$ a log resolution
  (of $\obbY$ together with its central fibre) and $\bbD \subset \bbY$ the reduced central fibre. Our main result
  in this case will be a (formal) mirror family of polarized compact Calabi-Yaus over $\Sp(\hR_{\bbY_s})$, where $\hR_{\bbY_s}$
  is the formal completion at the maximal monomial ideal (generated by all effective curve classes),
  each fibre endowed with
  a canonical basis of the ring of sections. See \cref{thm:CYD} and \cref{thm:lcsingsccy}.  
  \end{example}

  \begin{example}[Universal family for the pipe dream] \label{ex:pipe} We begin with a formal family $f: \obbY \to \Sp(k[[s]])$
    with central fibre $\obbY_s = \bbV_{\Sigma}$ a {\it vertex}, the Spectrum of the Stanley-Reisner ring as in
    \cref{sec:vertex}. 
    
    We assume the singular locus of $f$ is contained in the central fibre. 
    Let $\bbY \to \obbY$ be a relative dlt model, i.e. $K_{\bbY} + \bbY_s$ is dlt and relatively nef, 
    such that the central fibre $\bbY_s$ is reduced (note there is always such a model after base change of the original
    $f$). We assume $\Pic(\bbY)$ is a finitely generated lattice, and that the dual complex of $(\bbY,\bbY_s)$ 
    has suitable combinatorics, see \cref{lem:slcdclem} (we expect this to hold in our main application,
    see \cref{rem:Fano}, it holds for the Fano case below). 

    In this case $U$ contains no complete curves, our structure constants are naive counts, and 
    our main result in this case is a 
    canonical (algebraic) family of tuples $(\cX,\cE,\Theta,\cL) \to \cTV(\Sec(\bbY/\obbY))$, with fibres
    $(X,E,\Theta,L)$ satisfying:
    \begin{enumerate}
    \item $X$ is projective, $L$ is an ample line bundle.
    \item $K_X + E$ is trivial and SLC, and normal for fibres over the structure torus $\bbT_{\Pic(\bbY_s)}$.
    \item $\Theta \in | L|$
    \item The section ring $R(X,L)$ comes with a canonical vector space basis and structure constants
      naive counts of $k$-analytic disks. 
    \end{enumerate}

    See \cref{thm:mfs}. 

    Here $\Sec(\bbY/\obbY)$ is a generalization of the Gelfand-Kapranov-Zelevinsky secondary fan,
    a complete toric fan (with support $\Pic(\bbY)_{\bbR}$), a canonical coarsening
      of $\MoriFan(\bbY/\obbY)$, introduced in \cite[\S  2]{HKY20}, and $\cTV(\Sec(\bbY/\obbY))$ is a Deligne-Mumford
      stack with coarse moduli space the toric variety $\TV(\Sec(\bbY/\obbY))$. 

    Our main application is as follows: 
    We begin with $E \subset X$ an anti-canonical snc divisor on
    a smooth projective variety such that the dual fan $\Sigma_{(X,E)}$ has dimension $\dim(X)$
    (equivalently the log CY $V \coloneqq X \setminus E$ has maximal boundary), and $L > 0$ an
    ample line bundle. Then the Gross-Siebert mirror construction (or the construction here, \cref{ex:POA}, if
    $V$ is proper over affine)
    produces a canonical formal family $f:\obbY \to \Sp(k[[s]])$ with central fibre the vertex, i.e. spectrum of
    the Stanley-Reisner ring, for 
    the dual complex of $(X,E)$. Now take $\bbY \to \obbY$ a dlt model as above (one always exists after possibly
    replacing $f$ by a base extension), and apply \cref{thm:mfs}. 
    We conjecture that $(X,E + \epsilon \Theta)$ is log canonical (equivalently, the pair is KSBA stable),
    that the original $(X,E,L)$ is a fibre of the family, and that this gives
    the universal family conjectured in \cite[1.1-1.2]{HKY20}. See \cref{conj:smooth} and \cref{conj:dmc}.
    \footnote{Keel first made the conjecture in 2012, and refers to it
    as {\it The Pipe Dream}.} 

  One important special case (where the dual complex has the required combinatorics):
  Let $(\bbX,\bbD)$ be a pair of an anti-canonical divisor on a
    smooth projective Fano variety, with $\bbD$ containing a zero stratum (equivalently, the dual complex has
    the maximal possible dimension, $\dim(\bbD)$). We take $\bbY \coloneqq K_{\bbX} \to \oK \eqqcolon \obbY$ the
    contraction to a point of the zero section $\bbX \subset K_{\bbX}$ in the total space of its
    canonical bundle. In this case the mirror is a complete family of (in general singular) Fanos,
    see \cref{rem:roc}. 
  \end{example} 

  We also have an analog of \cref{ex:affine} for $\bbU$ proper over affine:

  \begin{example}[Mirrors to smooth log CYs with maximal boundary proper over affine varieties] \label{ex:POA}
    Notation and assumptions just as in \cref{ex:affine} except we only assume $\bbU$ is proper over affine.
    In this case our structure constants are (in general) virtual counts, and our mirror family is formal. 
    One application is a mirror to affine log CYs with canonical singularities and maximal boundary, see
    \cref{sec:cansing}. This extension is important: \cite{GHKIv2} conjectures that smooth affine log CYs with
    maximal boundary have a canonical basis of theta functions. One scheme for proving this is to run the mirror
    machine twice and prove the original $\bbU$ is a fibre of the double mirror family, see \cref{conj:dm}.
    But in general the
    fibres of the mirror family are singular, so the basic construction here does not apply.
    But conjecturally they are at worst canonical; if so we can use \cref{sec:cansing}.  
    \end{example} 

Next we describe the mirror algebra:
\cref{eq:kpd} endows $U$ with a volume form, $\omega$, see \cref{sec:svf}.  
Let $\Sigma$ be the {\it essential dual fan} to $(\bbY,\bbD)$ (the abstract cone complex encoding
the incidence relation among the {\it essential} boundary divisors: the components of $D$ along
which $\omega$ has a (necessarily simple) pole, see \cref{def:DualFan}.
We write $\Sk(U) \coloneqq |\Sigma|$. There is a natural subcomplex $\partial \Sigma$, corresponding
to non-algebraic strata of $\bbD$ (it is empty in cases \cref{ex:POA} or \cref{ex:CYD}), see
\cref{def:DualFan}. 
$|\Sigma| \setminus |\partial \Sigma|$ is canonically identified
with the Kontsevich-Soibelman-Temkin essential skeleton of $(U,\omega)$.
See \cref{sec:btdc} for basics
on the topology of $|\Sigma|$. The Maximal Boundary assumption in \cref{ass:basicsetup} means $|\Sigma|$
has the maximal possible dimension, $\dim U$. 

\begin{notation} \label{not:coeffs}
Denote $R_{\bbY_s}\coloneqq Q[\NE(\bbY_s/\obbY_s,\bbZ)]$ with basis $z^\beta$ for
effective curve classes $\beta\in\NE(\bbY_s/\obbY_s,\bbZ)$ ($\bbY_s \subset \bbY$ is the reduction
of the underlying scheme, which by assumption is projective over the affine $\obbY_s$).
Here we take $Q$ to be the rational numbers $\bbQ$ unless $U$ contains no complete rational
curves, in which case we take $Q$ to be the integers $\bbZ$ (in general our counts of disks will be
virtual, and so rational numbers, but in the absence of complete rational curves the counts are naive,
and so integers). 
\end{notation}

The mirror algebra \begin{equation} \label{eq:ma}
A_{(\bbY,\bbD)} \coloneqq \bigoplus_{P \in \Sk(U,\bbZ)} R_{\bbY_s} \cdot \theta_P,
\end{equation}
is defined, as module, to be
the free $R_{\bbY_s}$-module with basis $\Sk(U)(\bbZ)$, where we indicate by $\theta_P$ the basis element
associated to $P \in \Sk(U)(\bbZ)$. Our goal is to give $A$ an 
$R_{\bbY_s}$ algebra structure. Note this is equivalent to giving 
{\it structure constants}: integers $\chi(P_1,\dots,P_n,Q,\beta)$
for $P_1, \dots, P_n, Q \in \Sk(U,\bbZ)$ and
$\beta \in \NE(Y,\bbZ)$, then
\begin{equation} \label{eq:sc}
\theta_{P_1} \cdot \theta_{P_2} \dots \theta_{P_n} \coloneqq \sum_{Q,\beta} \chi(P_1,\dots,P_n,Q,\beta)
z^{\beta} \cdot \theta_Q.
\end{equation}

Our structure constants are counts of $k$-analytic disks, here is the definition:

See \cref{sec:geomot}
for elementary geometric motivation.

\begin{construction} [See \cref{fig:structure_constants}]
  \label{const:structure_constants}
  
Pick $\cG \subset \Sk(U)(\bbR)$ a full dimensional closed convex polyhedral
subset in the interior of a maximal cone, $\sigma \in \Sigma^d$, $d \coloneqq \dim Y$. 
Let $G \coloneqq \tau^{-1}(\cG) \subset U$,
where $\tau: U \to \Sk(U)(\bbR)$ is the Berkovich retraction
determined by $(\bbY,\bbD^{\ess})$ (see \cref{sec:lSYZ}).
Let $\tbbY \to \bbY$ be a toric blowup
(a composition of blowups of closed strata of $\bbD$) on which each $P_i$
has divisorial center, $\bbD_i$. 

We let $M(U,P_1,\dots,P_n,Q,\beta)$ be the moduli space of pointed closed Berkovich
analytic disks $f:(B,p_1,\dots,p_n,q) \to \tY$ such that
\begin{enumerate} 

\item $f^{-1}(D) = \sum m_i p_i$, with $f(p_i) \in D_i^{\circ}$ (the natural interior of $D_i$) for
  $P_i \neq 0$,
  where $P_i = m_i \cdot \oP_i \in \Sk(U)(\bbZ)$ with $\oP_i$ primitive
  (and we define
  the multiplicity of $0 \in \Sk(U)(\bbZ)$ to be zero). 
\item $[f:B \to Y] = \beta \in \NE(\bbY_s,\bbZ)$ (for the
  definition of this class see \cref{sec:dc}).
\end{enumerate}
which 
satisfy the following {\it boundary condition}, the existence
of an extension of $f$ to a rational curve as in
\cref{fig:glued_target}. Precisely: 

We consider $\oM_{0,n+2}$, with the labeled points $p_1,\dots,p_n,q,z$.
Let $\cV_M \subset M_{0,n+2}^{\trop}$ be the locus of metrized trees
such that the simple path from $q$ to $z$ meets a single point of
valence greater than $1$, this a point of valence $3$.
The point of considering $\cV_M$ is that 
$$
(C,p_1,\dots,p_n,q,z) \in V_M \coloneqq \tau^{-1}(\cV_M) \subset M_{0,n+2}^{\an}$$
is canonically decomposed as a union of two closed analytic
disks, $(B,p_1,\dots,p_n,q) \cup (E,q,z)$ (see \cref{def:VMdef}),
where
$\tau: M_{0,n+2}^{\an} \to M_{0,n+2}^{\trop}$ is the canonical retraction (which
has the modular meaning of sending the pointed analytic curve to
the metrized tree given by the convex hull of its marked points,
see \cite[8.24]{Keel_Yu_The_Frobenius}). 

Now we take $M$ the integer tangent space to $\Sk(U)$ at any
point in the interior of $\sigma$ (they are canonically identified),
and $\bbT$ the affine $\bbT_M \coloneqq M \otimes_{\bbZ} \bbG_m$-toric variety
whose fan is the single
ray spanned by $-Q \in M$. The boundary $\bbT \setminus \bbT_M$ is thus a single
irreducible divisor, $\bbD_T$. We let $T \subset T_M$ be the analytification of
$\bbT_M \subset \bbT$. 
We note there is a copy 
$G \subset T$, namely $\tau^{-1}(\cG) \subset T_M \subset T$,
for $\tau: T_M\to M_{\bbR}$ the canonical Berkovich retraction.
We let $Z \coloneqq Y \cup_{G} T$. Now in defining $M(U,P_1,\dots,P_n,Q,\beta)$ 
we require $f$ extends to $f: C \to Z$, with
$f(E) \subset T$, and $f^{-1}(D_T) = m_{Q} z$.

We show that for sufficiently general $\cG$, sufficiently close
to the ray $\bbR_{\geq 0} Q$, 
$M(U,P_1,\dots,P_n,Q,\beta)$ is a $k$-analytic space and the natural
map
$$
\Phi: M(U,P_1,\dots,P_n,Q,\beta) \to V_M \times G
$$
(taking the domain, and image of $q$) is proper, of relative virtual dimension zero, over a neighborhood of 
$\cV_M \times \cG \subset V_M \times G  $.  
We define $\mu(P_1,\dots,P_n,Q,\beta)$ to be the virtual degree, using the virtual fundamental class
of \cite{PY}, which we
show is independent of all choices (namely the choice of $\cG$, $\sigma$,
the gluing in the construction of $Z$, and the point in $\cV_M \times \cG$),
see \cref{prop:scwd}, \cref{rem:degree} and \cref{rem:formulate}.
When $U$ contains no complete rational curves (e.g. in cases \cref{ex:affine} or \cref{ex:pipe}),
$\Phi$ is \'etale, and this
is the same as the analytic degree (which is the cardinality of the fibres, after field extension),
in particular, a non-negative integer. 
\end{construction}

\begin{figure}[!ht]
  \centering \setlength{\unitlength}{0.5\textwidth}
  \begin{picture} (1,1)
  \put(0,0){\includegraphics[page=1,width=\unitlength]{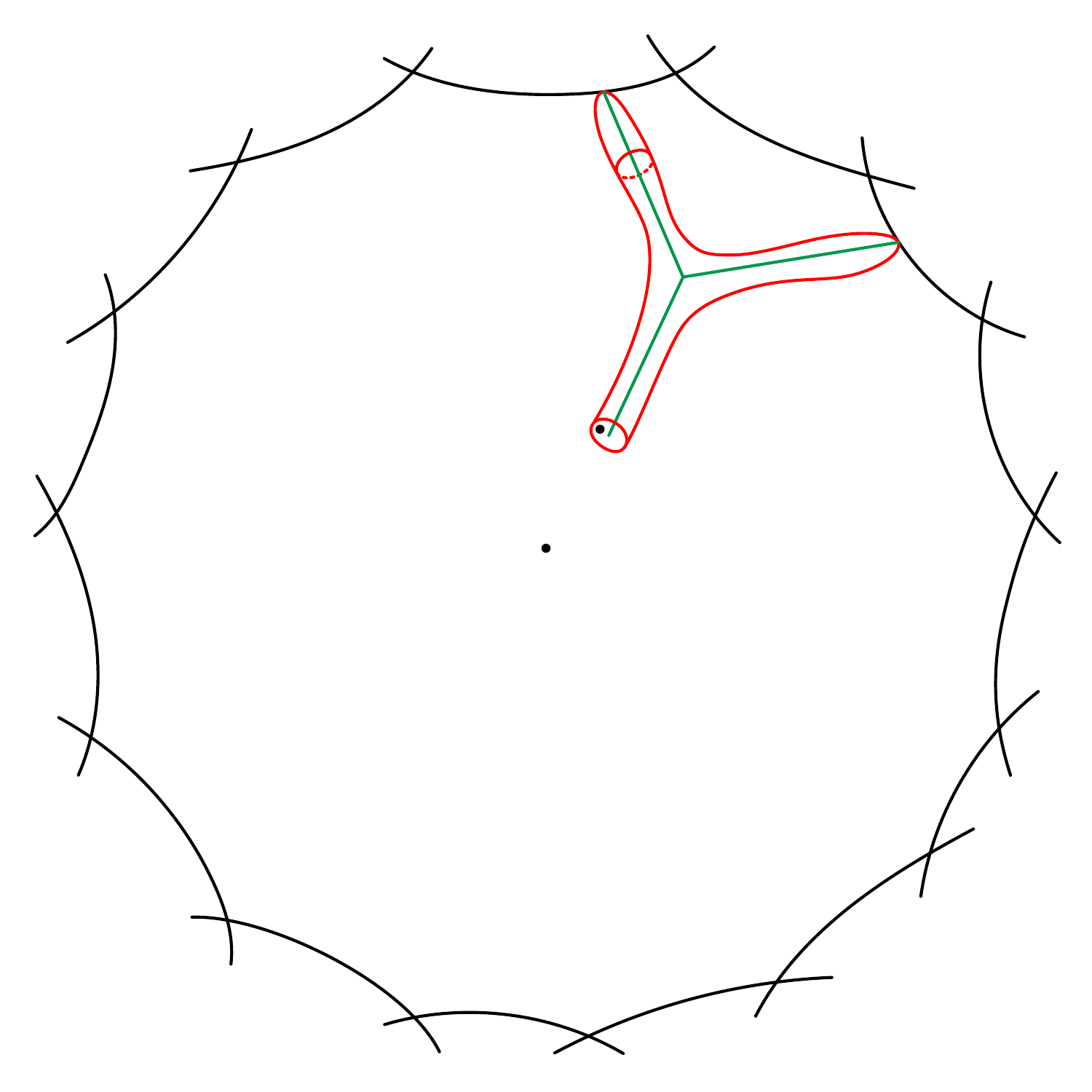}}
  \put(0.5,0.94){$D_1$}
  \put(0.84,0.77){$D_2$}
  \put(0.52,0.47){$0$}
  \put(0.50,0.60){$Q$}
  \end{picture}
  \caption{The red curve is a heuristic depiction of a non-archimedean analytic disk
    (which is in reality an $\bbR$-tree).
  The embedded green graph (accurately) depicts the associated spine.}
  \label{fig:structure_constants}
\end{figure}

\begin{figure}[!ht]
  \centering \setlength{\unitlength}{0.8\textwidth}
  \begin{picture} (1,0.707)
  \put(0,0){\includegraphics[width=\unitlength]{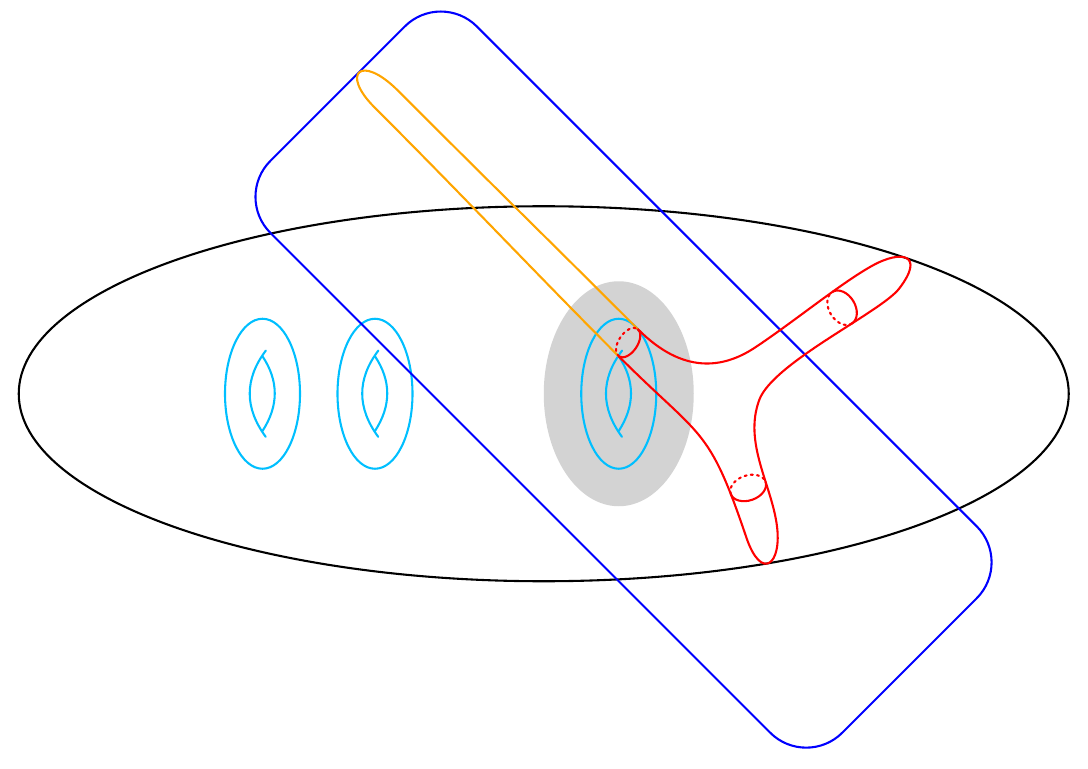}}
  \put(0.4,0.6){$E$}
  \put(0.7,0.32){$B$}
  \put(0.51,0.32){$G$}
  \put(0.75,0.07){$T^\an$}
  \put(0.9,0.36){$Y^\an$}
  \end{picture}
  \caption{The glued target space $Z = Y \cup_{G} T^\an$ and a rational curve in it.}
  \label{fig:glued_target}
\end{figure}

\begin{theorem} \label{thm:mirror_alg}
  For each $k > 0$,
  with the above structure constants, $A/m_{\bbY_s}^k$ is a finitely generated commutative and associative $R_{\bbY_s}/m_{\bbY_s}^k$-algebra.
  The induced map
  $$
  \Spec(A_{(\bbY,\bbD)}/m_{\bbY_s}^k) \to \Spec(R_{\bbY_s}/m_{\bbY_s}^k)
  $$
  is flat. The maps, and $\theta$ functions,
  for different $k$ are in the obvious way compatible, fit together to get a canonical
  flat
  $$
  \Spf(\varprojlim_k A_{(\bbY,\bbD)}/m_{\bbY_s}^k)  \to \Spf(\hR_{\bbY_s})
  $$
  (where $\hR$ is the completion with respect to the maximal monomial ideal). 
\end{theorem}

Here are special features of the example cases above:

\begin{theorem} \label{thm:as} In cases \cref{ex:affine} or \cref{ex:pipe}, with the above structure constants,
  $A_{(\bbY,\bbD)}$ is a finitely generated commutative associative $R_{\bbY_s} \coloneqq \bbZ[\NE(\bbY_s,\bbZ)]$ algebra. 
  The fibres of $\Spec(A) \to \Spec(R_{\bbY_s})$ are 
  affine Gorenstein $K$-trivial semi-log-canonical varieties
  (of dimension $\dim U$ in the \cref{ex:affine} case, of dimension $\dim{U} + 1$ in
  the \cref{ex:pipe} case). Fibers over the structure torus
  $\Spec(N_1(\bbY_s,\bbZ)) \subset \Spec(R_{\bbY_s})$ are log canonical (in particular, normal). 
\end{theorem}

\begin{remark}
  The case \cref{ex:affine} generalizes \cite[Theorem 1.2(2-3)]{Keel_Yu_The_Frobenius},
  which has the additional assumption that $\bbU$ contains an open algebraic torus.
  \cref{ex:affine} is much more general,
  e.g.\ the log Calabi-Yau varieties we consider here need not be rational varieties. We think
  the more general construction is also conceptually simpler, and closer to the intuition from symplectic geometry. The case of \cref{ex:pipe} generalizes the main construction
  of \cite{HKY20} (which treats the Fano case, as in \cref{ex:pipe} under the additional
  assumption that the Fano $\bbX$ contains a Zariski open algebraic torus). See \cref{thm:mfs}.  
\end{remark}

\begin{theorem} \label{thm:CYD} In case \cref{ex:CYD}, the mirror algebra $A$ is naturally $\bbZ_{\geq 0}$ graded (see \cref{rem:associatedgraded}). 
  Taking $\Proj$ gives an amply polarized
  family over $\Sp(\hR_{\bbY_s})$ with projective
  $K$-trivial slc fibres. If $(\bbY,\bbD)$ is almost minimal (see \cref{def:am})  then the generic fibre is (at worst) log canonical (in particular, normal). 
\end{theorem}

\begin{remark} We note that \cref{ass:basicsetup} implies that there is a birational
  model which is almost minimal, see \cref{prop:am}. \cref{ex:CYD}, and its proof, significantly
  simplify the main mirror construction of \cite{GHKSK3}, and give an extension to all dimensions,
  see \cref{rem:MCCYs}. 
 \end{remark}

We note that a mirror algebra, with the same basis, has been previously obtained
by Gross and Siebert in \cite{GSIMS}.
Their construction applies in greater generality: e.g. they do not require the analog of the exhaustion
functions at the end of \cref{ass:basicsetup}.

We expect the algebras are the same (when ours is defined), but this is not
at all obvious, as their structure constants count something different (but clearly
related, see \cref{rem:ec}). In
\cite{Johnston_Comparison} Johnston proves that the Frobenius Structure Conjecture,
(see \cite[1.1]{Keel_Yu_The_Frobenius}) implies the equality in the affine case of \cref{ex:affine}.
In particular, the equality holds in case \cref{ex:affine}
under the additional assumption that $\bbU$ contains a Zariski
open torus (by the main result of \cite{Keel_Yu_The_Frobenius}). 
Given this expectation, and
the greater generality of the Gross-Siebert construction, it is natural to wonder
why one should bother with ours. In fact, ours
has many advantages: The main advantage of our definition
of structure 
constants is that we are counting objects (punctured Berkovich disks) in
$U$, while GS count punctured log curves which can have components mapping
into the boundary, so it's clear our construction depends only on $U$ while theirs appears to
depend on the compactification. An immediate practical implication is that
certain basic convexity results, e.g.
\cref{prop:basicconvexity}, and \cref{cor:filtration} 
entirely elementary in our theory, are unknown in theirs. These convexity
statements are key in many applications.
As a simple example: because we count objects in $U$ it is 
obvious that our mirror algebra is independent of the compactification (up to changes in
the coefficient, see \cref{prop:act}): for different compactifications the objects we count are exactly the same.
E.g. because of this we are able to glue the mirror families over different compactifications
to obtain in \cref{ex:pipe} a family over a complete toric base, see \cref{thm:mfs}. 
Another example: by the basic convexity, any rational function on $\bbY$
with polar locus contained in $\bbD$ naturally induces a filtration on the mirror algebra, with associated
graded the mirror algebra for the non-vanishing locus (see \cref{rem:associatedgraded} for a precise statement).
In particular this gives a degeneration of the mirror algebra to the mirror algebra for the
non-vanishing locus (this is how the Frobenius Structure Conjecture is proved in \cite{Keel_Yu_The_Frobenius}).
The analogs of the above basic birational invariance result, or this filtration (and associated degeneration)
are unknown in the Gross-Siebert theory.
When $U$ contains no complete rational curves our structure constants are naive
counts. The mirror is thus defined over $\bbZ$, and our counts are non-negative integers
(e.g. this gives a simple geometric explanation for positivity in the Laurent Phenomenon, see
the discussion below \cite[1.19]{Keel_Yu_The_Frobenius}),
neither feature is apparent in the GS theory. We have a nice conjectural picture (partially
realized) of how basic
constructions from toric geometry extend to affine log CYs, $\bbU$, see \cref{sec:conj}. E.g. each
regular function $f$ on $\bbU$ determines by the above filtration
and the Rees construction a partial
compactification of the mirror (generalizing the partial compactification of a torus given by
rational convex polytope in its co-character lattice).
We expect the boundary is irreducible exactly when $f$ gives 
a tropical theta function, and in this case the boundary gives the corresponding point of
the skeleton of the mirror, see \cref{conj:bd}. 
Thus basic convexity is central to our understanding of the overall picture.
We expect the birational statements can be proven in the Gross-Siebert context.
The basic convexity statement looks to us more problematic, see \cref{rem:noway}.
Finally, our favorite feature: When $U$ contains no complete rational curves, we use naive counts of
very naive objects: a $k$-analytic disk of radius $r$ is just what you think it is (the locus in
$\bbA^1_{\an} \coloneqq \Spec(k[X])^{\an}$ defined by $|X| \leq r$), the description of the multiplication rule in \cref{const:structure_constants} is not some heuristic, it is the precise definition, with a very
natural geometric motivation, see \cref{sec:geomot}. This mirror construction is quite simple.
The (very powerful)
Gross-Siebert theory seems to us a good deal more complicated. 

Next we give a quick description of our applications, to indicate some interesting
connections between the mirror algebra and Mori theory.

\subsection{Application to Moduli of Log Calabi-Yau Pairs} 

The following is an excerpt from \cite{HKY20}:

\begin{conjecture} \label{conj:smooth}
Let $Q$ be a connected component of the coarse moduli space of triples $(X,E,\Theta)$ where
\begin{enumerate}[wide]
\item $X$ is a connected smooth projective complex variety,
\item $E \in \abs{-K_X}$ is an snc divisor with a zero stratum, and \item $\Theta\subset X$ is an ample divisor not containing any 0-stratum of $E$.
\end{enumerate}
Then there is a toric variety $T$, a Zariski open subset $T^\circ\subset T$ and a finite surjective map $T^\circ\to Q$.
\end{conjecture}

We have a more precise form of the conjecture above.
Note that in view of Conditions (1-2), (3) is equivalent to the condition that for sufficiently small $\epsilon > 0$, $(X,E + \epsilon \Theta)$ is a stable pair \footnote{It is sometimes called KSBA stable pair in honor of the works of Kollár–Shepherd-Barron \cite{Kollar_Threefolds_and_deformations} and Alexeev \cite{Alexeev_Moduli_spaces_MgnW}.}
(see \cite[\S 5]{Kollar_Singularities_of_the_minimal_model_program}), thus $Q$ immerses into $\oM_{\SP}$, the moduli space of stable pairs (which is a higher-dimensional generalization of the moduli space $\oM_{g,n}$ of stable pointed curves, see \cite{Kollar_Moduli_of_varieties_of_general_type}).
Let $\oQ$ denote the closure of $Q$ in $\oM_{\SP}$.

\begin{conjecture} \label{conj:main}
There is a complete toric variety $T$ with a finite surjective map $T \to \oQ$.
\end{conjecture}

We conjecture more: we have a precise description of the toric variety -- we give the fan, a generalization of the Gelfand-Kapranov-Zelevinsky secondary fan, and,
more importantly,
the universal family as Proj of an algebra with canonical basis and disk counting structure constants as in the mirror algebra. In particular (we conjecture) the data $(X,E,\Theta)$ determines a canonical basis of the homogeneous coordinate ring $R(X,\cO(\Theta))$.

We hope many readers find \cref{conj:main} surprising.
We know of no moduli or Mori theoretic reason why the moduli space in question should even be uniruled (much less a toric variety). We believe it because our mirror machine spits out what we think is the full universal
family (see \cref{rem:wct} for reasons to believe this). A strategy for proving the conjecture is given in \cite[8.3]{HKY20}: As explained in
\cref{ex:pipe} we run the mirror program twice -- first to obtain from Gross-Siebert
(or this paper if $X \setminus E$ is proper over affine) a one parameter formal family,
which we then use as input to \cref{ex:pipe}. The output is then a complete algebraic family of tuples over the desired
toric variety, see \cref{thm:mfs}. This we conjecture is the desired universal family, see \cref{conj:mfs} and
\cref{rem:roc}. 

Analogously, our results can be applied to degenerations of compact Calabi-Yaus, to give a much simplified
construction of the universal family of polarized K3 surfaces
in \cite{GHKSK3}, as well as higher dimensional generalizations.
See \cref{rem:MCCYs}. 

The two dimensional
instance of \cref{conj:smooth} is the main result of \cite{HPV24}. 

\subsection{Application to the Cox ring of a positive Looijenga pair}

In \cite{KW24} the results here are used to carry out the conjectural scheme from \cref{conj:cox} to give a basis of the Cox ring of smooth two dimensional Looijenga pair $(X,E)$ with affine interior $U \coloneqq X \setminus E$, canonical up to individual scaling. Moreover, such a basis is given for the Cox ring of the quasi-universal families of \cite[3.1]{Gross_Moduli_of_surfaces}.
We note by \cite{Hu_Mori_dream_spaces} that the Cox ring (with its natural grading by the Picard group)
{\it knows} all of the birational contractions of the variety.  For example the very rich birational geometry of these
quasi-universal families is thus encoded in counts of analytic disks.

\subsection{Application to counts of tropical curves}  \label{sec:atc} 
We note one other application: our counts naturally decompose according to the tropicalisations of the disks, and in particular our construction gives a simple definition of counts of tropical curves.
The key idea is the notion of a skeletal curve: A compact genus zero analytic curve, $B$,
(i.e.\ the complement of a union of disjoint open disks in $\bbP^1_{\an}$)
contains a canonical skeleton, a metrized tree $\Sk(B) \subset B$
(defined as the set of point which have neighborhood isomorphic to an open analytic disk, see \cite[1.2]{Baker_On_the_structure}).
A skeletal curve in $U$ is an analytic map $f\colon B \to U$ such that the restriction to $\Sk(B) \subset B$ factors through $\Sk(U) \subset U$
(the precise definition is slightly different than this,
but it has this implication, see \cref{def:skeletal}). 
The map $f$ thus has a canonical {\it tropicalisation}, which we call its spine (we reserve the term tropicalisation for something larger,
but less canonical, see \cref{def:iostrop}), the restriction $\Sp(f) \coloneqq f|_{\Sk(B)}:\Sk(B) \to \Sk(U)$.
Now there is an entirely naive way of counting functions $S\colon \Gamma \to \Sk(U)$ from a metrized tree into the skeleton of $U$: we count skeletal analytic curves, $f\colon B \to U^{\an}$, using a moduli space as in \cref{const:structure_constants}, with $\Sp(f) = \Gamma$.
For the precise definition see \cref{def:naive_counts_general}.
We stress: There is no notion of balancing here ($\Sk(U)$ has in general no canonical integer affine structure, balanced does not make intrinsic sense), nothing from the usual tropical geometry; our count is a count of analytic stable maps.
For sufficiently generic spines the counts are well behaved --
deformation invariant, and satisfy a natural gluing formula, see \cref{cor:vary_lengths} and \cref{sec:gluing}.
We note these are counts of analytic curves (possibly) with boundary, thus a new version of Gromov-Witten invariants.

\subsection{Mirrors to affine log CY with canonical singularities} \label{sec:cansing}
Let $\bbU'$ be an affine log CY with maximal boundary and at worst canonical singularities. Let
$b:\bbY \to \bbY'$ be a log resolution of a normal projective compactification $\bbU' \subset \bbY'$,
and $\bbU \coloneqq b^{-1}(\bbU')$. Write
$K_{\tbbY } + P = Z$
with effective divisors $P$ and $Z$ (with their supports having no divisor in common). That $\bbU$ is
canonical is equivalent to $|P| \subset \bbD \coloneqq (\bbY \setminus \bbU)$. Now $(\bbY,\bbD)$ is
an instance of \cref{ex:POA}, and we have the (formal) mirror algebra $A_{(\bbY,\bbD)}$.
\begin{conjecture} Notation as immediately above. The multiplication rule on the mirror algebra is polynomial
  (the finiteness in (2) of \cref{cor:bound} holds).  The restriction
  of the mirror family (together with its theta functions) to
  $\TV(\Nef(\bbY')) \subset \TV(\Nef(\bbY))$ is independent of the resolution. 
\end{conjecture}

\subsection{Geometric motivation for the definition of the multiplication rule} \label{sec:geomot}
Here, for simplicity, we discuss the case \cref{ex:affine}, and the {\it absolute mirror algebra},
$A \coloneqq A_{(\bbY,\bbD)} \otimes \bbZ[\NE(\bbY_s,\bbZ)]/m_{\bbY_S}$, which depends only on $\bbU$
(and has $\bbZ$ coefficients). 

Let's first consider the toric case $\bbU = \bbT_M \coloneqq M \otimes \bbG_m$,
for $M = \bbZ^d$ (the lattice of one parameter subgroups of this algebraic torus).
In this case the mirror algebra is just the group ring $\bbZ[M]$ and so the structure constants are $\chi(P_1,P_2,Q) = \delta_{Q,P_1+P_2}$, (Kronecker delta), addition in the group $M$.
We note $M$ is the set of integer points of the Kontsevich-Soibelman intrinsic skeleton,
$M = \Sk(\bbT_M,\bbZ)$.
For general log Calabi-Yau $\bbU$,
$\Sk(U,\bbZ)$ is not a group, $Q = P_1 + P_2$ does not make sense, but we can reformulate this geometrically in a way that does: Let $P_1,P_2, \dots, P_n \in M \setminus \{0\}$,
and let $\bbT_M \subset \bbY$ be a toric variety on which all $P_i$  have divisorial center (i.e.\ each $P_i$ spans a ray of the fan for $\bbY$, e.g.\ take $\bbY$ to have fan given by these $n$ rays).
Then it is an elementary fact that $\sum P_i =0$ if and only if there is an analytic map $f\colon (C,p_1,\dots,p_n) \to Y$ from an $n$-pointed $k$-analytic rational curve (i.e.\ $C = \bbP^1_{\an}$)
such that $f^{-1}(D) = \sum m_i p_i$, with $f(p_i) \in D_{P_i}^{\circ}$, where $D \subset Y$ is the boundary, $D_{P_i}^{\circ} \subset D_{P_i}$ is the complement of all other boundary divisors ($D_{P_i} = D_{P_i}^{\circ}$ if we take $\bbY$ with fan just the $n$ rays) and $P_i = m_i \oP_i$ with $\oP_i \in M$ reduced.

The analogous statement holds without passing to analytification, we could just use algebraic maps of $\bbP^1$ into $\bbY$.
But we will need the $k$-analytic theory for the general case.

We could use the existence of rational curves meeting the boundary with such contact as the {\it definition} in the general case (where $\Sk(U)$ has no group structure) of $\sum P_i = 0$ (this is closely related to the Frobenius structure conjecture).
But for structure constants we need (notation as above) $P + Q + (-R) = 0$,
and in general $-R$ does not make sense.
But it does make sense {\it near} $R \in \Sk(U)$, e.g.
in the integer tangent space $M \coloneqq T_{\sigma}(\Sk(U))$, for any maximal cone $R \in \sigma$ of the dual fan as in \cref{const:structure_constants}.
So now we take the toric variety $\bbT_M \subset \bbT$ with fan the single ray spanned by $-R$, and we can glue
$T \coloneqq \bbT^{\an}$ onto $Y \coloneqq \bbY^{\an}$
along the analytic domain $G \subset U$, notation from \cref{const:structure_constants} (which is an open set in the Tate topology, so this is a completely natural thing to do in the Berkovich world).
Now we will say $\chi(P,Q,R) \neq 0$ if there is a $k$-analytic rational curve $(C,p,q,s)$, decomposed into $2$ disks,
$C = B \cup E$ as in \cref{fig:glued_target},
with $f^{-1}(D + \partial T) = m_P p + m_Q q + m_R s$,
$f(p) \in D_{\oP}^{\circ}$, $f(q) \in D_{\oQ}^{\circ}$, $f(s) \in D_{-\oR}^{\circ}$,
and $f(E) \subset T$ (here $\partial T$ is the toric boundary,
the single divisor corresponding to $-R$).

In the toric case given $P,Q$, there is only one possible $R$, namely $P + Q$ and $\Sk(T_M)(\bbZ) = M$ is a group.
In general there can be finitely many such $R$ so instead $\Sk(U,\bbZ)$ becomes the basis of an algebra.

It is natural to wonder:
  \begin{question} Why above do we add the condition $f(E) \subset T$? \end{question} 
  \label{sec:why}
  
  It is in fact key. Without this the naive
  counts that are our structure constants will not be well behaved.
  See \cref{rem:rfund}. Note with this condition it is natural to view
  our counts as counts of disks, the $f: B \to Y$ (with a boundary
  condition, the existence of an extension to $C$).

  A pointed $k$-analytic curve with boundary, $B$, contains a canonical
  metrized graph $\Gamma \subset B$, the convex hull of the marked points
  and Berkovich boundary points. Consider for example
  $f:(B,p_1,p_2,q) \to Y^{\an}$ satisfying (1) in \cref{const:structure_constants}. Take $\tau: U \to \Sk(U)$
  the Berkovich retraction and the composition
  $\Sp(f): \Gamma^\circ \to \Sk(U)$, where $\Gamma^\circ \subset \Gamma$ is the complement
  of the marked points. When $U = \bbT_M$ is a torus, this is a balanced tropical
  curve; the outward derivatives at the two infinite legs (going to the $p_i$),
  $P_1, P_2 \in M$, sum to the inward derivative at $q$. The existence of
  the extension to $E$ (with contact as prescribed in \cref{const:structure_constants}, namely
  $f^{-1}(D_T) = m_Q z$) tells us this inward derivative is $Q$. 

  We can form $\Sp(f)$ in the same way in general. 
  Note that when $\Sp(f)(q)$ is close
  to $Q \in \Sk(U)$, it makes sense to ask if the inward derivative
  at $q$ is $Q$ 
  (this direction in the tangent space makes sense -- but 
  it only makes sense near the ray spanned by $Q$). The existence
  of the extension $f:E \to T^{\an}$ (with the prescribed boundary contact)
  guarantees that this holds (because the spine of
  this end disk, in the torus, is a straight ray in $M_{\bbR}$). For this
  guarantee we need the end disk to land in the torus.

  From this perspective, it is the disks $f:(B,p_1,p_2,q) \to Y$
  that we care about. The extension to $C$ with end $f:E \to T$ is
  a geometric way to guarantee the inward derivative of $\Sp(f)$ at $q$ is $-Q$.
  There is also the practical benefit that now we are counting (complete)
  rational curves, a process well understood, rather than disks (which we do
  not know how to count, the moduli spaces are, for example, infinite
  dimensional).

\smallskip
Let us remark how the proof of \cref{thm:mirror_alg} in the \cref{ex:affine} case, differs from the special case in \cite{Keel_Yu_The_Frobenius} which it generalizes. 
In \cite{Keel_Yu_The_Frobenius} we make the restrictive assumption that $\bbU$ contains a Zariski open torus
$\bbT_M \subset \bbU$.
We use this in several ways.
The most important: The structure constants, as in \cref{const:structure_constants}, count disks in $Y$ with prescribed contact with the boundary, having to do with the inputs to the structure constant (the $P_i$ in $\chi(P_1,\dots,P_n,Q,\beta)$)
and satisfying a boundary condition -- the existence of an extension to a map of $\bbP^1_{\an}$, with domain the union of the original disk, and one {\it end} disk, with one more point of contact, corresponding to $Q$ (the space of disks without some boundary conditions will be infinite dimensional).
We want this end disk to land in (the analytification of) the simple $\bbT_M$ toric variety, with a single boundary divisor corresponding to $-Q$.
When we have the torus we can glue on this toric variety along the torus, and so the resulting space is just $Y =\bbY^{\an}$ itself (possibly blownup so the valuations corresponding to the $P_i$ and $-Q$ have divisorial center).
In general we have no torus.
But we have plenty of standard analytic subsets of tori -- the analytic domains $G$ of \cref{const:structure_constants}
(see also \cref{const:Z}) and we can glue on a toric end (into which the end disk will map) along this standard domain
(which is open in a natural Grothendieck topology, the so called {\it Tate topology}), to obtain the Berkovich analytic space $Z$.
In \cite{Keel_Yu_The_Frobenius} we study stable rational curves in $Y$, which is essentially algebraic geometry (the space is the analytification of the analogous space of maps into $\bbY$).
The basic theory of stable curves in a Berkovich analytic space (e.g.\ our glued space $Z$)
is developed in \cite{Yu_Gromov_compactness}.
We recall the main statements we need in \cref{sec:deformation_theory}.
All the necessary deformation theory (for stable maps, in particular, the domain is complete) generalizes, see \cref{sec:smoothloc}.
The first thing we have to establish here is that the locus of stable maps where the {\it body} disk $B$ lands in $Y$, and the end disk $E$ (notation as in \cref{const:structure_constants})
lands in the toric end, is an analytic domain. We define our counts using the virtual fundamental class
from \cite{PY}. 
For this we need that our moduli space is of finite type and without boundary (relative to the map, $\Phi$, whose degree we count, see \cref{thm:scthm}). Note: In our moduli spaces we DO NOT allow degenerations where components of the
domain curve map into the boundary $D \subset Y$, so properness is delicate. 
Establishing this is the main technical task.
Once we have dealt with these foundational questions,
the argument from \cite{Keel_Yu_The_Frobenius} for proving the mirror algebra is associative and finitely generated is the same. 
Once we have associativity,
we want to prove that the fibres of the mirror family, over the structure torus $\bbT_{\Pic(\bbY_s)}$ are log
canonical affine log Calabi-Yau varieties, see \cref{thm:mirror_alg}.
The main issue is to show they are integral -- the rest then follows by degeneration to the {\it vertex} just as in \cite{Keel_Yu_The_Frobenius}, see the proof of \cref{thm:lcsings}.
In \cite{Keel_Yu_The_Frobenius} we had a shortcut:
Whenever we have a Zariski open subset $\bbV \subset \bbU$, say in the case \cref{ex:affine} (something
analogous holds in general), the mirror algebra for $\bbU$ degenerates to the mirror algebra for $\bbV$, see \cite[18.1]{Keel_Yu_The_Frobenius}.
When $\bbV = \bbT_M$ this is purely toric, a construction of Mumford -- the mirror algebra for fibres over the structure torus is a ring of Laurent series,
and \cref{thm:mirror_alg}, in particular the integrality, is obvious.
The case for $\bbU \supset \bbT_M$ follows by degeneration.
In general, when we have no torus open set, we have only the degeneration to the (highly reducible, but otherwise SLC Gorenstein $K$-trivial log Calabi-Yau) vertex, see \cref{sec:vertex}.
To prove integrality we make use of the strategy from \cite{Gross_Mirror_symmetry_for_log_Calabi-Yau_surfaces_I_v1}, and \cite{Gross_Theta_Functions}: We make use of a scattering diagram to show that our deformation of the vertex is (outside codimension two on the central fibre) locally isomorphic to the Mumford construction (the toric case of the mirror family) and from this easily deduce integrality.
Here as in \cite{Keel_Yu_The_Frobenius} we define the scattering diagram directly, as counts of holomorphic cylinders.
This is carried out in \cref{sec:scattering}.

\smallskip
\begin{remark} \label{rem:ec} 
There is also a direct geometric definition of the scattering diagram in \cite{GSIMS},
and of the structure constants for the mirror algebra.
We assume the scattering diagrams, and mirror algebras, are the same, though this is not obvious.
E.g.\ our structure constants are counts of (punctured) analytic (in general semi-stable, see
\cref{rem:ssd}) disks in $U$.
Gross-Siebert count punctured log curves.
Any of our disks is the Berkovich generic fibre of a formal model. 
The central fibre will have a log structure of the sort Gross-Siebert consider.
But in general, given a punctured log curve,
there is no {\it generic fibre}, no $k$-analytic curve to consider.
We note also that our scattering diagrams are based on quite different counts -- we count cylinders, which complete to curves with two points of contact with the boundary, which when $U$ contains no complete rational
curves, deform freely, while what Gross and Siebert count is a log version of disks in $U$, which complete to rational curves with a single point of contact with the boundary (which, by log Kodaira dimension, definitely do not freely deform).
The disks they count arise as {\it twig disks}
of our cylinders, see \cref{def:twigdef}, \cref{sec:scattering}, and \cref{rem:scmio}. 
\end{remark}

There is a secondary use of the torus assumption in \cite{Keel_Yu_The_Frobenius}: the open set $\bbT_M \subset \bbU$
induces an identification $\Sk(T_M) = M_{\bbR} = \Sk(U)$, which in particular gives $\Sk(U)$ a linear structure (it's identified with $\bbR^n$).
When we need a linear structure, we will use the one induced by $(\bbY,\bbD^{\ess})$, see \cref{sec:integeraffine}. This difference turns out to have little practical effect.

We note one additional change in generalizing from the special case \cite{Keel_Yu_The_Frobenius}: Our most fundamental results are properness statements, see \cref{prop:Spprop} (which is the key for counting tropical curves) and \cref{thm:scthm} (the key for defining structure constants).
The arguments for properness are different at several key points.
In algebraic geometry any map can be completed to a proper map.
This is false in the analytic world, but there is an alternative: one can often achieve topological properness by shrinking (rather than completing), e.g.
replacing an open disk by slightly smaller closed disk.
In \cite{Keel_Yu_The_Frobenius} we could complete our basic moduli space $M(U,\beta)$ (see \cref{def:MUk}),
by taking its closure in the space of all stable maps.
Instead here we consider a slightly shrunken version, which picks up a Berkovich boundary, e.g.
$M' \subset M$ in \cref{def:bms}.

Now for the reader's convenience we give a quick sketch of the argument for associativity in \cite{Keel_Yu_The_Frobenius} which we follow here with the modifications mentioned above, as this argument accounts for the logical structure of the paper.

Note the structure constants are manifestly symmetric in the $P_i$ so commutativity is obvious.
For associativity:
We consider the coefficient of one basis element $\theta_Q$ in the desired equality:
$$
\theta_{P_1} \cdot \theta_{P_2} \cdot \theta_{P_3} =
(\theta_{P_1} \cdot \theta_{P_2}) \cdot \theta_{P_3}.
$$
This is by definition a count of disks, extended to rational curves,
as in \cref{const:structure_constants}, the length of a fibre of $\Phi$ from \cref{thm:scthm}.
Associativity is proven by varying the point over which we take the fibre.
We vary the modulus of the domain curve over a big set,
$V_M \subset M_{0,4}^{\trop}$, see \cref{def:VMdef},
and the position of the basepoint over a small set, $H$, see \cref{const:data1},
near $Q \in \Sk(U)$.

We count the fibre over points in $$
\Sk(M_{0,4} \times U) = M_{0,4}^{\trop} \times \Sk(U) \subset M_{0,4}^{\an}
\times U^{\an}.
$$
$\Phi^{-1}\Sk(M_{0,4} \times U)$ consist of skeletal curves, see \cref{lem:source_of_skeletal_curve}
We take domain curve with skeleton as in \cite[Figure 15]{Keel_Yu_The_Frobenius}
and vary the modulus so the edge labeled $\lambda$ {\it stretches}.
Note the tropical domain curve, the intrinsic skeleton of the punctured (at the marked points) domain curve, determines the domain curve (something slightly weaker holds for the stable map,
the spine determines the map up to finitely many choices).
Invariance of the counts under this stretch is obtained by proving $\Phi$ is finite and \'etale, see \cref{prop:Spprop} and \cref{thm:scthm}.
The count naturally decomposes according to the tropical picture, i.e.\ the intrinsic spine $\Sp(f)$.
We cut each domain as in \cite[Figure 15]{Keel_Yu_The_Frobenius}.
Now the natural gluing formula relates the {\it uncut} count to counts of the two pieces,
which then gives the RHS.
The gluing formula follows from invariance of counts for a very simple deformation, which allows us to {\it break} the domain,
see \cref{thm:gluing_inside}
(and its very short proof, \cite[12.2]{Keel_Yu_The_Frobenius}).
Another important issue: to count we fix a basepoint, to make the argument work we need to know the count is invariant under changing the basepoint ($u$ vs $v$ in \cite[Figure 15]{Keel_Yu_The_Frobenius}).
This is again a version of deformation invariance (varying the point along the domain), for which we have a nice geometric argument, see \cite[8.23]{Keel_Yu_The_Frobenius}
and the proof of \cite[9.5]{Keel_Yu_The_Frobenius}.
We only know deformation invariance for sufficiently generic spines, what we call {\it transverse}, see \cref{def:transverse}.
For the above argument we need to know there are {\it enough} transverse spines,
i.e.\ that we can deform so that all the spines we count are transverse.
For this we use \cref{prop:transversality},
which says any spine with sufficiently generic domain and basepoint is transverse.

\bigskip \paragraph{\bfseries Related works}

Note that special cases were studied before in \cite{Keel_Yu_The_Frobenius,Gross_Canonical_bases}.
Gross and Siebert made remarkable progress on the mirror construction problem in greater generality in the setting of log geometry \cite{GSIMS}.
\Cref{conj:main} was proved in the case of del Pezzo surfaces in \cite{HKY20}, and for general surfaces in \cite{HPV24}.

\bigskip \paragraph{\bfseries Acknowledgments}
We benefited from profound, detailed technical discussions with Mark Gross, Bernd Siebert,
Paul Hacking, Johannes Nicaise, Maxim Kontsevich, Michael Temkin and d'Antoine Ducros.

S.\ Keel and L.\ White were supported by NSF grant DMS-1561632.
T.Y.\ Yu was supported by NSF grants DMS-2302095 and DMS-2245099.

\section{essential dual complex}

Assumptions as in \cref{ass:basicsetup}. We let $\Sigma \coloneqq \Sigma_{(\bbY,\bbD)}$ be the essential
dual fan, and $\Sk(U) \coloneqq | \Sigma_{(\bbY,\bbD)}|$.

\subsection{Skeleton of the volume form} \label{sec:svf}
We continue to assume \cref{ass:basicsetup}.
\begin{proposition} Let $\omega_1,\omega_2$ be sections of $\cO(K_{\bbY} + \bbD)$ with zero scheme $\bbW$. These induce
  volume forms (which we again denote by) $\omega_i$ on $U$. $|\omega_1| = |\omega_2|$, where $| \cdot |$ indicates Temkin's canonical
  semi-norm, \cite{Temkin_Metrization_of_differential_pluriforms}.
  In particular, the two forms have the same skeleton (the maximal
  locus of $|\omega|$, see \cite[\S 9]{Keel_Yu_The_Frobenius}). 
\end{proposition}
\begin{proof} By passing to the Stein factorisation we can assume $p_*(\cO_{\bbY}) = \cO_{\obbY}$.  $\omega_1 = u \cdot \omega_2$ for
  invertible $u \in \cO(\bbY) = \cO(\obbY)$, so in particular $u$ is pulled back from $\obbY$.
  Then on $U$ we have
  $\omega_1 = u \omega_2$ (where we use the same notation for the induced analytic forms and functions).  By \cref{lem:normlemma},
  $|u|$ is identically $1$. The result follows. \end{proof}

\begin{lemma} \label{lem:normlemma} Let $u$ be an invertible function on a connected formal scheme
  topologically formally of finite type as in \cref{sec:BTS}, and 
  $u^{\an}$ the induced analytic
  function on the Berkovich generic fibre. Then $|u^{\an}|$ is identically $1$.  \end{lemma}
\begin{proof} By the construction $|u| \leq 1$, and now apply the same to $u^{-1}$.
  \end{proof}

\section{Log Calabi-Yau pairs} \label{sec:log_CY}

We continue to operate under \cref{ass:basicsetup}.
Here we use some results from MMP to deduce properties of $(\bbY,\bbD)$ from that of the
minimal model. A reader trying to understand the basic idea of the paper could safely skip most of
this section,
and just assume $(\bbY,\bbD)$ is minimal (i.e. $K_{\bbY} + \bbD$ is trivial over $\obbY$). Important
ideas are the SYZ fibration, \cref{sec:SYZ}, and \cref{sec:lSYZ}, and the integer linear
structure on $\Sk(U)$, \cref{sec:integeraffine}.  


Fix $k_0$ a field of characteristic zero, equipped with the trivial valuation.
Let $k_0\subset k$ be any non-archimedean field extension.
We say that a variety (or a formal scheme, a divisor, a function, etc.) is \emph{constant over $k$} if it is isomorphic to the pullback of something over $k_0$.
We introduce this terminology (which we also used in \cite{Keel_Yu_The_Frobenius}) because it will help simplify notations while we frequently make base field extensions.

Now assume $k$ has discrete (possibly trivial) valuation.

Note that part of \cref{ass:basicsetup} is: 

\begin{assumption} \label{ass:SYZ} We assume that $(\bbY,\bbD)$ admits a minimal model, ie.
  there exists projective $\bbX \to \obbY$ and a birational contraction $c: \bbY \dasharrow \bbX$ (over $\obbY$)
  such that $K_{\bbX} + \bbD_{\bbX}$ is dlt, and trivial relative to $\obbY$, where $\bbD_{\bbX} \coloneqq c_*(\bbD)$.
\end{assumption}

\begin{remark} Conjecturally (by the termination of log flips conjecture)
  \cref{ass:SYZ} holds whenever we have the other assumptions in \cref{ass:basicsetup}.
  It holds when $\bbY \to \obbY$ is birational,
  see \cite{BCHM}. 
\end{remark}

For us the point of the assumption is it will guarantee the existence of a suitable SYZ fibration,
as in \cite{Nicaise_Xu_Yu_The_non-archimedean_SYZ_fibration}, which in particular gives us the standard affinoid domains we want for gluing on toric ends, see \cref{const:Z}.
For this we use some results from the minimal model program. 

\begin{proposition} \label{prop:term}
  Assume \cref{ass:SYZ}. Then the $K_{\bbY} + \bbD$ MMP (relative to $\obbY$) terminates.
  \end{proposition}

  The following will be a useful notion at various points in the paper:
  
\begin{definition} \label{def:abuse} Let $(\bbY_i,\bbD_i)$ $i = 1,2$ be as in \cref{ass:basicsetup}, with $\bbY_i$ projective over the
  same $\obbY$. We abuse notation and say they {\it compactify the same $\bbU$} if there is a birational map
    $(\bbY_1, \bbD_1) \dasharrow (\bbY_2,\bbD_2)$ , over $\obbY$, such that its exceptional locus is contained in
    $|\bbD_1|$, and the same holds for its inverse. This is an abuse because in general there is no $\bbU$
    (we can have $\bbY_s = \bbD_s$), though we do have the open subset $U = Y \setminus D$ of the generic fibre $Y = \bbY_{\eta}$. 
\end{definition}

\cref{prop:term} follows from the next more general result.
We learned of it, and its proof, from Christopher Hacon:

\begin{proposition} \label{prop:hacon} Let
  $(\bbY,\bbD_{\bbY})$ and $(\bbZ,\bbD_{\bbZ})$ be pairs of formal schemes with reduced divisors,
  projective over $\obbY$. 
  We assume the first pair is log canonical, the second is snc, and we have a birational map
  $f: \bbY \dasharrow \bbZ$ such that the exceptional loci of $f$ or $f^{-1}$ is contained in
  (the appropriate) $\bbD$. 
Then if $K_{\bbY} + \bbD_{\bbY}$ is semi-ample relative to $\obbY$, 
then $(\bbZ,\bbD_{\bbZ})$ has a good relative minimal model and the relative
$K_{\bbZ} + \bbD_{\bbZ}$ minimal model program terminates.
\end{proposition}
\begin{proof} Let $p:\bbX \to \bbZ$, $g\colon \bbZ \to \bbY$ be a
  resolution of the normalisation of the graph of the birational $\bbZ \dasharrow \bbY$
Then $p^*(K_{\bbZ} + {\bbD}_{\bbZ}) + E = K_\bbX + {\bbD}_\bbX = q^*(K_{\bbY} + {\bbD}_{\bbY}) + F$
where $D_{\bbX} = q^{-1}_* {\bbD}_{\bbY} + \Ex(q) = p^{-1}_*({\bbD}_Z) + \Ex(p)$, and $E,F \geq 0$ (by the definition of log canonical).
By \cite[Lemma 2.10]{HX}
$(\bbZ,{\bbD}_{\bbZ})$ has a good relative minimal model if and only if $(\bbX,{\bbD}_\bbX)$ has a good relative minimal model if and only if $({\bbY},{\bbD}_{\bbY})$ has a good relative minimal model.
For the termination of the $K_{\bbZ} + {\bbD}_{\bbZ}$ MMP use \cite[Lem 2.9]{HX}.
This completes the proof.
\end{proof}

We will use \cref{prop:term} via the next result.
We learned of it,
and its proof, from Koll\'ar:

\begin{proposition} \label{prop:dualcomplex}
  Notation as in \cref{ass:basicsetup}. We assume the $K_{\bbY} + \bbD$ MMP relative to $\obbY$ terminates. Let
  $C$ be an essential one stratum of $\bbD$ which contracts to a point in $\obbY$.  Then $C = \bbP^1$, $C$
  contains exactly two zero strata, both essential, and $C$ is disjoint from $\bbZ$ where
  $\bbZ$ is the effective Weil divisor $\bbW - (\bbD - \bbD^{\ess})$. 
\end{proposition}
\begin{remark} Note the generic fibre of $\bbZ$ is the zero locus of the volume form on $U$, the polar
  locus is the generic fibre of $\bbD^{\ess}$, and $K_{\bbY} + \bbD^{\ess} = \bbZ$. Note also, $C$ might still
meet non-essential components of $\bbD$, but only those where the volume form has neither zero nor pole. \end{remark} 
\begin{proof}
  The proof of \cite[6.1]{Nicaise_Xu_Yu_The_non-archimedean_SYZ_fibration} shows that on the minimal model
  any essential one stratum, contracted to a point in $\obbY$, is {\it toric}, i.e. $\bbP^1$, a complete
  intersection of essential divisors, containing exactly two essential zero strata. 
  As explained at the start of \cite[\S 3]{KdFX}, each step of the program is an isomorphism generically along each
  essential stratum. It follows that the same description holds for $C$. In particular, $(K + \bbD^{\ess}) \cdot C = 0$,
  which implies $C$ is disjoint from $\bbZ$. This completes the proof.
\end{proof}

The following is implied by \cite[4.5]{Nicaise_Xu_Yu_The_non-archimedean_SYZ_fibration} (and its proof): 

\begin{lemma} \label{lem:SLCminmod} Let $(\bbX,\bbD)$ be a relative minimal model for
  $(\bbY,\bbD)$ as in \cref{ass:basicsetup}. Then the pair is snc in a neighborhood of
  the union of closed one and zero strata which map to points in $\obbY$.
\end{lemma}

\subsection{Non-archimedean SYZ fibration} \label{sec:SYZ}

\begin{proposition} \label{prop:TonySYZ} Let $(\bbY,\bbD)$ be a formal snc pair. And $U \coloneqq Y \setminus D$
  for $(Y,D)$ the generic fibre as in \cref{sec:BTS}. 
Let $C$ be a one stratum of $\bbD$, isomorphic to $\bbP^1$,
containing exactly two zero strata.
Let $\rho \in \Sigma_{(\bbY,\bbD)}^{n-1}$ be the cone corresponding to $C$.
Then the Berkovich retraction $r\colon U \to |\Sigma_{(\bbY,\bbD)}|$ is smooth, in the sense of \cite[Def.
4]{Kontsevich_Affine_structures} over $\star(\rho)$:
More precisely: There is a lattice $M$,
an open embedding $i:\star(\rho) \subset M_{\bbR}$, and an analytic isomorphism $f\colon r^{-1}(\star(\rho)) \to p^{-1}(\star(\rho))$, where $p\colon \bbT_M^{\an} \to M_{\bbR}$ such that \[\begin{tikzcd}
r^{-1}(\star(\rho)) \rar{f} \dar{r} & p^{-1}(\star(\rho)) \dar{p} \\
\star(\rho) \rar{i} & \star(\rho).
\end{tikzcd}\]
commutes.

\end{proposition}

(Here $\star(\rho) \subset |\Sigma|$ is the union of cones whose closures contain $\rho$.)

\begin{proof} This proof follows by the argument in \cite[\S 6]{Nicaise_Xu_Yu_The_non-archimedean_SYZ_fibration},
  see \cref{prop:GluingEnds}.
\end{proof}

\subsection{Integer linear structure} \label{sec:integeraffine}
Assume \cref{ass:SYZ} and let $(\bbY,\bbD)$ be as in \cref{prop:dualcomplex}, whose notation we follow.
Let $\Sigma$ be the essential dual fan. 
Let $S \subset \abs{\Sigma}$ be the complement of the union of codimension two cones. By \cref{prop:goodtopology} this
is a topological manifold with boundary. 
Note that a $\Sigma$-piecewise linear integer function on $S$ is the same thing as a Weil divisor (with no linear equivalence, just a formal sum) supported on $\bbD^{\ess}$ (indeed, it is the same as an integer attached to each ray of $\Sigma^1$, since the fan is simplicial).
Following \cite[Def.\ 1.1]{Gross_Mirror_symmetry_for_log_Calabi-Yau_surfaces_I_v1},
we can now define an integer linear structure on $S$ defining such a function to be linear if and only if the intersection number of the Weil divisor with each $1$-stratum of $\bbD^{\ess}$ is zero.
This gives $S \setminus |\partial \Sigma|$ the structure of integer linear manifold without boundary. 
\begin{remark} Because of \cref{prop:oc2} we will never need to consider a linear structure on $|\partial \Sigma|$.
  \end{remark} 
This integer linear structure on $\Sk(U)$ depends 
in general on $(\bbY,\bbD)$ (it is canonical if $U$ is two dimensional, or $(\bbY,\bbD)$ is
toric). 
But the underlying piecewise integer linear structure, induced by the maximal cones of $\Sigma$, is canonical.
The same affine manifold appears in \cite[\S 1.3]{GSCWS}. 

\subsection{Models}
We consider $(\bbY,\bbD)$ as in \cref{ass:basicsetup}. 

\begin{definition-lemma} \label{def:am} 
  We say an essential one stratum is {\it almost minimal} if it meets only essential irreducible components of $\bbD$.
  In this case $C$ is a complete intersection of $\dim Y -1$ essential boundary divisors, and in addition meets exactly two other boundary divisors, each essential, and moreover $C = \bbP^1$.

We say that $(\bbY,\bbD)$ is {\it almost minimal} if every essential $1$-stratum is almost minimal.
\end{definition-lemma}
\begin{proof} This all follows from \cref{prop:dualcomplex} and \cref{prop:term}. \end{proof}
\begin{proposition} \label{prop:am} Let $(\bbY,\bbD)$ satisfy \cref{ass:SYZ}. 
  Then it has an almost minimal model, i.e. an snc almost minimal pair $(\bbX,\bbD_{\bbX})$ with $\bbX \to \obbY$
  projective, and a birational map $f:\bbY \dasharrow \bbX$ whose exceptional locus is contained
  in $\bbD$, and the exceptional locus of the inverse is contained in $\bbD_{\bbX}$.
\end{proposition}
\begin{proof}
Let $(\bbX,\bbD)$ be a dlt minimal model.
By Hironaka there is a birational map $p\colon \bbX' \to \bbX$, which is an isomorphism over a neighborhood of the snc locus of $(\bbX,\bbD)$ such that $(\bbX',\bbD')$ is snc. As $\bbX$ is $\bbQ$-factorial, the exceptional locus is the union of
the exceptional divisors. 
By the definition of dlt, there are no essential exceptional divisors, so, by the definition of minimal model,
the polar divisor (as in \cref{prop:dualcomplex})
in $\bbX'$
is the strict transform, $\tbbD$, of $\bbD$. Also, take any essential stratum of $(\bbX',\bbD')$, it is by definition
an intersection of essential divisors, and so it is not contained in the exceptional locus, and thus its image is
an essential stratum of $\bbD$ of the same dimension. Dlt implies snc in a neighborhood of any $1$ or $0$ stratum,
it follows that $p$ and $p^{-1}$ are isomorphisms in a neighborhood of any essential $1$ or $0$ stratum -- and
note by the definition of minimal, $\bbX$ has only essential strata. 
It follows that $(\bbX',\bbD')$ is almost minimal.
\end{proof}

\begin{proposition} \label{prop:modelprop} Let $(\bbY,\bbD)$ be as above,
  and $(\bbW,\bbD_{\bbW})$ an almost minimal model. 
  Then there exist $\tbbY \to Y$, $\tbbW \to \bbW$ compositions of blowups of closed
  essential strata such that the following hold:
\begin{enumerate}[wide]
\item The essential boundary divisors of $\bbD_{\tY}$ (the support of the inverse image of $\bbD$)
  are the same as the essential boundary divisors of $\bbD_{\tW}$ (in the sense of birational geometry).
\item The essential dual fans $\Sigma_{(\tbbY,D_{{\tbbY}^{\ess}})}$
and $\Sigma_{(\tbbW,D_{{\tbbW}_{\ess}})}$
(which by (1) have the same rays) are the same fans on $\Sk(U)$.
\item The birational maps $\tbbY \dasharrow \tbbW$ and $\tbbW \dasharrow \tbbY$
are isomorphisms in a neighborhood of the generic point of any essential stratum \end{enumerate}
\end{proposition}
\begin{proof} The essential dual fans are simplicial fan structures on $\Sk(U)$.
Easy birational geometry shows we can produce a common refinement by compositions of blowups of essential strata.
So we can assume $(\bbY,\bbD)$ and $(\bbW,\bbD_{\bbW})$ satisfy (1-2).
Note these are preserved by blowups of {\it the same} (under the correspondence (2)) essential strata.
It remains to show (3) holds after a sequence of such blowups.

The identifications of dual fans give a dimension preserving identification of strata. Let $p_Y \in\bbY$ be a zero stratum, $p_W \in \bbW$ the corresponding zero stratum.
Let $f_1,\dots,f_d$ be rational functions on $\bbY$ which are regular at $p_Y$, and whose zero locus cut out the boundary divisors at $p_Y$.
The polar divisors of the $f_i$ contain no essential boundary divisors through $p_Y$, and each vanishes simply along one of the $d$ boundary divisors through $p_Y$, by (1-2).
So we can make a sequence of blowups of corresponding boundary strata, through $p_Y, p_W$,
so that their polar divisors contain no essential zero strata in the inverse image of $p_Y$, and the only zero divisors through the stratum are essential.
This process will introduce new zero strata, but monomials in the $f_i$ will give local coordinates on $\tbbW$ (cutting out the boundary divisors through the zero stratum just as at the original $p_W$).
As they are monomials in the original coordinates, their zero or pole divisors (near a zero stratum over $p_Y$) are all essential.
Thus after blowing up strata (and replacing $\bbY,\bbW$ by $\tbbY,\tbbW$) we can for each essential zero strata $p$ find coordinates $f_1,\dots,f_d$ at $p_W$ as above (their zeros cut out the essential boundary divisors) with no poles at $p_w$ and simple zeros along exactly one of the essential boundary divisors at $p_w$.
It follows these give local analytic coordinates at $p_W$.
It follows $\bbY \dasharrow \bbW$ and $\bbW \dasharrow \bbY$ are isomorphisms in a
neighborhood of any essential zero stratum.
Every minimal essential stratum is zero dimensional, by \cite[Thm 2(1)]{KdFX}, 
so this completes the proof.
\end{proof}

\section{Basic topology of the essential dual complex} \label{sec:btdc}
Here we assume we have $(\bbY,\bbD)$ as in \cref{ass:basicsetup}.
We let $\Sigma \coloneqq \Sigma_{(\bbY,\bbD^{\ess})}$,
$\oSigma \coloneqq \overline{\Sigma}_{(\bb{Y}, \bb{D})}$  as in \cref{def:DualFan}. 
Let $\partial \Sigma \subset \Sigma$. $\partial \oSigma \subset \oSigma$ be the subcomplex spanned by
cones corresponding to formal (i.e. non-algebraic) strata.

\begin{definition} \label{def:goodtopology} We say that $\Sigma$, or $\oSigma$ have {\it good topology} if all the following hold:
  \begin{enumerate}
  \item $|\Sigma|$ is connected
  \item Every maximal cone is $\dim(\bbY)$ dimensional.
  \item Every codimension one cone of $\Sigma$ not part of $\partial \Sigma$ is the face of exactly two
    maximal cones, each algebraic.
    \item Every codimension one cone of $\partial \Sigma$ is a face of a unique maximal cone of $\Sigma$,
    and this cone is not part of $\partial \Sigma$.
  \end{enumerate}
\end{definition}
We note that Gross and Siebert have very similar notions, which they show hold in applications by similar use of MMP, and
reference to \cite{KX}, \cite{KdFX}. See \cite[1.1-1.3]{GSCWS}.  

\begin{proposition} \label{prop:goodtopology} Assume $(\bbY,\bbD)$ satisfies \cref{ass:basicsetup} and
  has a relative minimal model as in \cref{ass:SYZ}. 
  Then $\Sigma$ has good topology in the sense of \cref{def:goodtopology}.
\end{proposition}
\begin{proof}
  (1) holds by our assumption that $f^{-1}(p)$ is connected. 
  \cite[Prop. 11]{KdFX} (and its proof which extends the Proposition to the subboundary
  case) the essential dual complex is independent of the steps in the MMP (up to PL-isomorphism). So for
  any of the questions (1-4) we can work either with $(\bbY,\bbD)$, or a relative minimal model $(\bbY',\bbD')$.

  Now we apply \cite[4.40]{Kollar_Singularities_of_the_minimal_model_program} to $f: \bbY' \to \obbY$: Koll\'ar works
  with maps of schemes but we can reduce to this using Grothendieck Existence (since we assume $f:\bbY \to \obbY$ is
  projective). 
  
  By \cite[4.40]{Kollar_Singularities_of_the_minimal_model_program} any two minimal log canonical centers of
  $(\bbY',\bbD')$ are $\bbP^1$-linked, in particular they have the same dimension. 
  Since we assume we have at least one essential zero stratum, (2) follows.

  A formal one dimensional stratum is a closed formal embedding
  $\Spf(k[[s]]) \subset \bbY$ (as follows e.g. from the Cohen structure
  theorem). So obviously it contains a
  unique $0$ stratum, which gives (4). 
   For (3),  we work with the minimal model. 
   DLT implies snc in a neighborhood of a $1$-stratum.  Now by adjunction and the definition
   of minimal it follows that each algebraic $1$ stratum of $(\bbY',\bbD')$,
   is $\bbP^1$, containing exactly two $0$ strata, each essential.
   This completes the proof. 
   \end{proof} 

   The notation $\partial \Sigma$ is justified by:
   \begin{corollary} \label{cor:boundary} If $\Sigma$ has good topology in the sense of \cref{def:goodtopology} then
     outside the union of codimension two cones of $\Sigma$, $(|\Sigma|,|\partial \Sigma|)$ is
     a connected piecewise linear manifold with boundary.
   \end{corollary}
   \begin{proof} This is immediate from \cref{def:goodtopology}. \end{proof}

   \subsection{Lemma on the SYZ fibration}

   We assume that $(\bbY,\bbD)$ satisfies \cref{ass:basicsetup} and admits a relative minimal
   model. So \cref{prop:goodtopology} applies.

   We consider $r: U \to |\Sigma|$ the SYZ fibration, where the target is the support of
   the essential dual fan. We have $\partial \Sigma$ as in \cref{cor:boundary}. 

   \begin{proposition} \label{prop:oc2} Assumptions as immediately above.
     $r(U) \cap |\partial \Sigma| $ is contained in the union of cones of $\partial \Sigma$
     of dimension at most $\dim Y -2$.
   \end{proposition}
   \begin{proof} We consider a cone $\sigma \in \partial \Sigma$ of dimension  $\dim Y -1$, corresponding
     to a non-algebraic $1$-stratum, $C$, the intersection of $\dim Y -1$ non-algebraic irreducible
     components of $\bbD$, $\bbE_1,\dots,\bbE_{d-1}$. By \cref{prop:goodtopology}, $C$ contains a unique essential $0$-stratum,
     $Z$, necessarily algebraic.

     \begin{claim} $C \cap \bbY_s$ is a single point. \end{claim}

     Assuming the Claim, $C \cap \bbY_s = Z$. If $u \in r^{-1}(\sigma^{\circ}) \subset U$,
     then $\red(u) \in \bbE_1 \cap \dots \bbE_{d-1} \cap \bbY_s = Z$. But then $r(u) \in \Sigma_Z^{\circ}$,
     a contradiction. So it's enough to prove the Claim.
     \begin{proof}[Proof of Claim] $C_s = C \cap \bbY_s$ is zero dimensional (as $C \not \subset \bbY_s$, since
       $C$  is a formal stratum). As $C$ is connected (part of our definition of snc), this is a single
       point. 
     \end{proof}
   \end{proof}

   \section{Bounding realizable tropical curves} \label{ssec:brtc}

A central role in mirror symmetry is played by so called Maslov index zero disks, which are responsible for instanton corrections and wall-crossing structures,
see \cite{Strominger_Mirror_symmetry_is_T-duality}.
Here we give a coarse but simple way of controlling them: once we bound the class,
we find a collection of codimension one walls in $\Sk(U)$ which contain their tropicalisations.
These walls then control the bending of our version of tropical curves, see \cref{prop:tree_bound}.
A much more refined version of these walls is the scattering diagram treated in \cref{sec:scattering}.
The main result of the section is \cref{prop:tree_bound}, which allows us to make the
balancing statement \cref{def:Balcon}. 
\subsection{Ios tropicalisation} \label{sec:ios}
We fix $(\bbY,\bbD)$ as in \cref{ass:basicsetup}.

\begin{definition} \label{def:iostrop} Let $(B,p_1,\dots,p_n)$ be a semi-stable pointed curve of genus zero (i.e.\ $B$ is the complement of a union of open disks in a semi-stable proper curve of genus $0$, and the $p_i$ are quasi-smooth rational points, see \cref{rem:ssd}).
Let $\Gamma^s \subset B$ be the convex hull of the union of the $p_i$ with the Berkovich boundary $\partial B$.
Let $f\colon B \to Y$ be an analytic map such that $f^{-1}(D^{\ess}) = f^{-1}(D)$ and is supported on the union of $p_i$.
Let $B^{\circ} \coloneqq f^{-1}(U) \subset B.$ Let $\In \coloneqq B \setminus B^{\circ}$.
Note $\In$ is contained in the union of the $p_i$ ($\In$ here stands for infinite).

We consider the composition $c:B \to Y \to \oSk$,
where the second map is the Berkovich retraction determined by $(\bbY,\bbD^{\ess})$, see \cref{sec:lSYZ}.
Let $s:B \to \Gamma$
contract all edges of $B$ contracted by $c$, but not contained in $\Gamma^s$.
Note by construction, the composition $\Gamma^s \subset B \to \Gamma$, is an embedding,
with image the convex hull of the marked points and Berkovich boundary.
Following \cite[\S 2]{Keel_Yu_The_Frobenius}: We call the induced $h\colon \Gamma \to \oSk$ the {\it ios tropicalisation}
of $f\colon B \to Y^{\an}$. 
We call $h\colon \Gamma^s \to \oSk$ the {\it spine} of $f$ (and ios stands for immersed off of spine).
We call such trees mapping to $\oSk$, {\it realizable}.
We note that $h\colon \Gamma \to \oSk$ is piecewise integer affine.
We note there is a canonical $\oSigma$-piecewise affine embedding $\oSk = |\oSigma| \subset \obbR^{\bbD^{\ess}} $,
for $\Sigma \coloneqq \Sigma_{(\bbY,\bbD^{\ess})}$ the essential dual fan, see \cref{def:DualFan}. 
We sometimes apply the names to the composition, e.g.\ call $h\colon \Gamma \to \oSk \subset \obbR^{\bbD^{\ess}}$ the (ios)
tropicalisation.
We write $\Gamma^s_{\circ} \subset \Gamma^s$,
$\Gamma_{\circ} \subset \Gamma$ for the intersection with $h^{-1}(\Sk(U))$,
Note e.g.\ $\Gamma_{\circ} = \Gamma \setminus \In$.
\end{definition}

\begin{definition} \label{def:twigdef}
  Notation as immediately above.
  By a twig of $h$ we mean its restriction to the closure, $T \subset \Gamma$, of a connected component $T^{\circ} \subset \Gamma \setminus \Gamma^s$.
  We call $h\colon T \to \obbR^{\bbD^{\ess}}$ a {\it realizable twig}.
  We note that by construction the domain has a distinguished valence one vertex, $r \in T$, which we call the root, $r \coloneqq T \cap \Gamma^s \subset \Gamma$.
  Note by construction $h\colon T \to \Sk(U) \subset \bbR^{\bbD^{\ess}}$.

  The analytic curve $B \setminus \Gamma^s$ is a disjoint union of open disks.
  By the {\it twig disk} associated to the twig $h\colon T \to \Sk(U)$ we mean the connected component of $B \setminus \Gamma^s$ that contains $T^{\circ}$, which can be equivalently defined as $\pi^{-1}(T^{\circ})$ for $\pi\colon B \to \Gamma$ the canonical retraction.
  These twig disks are our version of Maslov index zero disks.
  We note each twig disk has a natural class (take the contribution to $[f\colon B \to Y]$ corresponding to $T^{\circ}$).
\end{definition}

\begin{example}
  There are no twig disks in the algebraic torus case $U = T_M$: For a twig $h\colon T \to \Sk(U)$ is an immersion, with domain a compact tree.
  But in the torus case (where $\Sk(U) = M_{\bbR}$) this is balanced (outside $r$) so constant, a contradiction.
\end{example}

\begin{definition} \label{def:NB} Let $\Gamma$ be a nodal tree, as in \cite[4.2]{Keel_Yu_The_Frobenius}
  (in particular it is metrized) 
  and $h\colon \Gamma \to \Sk(U)$ a $\Sigma_{(\bbY,\bbD^{\ess})}$ piecewise affine map.
Recall we have a canonical piecewise affine inclusion $\Sk(U) \subset \bbR^{D^{\ess}}$.
For a vertex $v \in \Gamma$ we define \[
\NB_v(\bbZ^{D^{\ess}}) \coloneqq \sum_{e \ni v} w_{(v,e)} \in \bbZ^{\bbD^{\ess}}.
\]
If $h(v) \in \Sk^{\sm}(U)$ (in terms of the $(\bbY,\bbD^{\ess})$ affine structure, see \cref{sec:integeraffine})
we define \[
\NB_v(\Sk(U)) \coloneqq \sum_{e \ni v} w_{(v,e)} \in T^{\bbZ}_{h(v)}(\Sk(U))
\]
where the target is the integer tangent space, and $w_{(v,e)}$ is the derivative at the vertex $v$ in the direction of
the incident edge $v \in e$. (See \cref{def:Balanced} for some remarks on the difference between these two notions of $\NB_v$.)
We will sometimes be restricting $h$ to subgraphs (for example the spine inside the ios tropicalisation), to remove ambiguity we will add further notational decoration.
\end{definition}

The following is easily checked:
\begin{lemma} \label{lem:NBcont} Let $h\colon \Gamma \to \Sk(U)$ be as in \cref{def:NB}.
Suppose $h$ factors as $h = g \circ \gamma$ for $\gamma\colon \Gamma \to \oGamma$,
contracting some collection of edges of $\Gamma$.
Then $$
\NB_{b}^g = \sum_{v \in \Gamma, g(v) = b} \NB_v^h $$
(for either $\Sk(U)$ or $\bbZ^{D^{\ess}}$ versions).
\end{lemma}

\begin{lemma} \label{lem:balancebound} Let $(\fB,\fP)$ be a formal strictly semi-stable model for $(B,P)$ as in \cref{def:iostrop}, with $P$ the union of marked points $p_i, i \in \In$.
and $\ff\colon (\fB,\fP) \to (\bbY,\bbD^{\ess})$ a regular map, with generic fibre $f:(B,P) \to (Y,D^{\ess})$ as in \cref{def:iostrop},
and $\oSigma_{(\fB,\fP)} \subset B$ the compactified skeleton as in \cite[1.6]{Gubler_Skeletons_and_tropicalizations}.
Let $r\colon Y^{\an} \to \oSigma_{(\bbY,\bbD^{\ess})} \subset \oR^{\bbD^{\ess}}$ be the canonical retraction.
The composition $$
B \overset{f} \to Y^{\an} \overset {r} \to \oSigma_{(\bbY,\bbD^{\ess})}
$$
factors through $r\colon B \to \oSigma_{(\fB,\fP)}$,
$r \circ f = \fh \circ r$ for a unique $\oSigma_{(\fB,\fP)}$ piecewise affine $\fh\colon \oSigma_{(\fB,\fP)} \to \oSigma_{(\bbY,\bbD^{\ess})}$.
%
There is a canonical retraction $r\colon B \setminus \{p_i\} \to \Sigma_{(\fB,P)}$, and $s = r \circ \gamma$ (notation from \cref{def:iostrop}, where the left hand side is restricted to $B \setminus \{p_i\}$) for unique $\gamma\colon \Sigma_{\fB} \to \Gamma_{\circ}$ contracting some set of edges.
The following hold:
\begin{enumerate}[wide]
\item Let $v \in \Sigma_{(\fB,P)}$ be a vertex.
Let $\fC_s^v \subset \fC_s$ be the corresponding irreducible component of the central fibre.
$\fC_s^v$ is proper if and only if $v \not \in \partial B$.
Assume this is the case.
Then for any $i \in I_{\bbD_{\ess}}$ (the set of irreducible components of $\bbD^{\ess}$)
the i-th component of $\NB_v^{\fh}(\bbZ^{D^{\ess}})$ (see \cref{def:NB}) is the degree of the pullback of $\cO(\bbD_i)$ to $\fC_s^v$.
\item Let $w \in \Gamma_{\circ}$ be a vertex, not in the image of $\partial B$.
Define $[\fC_s^w] \coloneqq \sum_{\gamma(v) = w} [f_s\colon \fC_s^v \to Y]$.
Then the i-th component of $\NB^{h}_w(\bbZ^{D^{\ess}})$
is equal to the intersection number of (the pullback of) $c_1(\cO(\bbD_i))$
with $[\fC_s^w]$.
\item $[f\colon B \to Y] = \sum_{w \in \Gamma_\circ, w \not \in r(\partial B)} [\fC_s^w].$
\end{enumerate}

\end{lemma}
\begin{proof}
The retractions $Y^{\an} \to \obbR^{D^{\ess}}$ and $B \to \obbR^P$
(whose images are the corresponding dual complexes)
are given by $\val D^{\ess}$ and $\val P$,
notation as in \cite[15.1]{Keel_Yu_The_Frobenius}.
$\val$
commutes with pullback.
This gives the existence of piecewise linear $\fh$, see \cite[\S 5]{Gubler_Skeletons_and_tropicalizations}.
The factoring $s = r \circ \gamma$ is clear from the definitions, as is (3).
(2) follows from (1) using \cref{lem:NBcont}.
For (1): Let $Z \subset \fC_s$ be the union of irreducible components which are disjoint from $\fC_s^v$.
We can replace $\fB$ by the complement of $Z$, and $B$ by its generic fibre (and $f$ by the restriction).
Now we have no marked points, and the curve maps into $U$.
Now (1) follows from \cite[15.3]{Keel_Yu_The_Frobenius} (applied with $\fF$ one of the components of $\bbD^{\ess}$).
Note the second term on the RHS of \cite[15.4]{Keel_Yu_The_Frobenius} is zero, and in the first term we have only a single proper component, $\fC_s^v$.
This completes the proof.
\end{proof}

\begin{lemma} \label{lem:twigclass}
Let $f\colon B \to Y$ be as in \cref{def:iostrop}.
Let $W_1,\dots,W_k \subset B$ be its twig disks.
Then $$
[f\colon B \to Y] - \sum_i [f:W_i \to Y]
$$
is effective.
The class of any twig disk is non-zero.
\end{lemma}
\begin{proof}
We can choose a model for $f\colon B \to Y$ which restricts to a model for each twig disk.
Then the components of the central fibre contributing to each $[f:W_i \to Y]$ are disjoint, and the difference in the statement is the contribution of the components not in the model for any of the twigs, in particular effective.
This gives the first statement.
For the second, note that for any twig, by definition, $f:T \to \Sk(U)$ is an immersion, in particular the weight at any valence one vertex is non-zero.
Now it follows from (2) of \cref{lem:balancebound} (noting that at a valence one vertex $\NB_w$ is just the weight).
\end{proof}

\begin{proposition} \label{prop:weightbound}
Fix $\beta \in \NE(\bbY_s,\bbZ)$, and consider all ios tropicalisations of all $f$ as in \cref{def:iostrop} (Note: we are not fixing the number of $p_i$) with $\beta - [f\colon B \to Y]$ effective.
The following hold.
\begin{enumerate}[wide]
\item There is a uniform bound (depending only on $\beta$ and $(\bbY,\bbD)$) on the number of twig disks of the ios tropicalisation.
\item There are only finitely many possibilities for the combinatorial type of any twig.
More precisely, there are only finitely many possibilities for the topological type of the domain tree and only finitely many possibilities for weights, i.e.\ the integer derivative on some domain of affineness,
and a uniform bound on the number of (maximal) domains of affineness.
\end{enumerate}
\end{proposition}
\begin{proof} (1) holds by \cref{lem:twigclass} (since the Mori cone is strictly convex).
Similarly, there are only finitely many classes $\gamma \in \NE(Y,\bbZ)$ such that $\beta - \gamma$ is effective.
So now for (2) we can assume all the $f$ in question are themselves twig disks.
Again by the strict convexity of the Mori cone, there are only finitely many possibilities for $\sum_{w \in Z} [\fC_s^w]$, (notation as in (2) of \cref{lem:balancebound}) over subsets $Z$ of $\Gamma$ (note $\Gamma = \Gamma_{\circ}$ as we are considering twigs) not containing the root.
Note that the weight at any valence one vertex is non-zero, by the definition of ios tropicalisation.
It follows from (2) of \cref{lem:balancebound} that $[\fC_s^w]$ is non-zero at any such vertex,
other than the root.
Thus we have in particular a uniform bound on the number of valence one vertices.
It follows there are only finitely many possibilities for the topological tree.

Now by (2) of \cref{lem:balancebound} there are only finitely many possibilities for the analogous $\sum_{w \in Z} \NB_w$.
Now by \cref{lem:rootlem} the same holds for all subsets (the root now allowed).
In particular there are only finitely many possibilities for weights at valence one vertices (as this is the same as $\NB$ for a valence one vertex).
Now it follows easily (using that we have only finitely many topological types of trees to consider) that there are only finitely many possibilities for weights.
This completes the proof of (2).
\end{proof}

\begin{lemma} \label{lem:rootlem} We consider a piecewise affine $f\colon \Gamma \to \bbR^n$ (affine on each edge).
If we are given the data of $NB_v$ (see the proof of \cref{prop:weightbound})
for all vertices except possibly one valence one vertex, then this uniquely determines all weights.
\end{lemma}
\begin{proof} We induct on the number of edges.
We cut in the middle of any edge other than one incident to the special valence one vertex, and take the piece not containing this vertex (with the cut point as the new special valence one) and apply induction.
\end{proof}

\begin{definition} \label{def:Balanced}
Notation as in \cref{def:iostrop}.
We say $h\colon \Gamma^s \to \obbR^{\bbD^\ess}$ is {\it balanced} at the point $x \in \Gamma^s \setminus \In$, if $\NB_x(\bbR^{D^{\ess}}) = 0$, where \[
\NB_x(\bbR^{D^{\ess}}) \coloneqq \sum_{x \in e} w_{(x,e)} \in \bbR^{D^{\ess}} .
\]

Note that $h$ factors through $|\oSigma_{(\bbY,\bbD^{\ess})} = \oSk(U) \subset \obbR^{D^{\ess}}$.
If $h(x) \in \Sk^{\sm}(U)$ we define \[
\NB_x \coloneqq \sum_{x \in e} w_{(x,e)} \in T_{h(x)}(\Sk(U)) .
\]
We say $h\colon \Gamma^s \to \oSk(U)$ is balanced at $x$ if $\NB_x = 0$.

If $h(x)$ is in the interior of a maximal cone of $\Sigma$, then $\NB_x$ maps to $\NB_x(\bbR^{D^{\ess}})$, so the two notions are the same.
But in general, as $\Sk(U) \to \bbR^{D^{\ess}}$ is not affine, this fails (and only balanced in $\Sk(U)$ is for us a useful notion).

\end{definition}

We have the following balancing result:

\begin{lemma} \label{lem:tbal} If $h\colon \Gamma^s_{\circ} \to \Sk$ is unbalanced at $x$, and
  $h(x) \in |\Sigma_{(\bbY,\bbD^{\ess})}| = \Sk(U)$ is in the complement of the union of codim $2$ cones,
  then the following hold:
  \begin{enumerate}
  \item $h(x) \not \in |\partial \Sigma|$.
    \item $h(x) \in \Sk^{\sm}(U)$ (the non-singular locus
      for the $(\bbY,\bbD^{\ess})$ linear structure, see \cref{sec:integeraffine}).
    \item $x$ is the root of a twig, and $-\NB_x(\Sk(U))$ is the sum of the outgoing weights of all twigs at $x$.
      \end{enumerate}
\end{lemma}
\begin{remark}
  Note by \cref{prop:oc2} that the assumption implies $h(x) \in \Sk^{\sm}(U)$, the non-singular
  locus for the $(\bbY,\bbD^{\ess})$ linear structure, see \cref{sec:integeraffine}. 
  \end{remark}
 \begin{proof}

  (1) holds by \cref{prop:oc2}, and this implies (2). 

   (3) follows from \cref{lem:balancebound} and the definition of the integer linear structure on $\Sk(U)$,
   \cref{sec:integeraffine}: It is enough to prove that if we instead consider the full ios tropicalisation,
   then $\NB_x(\Sk(U)) = 0$.
   By \cref{prop:goodtopology}, 
   $h(x) \in \Sigma \coloneqq \Sigma_{(\bbY,\bbD^{\ess})}$ lies in at most two maximal cones.
Each weight in the sum $\NB_x(\bbZ^{D^{\ess}})$
lies in one of the cones, write $w \coloneqq \NB_x(\bbZ^{D^{\ess}}) = a + b$
where each term on the RHS is a sum of weights in a single cone.
The statement is equivalent to showing that $a, b \in \bbR^{D^{\ess}}$
have opposite pairing with each piecewise linear function which restricts to linear on the union of the two cones (in the $\Sk(U)$ linear structure).
A piecewise linear function is the same as a Weil divisor supported on $D^{\ess}$ and it restricts to linear on the union of the cones if and only if its intersection is zero with the $1$-stratum, $X$, corresponding to the codimension one wall which is the intersection of the two maximal cones.
So it's enough to prove $w$ pairs zero with such a Weil divisor.
By \cref{lem:balancebound} this pairing is the degree of the pullback of the Weil divisor to the irreducible component $\fC_x$
(notation as in the proof of \cref{lem:balancebound}).
By assumption $f(\fC_x) \subset X$.
If it is a point, then this degree is zero for any Weil divisor.
Otherwise $f(\fC_x) = X$, and so the pairing is zero.
This completes the proof.
\end{proof}

\begin{definition} We say a pair $(v,e)$ of a vertex and incident edge on a twig is {\it outward} if, when we cut the domain at $v$, $e$ lies on the connected component with the root (at the root, all edges are inward).
\end{definition}

\begin{proposition} \label{prop:tree_bound} We fix $\beta$
and consider all ios tropicalisations of $f\colon B \to Y$ as in \cref{def:iostrop} with $\beta - [f\colon B \to Y]$ effective (as throughout the section, we are not fixing the number of points $p_i$, or any contact data with $D$).

There is a finite collection of pairs $(\fD,v)$, with $\fD$ a closed rational polyhedral cone,
of dimension $\dim Y -1$, contained in a cone of $\Sigma_{(\bbY,\bbD)}$, and $v \in T_{\fD,\bbZ}(\Sk(U))$ such that the following holds:

\begin{enumerate}[wide]
\item For each realizable twig $h\colon T \to \Sk(U)$ and each pair $(v,e)$ of a vertex and incident edge in $T$, there is a pair $(\fD,w)$ with $h(v) \in \fD$ and $w = \pm w_{(e,v)}$, where the sign is chosen to give the derivative in the outward direction.
\item For each $h\colon \Gamma^s \to \oSk$ and point $x \in \Gamma^s \setminus \In$, either $h$ is balanced at $x$, or there is a pair $(\fD,v)$ as above, with $h(x) \in \fD$, and either $h(x) \in \Sing(\Sk(U))$ or $\NB_x = -v$.
\end{enumerate}
\end{proposition}
\begin{proof} First we consider producing pairs to satisfy (1).
Valence one non-root vertices of twigs necessarily map to $\Sing(\Sk(U))$, by \cref{lem:tbal}.
By \cref{prop:weightbound} there are only finitely many possibilities for the weight on any edge of any twig, and, as we will use below,
only finitely many possibilities for sums of weights of twigs (of a single ios tropicalisation) with the same root.
Let $V$ be this finite set of possible sums of weights.
$\Sing(\Sk(U))$ is contained in $|\Sigma^{\dim Y -2}|$ (the union of codimension two cones),
see \cref{sec:integeraffine}.
So we begin by taking the finitely many $(\fD,v)$ consisting of the span,
inside a maximal cone of $\Sigma$, of a codim two cone, and one of the vectors $v \in V$ (lying in this maximal cone).
Now from these we generate further pairs as in \cite[4.16]{Keel_Yu_The_Frobenius}, by taking two pairs $(\fD_1,v_1),(\fD_2,v_2)$, with the $\fD_i$
contained in the same maximal cone, and such that $\fD_1 \cap \fD_2$ is codimension two, and then adding $(\fD,v)$ for any $v \in V$, with $\fD$ the intersection of the span of $v$ and $\fD_1 \cap \fD_2$ with a maximal cone of $\Sigma$
(we do not need to take all of $V$, we could take $v_1 + v_2$, but we are making no effort here to be efficient).
Now continue.
We stop after $k$ steps, for some $k$ larger than the maximal number of domains of affineness on any twig.
Now it is easy to see that (1-2) are satisfied, using \cref{lem:tbal}.
This completes the proof.
\end{proof}

\begin{definition} \label{def:walldef}
Given $\beta \in \NE(\bbY_s,\bbZ)$ by a set of $\beta$-walls we mean a finite set of pairs $(\fD,v)$ satisfying \cref{prop:tree_bound}.
For $k > 0$, and a fixed ample line bundle on $Y$, by a set of $(k,A)$-walls we mean a union of $\beta$ walls over all $\beta \in \NE(Y,\bbZ)$ with $\beta \cdot A < k$ (note the set of such $\beta$ is finite).
We will often leave $A$ off of the notation.
\end{definition}

By \cref{prop:tree_bound} the realizable spine $h\colon \Gamma^s \to \oSk$
satisfies a balancing condition:
\begin{definition-lemma} \label{def:Balcon}

For each vertex $v \in h^{-1}(\Sk^{\sm}(U))$ the {\it bend}
$\NB_v \coloneqq b_v \coloneqq \sum_{e \ni v} d_{v \in e}(h)
\in \bbZ^{\bbD^{\ess}}$ is either zero, or there is $(\fD,w) \in \Wall_\beta$ with $h(v) \in \fD$, and $\NB_v = w$.
\end{definition-lemma}

\subsection{Spines}

With this as motivation, we give the abstract definition of spines in $\Sk(U)$.
We do this as in \cite[\S 4]{Keel_Yu_The_Frobenius}.
To this point the discussion has applied to any $f\colon B \to Y$ (satisfying the conditions in \cref{def:iostrop}.
Now we specialize to the case of our main moduli space. 
See \cref{sec:bms}.
In particular we have the index sets for marked points $J = F \sqcup \Bnd \sqcup I$ (where here $\Bnd$ are the marked point
that map to the boundary $D \subset Y$, so play the role of $\In$ above). 
We have the space of metric trees $\NT_J^F$ defined in \cite[4.5]{Keel_Yu_The_Frobenius}.
We also assume we have a function $b\colon \Bnd \to \Sk(U,\bbZ) \setminus \{0\}$ (for its meaning
see \cref{const:data}).


\begin{definition} \label{def:spinedef} By an $\NT_J^F$ tree in $\oSk(U)$ we mean $\Gamma \in \NT_J^F$ and continuous $h\colon \Gamma \to \oSk(Y)$, such that $h^{-1}(\partial \oSk(U))$ is the set of $p_b, b \in \Bnd$.
We let $\NT_{J}^F(U)$ be the set of such trees.
We note we have a natural inclusion $\NT_J^F(U) \subset \Cont(C/\NT_J^F,\oSk)$, and so it inherits a natural compact-open topology, Hausdorff by \cite[4.9]{Keel_Yu_The_Frobenius}.

Let $\Wall_{\beta}$ be a set of $\beta$-walls.
We say $h$ is a $\Wall_{\beta}$ {\it spine} if
\begin{enumerate}[wide]
\item $h\colon \Gamma \to \oSk$ is $\Sigma_{(\bbY,\bbD)}$
piecewise affine.
\item $h$ satisfies the above balancing condition, \cref{def:Balcon},
with respect to $\Wall_\beta$,
for each $x \in \Gamma$ other than one of the marked points $p_b$, $b \in \Bnd$.
\item The weight (i.e.\ derivative) of $h$ for the (outward pointing) edge incident to $p_b$ is $b(b) \in Sk(U)(\bbZ)$.
\end{enumerate}

\begin{remark} Note that condition (3) makes sense independent of any linear structure on $\Sk(U)$, it requires only the canonical piecewise integer linear structure, because e.g.\ one can check it in the cone of $\Sigma$ that contains $h(E \setminus p_b)$, for $p_b \in E$
the incident edge.
\end{remark}

Let $\Sp_J^F(U,\beta) \subset \NT_J^F(U)$ be the subset of spines, with its subspace topology.
\end{definition}

\begin{remark} \label{rem:internal_leg_contracted}
It follows from the bending condition that for any $\Wall_{\beta}$ spine,
and each $i\in I$, if $h(v_i)\notin\Wall_\beta$, then $h$ is constant on the leg incident to $v_i$.
\end{remark}

\begin{definition} \label{def:transverse}
A spine $[\Gamma,(v_j)_{j\in J},h] \in \Sp(U,\beta)$
is called $\Wall_{\beta}$ \emph{transverse}, for a set of $\beta$-walls (see \cref{def:walldef}),
if it satisfies the following conditions:
\begin{enumerate}[wide]
\item \label{def:transverse:image} $h(\Gamma)$ is transverse to $\Wall_\beta$.
More precisely $h(\Gamma) \cap \Wall_{\beta}$ is finite and contains no points in $(d-2)$-dimensional strata of $\Wall_\beta$.
\item \label{def:transverse:bending} Every vertex of $\Gamma$ whose image lies in $\Wall_{\beta}$ is 2-valent (note in particular,
no marked point maps into $\Wall_{\beta}$ as these are valence one vertices).
\item \label{def:transverse:skeleton} For each $j \in F$,
$h(v_j) \not \in |\Sigma^{d-1}|$.
\item \label{def:transverse:singlocus}
$h(\Gamma) \cap \Sing(\Sk(U)) = \emptyset$
\end{enumerate}
Let $\SP^\tr(\Sk(U),\beta))\subset\SP(\Sk(U),\beta)$ denote the subset consisting of transverse spines.

\end{definition}

\begin{remark} \label{rem:spine_factorization}
In the context of \cref{def:transverse}, let $\Gamma^\Bnd$ be the convex hull of $(v_j)_{j\in \Bnd}$ in $\Gamma$ and $r\colon\Gamma\to\Gamma^\Bnd$ the retraction map.
Then by \cref{rem:internal_leg_contracted}, Condition (\ref{def:transverse:bending}) implies that $h\colon\Gamma\to\oSk$ factors through $r$.
\end{remark}

\cref{def:spinedef} and \cref{prop:tree_bound} immediately imply:

\begin{proposition} \label{prop:tropicalization_of_stable_map}
Let $[C,(p_j)_{j\in J},f\colon C\to Y^\an]$ be a stable map in $M^\sm(U^\an,\bP,\beta)$ as \cref{def:cMsm}, and its restriction $f\colon B \to Y^{\an}$ (which by construction has marked points for each $j \in J$).
Let $\Gamma \subset B$ be the convex hull of the $p_j, j \in J$, and consider the composition $$
h\colon \Gamma \subset B \overset{f} \to Y^{\an} \overset{r} \to \oSk(U).
$$

Then $h \in \Sp(U,\beta)$, for $\beta = [f\colon B \to Y^{\an}]$.
\end{proposition}

We will make use of the following refinement of \cite[7.2]{Keel_Yu_The_Frobenius}.
We note in the statement we are not making any log Calabi-Yau assumption.

\begin{proposition} \label{prop:dccoeff} Let $(\bbY,\bbD)$ and $\bbY \to \obbY$ be as in \cref{ass:basicsetup} (but
  we do not assume \cref{eq:kpd}). 
Let $C$ be a compact quasi-smooth strictly \kanal curve and $f\colon C\to Y$ a map with image not contained in $D$.
Let $\fC$ be a strictly semistable model of $C$ such that the map $f\colon C\to Y$ extends to $\ff\colon\fC\to\bbY$
and $(\fC,\ff\inv(\hD_{\kc}))$ is a formal strictly semistable pair.
Let $\Gamma\coloneqq \oSigma(\fC,\ff\inv(\bbD))$
be the associated extended skeleton and $h\colon\Gamma\to \Sigma_{(\bbY,\bbD)}$ the piecewise linear map induced by $\ff$ (see \cite{Gubler_Skeletons_and_tropicalizations}).
Let $E \subset \fC_s$ be a component of the central fibre, and $v \in \Gamma$ the associated vertex.
Assume $h(v)$ lies in the interior of a $\dim Y -1$ dimensional cone $\rho$.
Assume there is an edge of $\Gamma$ incident to $v$ whose image meets the interior of a $\dim Y$ dimensional cone $\tau \supset \rho$.
Then $\ff(E) = Z_{\rho}$, the one stratum of $\bbD$ corresponding to $\rho$.
Let $N$ be the primitive integer dual element to $\tau$ vanishing on $\rho$.

The coefficient of $Z_{\rho}$ in the disk cycle $f_*[f\colon C \to Y]$ (i.e.\ the multiplicity of $f_*(E)$) is
\begin{equation} \label{eq:curve_class}
\sum_{v\in e, h(v) \subset \tau} N(d_{v e} h).
\end{equation}
\end{proposition}
\begin{proof} $h(v) \in \rho^{\circ}$ implies $f(E) \subset C_{\rho}$, and the image is not a zero stratum.
The fact that an adjacent vertex maps into $\tau^{\circ}$ implies the corresponding zero stratum is in the image of $f\colon E \to C_{\rho}$, thus $f\colon E \to C_{\rho}$ is finite; the coefficient in question is the degree.
We take a local analytic equation for the irreducible component of $\bbD$
corresponding to the {\it extra} edge of $\rho \subset \tau$.
The degree of the zero divisor of its pullback to $E$ is then the degree of $f$.
We compute the degree by taking the scheme theoretic inverse image of $\bbD$ and restricting to $E$.
The multiplicities of the inverse image are encoded in $h$.
Since the pair is semi-stable, the result follows.
\end{proof}

\subsection{Rigidity and transversality of spines} \label{sec:rigidity_and_transversality}

Here we state two basic properties of spines from \cite[\S 5] {Keel_Yu_The_Frobenius}: the rigidity in \cref{prop:rigidity_spine} and the transversality in \cref{prop:transversality}.
\begin{remark} Although the context of \cite{Keel_Yu_The_Frobenius} is much less general than ours here (in \cite{Keel_Yu_The_Frobenius} we assume $U$ contains a Zariski open torus), for these transversality results, which are elementary piecewise affine geometry, the arguments of \cite{Keel_Yu_The_Frobenius} apply.
\end{remark}

We fix $J = F \sqcup \Bnd \sqcup I$ as in \cref{def:spinedef}, which in this section we will leave out of the notation.

\begin{proposition} \label{prop:rigidity_spine} ($J$ fixed as immediately above)
Let $\SP^\tr(\Sk(U),\beta)$ be as in \cref{def:transverse}.
Let $u\coloneqq v_i$ for some $i\in F\cup I$.
Let \[\Psi_u\coloneqq(\dom,\ev_u)\colon\SP^\tr(\Sk(U),\beta)
\longrightarrow\NT^F_J \times \Sk(U).\]
Let $S\in\SP^\tr(\Sk(U),\beta)$.
Then for a sufficiently small neighborhood $V_S$ of $S$ in $\SP^\tr(\Sk(U),\beta))$, the restriction of $\Psi_u$ to $V_S$ is a homeomorphism onto its image and is open.
\end{proposition}

\begin{proposition} \label{prop:transversality}
Let $\SP_J^F(\Sk(U),\beta)$ be as in \cref{def:spinedef}, with $J$ fixed as above.
Let $u\coloneqq v_i$ for some $i\in F\cup I$.
Let $N$ be a natural number, and $W\subset \bbZ^{\bbD^{\ess}}$ a finite subset.
Let $\SP(\Sk(U),\beta,N,W)\subset\SP(\Sk(U),\beta)$ be the subset consisting of spines such that the number of bending vertices (i.e.\ vertices where $h$ is not balanced)
is bounded by $N$, and all the weight vectors belong to $W$.
Let \[\Psi_u\coloneqq(\dom,\ev_u)\colon\SP(,N,W)\longrightarrow\NT^F_J\times M_\bbR.\]
Then there exists a lower dimensional finite polyhedral subset $Z\subset\NT^F_J\times \Sk(U)$ such that all the spines in $\Psi_u\inv(Z^c)$ are transverse.
\end{proposition}

\section{The interior}
Here we recall a result of Temkin characterizing the Berkovich interior.

\begin{definition-lemma} \label{def:interior}
  Let $\fY \to \fV$ be a formal model (over $k^{\circ}$) for
  a map $f: Y \to V$ of $k$-analytic spaces of
  finite type, and $s: Y \to \fY_s$
  the reduction map. 
  
Consider the union 
 of all $s^{-1}(P)$ over all Zariski closed subsets 
  $P \subset \fY_s$ proper over $\fV_s$. This is open, equal to the
  Berkovich interior,   $\interior(Y/V)$, and independent of the model.
\end{definition-lemma}
\begin{proof} Equality with the Berkovich interior implies open, and independence of
  model (each of which are easy to check directly, as well), and this equality
  is immediate from \cite[5.2]{Temkin_On_local_properties_II} (note here Temkin is
  using Berkovich's term {\it open} to mean without Berkovich boundary). 
\end{proof}

\begin{lemma} \label{lem:Kint} Notation as in \cref{def:interior}. Let
  $K \subset \skinterior(Y/V)$ be compact. Then there is a model,
  a union of irreducible components $P \subset \fY_s$, proper over $\fV_s$,
  and an open set $P \subset U \subset \fY_s$ with $s(K) \subset P$.
  \end{lemma}
\begin{proof} For each $x \in K$ there exists proper $P_x \subset
  \fY_s$ with $K \subset s^{-1}(P_x)$. This gives an open cover, so
  $K$ is covered by finitely many such. Let $P$ be their union.
  Blowing up we can assume
  each $P_x \subset \fY_s$ is an irreducible component. Blowing
  up other strata, we can assume $P$ is disjoint from all non-proper
  irreducible components. Now we take $U$ the complement of the union
  of the non-proper irreducible components. This completes the proof.
  \end{proof} 

  $\skinterior(Y/V) \subset Y$ is {\it convex} in the following sense:

  \begin{lemma} \label{lem:convexint} Notation as above.
    Let $f: C \to Y$ be a $k$-analytic map from a $k$-analytic
    space of finite type. $f(\skinterior(C/P)) \subset \skinterior(Y/P)$.
  \end{lemma}
  \begin{proof} Take a formal model for $f$.
    $r(f(x)) = f(r(x))$ and $f_s: \fC_s \to \fY_s$ takes proper sets
    to proper sets. The result follows.
  \end{proof}

\section{Basic moduli spaces} \label{sec:bms}
Here we introduce the basic moduli space that we will use for counting trees in $\Sk(U)$, and for defining structure constants for the mirror algebra.
When we have a map $f\colon X \to Y$ and a subset $S \subset Y$ we will often write $X_S \coloneqq f^{-1}(S)$ and $X_{S_1,\dots,S_n}$ (for the intersection) if we have several subsets.

The idea, in the structure constant case, is described in the introduction (above \cref{thm:mirror_alg}).
Here the main thing to show is that these conditions define a reasonable space --
roughly that the locus of stable maps where the {\it end disks}, $E_f$
(defined below) maps to the toric end $T_f$ and the {\it body}, $B$ maps into $Y$, is an analytic domain. 
The basic idea is to use formal models and then translate the conditions into analogous conditions on the central fibre, which moves the questions from analytic to algebraic geometry.
For a simple statement of this sort see e.g.\ \cref{lem:tms}, and \cref{cor:familytms} for the full details see \cref{prop:cad} and its proof.

\begin{construction}[Fixed Data] \label{const:data}
Here, following the notation of \cite[3.1,4.5]{Keel_Yu_The_Frobenius},
and \cref{def:spinedef}
we take finite $J = F \sqcup \Bnd \sqcup I$.
These will index the marked points for the tropical and analytic curves we count.
$\Bnd$ (which stands for {\it boundary}) indexes points that map to $\partial Y \eqqcolon D$
(or in the tropical case, $\partial \oSk(U)$), $I$ points that map into $U$ (or $\Sk(U)$) and $F$ will parameterize the {\it ends}.

We take two disjoint copies of $F$, that we label $F_S,F_E$ and then $N := F_S \sqcup F_E \sqcup I \sqcup \Bnd$
(this doubling of $F$ is for defining the {\it end} disks, the $E_f$ in \cref{def:VMdef} below).

We also fix a function $b\colon \Bnd \to \Sk(U,\bbZ)$ and write $P_b \coloneqq b(b)$.
These correspond (possibly on some blowup) to the boundary divisors where marked points from $\Bnd$ are sent
(more precisely, if we write $P_b = m_b \cdot \oP_b$
with $\oP_b$ primitive, then $\oP_b$ corresponds to the divisor, we use $m_b$ for contact order, see \cref{def:Mdef}).
\end{construction}
\begin{definition} \label{def:VMdef}

Let $\cV_M \subset \oM_{0,N}^{\trop}$ be the locus of trees as in the toric tail condition, \cite[\S]{Keel_Yu_The_Frobenius}
namely the simple path from $s_f$ to $e_f$
meets a unique topological vertex of valence greater than two, this point, $\os_f$,
has valence $3$.
See \cref{figure:tail1}.
Let $V_M := \tau^{-1}(\cV_M) \subset \oM_{0,N}^{\an}$.

\begin{figure}[!ht]
  \centering \setlength{\unitlength}{0.4\textwidth}
  \begin{picture} (1,1)
  \put(0,0){\includegraphics[page=1,width=\unitlength]{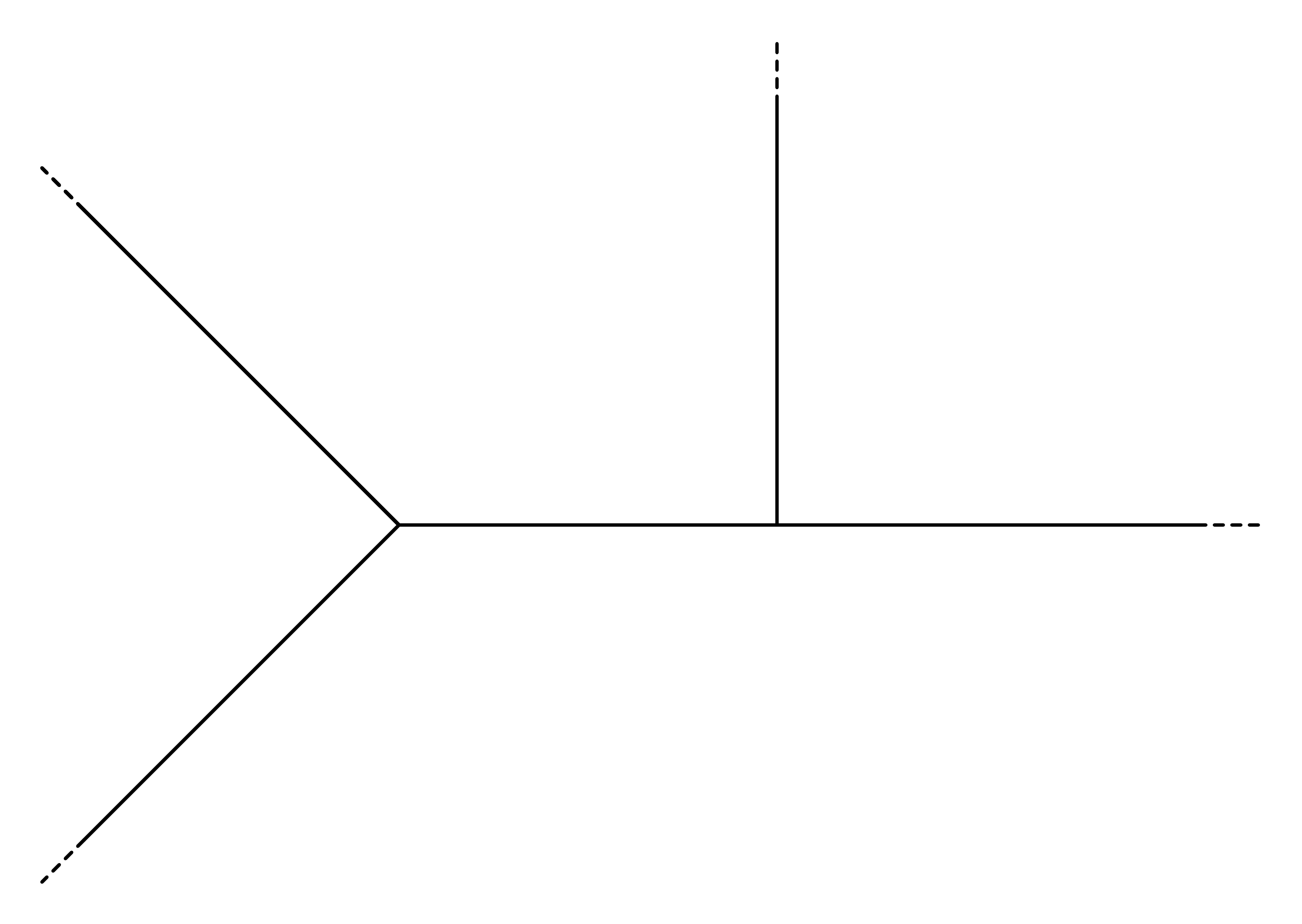}}
  \put(0.57,0.71){$s_f$}
  \put(0.57,0.25){$\os_f$}
  \put(0.92,0.25){$e_f$}
  \end{picture}
  \caption{}
  \label{figure:tail1}
\end{figure}

The point of \cref{def:VMdef} is that, as in \cite[11.1]{Keel_Yu_The_Frobenius},
$\oC \in V_M$ has a canonical decomposition into closed analytic subcurves $\oC = \oB \cup_f \oE_f$, where $\oE_f \coloneqq r^{-1}([s_f,e_f])$ and $\oB \coloneqq r^{-1}(\overline{\Gamma \setminus [s_f,e_f]})$.
Here $\Gamma \subset \oC$ is the convex hull of the marked points, and $r\colon \oC \to \Gamma$ is the canonical retraction.
$\oE_f$ is a closed analytic disk.
Let $\oS_f$ be the {\it circle} $\oS_f \coloneqq \oB \cap \oE_f$.
\end{definition}

If $C \to \oC$ is the stabilization of a semi-stable $N$-pointed curve, then this induces a decomposition,
$C = B \cup_{f \in F} E_f$ by taking inverse image.

\begin{remark} \label{rem:ssd} Note when $[\oC] \in M_{0,N} \subset \oM_{0,N}$, $\oB \subset \oC$ is an honest
  $k$-analytic disk. But $B$ might contain complete rational curves (which we will sometimes call
  {\it Twigs} or {\it Bubbles}) contracted to points by the stabilization $B \to \oB$. This is what
  we mean when we refer to a {\it semi-stable disk}. \end{remark}
Let $S_f \to \oS_f$ be the {\it semi-stable circle}: 

$S_f \coloneqq B \cap E_f$ is the inverse image of $\oS_f$.

By construction we have $s_f \in S_f$, $e_{f} \in E_f$.
Let $B^{\circ} \coloneqq B \setminus \cup_f S_f$,
$E_f^\circ \coloneqq E_f \setminus S_f$.

Now we study stable maps with domain curves decomposed as above, and impose conditions on the images
of loci such as $E_f$. As most of the novel elements of the construction have nothing to do with log CYs or
skeletons, intially we work with more general targets: 

\subsection{Stable maps with {\it end disk} conditions}  \label{sec:enddisk}

\begin{assumptions} \label{ass:basicmod}
  We assume we are given $Y$ and $T_f$, $f \in F$, finite type $k$-analytic spaces
  and pairwise disjoint compact analytic domains $G_f \subset Y$, $f \in F$, with 
  open sets $H_f \subset G_f$. 

  We assume we have for each
  $f$ a copy $G_f \subset T_f$, a compact analytic domain. 
  We assume each
  $G_f \subset \skinterior(Y)$, and $G_f \subset \skinterior(T_f)$
  (notation   as in \cref{def:interior}).  We assume each $H_f \subset G_f$ is {\it convex} 
  in
  the sense that there is a set of invertible analytic functions $E \subset \cO^*(G_f)$ such that
  $H_f = \{x \in G_f| |f(x)| < 1 \forall f \in E \} $.

  Moreover, we assume
  we have formal models $\fY$ and $\fT^f$, with quasi-projective
  central fibres, and isomorphic Zariski open subsets
  $\fG^f_s \subset \fY_s$, $\fG^f_s \subset \fT_s$, with generic
  fibres $G_f \subset Y$, $G_f \subset T_f$, and moreover line
  bundles $\fA$ on $\fY_s$, $\fA^f$ on $\fT^f_s$, whose restrictions
  to $\fY_s$ and $\fT^f_s$ are ample, and an isomorphism between
  their restrictions to $\fG^f_s$. Finally we assume we
  have Zariski closed $\fH_f \subset \fG_f$ with $\red$ inverse
  image $H_f \subset G_f$ ($\red$ the reduction from the generic fibre to the central fibre of the model). 

  \begin{remark} In our main application, see \cref{const:data1}, the above models exist by
    \cref{prop:GluingEnds}.
  \end{remark}

Now we form $Z \coloneqq Y \cup_{f\in F,G_f} T_f$. By the assumption
it comes with a formal model $\fZ = \fY \cup_{f \in F,\fG_f} \fT_f$ whose
central fbre $\fZ_s = \fY_s \cup_{f \in F,\fG^f_s} \fT^f_s$
carries a line bundle $\fA$ restricting to ample on the central fibres 
$\fY_s$ and $\fT^f_s$. Note: We are not assuming we have a line bundle on $\fZ$, just
on $\fZ_s$. 
\end{assumptions}

By \cite[8.10]{Yu_Gromov_compactness} there is a $k$-analytic stack, 
$\oM_{0,n}(Z)$, of stable maps with a formal model with central fibre
$\oM_{0,n}(\fZ_s)$. The locus $\oM_{0,n}(\fZ_s,m)$ where the pullback of $\fA$ has
degree $m$ is a union of connected components, its inverse image under
the reduction map $r$, $\oM_{0,n}(Z,m) \subset \oM_{0,n}(Z)$ is thus a union of
connected components.
\begin{remark} 
Note: Our construction of
$\oM_{0,n}(Z,m)$ depends (mildly) on the formal
model, though this is hidden in the notation.
If $\fA$ extended to the formal model $\fZ$ then this locus would
depend only on $Z$ and this line bundle. For our counts we will use
a version with no such dependence, see \cref{def:cmsm}. \end{remark}  

Let $M \subset \oM_{0,N}(Z)_{V_M}$ be the locus such that 
$f(B^{\circ}) \subset Y$, $f(E_f^{\circ}) \subset G_f$, and
$f(S_f) \subset H_f$, for all $f \in F$.

We note:
\begin{lemma} \label{lem:equivdef} $M$ is equivalently defined by the conditions
  $f(B^{\circ}) \subset Y$, $f(E_f^{\circ}) \subset G_f$ and
  $f(S_f) \subset H_f$.
\end{lemma}
\begin{proof} Either set of conditions imply $f(S_f) \subset G_f$. 
  The result
  then follows from \cref{ass:basicmod} and the maximum principle (and the {\it convexity} assumption on
  $H_f \subset G_f$) since the Berkovich
  boundary of $S_f$ is a single point, and so the pullback of each $|f|$, $f \in E$ (notation
  from \cref{ass:basicmod} ) is constant.
  \end{proof}


We begin by checking this is a reasonable locus:

\begin{proposition} \label{prop:noboundary} Notation as immediately above.
  $M \subset  \oM_{0,N}(Z,k)_{V_M}$ is the intersection of an open
  analytic domain with a closed analytic domain. $M$ is a
  $k$-analytic stack without boundary.
\end{proposition}

\begin{remark} \label{rem:rfund}
  The key statement here is the absence of boundary.  It is perhaps initially surprising, as $Z$ has boundary. 
  Our counts, e.g. for defining structure constants, will be 
  degrees of a map with domain a refinement of $M$, in general virtual, naive when $U$ in \cref{ass:basicsetup} contains
  no complete rational curves. To show they
  are well defined, e.g. invariant under small perturbation, we need
  to show the map is proper (and for naive counts, \'etale). Properness is delicate as
  our moduli spaces are not  compact (we are counting disks in $U$, which e.g. in \cref{ex:affine} is the analytification
  of an affine variety). Properness has two parts, topologically proper,
  and without boundary. \cref{prop:noboundary} will give the second.
  We will have a variant of $M$ where the first holds, see \cref{def:cM'def} and \cref{prop:VM'ft}.
  Then
  in making our counts, we prove (under further tropical
  conditions in the context of the counts)
  the two moduli spaces are equal, see \cref{prop:cad}.
  \end{remark}

  \begin{proof}
Conditions such as $f(B^{\circ}) \subset Y$, 
the condition that open map to proper, are closed. While conditions
such as $f(S_f) \subset H_f$, that proper map to open, are open.
For the further statements we use formal models, and
\cref{cor:familytms}, 
to reduce the statements to questions about the central fibre.
We pick $x \in M$, with image $y \in V_M$.  Because
$\cV_M \subset \oM_{0,N}^{\trop}$ is open, $V_M$ is not of finite type, so
does not have a model. But the statements are local on $V_M$;
we can replace $\cV_M$ by a compact full dimensional
polyhedral subset $\tau(y) \in \cP \subset \cV_M$,
and $V_M$ by $P \coloneqq \tau^{-1}(\cP)$. $P \subset (\oM_{0,N})_{V_M}$
is a closed analytic domain of finite type. We pull everything back to $P$.

By  \cite[Th 8.9]{Yu_Gromov_compactness}, we have formal families $\fP$,  $\fZ$,
$\ofM$,  $\ofC \to \ofM$, and a map $f: \ofC \to \ofM \to \fP$,  which are formal models of
$P,Z,\oM_{0,N}(Z,m)_{P}$,  $\oC \to \oM_{0,N}(Z,k)$, such that the central
fibre, $\ofM_s$, is $\oM_{0,N}(\fZ_s,m)_{\fP_s}$ with 
universal map the central fibre of $f:\ofC \to \fZ$.
By assumption $f_x(\partial(B)) \subset G \subset \skinterior(Y)$,
$f_x(\partial(E_f)) \subset G_f \subset \skinterior(T_f)$.
Thus by \cref{lem:convexint}, $f(B) \subset \skinterior(Y)$,
$f(E_f) \subset \skinterior(T_f)$. Now by \cref{lem:Kint}
(after blowing up our models) 
we can assume there are unions of projective irreducible
components $J^Y_s \subset \fY_s$, $J^{T_f}_s \subset T^f_s$
whose inverse images contain $f_x(B) \subset Y$ and
$f_x(E_f) \subset T_f$. The analogous inclusions holding
for the central fibre (i.e. replacing $B$ and $E_f$ by their
reductions,  see \cref{rem:weird}). Now by
\cref{prop:prop_locus} and \cref{cor:familytms} 
there are
Zariski open and Zariski closed subsets of
$\ofM_s$, whose $\red$ inverse images give closed and open
analytic domains in $\oM_{0,N}(Z,m)_\cP$ intersecting
in $M_{P}$, where $\red: \oM_{0,N}(Z,m) \to \ofM_s= \oM_{0,N}(\fZ_s,m)_{\fP_s}$
is the reduction map. Moreover $Q^{G,J}$ of
\cref{prop:prop_locus} gives a subset of the
central fibre, $\oM_{0,N}(\fZ_s,m)_{\fP_s}$, proper over $\fP_s$,
whose inverse image
lies in $M \subset \oM_{0,N}(Z,m)$, and contains $x$. 
Thus $x \in \skinterior(M/P) $. 
As $x \in M$ was arbitrary, this shows in particular that
$M$ has no Berkovich boundary over $V_M$. As $V_M$
has no boundary, $M$ has no boundary. 
\end{proof}

By the above, 
$M(m) \subset \oM_{0,N}(Z,m)_{V_M}$ is an analytic domain without boundary. But 
not in general  topologically compact. We will also need a version, $M'$,
which is topologically compact (but in general has a boundary).   We turn
to this next:

$M'$ requires a more involved version of \cref{def:VMdef},
where we
have in addition to the semi-stable circle, a semi-stable open annulus (see \cref{rem:rfund}):

\begin{definition} \label{def:VMpdef}
  For $\Gamma \in \cV_M$, let the metrized topological edges incident to the distinguished valence $3$ vertex
  $\os_f$ be $[0 = \os_f,s_f = \infty]$,
$[0=\os_f,e_f = \infty]$, and $[\os_f,z_f]$ (where this last edge defines $z_f$).

We add two extra points to $N$ for each $f \in F$: Let $N' \coloneqq N \sqcup_{f \in F} \{a_f,b_f\}$.
Let $\cV_{M}' \subset (\oM_{0,N'}^{\trop})_{\cV_{M}}$ be the locus of $\Gamma'$ such that $\tau(a_f) \in (z_f,\os_f)$,
$\tau(b_f) \in (\os_f,e_f)$,
where $\tau$ is the retraction of $\Gamma$ to the convex hull,
$\Gamma'$,
of the points $N \setminus \cup_f s_f$.
Let $\cV_{M,\epsilon}' \subset \cV_{M}'$ be the locus where in addition $[\tau(a_f),\tau(b_f)]$ has length less than $\epsilon$.
The point of the definition is $\oC \in V_{M}'$ contains a canonical open annulus $\oS_f \subset \oA_{f}$, for each $f \in F$, namely: Let $\Gamma_{A} := \tau^{-1}(\tau(a_f),\tau(b_f))$ and then $\oA_{f,\epsilon} \coloneqq r^{-1}(\Gamma_A)$, for $r\colon \oC \to \Gamma$ the canonical retraction (and $\Gamma \subset \oC$ the convex hull of the marked points (parameterized by $N'$)).
In $V_{M,\epsilon}'$ the annulus has radius less than $\epsilon$.
For a picture illustrating the notation see \cref{figure:tail2}.

\begin{figure}[!ht]
  \centering \setlength{\unitlength}{0.4\textwidth}
  \begin{picture} (1,1)
  \put(0,0){\includegraphics[page=1,width=\unitlength]{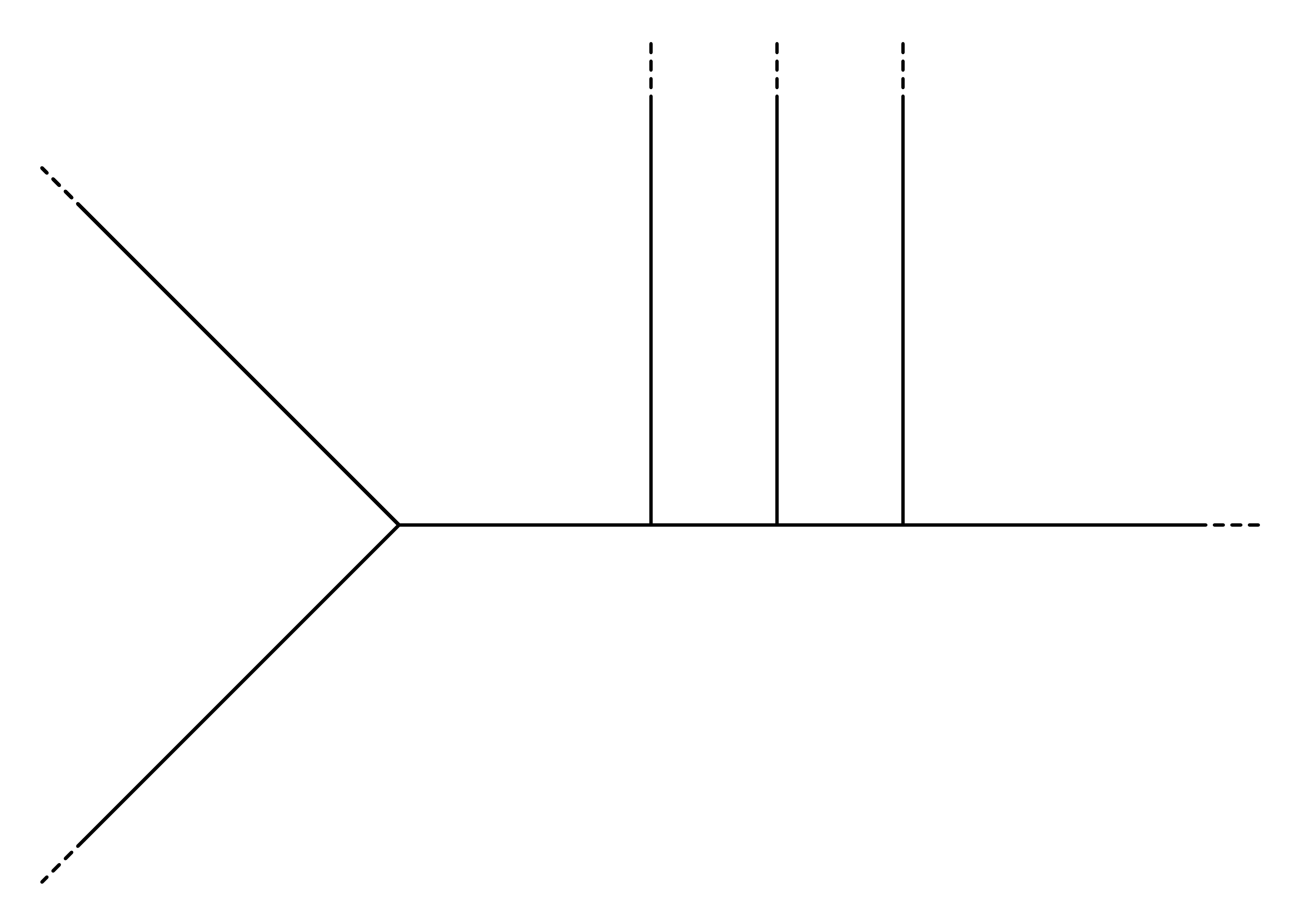}}
  \put(0.57,0.71){$s_f$}
  \put(0.47,0.71){$a_f$}
  \put(0.67,0.71){$b_f$}
  \put(0.57,0.25){$\os_f$}
  \put(0.92,0.25){$e_f$}
  \end{picture}
  \caption{}
  \label{figure:tail2}
\end{figure}

We have a canonical decomposition $\oC = \oB^\circ \cup_{f \in F} \oA_{f} \cup_f \oE^{\circ}_f$, which taking inverse images under $s:C \to \oC$ gives $C = B^\circ \cup_{f \in F} A_{f} \cup_f E^{\circ}_f$.
\end{definition}

Now
\begin{definition} \label{def:cM'def} We let $M' \subset \oM_{0,N'}(Z)_{V_M'}$ be the locus defined by
  the conditions $f(B^0) \subset Y$, $f(E_f^0) \subset T_f$, $f(A_f) \subset G_f$ and $f(s_f) \in H_f$.

  Note we could equivalently define $M' \subset \oM_{0,N'}(Z)_{V_M'\times_{f \in F} H_f}$ using just the first two
  conditions
\end{definition}

\begin{proposition}  \label{prop:VM'ft} Notation as immediately above.
  $M' \subset \oM_{0,N'}(Z')_{V_M' \times H}$ is a closed analytic domain.
$M'$ is a $k$-analytic stack  and 
  $$
  M'(m) \to V_M' \times_{f \in F} H_f
  $$
  is finite type and separated. 
\end{proposition}
\begin{proof} Both statements are local on $V_M' \times H$
so it's enough to prove the absolute versions
after replacing ${V_M' \times H}$ by a closed compact analytic domain $P$ (which we then
vary over a cover for ${V_M' \times H}$), and $M'$ by $M'_{P}$.  As in the proof of 
\cref{prop:noboundary},  
there are formal models $\fZ$, $\fW$ (the first finite type, the second locally of finite type)
of $Z$, and $\oM_{0,N'}(Z,m)$, with central fibres $\fZ_s $, $\oM_{0,N'}(\fZ_s,m) $. By \cref{ass:basicmod} there is a Zariski closed
$\bbH_f \subset \fZ_s$ with $\red^{-1}(\bbH_f) = H_f$.  By \cref{lem:anticontinuity}
there is a Zariski closed $\bbV_M' \subset \obbM_{0,N'}$ (in fact a codimension
$|F|$  closed boundary stratum) with $\red^{-1}(\bbV_{M'}) = V_{M'} \subset \oM_{0,N'}$
(for $\red: \oM_{0,N'} \to \obbM_{0,N'}$ the
reduction map associated to the constant model).
\begin{remark} Here, and just below, we use the blackboard bold font, e.g. $\bbP$, to indicate a locus in the
  central fibre of a formal model (while most of the time in the paper such loci are indicated with $_s$). \end{remark}
  Let
  $\bbP \subset \bbV_{M'} \times_{f \in F} \bbH_f$ be Zariski open (we will vary it over an open cover)
  and $P \coloneqq {\red}^{-1}(\bbP) \subset V_{M'} \times_{f \in F} H_f$.  

  Denote by $\fW_P \subset \fW$ the formal open subset with central fibre $\fW_s = \oM_{0,N'}(\fZ_s,m)_{\bbP}$, which we
  denote by $Q$. 
We have $Q^G \subset Q$, and
$Q^{G'} \subset Q$ as in \cref{sec:fcf}, \cref{sec:annulus}.  By \cref{lem:tms}, \cref{cor:familytms},
$M'_{P} = \red^{-1}(Q^G)$, $M_{P} = \red^{-1}(Q^G)$.
\begin{remark} \label{rem:convoluted} Note $M_{P}  \supset M'_{P} $ by \cref{lem:equivdef}.
  But we do not necessarily know that
  $Q^G \supset Q^{G'}$ -- the problem is that $\fW$ might not be flat, and thus $r$ might not
  be surjective. This is the reason for the slightly convoluted argument below. \end{remark} 
By \cref{prop:prop_locus_an}, $Q^{G'} \subset Q$ is Zariski open. Thus
$M'_{P}(m) \subset \fW_{\eta} = \oM_{0,N'}(Z,m)_{\cP}$ is an analytic domain
by \cref{lem:tms}, \cref{cor:familytms}. This gives the first statement of the theorem.
Now define $Z'$ as in the proof of \cref{prop:prop_locus}. By that argument we have a closed embedding
$Q^G \subset \oM_{0,N'}(Z',m)_Q$ and thus $Q^G$ is separated, and of finite type. 
  By \cite[7.6]{Yu_Gromov_compactness}, $\oM_{0,N'}(\fZ)$ is an algebraic stack locally of finite type. It follows there is a formal
  open subscheme $\fS \subset \fW$, of finite type,  whose central fibre, $S$, contains $Q^G$, and thus
  (using \cref{rem:convoluted}) 
    $r(M'_{P}) \subset Q^G \cap Q^{G'} \subset S$.  Now $Q^{G'} \cap S \subset Q$ is a Zariski open subset of
    finite type such that $\red^{-1}(Q^{G'} \cap S) = M'_{P}(m)$. Thus $M'_{P}(m)$ is finite type. All that remains is to show
    it is separated. 
    By \cref{lem:FK}, there is an admissible formal blowup $\pi:\fS' \to \fS$ from an admissible formal model
    over $k^{\circ}$, and $V \coloneqq \pi_s(\fS'_s) = \red(\fS_{\eta}) \subset \fS_s$ is Zariski closed.
    $r(M'_{\cP}(m))  = Q^{G'} \cap V $, which is Zariski open in $Q^G \cap V$, which itself is Zariski closed in (separated) $Q^G$.
    Thus $Q^{G'} \cap V$ is separated. $\pi_s: \fS'_S \to V$ is proper, so $\pi_s^{-1}(Q^{G'}) = \pi_s^{-1}(Q^{G'} \cap V)$ is
      also separated.  Now $\pi_s^{-1}(Q^{G'}) \subset \fS'_s$ is a separated open subscheme whose preimage under
      reduction is $M'_{P}$, so $M'_{P}$ is separated.  This completes the proof. 
\end{proof} 

\begin{lemma} \label{lem:FK}  Let $\fS$ be a formal analytic stack of finite presentation over $k^{\circ}$. The
  following hold:

  \begin{enumerate}
  \item The image of the reduction map $\red(\fS_{\eta}) \subset \fS_s$ is closed.
    \item     There
      is an admissible blowup $\fS' \to \fS$ such that $\fS'$ is an admissible formal stack over $k^{\circ}$.
    \end{enumerate}
  \end{lemma}
  \begin{proof}  For an admissible formal stack the reduction is surjective, and the admissible blowup is proper, so
    the second statement implies the first. For $k$-analytic spaces the second statement is \cite[2.1.10]{FK}, and as the
    construction of the blowup is canonical, the same argument works for stacks.
    \end{proof}

    \begin{construction}[More Fixed Data] \label{const:data1}

Now we return to the context of the paper, with $(\bbY,\bbD)$ formal snc pair with generic
fibre $(Y,D)$ as in \cref{ass:basicsetup}.


        In addition to \cref{const:data} we fix the following:
For each $f \in F$ we fix $\cH_f \subset \cG_f \subset \Sk^{\sm}(U)$, $\cH_f,\cG_f$
full dimensional convex polyhedra (makes sense using \cref{ass:G_f} below),
$\cG_f$ compact and $\cH_f$ open with closure in the interior of $\cG_f$.

We define $M_f \coloneqq T_{\cG_f}(\bbZ)$, the integer tangent space.
We also fix $0 \neq v_F \in M_f$.
We indicate this fixed data, together with the data from \cref{const:data}
by $\bP$ (which we will often leave out of the notation).
\end{construction}

We often write $\Sigma \coloneqq \Sigma_{(\bbY,\bbD^{\ess})}$.

\begin{assumption} \label{ass:G_f} We will always assume that each $\cG_f$ is disjoint from $\Sigma^{n-2}$, and moreover satisfies one of the following:
\begin{enumerate}[wide]
\item Either $\cG_f$ lies in the interior of a maximal cone of $\Sigma$, or \item $\cG_f$ lies in the interior of $\star(\rho)$
for some $\rho \in \Sigma^{n-1}$, corresponding to an almost minimal essential $1$ stratum, as in \cref{def:am}.
We further assume $-v_f \in \rho \subset M_f$.
\end{enumerate}

Note that in either case the assumptions imply $\cG_f \subset \Sk^{\sm}(U)$,
and moreover the SYZ fibration $\tau:U \to \Sk(U)$ is smooth over a neighborhood of $G_f$, see \cref{prop:TonySYZ},  \cref{sec:integeraffine}, and \cref{sec:lSYZ}. 
\end{assumption}

\begin{remark} The second case we will only use for structure constants.
In that case we will have $Q \in \rho$, and $v_f = -Q$.
We expect this is related to the {\it negative ends} one has for the output in the pair of pants multiplication in symplectic cohomology.
In any case, it is the same condition we use for toric ends in \cite[1.1]{Keel_Yu_The_Frobenius}.
\end{remark}

\begin{construction} \label{const:Sigma_f}

Let $R_f \coloneqq \bbR_{\geq 0} v_f$.
Let $\Sigma_f$ be a complete simplicial fan containing the ray $R_f$, and in case (2) of \cref{ass:G_f} the cone $\rho_f$ (the precise choice of $\Sigma_f$ will not matter),
and $\bbT_f$ for the corresponding $\bbT_{M_f}$-toric variety.
Let $\bbD_f \subset \bbT_f$ be the boundary divisor corresponding to $R_f$.
Choose $\cG_f \subset M_f$ so that $\cG_f,\rho_f \subset \Sk(U)$
and $\cG_f,\rho_f \subset M_f$ are affinely isomorphic.
We choose $\Sigma_f$ so that $\cG_f \subset M_f$ lies in the interior of $\star(\rho) \subset M_f$.
\end{construction}

\begin{construction}[Gluing on Ends] \label{const:Z}
  Now we construct $Z$ as in \cref{ass:basicmod} using $Y \coloneqq \bbY_{\eta}$ (see \cref{sec:BTS})
  and $T_f \coloneqq \bbT_f^{\an}$.
  We take $V \coloneqq U \cup_{f\in F,G_f} \bbT_{M_f}^{\an}$.

\begin{remark} \label{rem:glueunique}
We note the gluing is unique up to an automorphism of $G_f$
commuting with the canonical retraction $\tau_{\toric}\colon G_f \to \cG_f$.

 The possibilities for such an automorphism are described by \cref{prop:SYZAuts}. See also \cite[6.3]{Yu_Enumeration_of_holomorphic_cylinders_II}.
Any two choices of gluing are {\it isotopic} by \cref{prop:famaut}.
\end{remark}

We use the same gluing for $V$ and $Z$.

By \cref{lem:pictriv},
any choice of line bundles, one on $Y$ and one on each $T_f$ will glue to give a line bundle on $Z$.
\end{construction}

\begin{definition} \label{lem:gsk}
We let $\Sk(V) := \Sk(U) \cup_{f \in F, G_f} M_{f,\bbR}$

We write $\oSk(V) \coloneqq \overline{\Sk(V)} \subset Z$.
\end{definition}

\begin{lemma} \label{lem:pictriv} Let $\cG \subset M_{\bbR}$ be a closed convex rational polyhedron, and $G \coloneqq \tau^{-1}(\cG) \subset \bbT_M^{\an}$ where $\tau\colon \bbT_M^{\an} \to M_{\bbR}$ is the canonical retraction.
Then every line bundle on $G$ is trivial.
\end{lemma}
\begin{proof}
See \cite[Theorem 3.7.2]{Fresnel_Rigid_analytic_geometry_and_its_applications}, and \cite[Tag 0BCH]{Stacks_project}.
\end{proof}

We will also want at various points a formal model for the glued space $Z$.
For this we use:

\begin{construction}[Formal Models] \label{const:polyhedraldecomp}
We have the (identified)
polyhedra $\cG_f \subset \Sk(U)$, and $\cG_f \subset M_{f,\bbR}$,
satisfying \cref{ass:G_f}. By \cref{prop:GluingEnds} we can find blowups
$\bbY' \to \bbY$, $\bbT_f' \to \bbT_f$, (where $\bbT_f$ is the constant model for the corresponding toric variety) and isomorphic Zariski open subsets $\bbG_f \subset \bbY'$,
$\bbG_f \subset \bbT_f'$, so that $\bbZ \coloneqq \bbY' \cup_{f,\bbG_f} \bbT_f'$ has generic
fibre $Z = Y \cup_{f,G_f} T_f$. Moreover we can find a 
a line bundle $A_s$ on $\bbZ_s$ whose restriction to each irreducible component of the
central fibre $\bbZ_s$ is ample. 
\end{construction}

With the formal models from \cref{const:polyhedraldecomp}, \cref{ass:basicmod} are satisfied.  In particular we have the
modular $k$-analytic stacks $\oM_{0,N}(Z,m)$ and $\oM_{0,N'}(Z,m)$.
In order to define our counts we impose further
conditions on contact with the boundary: 

\begin{definition} \label{def:Mdef} Let $\oM_{0,N}(Z,D_{Z},k) \subset \oM_{0,N}(Z,k)$ be stable maps with the properties:
\begin{enumerate}[wide]
\item $\deg(f^{*}(A)) =k$, $\deg(f^{*}(K_{Z} + D_{Z})) = 0$.
\item Scheme theoretically $f^*(D_f)$ vanishes to order at least $m_f$ at $p_{e_f}$, and $\deg(f^*(D_f)) = m_f$ for all $f \in F$
\item For each $b \in \Bnd$, let $v_1,\dots,v_s \in \Sk(U,\bbZ)$
be primitive generators for the edges of the minimal cone of $\Sigma$ that contains $P_b$, and $D_1,\dots,D_s \subset D^{\ess}$ the corresponding irreducible components.
Let $P_b = \sum m_i v_i$ (note the expression is unique, as the cones are simplicial).
We require for each $i$ that $f^{-1}(D_i)$ vanish scheme theoretically at $p_b$ to order at least $m_i$, and $\deg(f^*(D_i)) = m_i$.
\end{enumerate}

If we drop $k$ from the notation then we drop the first condition in (1).
\end{definition}

\begin{lemma}   \label{lem:smft}    $\oM_{0,N}(Z)$ is a $k$-analytic stack. 
  $\oM_{0,N}(Z,D_{Z}) \subset \oM_{0,N}(Z)$ is a Zariski closed $k$-analytic substack. 
\end{lemma}
\begin{proof} It is easy to see the defining conditions are Zariski closed, 
  see e.g. \cite[5.1]{Keel_Rational_curves}.   $\oM_{0,N}(Z)$ is a $k$-analytic stack,
  by \cite[8.10]{Yu_Gromov_compactness}, so the same holds for this Zariski closed substack. 
  \end{proof}

  Now we introduce the moduli space we will use for counts.  We have the moduli spaces $M, M'$ of \cref{lem:equivdef}
  and  \cref{def:cM'def}, which we
  henceforth refer to as $M_{\gen}$, $M'_{\gen}$, reserving the simpler notation for: 
  
  \begin{definition-lemma}  \label{def:bms}
    Let
    $$
    M(m)  \coloneqq  M_{\gen}(m) \cap \oM_{0,N}(Z,D_{Z},m)_{\cV_M} \subset \oM_{0,N}(Z,m).
    $$
    $M(m) \subset M_{\gen}(m)$ is Zariski closed. It is a $k$-analytic stack without boundary. 
    
    Let
    $$
    M'(m) \coloneqq M'_{\gen}(m) \cap \oM_{0,N'}(Z,D_{Z},m)_{\cV_{M}'} \subset \oM_{0,N'}(Z,m)_{\cV_{M}'}.
    $$

    $M'(m) \subset M'_{\gen}(m)$ is Zariski closed. It is a $k$-analytic stack, and
    $M'(m) \to V'_M \times_{f \in F} H_f$
    is of finite type. 
    
    Let $M'_{\epsilon} \coloneqq M'_{V_{M,\epsilon}'}$.
\end{definition-lemma}
\begin{proof} The first paragraph is immediate from \cref{prop:noboundary} and \cref{lem:smft}. The second paragraph is immediate
  from \cref{prop:VM'ft} and \cref{lem:smft}. \end{proof}

For our counts the following will be the main object:

\begin{definition-lemma} \label{def:MUk}

  We define $M(U,m) \subset M(m)$ to be the locus of stable curves $[h: C \to Z] \in M$
  such that no irreducible component of the domain maps into $\partial Z$.
More precisely,
such that $h(C^{\circ}) \subset V$, where $C^{\circ} \subset C$ is the complement of the marked points from $\Bnd \cup F$. 

This implies  that scheme theoretically $h^{-1}(D_f) = m_f p_{e_f}$ for $f \in F$, and in the notation of (2) of \cref{def:Mdef}, $f^{-1}(D_i) = m_i p_b$,
for all $b \in \Bnd$ and all $i$. Moreover, it implies 
$h(p_b) \in S_{P_b}^{\circ}$, $h(p_{e_f}) \in D_f^{\circ}$, for all $b \in \Bnd$, $f \in F$,
where $D_f^{\circ} \subset D_f$ is the natural interior of this boundary divisor (the complement to
the union of all other boundary strata), and $S_{P_b}^{\circ} \subset S_{P_b}$ is similarly defined as
the natural interior of this stratum:
$S_{P_b} \coloneqq \cap_{i} D_i$ (notation as in  (2) of \cref{def:Mdef}), and
$S_{P_b}^{\circ} \subset S_{P_b}$ is the complement to the union of all other boundary strata.

$M(U,m) \subset M(m)$ is a Zariski open substack, without boundary.

Let $M'(U,m) \coloneqq M(U,m) \cap M' \subset M(m)$,
$M'(U,m)_{\epsilon} \coloneqq M(U,m) \cap M'_{\epsilon} \subset M(m)$.  These are Zariski open substacks.
We write $M(U), M'(U)$ for the union over all $m$.  

\end{definition-lemma}
\begin{remark} \label{rem:prime}
Note: Whenever we discuss a primed version of the moduli space, we have increased the label set from $N$ to $N'$.
Note all the unprimed versions still make sense with the larger index set. 
If this is the index set we want we will add extra decoration, e.g.\ $M(U,m,N')$.
\end{remark}
\begin{proof}  The complement of the defining condition is that an irreducible component of $C$ maps into
  the boundary $D_{Z}$, it's easy to see this condition (that Zariski closed map to Zariski closed) is Zariski closed. So these
  are open substacks.  For $f \in M(U,m)$, $f^{-1}(D_{Z}) \subset C$ is an effective Cartier divisor, so the implication in the
  second paragraph follows from the degree conditions in (1) and (3) of \cref{def:Mdef}.
  \end{proof} 


The main result of this section is:

\begin{proposition} \label{prop:cad} We assume \cref{ass:G_f}.

The following hold
\begin{enumerate}[wide]
\item Let $\cW^c \subset \bigtimes_{f \in J} \oSk(U)$ be the product of compact subsets contained in complements of $\Wall_k$, notation as in \cref{def:walldef} and $W^c \subset \bigtimes_{f \in J} Y$ its inverse image.
  The natural inclusions
  \begin{align*}
    M'(U,k)_{W^c, V'_{M,\epsilon}} &\subset M(U,k,N')_{W^c,V'_{M,\epsilon} }\\
    M'_{W^c, V'_{M,\epsilon}} &\subset M_{W^c,V'_{M,\epsilon}}
  \end{align*}
  are equality (see \cref{rem:prime}) for sufficiently small $\epsilon > 0$.
\item For sufficiently small $\epsilon > 0$,
$M_{W^c,V'_{M,\epsilon}} \to W^c \times V'_{M,\epsilon}$
is proper.
\end{enumerate}
\end{proposition}
\begin{proof}
For (1): the second equality implies the first. For the first:
We choose $\epsilon > 0$ satisfying \cref{lem:epsilon}, and consider $f:C \to Z$ in $M$.
The image of the skeleton (the interval $[-\epsilon,\epsilon]$ of \cref{lem:epsilon}) meets $\cH_f $ and $W^c$, so by \cref{lem:epsilon},
the skeleton maps into $\cG_f \cap \Wall_k^c$.
Since this image is disjoint from walls, all twigs are contracted:
see \cref{ssec:brtc}, \cref{prop:tree_bound} and \cref{def:walldef}.
It follows that the full annulus maps into $G_f$, which gives the equality. 
Now (2) follows from \cref{prop:noboundary} and \cref{prop:VM'ft} (proper is
equivalent to no boundary together with finite type). 
This completes the proof.
\end{proof}




\begin{lemma} \label{lem:tms}
	
	Let $\ff\colon \fC \to \fZ$ be a map between formal schemes locally of finite presentation over $\kc$, with generic fibre $f\colon C \to Z$. Let $S_s \subset \fC_s$, $T_s \subset \fZ_s$ be Zariski locally closed and $S \coloneqq r^{-1}(S_s) \subset C$, $T \coloneqq r^{-1}(T_s)$, where $r$ denotes the reduction map from the generic fibers to the special fibers.
	The following hold:
	\begin{enumerate}
		\item $S \subset C$, $T \subset \fZ$ are analytic domains.
		\item If $f_s(S_s) \subset T_s$ then $f(S) \subset T$, and the converse holds if $\fC$ is flat over $\kc$.
		\item  If $S_s \subset \fC_s$ is Zariski closed (respectively open) then $S \subset C$ is an open (respectively closed) analytic domain.
		\item If $S_s \subset \fC_s$ is Zariski open, $f_s(S_s) \subset T_s$ and $f_s\colon S_s \to T_s$ is separated of finite type, then $f\colon S \to T$ (see (2)) is separated of finite type.
		\item If $f_s(S_s) \subset T_s$ and $f\colon C_s \to T_s$ is proper, then $f\colon S \to T$ is without boundary. 
	\end{enumerate}
\end{lemma}
\begin{proof}
	(1-3) follow from the properties of the reduction map (see \cite[\S 2.4]{Berkovich_Spectral_theory} and \cite[\S 1]{Berkovich_Vanishing_cycles_for_formal_schemes}).
	(4-5) follow from Temkin's theory of reduction of germs (see \cite[\S 4-5]{Temkin_On_local_properties_II}).
      \end{proof}

      \begin{corollary} \label{cor:familytms}
        Let $\fc: \fC \to \fM$, and $\ff: \fC \to \fZ$ be maps
        between formal schemes locally of finite presentation over $\kc$, with generic fibres $c\colon C \to Z$, and $f: C \to Z$.
        Let $S_s \subset \fC_s$, $T_s \subset \fZ_s$ be Zariski locally closed and $S \coloneqq r^{-1}(S_s) \subset C$, $T \coloneqq r^{-1}(T_s)$, where $r$ denotes the reduction map from the generic fibers to the special fibers.
        Let $G \coloneqq  \{x \in M| f(S_x) \subset T\}$, and
        $\fG_s \coloneqq \{r \in \fM_s| \ff(S_{s,r})
        \subset T_r\}$.

        Then $r^{-1}(\fG_s) \subset G$, with equality if $\fc$ is flat.

        Assuming $\fc$ is flat, the following hold:
        \begin{enumerate}
        \item $G \subset M$ is an intersection of an open and
          a closed analytic domain. 
        \item If $S_s \to \fM_s$ is proper and $T_s \subset \fZ_s$
          is open, then $\fG_s \subset \fM_s$ is Zariski open, and
          $G \subset M$ is a closed analytic domain.
        \item If $S_s \subset \fC_s$ is open and $T_s \subset \fZ_s$
          is closed, then $\fG_s$ is closed, and
          $G \subset M$ is an open analytic domain.
        \item If $\fG_s$ is proper then $G \subset M$ is an open
          analytic domain, and $G$ has no boundary.
        \end{enumerate}

      \end{corollary}

      \begin{proof} Take $x \in M$. We have the associated
        $x: \Sp(H(x)^{\circ}) \to \fM$. We pullback the formal
        model to obtain formal model for $f: C_x \to Y \times_k H(x)$,
        with central fibre $\ff: (\fC_s)_{r(x)} \to \fY_s \times_k \tilde{H(x)}$.
        Now the first statement follows from \cref{lem:tms}. Now
        assume $\fc$ is flat. The condition that proper map to open
        is Zariski open. The condition that open map to closed is open
        (because $\fC_s \to \fM_s$ is flat, thus open). The remaining
        statements now follow from \cref{lem:tms}. This completes
        the proof. \end{proof}

The following is obvious:

\begin{lemma} \label{lem:epsilon}
  Let $W^c \subset \Sk(U)$ be  compact subset in the complement
  of $\Wall_k$. 
There exists $\epsilon > 0$ so that any affine map 
$h:[-\epsilon,\epsilon] \to \Sk^{\sm}(U)$ with derivative
$v_f$, as in \cref{const:data1}, whose image intersects $W^c$, has
image disjoint from $\Wall_k$, and any such map whose image intersects
$\cH_f$ is contained in $\cG_f$. 
\end{lemma}

\subsection{The smooth locus $M^{\sm}(U,k)$} \label{sec:smoothloc}

\begin{definition-lemma} \label{def:Msd} Notation as in \cref{def:MUk}.
Let $M^{\sd}(U,k)
\subset M(U,k)$ be the substack of stable maps whose $N$-pointed domain curve is stable.
This is Zariski open, and an analytic space (as opposed to a stack).
\end{definition-lemma}
\begin{proof} It is clear the locus is open, it is an analytic space as stable $N$-pointed rational curves have no automorphisms.
\end{proof}

\begin{lemma} \label{lem:Msm_smooth}
Let $\mu=[C,(p_j)_{j\in N}]\in\oM_{0,n}$ be a rigid point.
Let $q\in C$ be any rational point not belonging to $\{p_i\}_{i\in F_E \cup \Bnd}$.
Let $M^\sd(U,\bP,\beta)_\mu$ be the fiber of the map $\dom\colon M^\sd(U,\bP,\beta)\to\oM_{0,n}$ over $\mu$.
Let $\nu=[f\colon C\to Y]$ be a rational point of $M^\sd(U,\bP,\beta)_\mu$.
The following are equivalent:
\begin{enumerate}[wide]
\item \label{lem:Msm_smooth:pullback} The pullback $f^*(T_Z(-\log D))$ is a trivial vector bundle on $C$.
\item \label{lem:Msm_smooth:surjective} The derivative $d\ev_q$ of the evaluation map $\ev_q\colon M^\sd(U,\bP,\beta)_\mu\to Y$ is surjective at $\nu$.
\item \label{lem:Msm_smooth:smooth_fiberwise} The evaluation map $\ev_q\colon M^\sd(U,\bP,\beta)_\mu\to Z$ is smooth at $\nu$.
\item \label{lem:Msm_smooth:smooth} For any $i\in I \cup F$, the map $\Phi_i\coloneqq(\dom,\ev_i)\colon M^\sd(U,\bP,\beta)\to\oM_{0,n}\times Z$ is smooth at $\nu$.
\end{enumerate}
Moreover, under the equivalent conditions above, the following hold:
\begin{enumerate}[wide, label=(\roman*), ref=\roman*]
\item \label{lem:Msm_smooth:etale} The maps $\ev_q$ and $\Phi_i$ above are in fact étale at $\nu$.
\item \label{lem:Msm_smooth:boundary} For any $i\in \Bnd \cup F_E$, the maps \[\ev_i\colon M^\sd(U,\bP,\beta)_\mu\to D_i,\]
\[\Phi_i^\partial=(\dom,\ev_i)\colon M^\sd(U,\bP,\beta)\to\oM_{0,n}\times D_i\]
are smooth at $\nu$.
\end{enumerate}
\end{lemma}
\begin{proof} We note the statement is exactly the same as that of \cite[3.6]{Keel_Yu_The_Frobenius}.
We can use the same proof once we generalize the basic deformation-obstruction theory for $\mathbf{Map}_S(X,Y)$ to the present context.
We carry this out in \cref{sec:deformation_theory}.
\end{proof}

We make the following definition, just as in \cite[3.4]{Keel_Yu_The_Frobenius}:
\begin{definition-lemma}  \label{def:cMsm}
Let $M^\sm(U,k) \subset M^{\sd}(U,k)$ be the locus of maps satisfying \cref{lem:Msm_smooth} (\ref{lem:Msm_smooth:pullback}).
The following hold:
\begin{enumerate}[wide]
\item $\Phi_i\colon M^\sm(U,k) \to \cV_M \times \cH$ are \'etale and boundaryless for any $i \in I \cup F$.
\item $M^{\sm}(U,k)$ is smooth and boundaryless.

\item $M^\sm(U,k) \subset M^{\sd}(U,k)$ is Zariski open.
\end{enumerate}
\end{definition-lemma}
\begin{proof} Only the final statement requires proof, and that follows from \cref{lem:Msm_smooth}(\ref{lem:Msm_smooth:smooth}).
\end{proof}

The stable maps in $M^{\sm}(U,k)$ are {\it free}, in the sense of \cite[II.3]{Kollar_Rational_curves_on_algebraic_varieties}:

\begin{definition-lemma} \label{def:MsmX} Let $X \subset Z$ be a closed analytic subspace of codimension at least two.
Then for any $i \in N$, $\ev_i^{-1}(X) \subset M^{\sm}(U,k)$ is a lower dimensional analytic subspace.
Its complement $M^{\sm}(U,X^c,k) \subset M^{\sm}(U,k)$ is a dense Zariski open subset.
\end{definition-lemma}
\begin{proof} This is immediate from the smoothness of the evaluation maps in \cref{lem:Msm_smooth}.
See \cite[3.9]{Keel_Yu_The_Frobenius}.
\end{proof}

Stable domain is a useful notion when $U$ contains no complete rational curves -- it will hold for sufficiently generic
domain curve and marked point, see \cref{lem:ISknodes}. For the general case, a weakening is more useful:
\begin{definition} \label{def:sot}
  
 Let  $x \coloneqq (f:(C,N) \to Z) \in M(U,\bP)$. Let $C_m \subset C$ be the {\it convex hull} of
the marked points.  More precisely, we take the union of irreducible components corresponding
to vertices of the dual graph that lie on the convex hull of the marked points (the $_m$ stands for {\it main
  components}, we will sometimes refer to the other irreducible components as {\it twigs} ). Then
we have $x_m \coloneqq (f: (C_m,N) \to Z) \in M(U, \bP, \beta')$ for some class $\beta'$ depending on $x$.
We let $M^{\sot}(U,\bP) \subset M(U,\bP) $ be the locus with $(C_m,N)$ stable.  It is easy to check this is
Zariski open ($\sot$ stands for stable off twigs).   Let
$M^{\smot}(U,\bP) \subset M^{\sot}(U,\bP)$ be the locus where $x_m$ lies in $M^{\sm}(U)$ as in \cref{def:cMsm}. This
is Zariski open ($\smot$ stands for smooth off twigs). As in \cref{def:MsmX}, we define
Zariski dense open $M^{\smot}(U,X^c,k) \subset M^{\smot}(U,k)$. 
\end{definition}

\subsection{Non-archimedean deformation theory} \label{sec:deformation_theory}

\begin{proposition} \label{prop:truncation_of_L}
Let $X$ be a derived \kanal space locally of finite presentation.
We have $\pi_0(\bbL^\an_X)\simeq\Omega_{\trunc X}$, and for any rigid point $x\in\trunc X$, $\pi_0(\bbT^\an_X(x))$ is isomorphic to the Zariski tangent space of $\trunc X$ at $x$.
\end{proposition}
\begin{proof}
For any $M\in\Cohh(X)$, we have \[\Map(\pi_0(\bbL^\an_X),M)\simeq\Map(\bbL^\an_X,M)\simeq\Map_{X/}(X[M],X)\]
where $X[M]$ denotes the analytic split square-zero extension of $X$ by $M$.
By \cite[Proposition 6.1]{Porta_Yu_Representability_theorem}, $X[M]$ is underived, so $\Map_{X/}(X[M],X)$ classifies derivations from $\cO_X$ to $M$.
Therefore, $\pi_0(\bbL^\an_X)$ has the same universal property as \cite[Proposition 3.3.1(ii)]{Berkovich_Etale_cohomology}.
This shows the first isomorphism.
Taking fiber at $x$ and then dual, we obtain the second isomorphism.
\end{proof}

\begin{proposition}
Let $X$ be a derived lci \kanal space, and $x\in\trunc X$ a rigid point.
If $\pi_1(\bbL^\an_X(x))=0$, then $\trunc X$ is regular at $x$.
\end{proposition}
\begin{proof}
Since $X$ is derived lci, by \cite[Proposition 2.8]{Porta_Yu_Non-archimedean_quantum_K-invariants}, we have \[\dim \pi_0(\bbL^\an_X(x)) - \dim \pi_1(\bbL^\an_X(x))\le\dim_x\trunc X\le \pi_0(\bbL^\an_X(x)).\]
By assumption, $\pi_1(\bbL^\an_X(x))=0$, so $\dim_x\trunc X = \dim \pi_0(\bbL^\an_X(x))$.
By \cref{prop:truncation_of_L}, $\pi_0(\bbL^\an_X(x))$ is isomorphic to the Zariski cotangent space of $\trunc(X)$ at $x$.
Therefore, $\trunc(X)$ is regular at $x$.
\end{proof}

\begin{proposition}
Let $S, X, Y$ be rigid \kanal spaces, $f\colon X\to S$ a proper flat morphism and $g\colon Y\to S$ an lci morphism.
Then the mapping stack $\mathbf{Map}_S(X,Y)$ is representable by a derived \kanal space locally of finite presentation over $S$.
Consider the canonical maps \[ \begin{tikzcd}
X \times_S \mathbf{Map}_S(X, Y) \arrow{r}{\mathrm{ev}} \arrow{d}{\pi} & Y \\
\mathbf{Map}_S(X, Y).
&
\end{tikzcd} \]
Then $\bbL^\an_{\mathbf{Map}_S(X,Y)/S}\simeq(\pi_* ( \mathrm{ev}^*( \bbL^\an_{Y/S} )^\vee ))^\vee$.
\end{proposition}
\begin{proof}
It follows from \cite[Lemma 8.4 and Proposition 8.5]{Porta_Yu_Derived_Hom_spaces}.
\end{proof}

\section{The proper central fibre} \label{sec:fcf}

We note above that stable analytic curves whose tropicalisation has
type $\cV_M \subset M_{0,N}^{\trop}$, i.e. curves in $V_M \subset \oM_{0,N} = \obbM_{0,N}^{\an}$,
notation as in \cref{def:VMdef},
have a canonical decomposition 
$\oC = \oB \sqcup_{f \in F} \oE_f$.  Here we consider the analogous
decomposition for stable algebraic curves (see \cref{rem:weird}): 

\begin{notation} In this section there will be no $k$-analytic spaces, or formal schemes, just
  schemes. So we will depart from the notation in the result of the paper and use font like $Y$ to
  indicate a scheme (whereas in the rest of the paper this font is usually for $k$-analytic spaces).
  Also $\oM_{0,N}$ is moduli of stable curves (whereas in the rest of the paper $\obbM_{0,N}$ is this
  moduli space and $\oM_{0,N} \coloneqq \obbM_{0,N}^{\an}$ is the $k$-analytic version). 
  \end{notation} 

\begin{construction} \label{const:stablecircles}
Consider a family of $N$-pointed semi-stable rational curves $C \to T$, with
(fibrewise) stabilization $s:C \to C'$. We assume each fibre of the stabilization
$C' \to T$ 
has dual graph of type $\cV_M$, and thus in particular has 
for each $f \in F$, a distinguished
irreducible component
corresponding to the distinguished point $\os$ of
the dual graph. This component
contains three special points, marked points $s_f,e_f$ 
and one other special point $n_f$, a node of the fibre  (or, in the special, and
especially simple, case $|N|=3$ $n_f \in N$ is the third marked point).
  
Together this gives a distinguished closed subfamily $S_f' \subset C'$,
a trivial $\bP^1$ bundle with three disjoint sections, $e_f,s_f,n_f$.
Let $E_f \coloneqq s^{-1}(S_f' \setminus \{n_f\}) \subset C $, 
$S_f \coloneqq s^{-1}(S'_f \setminus \{e_f,n_f\}) \subset C$.
Further define $E^{\circ}_f \coloneqq s^{-1}(e_f) \subset C$, the
scheme theoretic inverse image, which is thus Zariski closed. Note
(set theoretically) 
$E^{\circ}_f = E_f \setminus S_f$.

Let
$B \subset C$ be the Zariski open
$C \setminus \cup_{f \in F} E_f^{\circ}$,
and $B^{\circ}$ the
reduction of $C \setminus \cup_{f \in F} E_f$. We note
$C = B \cup_{f \in F} E_f$ is a decomposition into Zariski open subsets,
and $B \cap E_f = S_f$. Let $B^{\circ} \coloneqq B \setminus \cup_f S_f$,
with reduced structure (we will not need its scheme structure).
By construction we have 
a decomposition $C = B \cup_{f \in F} E_f$ into Zariski open subschemes
and a (set theoretic) disjoint union
$C = B^{\circ} \cup_{f \in F} S_f \cup_{f \in F} E_f^{\circ}$. 
\end{construction}

\begin{remark} \label{rem:weird} The (perhaps weird) notation is chosen for application in \cref{sec:enddisk}, for
  example in the proof of 
\cref{prop:noboundary}, e.g.
$E_f^{\circ}$ (resp. $E_f$) will be the locus in the central fibre of a model of an analytic curve
with inverse image, under the retraction from the generic fibre to
the central fibre, (an analytic curve which when we stabilize becomes) an open (resp closed) analytic disk, that in the application we call
$E_f^{\circ}$ (resp. $E_f$), so the name we use here is the name of the
inverse image in the application.
\end{remark} 

\begin{proposition} Notation as immediately above. The formation of
  $E_f, B, S_f$ commute with base extension (i.e. pulling back the family
  $C \to T$ under some $T' \to T$). The same holds set theoretically for
  $E_f^{\circ},B^{\circ}$.
\end{proposition}
\begin{proof} The second statement follows from the first and the first is
  clear from the construction.
\end{proof}

\begin{construction} \label{const:thelocus}
Now we  fix schemes $Y,T_f$, $f \in F$,  and open embeddings
$G_f \subset Y$, $G_f \subset T_f$, and projective $H_f \subset G_f$. Let $Z$
be obtained by gluing $Z \coloneqq Y \cup_{f \in F} T_f$ (glued along $G_f$).
We also fix closed projective $J^Y \subset Y$ and $J^f \subset T_f$. 

We consider a family of stable maps $f: C \to Z$, over a base  $Q$, such that 

the stabilization of the domain is a family of curves of type $\cV_M$, as in \cref{const:stablecircles}.
Now we consider the locus, $Q^G$ of $t \in Q$ over which 
$f(B_t) \subset Y$,
$f(E_{t,f}) \subset T_f$ and $f(S_f) \subset H_f$.
\begin{remark} \label{rem:reformpc} Note we could equivalently define the locus by 
the  two conditions $f(B_t^{\circ}) \subset Y$, $f(E_{t,f}^{\circ}) \subset T_f$, together with
the third $f(S_f) \subset H_f$.
\end{remark}

Let $Q^J \subset Q$ be the locus where
$f(C) \subset J^Y \cup_{f \in F} J^f$, and $Q^{G,J} \coloneqq Q^G$.
Note  $Q^{G,J} \subset Q^G$ is equivalently the locus
where $f(B_t) \subset J^Y$ and each $f(E_{t,f}) \subset J^f$.

Finally we assume we have a line
bundle $A$ on $Z$ whose restriction to $Y,T_f$ are each ample. And we require 
that the pullback of $A$ have degree bounded by a given number $k$. 
\end{construction} 

Here we abuse slightly notation and write $V_M \subset \oM_{0,N}$ for the locus of stable curves
whose dual graphs are of type $V_M$. 

\begin{lemma} \label{lem:anticontinuity} $V_M \subset \oM_{0,N}$ is a union of closed strata. In particular
it is Zariski closed. 
\end{lemma} 
\begin{proof} This is an instance of the {\it anti-continuity} of dual graph: open
strata of $\oM_{0,N}$ are in bijection with (topological) dual graphs, 
$S_{\alpha} \subset \overline{S_{\beta}}$ iff $\beta$ is a degeneration of $\alpha$,
or equivalently if $V_{\beta} \subset \oV_{\alpha}$ where 

$V_{\Gamma} \subset M_{0,n}^{\trop}$ is the locus of trees of topological type
$\Gamma$. Now $V_M \subset \oM_{0,N}$ is closed because the corresponding locus
in $M_{0,n}^{\trop}$ is open.
\end{proof}

\begin{proposition} \label{prop:prop_locus} Notation as immediately above.
  The given locus, $Q^G \subset Q$ is the intersection of Zariski open and Zariski closed. $Q^{J} \subset Q$ is Zariski closed. 
  Moreover, suppose $Q = \oM_{0,N}(Z,m) \times_{\oM_{0,N}} P$ for
  a scheme $P \to \oM_{0,N}$. 
  Then $Q^{G,J} \to P$ is proper
  (where we give $Q^G$ and $Q^{G,J}$ the reduced structure).
   \end{proposition} 
\begin{proof} 
The conditions
$f(B_t^{\circ}) \subset Y$, $f(E_{t,f}^{\circ}) \subset T_f$
are the condition that proper map into open, which
is open, 
while $f(S_{t,f}) \subset H_f$, or $f(C) \subset J^T\cup_f J^f$ are 
the condition that open map into closed,
which is closed (since the family of curves is flat over the base, so the
projection map is open). 

Thus $Q^G \subset Q$ is the intersection of Zariski open with
Zariski closed, and $Q^J \subset Q$ is Zariski closed. 

Now we check the second statement for $Q = \oM_{0,N}(Z,k) \times_{\oM_{0,N} \times Y^F} P$

Let $Z'$ be obtained by gluing along 
the projective $H$ (rather than the open $G$). We have a natural immersion
$Z' \to Z$, an open immersion on $Y$ and each $T_f$.
Now note each $f \in Q^G$, $f$ factors canonically 
through $Z' \to Z$: the restriction $f: B \to Z$ has image in $Y$,
the restriction $f: E_f \to Z$ has image in $T_f$, so give canonically
$f: B \to Z'$, $f:E_f \to Z'$. These agree on the intersection (each
restricts to $f: S_f \to H_f$). This works equally well in families. 
It follows easily that
$Q^G = \oM_{0,N}(Z,k)^G_P$
is identified with $\oM_{0,N}(Z',k)^G_P$. 
But here we can phrase all conditions as the condition that an
open map to a closed, so all are Zariski closed.
Thus
$$
\oM_{0,N}(Z',k)^G_P \subset
\oM_{0,N}(Z',k)_P
$$
is Zariski closed. 

Let $Z_1 \coloneqq J^Y \cup_f J^f$. 
Note that 
$(\oM_{0,N}(Z',m)^J \subset
\oM_{0,N}(Z',k)$
is $\oM_{0,N}(Z_1',m)$ (the defining condition after all is that
the curve have image in $Z_1$), where the notation means the same
construction but with $Z$ replaced by $Z_1$. But now $Z_1'$ is projective, 
and so $\oM_{0,N}(Z_1',m)_P$
is proper over $P$.  $Q^{G,J} \subset Q^J$ (for the $Q$ in
the statement) is Zariski closed by the above,
so $Q^{G,J}$ is also proper over $P$. This completes the proof. 
\end{proof} 

\subsection{The annulus case} \label{sec:annulus}

Here we note the  algebraic curve analog of the analytic decomposition
$C = B^{\circ} \cup_f E^{\circ}_f \cup_{f \in F} A_f$ of \cref{def:VMpdef}.

We take $s: C \to C'$ the stabilization of an $N$'-pointed semi-stable
rational curve over a field, whose dual graph has type $V'_M$.
Analogous to the picture along $S'$ in 
\cref{const:stablecircles}, there is for 
  each $f \in F$  a chain of three irreducible components of $C'$,
  which we denote 
$N_f' + S_f' + E_f'$, each of which has three special points, nodes 
of this chain together with $n_f,a_f \in N_f'$, $s_f \in S_f'$, 
$b_f,e_f \in E_f'$. Let
$E_f \coloneqq s^{-1}((S_f' + E_f') \setminus S'_f \cap N_f')$,
$E^\circ_f \coloneqq s^{-1}(E'_f)_{\red}$, 
$A_f \coloneqq s^{-1}(S_f')_{\red}$,
and let $B^{\circ}$ be the reduction of $C \setminus \cup_f E_f$.
The construction makes sense in families of semi-stable curves as in
\cref{const:stablecircles} above. 
Now, following notation in \cref{const:thelocus} we consider a family of $N'$-pointed 
stable
curve $f:C \to Z$ over a base $Q$, such that for each
fibre the stabilisation of the domain is a curve of type 
$V_M'$.

Now we consider the locus, $Q^{G'} \subset Q$ of $t \in Q$ over which 
$f(B_t^{\circ}) \subset Y$, $f(E_{t,f}^{\circ}) \subset T_f$ and
$f(A_{t,f}) \subset G_f$.   

\begin{lemma} \label{prop:prop_locus_an}  
 $Q^{G'} \subset Q$ is Zariski open.
    \end{lemma}
    \begin{proof} The defining conditions are all that proper map to open, and this is Zariski open
           (as in the proof of \cref{prop:prop_locus}). 
\end{proof}

\section{Disk classes} \label{sec:dc}

For a compact curve $f\colon \bD \to \bbY$, by its class, we mean the class in $\NE(\bbY_s,\bbZ)$, 
which we recall is defined, as a cycle, as the pushforward \begin{equation} \label{eq:dc}
(f_s)_*[\bbD_s^{\pr}] \in Z_1(\bbY_s,\bbZ)
\end{equation}
for $f \colon \bbD \to \bbY$ a formal model, 
and $\bbD_s^{\pr} \subset \bbD_s$ the union of proper components of the central fibre (this cycle
is easily seen to be independent of the model).
Our counts, of tropical curves, or structure constants, are always with respect to this class.
Here we study how the class varies over our basic moduli spaces.
Some care is required, because for curves with boundary this class is not in general deformation invariant.
Roughly speaking this will hold as long as $f_s$ contracts all non-proper components of $\bbD_s$.
In order to force this we introduce the following:

\begin{definition-lemma} \label{def:M_1} Notation as in \cref{const:Sigma_f}.
  For fixed closed piecewise affine
  subsets $\cJ_f \subset \cG_f$ of codimension one, containing $\cG_f \cap \Sigma^{d-1} \subset \Sk(U)$ and
  $\cG_f \cap \Sigma^{d-1}_f \subset M_f$ (recall $M_f$ is the integer tangent space to $\Sk(U)$ at points of $\cG_f$)
    let $V_f \subset G_f$ be the open set
  $r^{-1}(\cG_f \setminus \cJ)$ (where $r$ is the SYZ fibration of \cref{ass:SYZ}). 
  Note:
  In case (2) of \cref{ass:G_f} that $\cG_f \subset \Sk(U)$ only meets the cone $\rho \in \Sigma^{d-1}$ and
  $\cG_f \subset M_f$ only meets $\rho \in \Sigma_f^{d-1}$.  In case (1) it will not meet any cone of $\Sigma^{d-1}$ in $\Sk(U)$,
  but it might meet codim one
  cones of $\Sigma_f$ (introduced because we want the ray $R_f \in \Sigma_f$ and it might intersect $G_f \subset M_f$). 
  Let $M_1(U) \subset M(U)$ be the locus of stable maps such that $f(S_f) \subset V_f$ for all $f \in F$ (though it depends on
  $J$ we leave that out of the notation). 
This is open in $M(U)$.

Let $[h\colon C \to Z] \in M_1(U)$.
For each $f \in F$
$(r \circ h)(S_f)$ is contained in the interior of a maximal cone of $\Sigma$.
Let $h\colon \fC \to \bbZ$ (notation as in \cref{const:polyhedraldecomp})
be a formal model such that the closed analytic semi-stable circle $S_f \subset C$ is the inverse image
$\pi^{-1}(\fS_{f,s})$ of a Zariski open set $\fS_{f,S} \subset \fC_s$ of the central fibre,
where $\pi\colon C = \fC_{\eta} \to \fC_s$ is the canonical reduction map.
Then $h_s(\fS_{f,s}) \subset \bbY'_s$ maps to the zero stratum of $\bbY_s$
corresponding to this maximal cone of $\Sigma$.
\end{definition-lemma}

\begin{proof} The defining condition is that a compact set map into an open set, which is open.
This gives the first statement.
By the definition of $M_1$
$(r \circ f)(S_f) \subset r(\cV_f)$, and by the definition of $\cV_f$, $r(\cV_f)$ is contained in a maximal cone of $\Sigma$.
Now $h_s(\fS_{f,s})$ is the corresponding zero stratum of $Y$, by the definition of the Berkovich retraction $r$.
\end{proof}

\begin{remark} \cref{def:M_1} implies that $h_s \circ \pi_{\fC} = \pi_{\fY} \circ h\colon C \to Y$
contracts all the {\it boundary circles} $S_f \subset C$
of the {\it body (semi-stable) disk} $B$, to points.
For this reason the image $(h_s \circ \pi_{\fC})(B) = [h\colon B \to Y^{\an}] \subset Y$,
is a cycle, and this perhaps gives some philosophical reason for expecting good behavior.
Since $\pi$ is not continuous (in fact, it is anti-continuous), we do not know how to make this final intuitive remark rigorous.
\end{remark}

\begin{proposition} \label{prop:constclass} Notation as in \cref{const:polyhedraldecomp}.
  We can choose the codimension one subset $J \subset \Sk(U)$ of \cref{def:M_1}  so that the
  following holds:

  \begin{enumerate}
    \item
    Let $L$ be a line bundle on $\bbZ_s$ glued from line bundles $A$ on $\bbY_s'$ and
  $A_f$ on $\bbT'_{f,s}$ (which agree on the Zariski open gluing region $\bbG_{f,s}$).
  We assume $A$ is pulled back from a line bundle on $\bbY_s$, and $A_f$ from a
  line bundle on $\bbT_{f,s}$ (which we will indicate by the same names). Then: 

For $[f:C \to Z] \in M_1(U)$ we have $$
[f:C \to Z] \cdot c_1(L) = [f\colon B \to Y] \cdot c_1(A)
+ \sum_{f \in F} [f\colon E_f \to T_f] \cdot c_1(A_f).
$$

\item  $[f\colon E_f \to T_f] \in \NE(\bbT_f,\bbZ)$ is the class of the closure of the one parameter subgroup corresponding to $v_f \in M_f$, assuming $-v_f$ lies in each cone of $\Sigma_f$ meeting $G_f$. In any
  case there are at most finitely many possiblities for $[f\colon E_f \to T_f] \in \NE(\bbT_f,\bbZ)$.

\item For each line bundle $A$ on $\bbY_s$ we can choose $A_{f}$ on $\bbT_{f,s}$ so that their
  pullbacks agree on the open gluing locus $\bbG_{f,s}$, and thus give a line bundle $L$ as in (1).

  \item $[f\colon B \to Y]$ is constant on the intersection with $M_1$ (notation as in \cref{def:M_1}))
    of any connected component of $M(U) \subset \oM_{0,n}(Z)$.
    \end{enumerate}
\end{proposition}
\begin{remark} \label{rem:case1} We note that the condition on $v_f$ holds in (2)
  by assumption in case (2) of \cref{ass:G_f}.
\end{remark}
\begin{proof} We consider a model $f:\fC \to \bbZ$.
  The components of the central fibre corresponding to $S_f$ map to zero strata of the central fibre
  $\bbY_s$ by \cref{def:M_1}.
  It follows that $L$ is trivial on these components, and thus the intersection number is the sum of the intersection numbers with the classes of the disks, and then by the projection formula,
  these are $[f\colon B \to Y] \cdot c_1(A)$ and $[f:E_f \to T_f] \cdot c_1(A_f)$. This gives (1). 

By the definition of $M_1$, and the fact that $-v_f \in \rho_f$,
the tropicalisation of $f:E_f \to T_f$ is a ray in the direction of $v_f$,
with endpoint in a maximal cone of $\Sigma_f$.
Now an easy toric computation shows the class is the closure of the given one parameter subgroup:
We extend the one parameter subgroup to a complete (smooth $k$-analytic) rational curve, with domain
the union of two disks, tropicalising to this ray, and a ray with the same end point and direction
$-v_f$. This {\it new} ray lies in a maximal cone of $\Sigma_f$, so the class of this disk is trivial,
and $[f: E_f \to T_f] \in \NE(T_f,\bbZ)$ is the class of the complete curve.  This completes the proof
of (2).

For (3): Suppose first we are in case (1) of \cref{ass:G_f}. Then the gluing locus maps to a zero
stratum of $\bbY_s$, so the pullback of $A$ is trivial here, and so we can extend $A$ by trivial
$A_f$. In case (2), the line bundle on $\bbG_{f,s}$ depends only on the intersection number of
$A$ with the one stratum corresponding to $\rho$, with this it's easy to find an extension.

For (4): it's of course enough to show that for each $A \in \Pic(\bbY_s)$,
$[f: B \to Y] \cdot A$ is constant for $f$ in the given locus. Extend $A$ to $L$ as described
in the proof of (3). The left hand side of the displayed formula is constant (this is a family
of complete stable curves in $\bbZ_s$). In the first case of the proof of (3), each term in the
sum is trivial (in particular, constant over $f$). In the second case each term is constant by (2)
(the class of the end disk itself is constant). This completes the proof. 
\end{proof}

\begin{lemma} \label{lem:fL} There exists a line bundle on $\bbZ_s$ whose restriction to
  $\bbY'_s$ or any of the $\bbT'_{f,s}$ is ample.
  \end{lemma}
  \begin{proof} We start with ample line bundles on $\bbY_s$ and $\bbT_{f,s}$, such that in case
    (2) of \cref{ass:G_f} they have the same intersection number with the one strata corresponding
    to $\rho_f$. Their pullbacks to $\bbY'_s$, $\bbT_{f,s}'$ glue, and now we can modify them by
    exceptional divisors of the blowup to make them ample, and modify in {\it the same way}, so the
    modifications continue to glue. 
\end{proof}

Using \cref{prop:constclass} we can now make:

\begin{definition-lemma} \label{def:cmsm} Fix $\beta \in \NE(\bbY_s,\bbZ)$. 
  Pick $L$ as in \cref{lem:fL}. We let $M(U,\beta) \subset M(U)$ be the union of
  connected components $Q \subset M(U)(m)$, over all $m > 0$, such that 
  $[f:B \to Y] = \beta$ for all $f \in Q \cap M_1$. $Q$ is empty for all but finitely
  many $m$. $M(U,\beta)$ depends only on data from \cref{ass:basicsetup}, the
  convex $G_f$ from \cref{ass:G_f}, and $\beta \in \NE(\bbY_s,\bbZ)$. 
  \end{definition-lemma}
\begin{proof}
  It is clear that $M(U,\beta) \subset M(U)$ depends only on the stated data.
  By \cref{prop:constclass} there are
  only finitely many possibilities for $m \coloneqq [f: C \to Z] \cdot c_1(L)$ for $f \in M(U,\beta)$.
  The result follows. 
  \end{proof}

  \subsection{Independence under blowup} \label{sec:itb}
  The following is easy to check:
  
  \begin{definition-lemma} \label{def:perm} Let $(\bbY,\bbD)$ be as in \cref{ass:basicsetup}. We
    consider the blowup $\bbY' \to \bbY$
    along a smooth subscheme $Z$ contained in $|\bbD|$. We say it is {\it permissible} if
    $Z$ is transverse to the boundary of the minimal closed stratum of $\bbD$ that contains it. In this case
    the strict transform of $\bbD$, together with the exceptional divisor, gives another pair $(\bbY',\bbD')$
    satisfying \cref{ass:basicsetup}. The natural map $A_1(\bbY'_s,\bbZ) \to A_1(\bbY_s,\bbZ)$ is surjective,
    with kernel generated by a line in any fibre of the blowup. Note a {\it toric blowup}, i.e. a blowup of
    a closed stratum, is permissible. Any two compactifications of the same $\bbU$ (in the sense
    of \cref{def:abuse}) are related by a sequence of permissible blowups and blowdowns, by the
    weak factorisation theorem, \cite{Wlod}. 
    \end{definition-lemma}

Consider $p:\tbbY \to \bbY$ a composition of permissible blowups.
We assume \cref{ass:G_f} holds for $\tbbY$, and use  the same $G_f$ in \cref{const:Z} for $\bbY$ and $\tbbY$. 

\begin{proposition} Notation as immediately above.  $f \to b \circ f$ induces
  an isomorphism $M(U \subset \tY) \to M(U \subset Y)$. 
\end{proposition}
\begin{proof} It's enough to consider the case of a single blowup.
  We have $b: M(Y \subset \tY) \to M(U \subset Y)$ given by $f \to b \circ f$.
  Note the contact of the stable curve with the boundary ($Y \setminus U$ or $\tY \setminus U$) is
  prescribed by the fixed data (in particular the punctured (at the marked points) curve maps into $U$,
  where $b$ is an isomorphism). So 
  the
  universal property of blowing up induces a map in the other direction (each $(f:C \to Z) \in M(U \subset \cY)$
  canonically lifts to $f: C \to \tcZ$, and this works in families), and it's clear these maps are inverse. 
\end{proof}

The next result tracks the change in the disk class: 

\begin{proposition} \label{prop:tbeta} For $\beta \in \NE(\bbY_s,\bbZ)$,
  $M_1^{\smot}(U \subset Y,\beta) \neq \emptyset$
  if and only if there exists $\tbeta \in \NE(\tbbY_s,\bbZ)$ with $p_*(\tbeta) = \beta$ and $M_1^{\smot}(U \subset \tY,\tbeta) \neq \emptyset$.
  In this case $\tbeta$ is the unique $\gamma \in \NE(\tbbY,\bbZ)$ such that $p_*(\gamma) = \beta$ and
  $M^{\smot}(U \subset \tY,\gamma) \neq \emptyset$, and moreover we can choose closed sets
  $X \subset Y$, $\tX \subset \tY$ in \cref{def:MsmX}
  such that $f \to p \circ f$ induces an isomorphism $M^{smot}(U \subset \tY,\tX^c,\tbeta) = M^{smot}(U \subset Y,X^c,\beta)$.
  \end{proposition}
\begin{proof} Let $[f\colon C \to Z]$.
Note this factors $f:C \to \tZ$.
Since $p^*(T_Z(-\log D)) = T_{\tZ}(-\log D)$,
$(f:C \to Z) \in M^{\smot}$ (more precisely, the pullback of the log tangent bundle to the {it convex hull
  of the marked points} (see \cref{def:sot}) is globally generated, as in \cref{lem:Msm_smooth})
if and only if $(f:C \to \tZ) \in M^{\smot}$.
Now the result follows once we show that $[f\colon B \to Y] \in \NE(\bbY_s)$ determines
$[f\colon B \to \tY] \in \NE(\tbbY_s).$

The kernel of $A_1(\tbbY_s) \twoheadrightarrow A_1(\bbY_s)$ is one dimensional,
generated by a line in a fibre of the blowup $\tbbY \to \bbY$. So it is enough to
check that $[f\colon B \to \tY] \cdot \cO(E)$ is determined, where $E$ is the exceptional divisor.
We use \cref{prop:constclass}. $[f\colon C \to \tY] \cdot \cO(E)$ is determined by the fixed data $\bP$.
Suppose first we are in case (1) of \cref{ass:G_f}. Then
$[f\colon C \to \tY] \cdot \cO(E) = [f:B \to \tY] \cdot \cO(E)$ (this is true of every line bundle
pulled back from $\bbY$, see the proof of (3) of \cref{prop:constclass}). In case (2)
the class of the end disks is independent of $f$, see \cref{rem:case1}, so the result follows
from the displayed formula in (1) of \cref{prop:constclass}. This completes the proof.
\end{proof}

\begin{lemma} \label{lem:ISknodes}
 Pick $i \in I$ and consider $$
\Phi_i\colon M(U,\beta,\bP) \to \cV_M \times U^{\an}.
$$
Then 
$$
\ISk \coloneqq \Phi_i^{-1}(\oSk) \subset M^{\smot}(U, X^c,\beta,\bP)
$$
for any choice of $X \subset Z$ in \cref{def:MsmX} and 
the germ of $\Phi_i: M^{\smot}(U, \beta,\bP) \to \cV_M \times U^{\an} $ around $\Phi_i: \ISk \to \oSk$
is independent of permissible blowup.
%
If $U$ contains no complete rational curves
these statements hold with $\sm$ instead of $\smot$ (and for the
final statement, $C$ in place of $C_m$). 
\end{lemma}

Here $\oSk$ means the closure of $\Sk(M_{0,N} \times U) \subset (\oM_{0,N} \times U)^{\an}$.

\begin{proof}
  Assume first $U$ contains no complete rational curves. 
  
We fix $\mu \in \cV_M$ and prove the analogous statement,
for $\Phi\colon M(U,k)_{\mu \times \cH} \to \mu \times \cH$
(where now for $\ISk$ we take the inverse image of $\mu \times \Sk(U)$).
This implies the result.

First consider $\Phi^{-1}(\mu \times \Sk(U)) \cap M^{\sd}(U)$
(after restricting to $\cV_M \times \cH$, we are leaving this restriction out of the notation).

By \cite[10.3.7,6.3.7]{Ducros_Families_of_Berkovich_spaces},
if we restrict to the reduction of an irreducible component of $M^{\sd}(U)_\mu$ that meets the intersection, then $\Phi$ will be generically smooth, since $\mu \times \Sk$ consists of Abhyankar points (of $\mu \times U$) which then implies the map of Zariski tangent spaces (for the unreduced space) is generically surjective,
which then in turn implies that generically this irreducible component is in $M^{\sm}(U)$,
exactly as in the proof of \cite[3.9]{Keel_Yu_The_Frobenius}
(using \cref{sec:deformation_theory} which generalizes the necessary deformation theory of maps to our present context).
This implies then the map is generically \'etale (on the irreducible component we consider), and then moreover that the inverse image in question consists of Abhyankar points.
Thus it is contained in any Zariski open subset, so in the \'etale locus, and also in $M^{\sm}(U,X^c)$ from definition \cref{def:MsmX} for any choice of $X$.
This gives the independence of the germ on toric blowup.

All that remains is to show that the inverse image in $M(U \subset \tY,\tbeta)$ lies in $M^{\sd}$.
For this we use \cite[3.9]{Keel_Yu_The_Frobenius} (the same proof applies, see \cref{def:MsmX}),
and the argument for \cite[3.10]{Keel_Yu_The_Frobenius}: \cite[3.10]{Keel_Yu_The_Frobenius} describes precisely $M(U \subset \tY) \setminus M^{\sd}(U \subset \tY)$.
In each case the domain curve consists of $\bA^1$
curves (notation as in the proof of \cite[3.10]{Keel_Yu_The_Frobenius}, the term means a smooth rational curve
meeting the boundary of $U$ in a single set theoretic point),
and then the closure of the remaining component inherit {\it new} marked points, mapping into the interior,
the contact points with the $\bA^1$-curves. 
This closure is canonically identified with the stabilization of the original curve, and it (for the new marked points)
lies in $M' \coloneqq M^{\sd}$ for the analogous space.
So then we have \'etale at every point of the inverse image of $\Sk$, for $M'$ by the previous paragraph.
This implies evaluation at any of the {\it new} points is \'etale, at any point of $\ISk(M')$, $\ISk \coloneqq \Phi^{-1}(\Sk)$.
The union of $\bA^1$ curves (for bounded degree) is not Zariski dense, by log Kodaira dimension:
We note that the $\bA^1$ components (the full $\bbP^1_{\an}$) lie on
the body (semi-stable) disk $B \subset C$, notation as just above \cref{sec:enddisk}, and so map into $Y \subset Z$. 
Suppose we have a dominant family $F \times \bbP^1_{\an}$ to $Y$. The pullback of $K + D$ will have negative
degree on the general member, see \cite[5.11]{Keel_Rational_curves}.
But by \cref{eq:kpd}, $K + D = W$, with $W$ effective supported on $D$, and the general member will meet
$W$ non-negatively. 

This implies evaluation at any of the {\it new} points (at any point of $\ISk(M)$) lies outside the union of $\bA^1$-curves,
so in fact there are no such new points, and the original stable map has stable domain.
Now the result follows from \cref{prop:tbeta}.
This completes the proof in the no rats case.

For the general case, we consider $x_m: C_m \to Z$ as in \cref{def:sot}.
There is no irreducible component of
the domain mapping entirely into $U$. 
Now
the argument that $x_m \in M^{\sd}$ is the same as the proof that $x \in M^{\sd}$, and once we have this, the
same smoothness argument applies.
\end{proof}

\begin{remark} The astute reader will have noticed that the above argument would hold without change with $\oSk$ replaced by any closed subset whose intersection with $\mu \times U^{\an}$
consists of Abhyankar points.
\end{remark}

\section{Skeletal curves} \label{sec:skeletal_curves}

A key idea of \cite{Keel_Yu_The_Frobenius} is the notion of skeletal curves.
As the methods are elementary deformation theory of maps, the results carry over to
the current, more general context.

Here we consider
\begin{align*}
U &\subset V \subset Z \\
V &\coloneqq U \cup_{f\in F,G_f} T_{M_f} \\
Z &\coloneqq  Y   \cup_{f \in F,G_f} T_f \end{align*}

as in \cref{const:Z}.

\begin{assumption} \label{ass:gluing} We assume gluing is chosen so that the volume forms patch,
  to get $\omega$ on $V$, with simple pole along the full boundary of the snc $V \subset Z$.
We note there is such a gluing by \cref{prop:volumeform}
\end{assumption}

Now the construction of the skeleton of a volume form in \cite[\S 8]{Keel_Yu_The_Frobenius}
works without change, to give $\Sk(\omega) \subset V$, whose intersection with $U$ or $T_{M_f}$ is the
skeleton of its volume form.

Now we fix $C \in \oM_{0,N}$, and as in \cite{Keel_Yu_The_Frobenius}, we consider
$$
H \subset \Hom(C,Z)
$$
the subspace of $f \in \Hom(C,Z)$ with boundary contact as in \cref{def:MUk}, e.g.
$f^{-1}(D_f) =m_f p_{e_f}$, $f^{-1}(D_b) = m_b p_b$, $f \in f,b \in \Bnd$. 

Exactly as in \cite[8.11-8.12]{Keel_Yu_The_Frobenius}, $H^{\sm}$ has a canonical top degree differential form,
$\omega_H$ and thus a skeleton $\Sk(\omega_H) \subset H^{\sm}$ (where as in \cref{sec:smoothloc}, $H^{\sm}$ is the
locus where the pullback $T_Z(-\log D)$ is trivial).

We have the following result from \cite[8.18]{Keel_Yu_The_Frobenius} (whose proof carries over without change):

\begin{theorem} \label{thm:f_in_skeleton}
Notation as immediately above.
Let $f \in H$, and $C_f$ the fibre of $C \times H \to H$ over $f$.
Let $f:C_f \to Z$
denote the restriction of the universal map $C \times H \to Z$.
Let $g\colon C_f\to C\times Z$ denote the product of $C_f\to C$ and $f\colon C_f\to Z$.
Let $C^{\circ} \subset C^{\sm}$ be the complement of the marked points $p_j$, $j \in J$. 
The following are equivalent:
\begin{enumerate}[wide]
\item $f\in\Sk(\omega_H)\subset H^\sm$.
\item For some $x \in C(k) \subset C_f$, $f(x) \in \Sk(V)$.
\item For every $x \in C(k) \subset C_f$, $f(x) \in \Sk(V)$.
\item $g\inv(\Sk(C^{\circ} \times V)) = \Sk(C^\circ_f)$.
\item \label{thm:f_in_skeleton:g_inv_nonempty} $g^{-1}(\Sk(C^{\circ} \times V)) \neq \emptyset$.
\end{enumerate}
Assume these equivalent conditions hold, let $\Gamma(k) \subset C_f$ be the convex hull of $C(k)\subset C_f$; then $f(\Gamma(k)) \subset \oSk(V)$.
\end{theorem}

\begin{definition} \label{def:skeletal} Following \cite[1.12]{Keel_Yu_The_Frobenius} we call
  $f\colon C_f\to Z$ satisfying any of the equivalent conditions in \cref{thm:f_in_skeleton} {\it skeletal}.
\end{definition}

The next three results have direct counterparts in \cite[\S 8]{Keel_Yu_The_Frobenius}:

\begin{lemma} \label{lem:source_of_skeletal_curve}
Notation as in \cref{sec:bms}.
The following hold:
\begin{enumerate}[wide]
\item For any $\mu \in \oM_{0,n}$, and
  $(f:C \to Z) \in \Phi_i^{-1}(\Sk(\mu \times_k U)) \cap M(U,\bP,\beta)_{\mu}$, the
  main component 
  $f_m:C_m \to Z$ (notation as in \cref{def:sot}) is skeletal. The same holds for $f:C \to Z$
  if $U$ contains no complete rational curves.
 \item For $(f:C \to Z) \in \Phi_i^{-1}(\oSk(M_{0,n}) \times \Sk(U)) \cap M(U,\bP,\beta)$,
   $f_m:C_m \to Z$ is skeletal, and the same holds for $f:C \to Z$ if $U$ contains no complete
   rational curves. 
\end{enumerate}
\end{lemma}
\begin{proof} The proof of the completely analogous \cite[8.20]{Keel_Yu_The_Frobenius} applies,
  using \cref{lem:ISknodes}. \end{proof}

\begin{lemma} \label{lem:restrict_to_skeleton}
Notation as in \cref{lem:source_of_skeletal_curve}.
The following hold:
\begin{enumerate}[wide]
\item \label{lem:restriction_to_skeleton:sm} Assume $[C,(p_j)_{j\in J},f]\in M(U,\bP, \beta)_\mu$ is skeletal.
  Then it belongs to $M^\sm(U,\allowbreak\bP,\allowbreak\beta)$; in particular, we have
  $(f_m:C_m \to Z) \in M^\sm(U,\bP)$ for $(f:C \to Z) \in \ISk$ (and $(f:C \to Z) \in
  M^\sm(U,\bP)$ if $U$ contains no complete rational curves).
Moreover, for any closed subvariety $G\subset Y$ not containing any irreducible component of $D^\ess$, the pullback $f_m\inv(G^\an)$ is a finite set of points without multiplicities and disjoint from the nodes of $C$;
and for any closed subvariety $Z\subset Y$ of codimension at least 2, the image $f_m(C)$ does not meet $Z^\an$.
If $U$ contains no complete rational curves then 
$\Phi_i$ is representable (i.e.\ non-stacky) and étale over a neighborhood of $\ISk$.
\end{enumerate}
\end{lemma}
\begin{proof} The proof of \cite[8.21]{Keel_Yu_The_Frobenius} applies. \end{proof}

We also note:

\begin{proposition} \label{prop:symskel}
  As subsets, $\Phi_i^{-1}(\oSk \times \Sk(U)) = \Phi_j^{-1}(\oSk \times \Sk(U))$
  for any $i,j \in F_S \cup I \subset N$ (notation as in \cref{const:data}).
\end{proposition}
\begin{proof} This follows from  \cref{lem:source_of_skeletal_curve} and the equivalence of (1-5) in \cref{thm:f_in_skeleton}.
  \end{proof}

\begin{lemma} \label{lem:skeletal_curve_contraction}
Notation as in \cref{lem:source_of_skeletal_curve}.
Let $(f\colon[C,(p_1,\dots,p_n)] \to Z) \in \ISk$.
Let $\Gamma$ (resp.\ $\Gamma^\Bnd$) denote the convex hull in $C$ of the all the marked points (resp.\ all the marked points from $\Bnd$).
Then $f|_\Gamma\colon\Gamma\to {\oSk(U)}\subset Y$ factors through the retraction $\Gamma\to\Gamma^\Bnd$.
\end{lemma}
\begin{proof} The argument of \cite[8.22]{Keel_Yu_The_Frobenius} applies. \end{proof}

\begin{remark} \label{rem:lskel} Here we have followed the theory from \cite[\S 8]{Keel_Yu_The_Frobenius}, which is based on
  deformation theory of maps with domain a complete rational curve. The theory has been simplified  and considerably
  extended in
  \cite{LS}, to cover maps with domain semi-stable genus $0$ curves with boundary, as in \cref{def:iostrop},
  e.g. it can be applied directly
  to a disk $f: B \to Y$, or a punctured disk $f: B^\circ \to U$, which need not admit extensions to complete
  curves mapping to $Z$.
  \end{remark}

\section{Properness of the spine map}
\label{sec:spmapprop}

Here we take fixed data as in \cref{const:data1}.
We assume (1) of \cref{ass:G_f}.
We use the space $\NT_J^F$ of \cite[4.5]{Keel_Yu_The_Frobenius}, a simple variant of $\oM_{0,J}^{\trop}$ allowing marked points at finite distance.

\begin{definition} \label{def:treesinsku} By an $\NT_J^F$ tree in $\oSk(U) \subset Y$ we mean $\Gamma \in \NT_J^F$ and continuous $h\colon \Gamma \to \oSk(U)$, such that $h^{-1}(\partial \oSk(U))$ is the set of $p_b, b \in \Bnd$.
Moreover we assume $h(p_f) \in \cG_f$, for each $f \in F$, and $h$ is affine in a neighborhood of $p_j$, for each $j \in J$ (we note this makes sense for $b \in \Bnd$, using the canonical piecewise affine structure).
We let $\NT_J^F(U)$ be the set of such trees (it depends on the data of $\cG_f$ but we leave this out of the notation).

We note:
\begin{lemma} \label{lem:treesindep}
  The set $\NT_J^F(U)$ does not depend on the snc compactification $(\bbY,\bbD)$ (only on $U$ and the fixed data $\bP$).
\end{lemma}
\begin{proof}
  This is clear as $h$ is uniquely determined by its restriction to $\Gamma \setminus \Bnd = h^{-1}(\Sk(U))$.
\end{proof}

Now we fix $\beta \in \NE(\bbY_s,\bbZ)$ and a set of $\beta$-walls $\Wall_{\beta}$, see \cref{def:walldef}.
Let $\tbbY \to \bbY$ be a toric blowup such that each $P_b$, $b \in \Bnd$, has divisorial center on $\tbbY$. 
We let $\NT_J^F(U,\beta) \subset \NT_J^F$ be the subset where additionally
$h(p_j) \not \in \overline{\Wall_{\beta}} \subset \tY$ (we stress that here we compute the closure in the generic fibre
$\tY \coloneqq (\bbY_{\eta}$) for $j \in \Bnd \cup F$.

\begin{lemma}
  The set $\NT_J^F(U,\beta) \subset \NT_J^F(U)$ does not depend on the choice of $\tbbY$.
  Moreover $\SP^\tr_{\Wall_{\beta}}(\Sk(U),\beta)^F_J \allowbreak \subset \NT_J^F(U,\beta)$, notation as in \cref{def:transverse}.
\end{lemma}
\begin{proof} Near $v_b$, $b \in \Bnd$, the derivative of $h$ is $P_b$, which is parallel to a ray of $\Sigma_{(\tY,D^{\ess})}$.
It follows that $h(v_b) \in \obbR_{\geq 0}^{D_{\tY}^{\ess}}$ has a single infinite coordinate.
It follows that $h(v_b) \in \overline{\fD}$ for a cone $\fD$ (contained in a cone of $\Sigma_{(\tY,D^{\ess})}$) if and only if this holds for all points in a neighborhood of $v_b$, in which case the spine is not transverse to $\fD$.
$\SP^\tr_{\Wall_{\beta}}(\Sk(U),\beta)^F_J \subset \NT_J^F(U,\beta)$ follows.
The independence on the toric blowup follows from similar reasoning.
\end{proof}

We have a natural inclusion $\NT_J^F(U) \subset \Cont(C/\NT_J^F,\oSk(U))$, and so it inherits a natural compact-open topology, see \cite[4.19]{Keel_Yu_The_Frobenius}.
\end{definition}

\begin{assumption} \label{ass:vformskel}
For the rest of the section we assume \cref{ass:gluing}.
Thus $V$ inherits a canonical volume form, and the theory of skeletal curves from \cref{sec:skeletal_curves} applies.
\end{assumption}

Fix $i \in I$.
We consider $\Phi_i\colon M(U,\beta) \to V_M \times U$.
By \cref{lem:sksm},
$\ISk \coloneqq \Phi_i^{-1}(\oSk) \subset M^{\smot}(U,\beta)$,
where $\oSk \coloneqq \oSk(\oM_{0,N}) \times \Sk(U)$,
and consists of skeletal curves.
Thus we have a canonical $$
\Sp\colon \ISk \to \NT(U).
$$

Let $\ISk^{\tri} \coloneqq \Sp^{-1}(\NT(U,\beta))$ (here $\tri$ denotes {\it transverse at marked points}, though we actually have transverse type conditions only at the points $p_i$ for $ i \in F \cup \Bnd$, we do not need any conditions on $i \in I$).

Note since $\ISk$ consists of skeletal curves,
$\ISk^{\tri} \subset \ISk$ is equivalently defined by the conditions that $f(p_j) \not \in \Wall_{\beta}$ for all $j \in \Bnd \cup F$.

\begin{proposition} \label{prop:Spprop} Assumptions as in \cref{ass:basicsetup}
  We consider 
$$
\Sp\colon \ISk^{\tri} \to \NT(U,\beta).
$$
The following hold:
\begin{enumerate}[wide]
\item Any net in $\ISk^{\tri}$ whose image converges in $\NT(U,\beta)$ has a subnet converging in $\ISk^{\tri}$, equivalently, $\Sp$ is proper.
\item If $U$ contains no complete curves then furthermore the fibres of $\Sp$ are finite.
\end{enumerate}
\end{proposition}

\begin{proof}
Using \cref{prop:tbeta} and \cref{lem:ISknodes}
we can replace $\bbY$ by a toric blowup and assume each $P_b$, $b \in \Bnd$
has divisorial center on $\bbY$.
Now for $f \in \ISk^{\tri}$,
$f(p_j) \not \in \overline{\Wall_{\beta}} \subset Y$ for $j \in \Bnd \cup F$.

To check properness, by \cref{lem:proppb} we can add extra points and use $N' \subset N$ instead of $N$, notation as in \cref{sec:bms}.

For this we follow the argument in the proof of \cite[10.1]{Keel_Yu_The_Frobenius}.
We consider a net in $\ISk^{\tri}$
whose image converges in $\NT(U,\beta)$ and prove a subnet converges in $\ISk^{\tri}$.
Since $M'$ is of finite type, see \cref{def:bms}, and we have the equality (1) of \cref{prop:cad},
after passing to a subnet, the net converges in $M'$, and it remains to show the limit lies in $M'(U,k,N')$ (again we have the equality from (1)).
Now the argument of \cite[10.1]{Keel_Yu_The_Frobenius} applies.
This completes the proof of (1).

For (2), $\ISk \subset M^{\sm}(U,\beta)$
by \cref{lem:ISknodes}, so $\Sp$ has discrete fibres, and thus finite fibres, by (1). 
\end{proof}

\begin{lemma} \label{lem:sksm} We consider $\Phi\colon M(U,\beta) \to V_M \times U$.
  If $U$ contains no complete rational curves, $\Phi^{-1}(\oSk) \subset M^{\sm}(U,\beta)$, and in any
  case this inverse is contained in $M^{\smot}(U,\beta)$ (notation as in \cref{def:sot}),
  and $x_m$ is a skeletal curve.
  \end{lemma}
  \begin{proof} The first statement is part of \cref{lem:ISknodes} and
    \cref{lem:source_of_skeletal_curve}.
Once we have this, then the same deformation theory argument used to prove \cite[8.21(1)]{Keel_Yu_The_Frobenius}
applies.
\end{proof}

\begin{lemma} \label{lem:proppb}
Let $\Phi\colon M \to Q$ be a continuous map between Hausdorff spaces.
Let $f\colon Q_1 \to Q$ be a continuous surjection, with $Q_1$ Hausdorff and locally compact.
Then $\Phi$ is proper if and only if the fibre product $\Phi_1\colon M \times_Q Q_1 \to Q_1$ is proper.
\end{lemma}
\begin{proof} Proper is preserved by fibre product.
For the other direction, given a compact set $K \subset Q$, by the local compactness and surjectivity we can find compact $K_1 \subset Q_1$
with $f(K_1) = K$.
Now $$
\Phi^{-1}(K) = f(\Phi_1^{-1}(K_1)).
$$
The result follows.
\end{proof}

Now we can make:

\begin{definition} \label{def:naive_counts_general}

  Suppose $U$ contains no complete rational curves: 
  For $h \in \NT_J^F(U)$, and $\beta \in \NE(Y,\bbZ)$, by \cref{lem:sksm}, the fibre $\Sp^{-1}(h)$
is contained in the \'etale locus of $\Phi_i$.
We define $N_i(h,\beta)$ to be its length.

We note that a priori this
depends on the choice of gluing in the construction of $Z$ in \cref{const:Z}, we
will show in \cref{thm:ncwd} that, for sufficiently generic spines, this naive count is
independent of that choice.

\end{definition}

In the general case, we will only define the count of transverse spines:

\subsection{counts of spines } \label{sec:cospoa}
Throughout this subsection we assume \cref{ass:basicsetup}.

\begin{lemma} \label{lem:spinediscrete} 
  We consider realizable spines, i.e. $\Sp(\ISk) \subset \Sp(\Sk(U),\beta)$.
  $\Phi_i^{\trop}: \Sp(\ISk) \to \Sk \subset \oM_{0,N} \times U$ has finite
  fibres.
\end{lemma}
\begin{proof} This follows from \cref{lem:ISknodes}, \cref{lem:sksm}, and
  \cref{lem:source_of_skeletal_curve}
  -- note the spine
  depends only on the main component, $x_m$ in the notation of \cref{lem:sksm}, so the finiteness
  holds for the same reason finiteness holds in the no complete rational curves
  case of \cref{prop:Spprop}.
\end{proof}

For a non-realizable spine, we define the count to be zero. Otherwise, for
$S \in \Sp(\ISk)$, by \cref{lem:spinediscrete}, $\Sp^{-1}(S) \subset M(U,\beta)$ is a union
of connected components of the analytic fibre $\Phi_i^{-1}(\Phi_i^{\trop}(S))$. And this union
is proper by \cref{prop:Spprop}, and so a union of connected components of the analytic fibre
of
$$
\Phi_i: \oM_{0,N}(Z,D_{Z} ) \to \oM_{0,N}\times Z
$$
(notation as in \cref{def:Mdef}).  Let $F_S \coloneqq \Sp^{-1}(S) \subset M(U,\beta)$. 

Now by simple point set topology, \cite[9.9]{Keel_Yu_The_Frobenius}, for sufficiently small open
$S \in W \subset \oM_{0,N} \times Z$, if we take $\tW \subset \Phi_i^{-1}(W)$ the union
of connected components that intersect $F_S$, $\Phi_i: \tW \to W$ is proper. By \cite[1.1]{PY},
$\oM_{0,N}(Z,D_{Z} )$ has a virtual fundamental class (of degree the dimension of the target 
$W$) and thus $\deg(\Phi_i) \in \bbQ$ is well defined.

\begin{definition-lemma} \label{def:spinecount} Notation as immediately above. Define 
  $N_i(h,\beta) \coloneqq \deg(\Phi_i) \in \bbQ$. This agrees with \cref{def:naive_counts_general} if
  $U$ contains no complete rational curves.
\end{definition-lemma}
\begin{proof} When $U$ contains no complete rational curves $\Phi_i: \tW \to W$ is a finite \'etale map
  so the virtual degree is the naive degree. \end{proof}

The remaining results in this section now follow by the same argument used to prove their (nearly identical) counterparts in \cite{Keel_Yu_The_Frobenius}:

\begin{proposition}(see \cite[9.5]{Keel_Yu_The_Frobenius}) \label{prop:moving_w}
Assume $S$ is a transverse spine with respect to $\Wall_\beta$.
Let $\ow\in\Gamma\setminus\set{v_j | j\in \Bnd\setminus F}$ away from the nodes.
We glue $[0,w=+\infty]$ to $\Gamma$ along $0$ and $\ow$, extend $h$ constantly on the new leg, and obtain a new spine which we denote by $S_{\ow}$.
Then the count $N_w(S_{\ow},\gamma)$ is independent of the choice of $\ow\in\Gamma$.
\end{proposition}

\begin{definition} \label{def:naive_count_transverse}
By virtue of \cref{prop:moving_w}, we define $N(S,\gamma) \coloneqq N_w(S_{\ow},\gamma)$ (for any choice of $\ow$).
\end{definition}

\begin{theorem}(see \cite[9.14]{Keel_Yu_The_Frobenius}) \label{thm:forgetting_interior_marked_points}
Notation as in \cref{def:naive_count_transverse}.
Let $\Gamma^\Bnd \subset \Gamma$ be the convex hull of the $\Bnd$-type marked points.
If $N_i(S,\beta) \neq 0$ then $h$ factors through the canonical retraction $r:\Gamma \to \Gamma^\Bnd$.

Now assume $h$ factors through this retraction, and let $S^\Bnd$ be the restriction of $S$ to $\Gamma^\Bnd$.
Assume furthermore that $S^\Bnd$ is a transverse spine with respect to $\Wall_\beta$.
Then $N_i(S,\gamma) = N(S^\Bnd,\gamma)$.
In particular $N_i(S,\gamma)$ is independent of the choice of $i\in I$.
\end{theorem}

Exactly as in \cite[10.10]{Keel_Yu_The_Frobenius}, \cref{prop:rigidity_spine} and \cref{prop:Spprop} imply:

\begin{theorem} \label{cor:vary_lengths} Notation as in \cref{def:naive_counts_general}

$N_i(S,\beta)$ is locally constant on the transverse locus $\Sp^{\tr}(\Sk(U),\beta)$.

If we deform $S \in \Sp^{\tr}(\Sk(U),\beta)$ by varying the lengths of edges of $\Gamma$ on which $h$ is constant (including possibly shrinking them to points) we remain in $\Sp^{\tr}(\Sk(U),\beta)$ and the count stays the same.
\end{theorem}

\section{Gluing formula} \label{sec:gluing}

Here (and anywhere we count spines) we assume \cref{ass:vformskel}.
The gluing formula for naive counts follows from \cref{cor:vary_lengths}
as in \cite[\S 12]{Keel_Yu_The_Frobenius}, with minor modifications which we indicate.

Here we state the results (see \cite[Figures 12-13]{Keel_Yu_The_Frobenius} for pictures explaining the dense notation).

\begin{theorem} \label{thm:gluing_inside}
Let $\gamma\in\NE(Y)$.
Let $S^i=[\Gamma^i,(v^i_j)_{j\in J^i},h^i]$, $i=1,2$ be two transverse spines in $\Sk(U,\gamma)$ (see \cref{def:transverse}).
Let $p^i\in\Gamma^i$ be points in the interiors of edges such that $h^1(p^1)=h^2(p^2)\in \Sk(U) \setminus\Wall_\gamma$.
Let $S$ be the spine obtained by gluing $S^i$ at $p^i$.
Then we have \[
N(S,\gamma)=\sum_{\gamma^1+\gamma^2=\gamma} N(S^1,\gamma^1)\cdot N(S^2,\gamma^2).
\]
\end{theorem}
\begin{proof} We follow the proof of (the nearly identical) \cite[12.1]{Keel_Yu_The_Frobenius}, and its notation.
  We need a new argument that $f$ contracts $C_{\Delta}$ (because here we do not have the assumption that $U$ contains
  no complete rational curves). We have $f: C_{\Delta} \to U$. 
  Its spine (with domain the convex hull of its three marked points), is by assumption constant. Since the spine contracts
  to a point off of $\Wall_{\beta}$, this implies the entire IOS tropicalisation is constant. But then 
  $f(C_{\Delta})$ lies in a fibre of the SYZ fibration, and so is constant (as these fibres are affinoid). 
  Now the proof of \cite[12.1]{Keel_Yu_The_Frobenius}
  applies up to the
  displayed formula for $F_w(S,\gamma)$ as a disjoint union of products. We now apply (2) of
  \cite[Thm 1.1]{PY}. 
  \end{proof} 

  Now the next results follows by the same argument used for the (nearly identical)
  \cite[Thm 12.2]{Keel_Yu_The_Frobenius}.
  
\begin{theorem} \label{thm:gluing_concatenate}
Let $\gamma\in\NE(Y)$.
Let $S^i=[\Gamma^i,(v^i_j)_{j\in J^i},h^i]$, $i=1,2$ be two transverse spines in $\Sp(U,\gamma)$.
Let $p^i\in\Gamma^i$ be a finite 1-valent vertex, and $e^i$ the edge incident to $p^i$.
Assume $h^1(p^1)=h^2(p^2)$ and $w_{(p^1,e^1)}+w_{(p^2,e^2)}=0$.
So we can concatenate $S^1$ and $S^2$ at the vertices $v^1$ and $v^2$, and form a transverse spine which we denote by $S$.
Then we have \[
N(S,\gamma)=\sum_{\gamma^1+\gamma^2=\gamma} N(S^1,\gamma^1)\cdot N(S^2,\gamma^2).
\]
\end{theorem}

Next we observe that naive counts of sufficiently transverse tropical curves are independent of the choice of the {\it ends},
i.e.\ of the choice of $H \subset G \subset \Sk(U)$, and the gluing in \cref{const:Z} (so long as the volume form on $U$ extends).
This follows from the argument for the analogous results \cite[\S 13]{Keel_Yu_The_Frobenius}, once we establish:

\begin{lemma} \label{lem:straightspinecount} Let $\beta \in \NE(Y,\bbZ)$.
  Let $f:[0,1] \to \Sk(U)$ be piecewise affine with image in the interior of a maximal cone
  $\sigma \in \Sigma$, disjoint from $\Wall_\beta$.
If $N(f,\beta) \neq 0$ then $\beta =0$, $f$ is affine,
and $N(f,0) = 1$.
\end{lemma}
\begin{proof} Since the image is disjoint from $\Wall_{\beta}$ the ios tropicalisation of any contributing $[B \to Y]$ has no twigs,
by \cref{lem:tbal}, and thus $f$ is affine.
Moreover the ios tropicalisation is equal to $f$.
Since the image is in the interior of a maximal cone it follows that $\beta= [B \to Y] = 0$ (each irreducible component of the central fibre in a model for $B \to Y$ is contracted to the zero stratum corresponding to $\sigma$).
We can find a closed convex $\dim U$ dimensional $f([0,1] \subset G \subset \sigma^{\circ}$.
Then for any stable map $g\colon C \to Z$ contributing to the count, i.e.
any $g \in \Sp^{-1}(f)$, notation as in \cref{def:naive_counts_general},
$g\colon B \to Y$ has image in $G \coloneqq \tau^{-1}(\cG)$ (notation as in \cref{const:data1} ) and the full rational curve has image in $G$ union the two simple toric ends, $T_{f_0}$, $T_{f_1}$ (corresponding to the two ends points of the domain of the spine, $[0,1]$).
Now it follows from \cite[6.5]{Yu_Enumeration_of_holomorphic_cylinders_I} that there is an open analytic domain $A \times \bbP^1_{\an} \eqqcolon V \subset Z$, for an open polyannulus $A$, such that $g\colon C \to Z$ has image in $V$.
Now the count is the same as the toric case, and so follows from \cite[6.2]{Keel_Yu_The_Frobenius}.
This completes the proof.
\end{proof}

\begin{theorem} \label{thm:ncwd} Naive counts of transverse tropical curves are independent of the choice of {\it end} (as above).
More precisely,
given two choices of data $\cH_f \subset \cG_f \subset \Sk(U)$, and the gluing automorphism \cref{const:Z}, if $h$ is transverse for $\Wall_{\beta}$
for both choices, and in each case the volume form on $U$ extends,
then $N_i(h,\beta)$ is the same for either choice.
\end{theorem}
\begin{proof} This follows from the gluing formula exactly as in \cite[\S 13]{Keel_Yu_The_Frobenius}, using \cref{lem:straightspinecount}.
\end{proof}

The same argument shows:

\begin{theorem} \label{thm:oldnew} Notation as in \cref{thm:ncwd}.
Suppose $U$ contains an open algebraic torus.
Then for $h$ transverse with respect to $\Wall_{\beta}$ and transverse in the sense of \cite[4.20]{Keel_Yu_The_Frobenius},
$N_i(h,\beta)$ defined here is equal to the count as defined in \cite[9.3]{Keel_Yu_The_Frobenius}.
\end{theorem}

Now we can give a stronger form of \cref{lem:straightspinecount}

\begin{proposition} \label{prop:zeroclass} For a transverse spine $S$, $N(S,0) = 0$ unless $S \setminus \Bnd$
has image in the interior of a maximal cone, and is balanced.
In this case $N(S,0) = 1$.
\end{proposition}
\begin{proof} If $[B \to Y] = 0$, then the ios tropicalisation $h\colon \Gamma \setminus \Bnd \to \Sk(U)$ has image in the interior of a maximal cone, by \cref{prop:dccoeff}.
In particular, there are no twigs (because the valence one vertices other than the root must map to the singular locus of $\Sk(U)$, by balancing, \cite[4.1]{Yu_Tropicalization_of_the_moduli_space_of_stable_maps}),
thus $\Gamma^s \setminus \Bnd \to \Sk(U)$ is balanced, by \cref{lem:tbal}.
Now suppose $h\colon S\setminus \Bnd \to \Sk(U)$ is balanced, with image in a maximal cone, and we consider the moduli space,
with disk class $\beta = 0$, as in \cref{def:naive_counts_general}.
Then any $f \in \Sp^{-1}(h)$ has ios tropicalisation without twigs i.e.\ the ios tropicalisation is $S$, with image in the interior of a maximal cell, $\sigma^{\circ}$, and thus $f\colon B^{\circ} \to Y$
has image in $r^{-1}(\sigma^{\circ})$ (where the domain indicates the complement of the marked points mapping to $\partial Y$).
Thus the moduli space is exactly the same as in the toric case, and so by \cref{thm:oldnew} the result follows from \cite[6.2]{Keel_Yu_The_Frobenius}.
This completes the proof.
\end{proof}

\subsection{Gluing volume forms} \label{sec:gvf}

Let $A=\Sp k\braket{x_1,\lambda_1 x_1^{-1},\dots,x_n,\lambda_n x_n^{-1}}$ be an $n$-dimensional polyannulus, with $\lambda_i\in k$, $\abs{\lambda_i}<1$.
Let $dx/x$ be the standard volume form, i.e.\ the wedge product of $dx_i/x_i$.

\begin{lemma} \label{lem:dominant_term}
A function $f=\sum_{i\in\bbZ^n} c_i \bx^i$ on $A$ is invertible if and only if there exists $n\in\bbZ^n$ such that $\abs{(c_n \bx^n)(y)} > \abs{(c_i\bx^i)(y)}$ for all $i\neq n$ and $y\in A$.
\end{lemma}
\begin{proof}
The proof is parallel to the one-dimensional case in \cite[Lemma 9.7.1/1]{Bosch_Non-Archimedean_analysis}.
The ``if'' direction follows from the power series expansion of $(1+g)^{-1}$ for $\abs{g}_{\sup}<1$.
Next we explain the ``only if'' direction.
If every $\lambda_i=1$, it follows by reduction to $\Spec\tk[x_1,x_1^{-1},\dots,x_n,x_n^{-1}]$, since $f$ is invertible if and only if its reduction $\tf$ is invertible.
Note that every point $y\in A$ is contained in a sub-polyannulus of $A$ of thickness 0, which is isomorphic to $\Sp k\braket{x_1, x_1^{-1},\dots,x_n, x_n^{-1}}$ after a ground field extension.
Therefore, for every $y\in A$, the series $f=\sum_{i\in\bbZ^n} c_i \bx^i$ has a term indexed by $n(y)$ that is dominant at $y$, i.e.\ whose norm is greater than the norm of any other term.
Since $\lim_{i\to\infty}\abs{c_i}=0$, $n(y)$ takes finitely many values on $A$, and the subsets $A_{n_0}\coloneqq\set{y\in A | n(y)=n_0}$ gives a finite cover of $A$ by disjoint affinoid subdomains.
Since $A$ is connected, we deduce that $A_{n_0}$ is nonempty for exactly one $n_0\in\bbZ^n$, hence $n(y)$ is constant on $A$, completing the proof.
\end{proof}

\begin{lemma} \label{lem:vfskel}
A volume form $\omega=f dx/x$ on $A$ has skeleton $\Sk(A)$ if and only if $\abs{f}$ is constant on $A$.
\end{lemma}
\begin{proof}
The skeleton $\Sk(A)$ is a product of closed intervals.
By \cite[2.4.4(iii)]{Berkovich_Spectral_theory}, the Shilov boundary of $A$ is equal to the product of the endpoints of the intervals above, viewed as subsets of $A$.
In particular, it lies in $\Sk(A)$.
Hence the maximum of $\abs{f}$ on $A$ is attained on $\Sk(A)$.
Then the lemma follows from \cite[8.1]{Keel_Yu_The_Frobenius}.
\end{proof}

\begin{lemma} \label{lem:annc}
Let $f$ be a function on the polyannulus $A$.
The following are equivalent:
\begin{enumerate}[wide]
\item $\abs{f}$ is identically $1$ on $A$.
\item $f$ is invertible, and $\abs{f}$ is identically $1$ on $\Sk(A)$.
\item $f = c + h$ where $c$ is a constant of norm $1$, and $\abs{h}_{\sup}<1$ on $A$.
\item $f = \sum_{i \in \bbZ^n} c_i x^i$ with $\abs{c_0} = 1$, and $\abs{c_i} r^i < 1$ for all $r \in \Sk(A)$
\end{enumerate}
\end{lemma}
\begin{proof}
It is trivial that (1) implies (2), (3) implies (1) and (4) implies (3).
Since the Shilov boundary of $A$ lies in $\Sk(A)$, (2) implies (1).
Finally, it follows from \cref{lem:dominant_term} that (2) implies (4).
\end{proof}

\begin{proposition} \label{prop:volumeform}
Let $\omega = f dx/x$ be a volume form on the polyannulus $A$ with $\Sk(\omega) = \Sk(A)$.
There is an automorphism $\varphi$ of $A$ such that $\varphi^*(c\, dx/x) = \omega$ for some constant $c$.
Moreover we can pick $\varphi$ to be of the form \[ x_i \mapsto x_i \cdot (1 + h_i(x_i,\dots,x_n) ) \]
with $\abs{h_i}_{\sup}<1$.
$h_i$ satisfying the equivalent conditions of \cref{lem:annc} (e.g.\ $|h_i|$ is identically one on $\Sk(A)$).
\end{proposition}
\begin{proof}
Up to a change of coordinates, we can assume that $\abs{x_i}_{\sup}\le 1$ for all $i$.
We induct on $n$, the dimension of $A$.
By \cref{lem:vfskel}, $\abs{f}$ is a nonzero constant on $A$.
Up to rescaling, we can assume that $\abs{f}=1$.
We organize the power series expansion of $f$ as \[ f = c + K(x_2,\dots,x_n) + L(x_1,\dots,x_n) \]
where $c$ is the constant term, $L$ is the sum of terms that contain any non-zero power of $x_1$, and $K$ are the remaining terms.
By \cref{lem:annc}, we have $\abs{c}=1$, $\abs{K}_{\sup} < 1$, and $\abs{L}_{\sup} < 1$.
Let $\omega_i$ denote the standard volume form in the variables $x_i,\dots,x_n$.
By induction we can find an automorphism $\psi$ of the prescribed form so that $\psi^*(c \omega_2) =(c + K) \omega_2$.
Since $c+K$ is invertible, we can factor $f$ as \[f=(c + K)(1 + M(x_1,\dots,x_n))\]
where $M=L\cdot(c+K)^{-1}$, and $\abs{M}_{\sup}<1$.
By construction, every term in the power series expansion of $M$ has a non-zero power of $x_1$.
For each term of the form $c_j(x_2,\dots,x_n)x_1^j$ with $j\neq 0$, an antiderivative of $c_j(x_2,\dots,x_n)x_1^{j-1} d x_1$ is $\frac{1}{j}c_j(x_2,\dots,x_n)x_1^j$.
Summing over, we obtain an antiderivative $Z(x_1,\dots,x_n)$ of $M d x_1 / x_1$.
Since $\abs{x_1}_{\sup}\le 1$ by assumption, evaluating on the Shilov boundary points of $A$, we deduce that $\abs{Z_1}_{\sup} \le \abs{M}_{\sup}<1$.
Then $\exp(Z(x))$ is convergent on $A$.
Now let $\varphi$ be the automorphism sending $x_1\mapsto x_1(\exp(Z(x)))$, and $x_i\mapsto\psi(x_i)$.
We compute that \[ \varphi^*(c\, dx/x) = \varphi^*(dx_1/x_1) \wedge \psi^*(c\omega_2) = (1+M)(c + K)\omega_1 = \omega_1\]
This completes the proof.
\end{proof}

Let $J$ be the closed disk of radius one, with coordinate $t$.

\begin{proposition} \label{prop:famaut} Let $h_i \in \cO(A)^{\an}$ be an analytic function with $|h_i| < 1$, $i =1,\dots n$.
Then $X_i \to X_i \cdot (1 + h_i)$ induces an automorphism $h\colon A \to A$.
Let $J$ be a closed unit disk.
There is an automorphism $\widetilde h\colon A \times J \to A \times J$,
With $\pi_J \circ \widetilde h = \widetilde h$, whose restriction to one fibre is $h$,
and restriction to another fibre is the identity.
\end{proposition}
\begin{proof} Let $t$ be a coordinate on $J$.
$X_i \to X_i(1 + th_i)$ gives the desired $\widetilde h$.
\end{proof}

\section{Structure constants} \label{sec:sc}
Throughout the section we assume \cref{ass:basicsetup}

We fix non-zero $P_1,\dots,P_n,Q \in \Sk(U,\bbZ)$, and $\beta \in \NE(Y,\bbZ)$.
Our goal is to define the structure constant
$\chi(P_1,\dots,P_n,Q,\beta)$ as in \cref{const:structure_constants}.

\begin{assumption} \label{ass:dataforsc} 
We use the basic moduli space from \cref{def:MUk},
with $\Bnd = \{1,2,\dots,n\}$, $F = \{q\}$, $I = \emptyset$. And the function
$b \colon \Bnd \to \Sk(U,\bbZ)$ (from \cref{const:data})
is the obvious one, $b(i) \coloneqq P_i$.
We take $v_q \coloneqq -Q$ (notation as in \cref{const:data1})
noting that by our choice of $G$, $-Q \in T_G(\Sk(U))(\bbZ)$ makes sense.

\end{assumption}  

We can pick $\cG \subset(\Sk(U))$ (from \cref{const:data1})
near $Q$, satisfying \cref{ass:G_f}, sufficiently small so that:

\begin{assumption} \label{ass:wallass}
  Any wall $\fD \in \Wall_{\beta}$ (see \cref{def:walldef}) intersecting $\cG$
  has linear span containing the ray spanned by $Q$. 
\end{assumption}

\begin{lemma} \label{lem:sclemma} Let $f\colon S \to T_M$ be a map with domain a semi-stable circle (see just below \cref{def:VMdef}).
Then $(\tau \circ f)(S) \subset M_{\bbR}$ is a single point.
$f(S) \subset \tau^{-1}( \tau \circ f)(S)$.
\end{lemma}
\begin{proof}
The second statement is purely set theoretic, and obvious.
The first follows from the maximum principle, since the Berkovich boundary of a stable circle is a single point.
\end{proof}

Here we will sometimes abbreviate $\Sk \coloneqq \Sk(M_{0,N}) \times \Sk(U)$.
We stress here we use $M_{0,N}$, parameterizing smooth $k$-analytic curves, we are not allowing nodes. Note
$\Sk(M_{0,N}) = M_{0,N}^{\trop}$, see \cite[8.24]{Keel_Yu_The_Frobenius}. 
Recall from \cref{def:VMdef} and \cref{const:data1}, we have the locus $\cV_M \subset M_{0,N}^{\trop}$,
and convex polytopes $\cH_f \subset \Sk(U)$, with SYZ inverse images  $V_M \subset M_{0,N}$, $H_f \subset U$.

The main step in defining structure constants is:

\begin{theorem} \label{thm:scthm} Notation as
  immediately above. 
  Assume \cref{ass:wallass}. Pick $f \in F$ and take the corresponding $i \in F_s$ (see \cref{const:data}).
  We consider 
  $\Phi\colon \Phi_i: M(U,\beta) \allowbreak \to V_M \times H$, where the
  domain is defined in \cref{def:cmsm}, and $H \coloneqq H_f$. 
 
  The following hold
  \begin{enumerate}[wide]
    \item $\Phi^{-1}(\Sk_{\cV_M \times \cH}) \subset M^{\smot}(U,\beta)$, notation as in \cref{def:sot}. 
     \item $\Phi^{-1}(\Sk_{\cV_M \times \cH}) \to \Sk_{\cV_M \times \cH}$ is topologically
       proper, open, and when $U$ contains no complete rational curves, finite. 
     \item If $U$ contains no complete rational curves,
       $\Phi^{-1}(\Sk_{\cV_M \times \cH}) \subset M^{\sm}(U,\beta)$.
     \item For an open neighborhood  $p \in W \subset V_M \times H$ of
       $p \in \Sk_{\cV_M \times \cH}$,  let $\tW$ be the union of connected components of $\Phi^{-1}(W)$
       meeting $\Phi^{-1}(\Sk_{\cV_M \times \cH})$.  For any sufficiently small $W$ (containing $p$),
       $\Phi: \tW \to W$ is flat, open and proper. In the no rational curves case,
       it is finite and  \'etale, of degree
       independent of $p \in \Sk_{\cV_M \times \cH}$. 
  \end{enumerate}
\end{theorem}

\begin{remark} The retraction map is an isomorphism of $\Sk$ onto its image, and so in particular
  $$\Phi\colon \Sk_{\cV_M \times \cH} \to \cV_M \times \cH$$ is an isomorphism.
  We write $\Sk_{\cV_M \times \cH}$ instead of $\cV_M \times \cH$ to stress we view it as a locus in
  $V_M \times H \subset M_{0,N} \times U$.
\end{remark}

\begin{remark} \label{rem:degree}
  
The statement implies, by  \cite[9.9]{Keel_Yu_The_Frobenius}, that for any sufficiently small
open neighborhood $W \subset V_M \times H$ of $p \in \Sk_{\cV_M \times \cH}$,
and $\tW$ the union of connected components of $\Phi^{-1}(W)$ meeting
$\Phi^{-1}(\Sk_{\cV_M \times \cH})$, $\Phi\colon \tW \to W$ is proper.  Now we can define the degree
of $\Phi: \tW \to W$ using the virtual fundamental class exactly as in
\cref{def:spinecount}.
In the no rats case,
by (3) this count is naive. 
\end{remark}

Let $L \subset M$ (notation as in \cref{def:Mdef}) be the union of irreducible components which
intersect the Zariski open subset $M^{\smot}(U,\beta) \subset M(U,\beta)$.
We note that in the no rational curves 
case this implies they intersect $M^{\sm}(U,\beta) $  (which is quasi-smooth as
$\Phi\colon M^{\sm}(U,\beta) \to V_M \times H$ is \'etale).
Note $\Phi^{-1}(\Sk) \cap M(U,\beta) \subset M^{\smot}(U,\beta)$ by
\cref{lem:ISknodes}, 
so we have
$$
\Phi^{-1}(\Sk) \cap M(U,\beta) \subset M^{\smot}(U,\beta) \subset L.
$$

We begin with the following strengthening:

\begin{proposition} \label{thm:ISk} Notation as immediately above.
Consider $\Phi\colon L \to V_M \times H$.
$\Phi^{-1}(\Sk) \subset M^{\smot}(U,\beta)$. If $U$ contains no complete rational curves,
$\Phi^{-1}(\Sk) \subset M^{\sm}(U,\beta)$.
\end{proposition}

Note the proposition implies in particular that, in the no rational curves case,
$L$ is a space (as opposed to stack) in a Zariski open neighborhood of
$\Phi^{-1}(\Sk_{\cV_M \times \cH})$,
by \cref{def:Msd}.

\begin{proof}[Proof of \cref{thm:ISk} ]
  By \cref{lem:ISknodes}, 
  it is enough to show $\Phi^{-1}(\Sk) \subset M( U,\beta)$.

 First consider the no rational curves case. We can throughout replace $\smot$ by $\sm$. 
Note $L$ is equidimensional, of the same dimension as $\cV \times \cH$ (this was the reason we restricted to this union of irreducible components).
Moreover (because this holds for $M^{\sm}$), each irreducible component of $L$ dominates (the base curve) $J$.
It follows that each fibre, $L$ of $L \to J$
is equidimensional, of the same dimension as $V \times H$.
For the rest of the argument we work with one fibre. 
Thus, since $\Sk$ consists of Abhyankar points (note this fails for
$\oSk(\oM_{0,N}) = \oM_{0,N}^{\trop} \subset \oM_{0,N}$, this is why we are not allowing nodes)
$\Phi^{-1}(\Sk) \subset L$ consists of Abhyankar points by \cref{lem:dflat},
and thus these lie in any Zariski dense Zariski open set, e.g.  $M(U^{\sm},\beta) \cap L \subset L$.
This completes the proof.

Note, all we used here was that $L$ is equidimensional, same dimension as $\cV \times \cH$.
This is a lot weaker
than the assumption that $U$ contains no rational curves.
It holds so long as the generic element of $M(U,\beta)$ has stable domain.

For the general case, we can replace $L$ by an irreducible component. Let $\cC \to L$ be the domain of the
universal map, and $\cC_m \subset \cC$ the union of irreducible components which correspond to a vertex on the convex hull of
the marked points.  The restriction $f: \cC_m \to Z$ gives a family of stable maps, let $\beta'$ be the class. So this restriction
induces $L \to L'$ where the target is the analogous space but for the class $\beta'$.  Now by
\cref{lem:ISknodes}, $L'$ is equidimensional, same dimension as $\cV \times \cH$. So on
$\Phi^{-1}(\Sk) \cap L' \subset M^{\sm}(U,\beta')$.  In particular, for any proper Zariski closed
$W \subset  D$ (recall $D = Y \setminus U$),
$\phi_b^{-1}(W^{\an}) \cap \Phi^{-1}(\Sk) \cap L'  = \emptyset$,  for any $b \in \Bnd$, which implies the same statement for $L$
(the marked points are all on $\cC_m$).  In particular, we can take for $W$ the intersection of $D$ with the closure of
the union of all complete rational curves in $U$ ( that could occur in  a stable curve of class $\beta$, so
in particular $\beta$ minus
the class of the curve is effective). Similarly, if we take the closure of the union of complete rational curves in $U$
(with
class bounded by $\beta$ as immediately above), the marked points from $F \cup I$ will be disjoint from these.
We note this is similar to the $\tri$ condition in properness of the spine map,  \cref{prop:Spprop}. 

Now we claim $\Phi^{-1}(\Sk) \cap L \subset M(U,\beta)$.  We take $f \in \Phi^{-1}(\Sk) \cap L $.
By definition $L \cap M^{\smot}(U,\beta) \subset L$ is a Zariski (and so in particular, Berkovich) dense
open subset.  We can choose a net $f^{\lambda} \subset M^{\smot}(U,\beta) $ converging to $f$.  We decompose
the domain curves into the main component $C_m^{\lambda}$ and the twigs $T^{\lambda}$ (the closure of the complement of the
main component).  By the above
$f^{\lambda}|_{C_m^{\lambda}}$ is in $M^{\sm}(U,\beta')$, and this holds also for the limit (in particular the
domain curve does not degenerate, it remains $\bbP^1_{\an}$). 
The twigs of $f^{\lambda}$ map to complete rational curves in $U^{\an}$. So by the above,  the intersection points 
$T^{\lambda} \cap C_m^{\lambda}$ , which converge to the analogous points of the limit, do not approach the marked points. 
Now we can follow the ideas of the maximum principle argument in the proof of \cref{prop:Spprop} to show
that the limit {\it stays in $U$} (more precisely, in $M(U,\beta)$). 

\begin{remark} Note the spine is the spine of the main component $(f:C_m \to Z)$, and this converges to the
  spine of the main component of the limit (which is well behaved, a stable map in the smooth locus).  However
  We cannot directly
  apply (2) of \cref{prop:Spprop} because $\tri$ has to do with the $\beta'$-walls. Here we have something closely
  related, the closure of the locus of complete curves.  The same maximum principle argument applies.
\end{remark}
\end{proof} 

\begin{remark} Here, because we are focusing on structure constants, we have
  the special fixed data as in \cref{ass:dataforsc}. But the proof of \cref{thm:ISk} works for
  any fixed data as in \cref{const:data1}, with $\Phi := \Phi_i$ for any $i \in F_S \cup I$,
  notation as in \cref{const:data}. \end{remark}

\begin{lemma} \label{lem:dflat}
Let $\Phi\colon Y \to X$ be a map of $k$-analytic spaces with $X$ reduced.
Then $\Phi$ is flat at the inverse image of any Abhyankar point.
If $X$ and $Y$ are equidimensional of the same dimension, then the inverse image of every Abhyankar point consists of Abhyankar points.
\end{lemma}
\begin{proof}
The first statement is \cite[10.3.7]{Ducros_Families_of_Berkovich_spaces}, and this implies the second.
\end{proof}

\subsection{Proof of \cref{thm:scthm}}
It is enough to prove the analogous statement with $N$ replaced by $N'$, notation as in \cref{def:VMpdef}.
This holds because $\Phi\colon M(U,\beta,N') \to \cV_M' \times \cH$ is the base extension of $\Phi\colon M(U,\beta,N) \to \cV_M \times \cH$ with respect to the smooth surjection $\cV_M' \to \cV_M$.

By (3) of \cref{prop:cad},
$\Phi\colon M' \to \cV_M \times \cH$
is of finite type, and thus topologically proper.
Thus the same holds for $\Phi\colon M' \cap L \to \cV_M \times \cH$, since $L \subset M$ is (even Zariski) closed.
So restricting gives:
\begin{claim} \label{cl:PhionG}
$\Phi\colon L \cap M' \cap \Phi^{-1}(\Sk) \to \Sk_{\cV_M \times \cH}$
is topologically proper.
\end{claim}

\begin{claim} \label{cl:PhionGcM'}
$$
\Phi\colon M'(U,\beta) \cap \Phi^{-1}(\Sk) \to \Sk_{\cV_M \times \cH}
$$
is topologically proper.
\end{claim}
\begin{proof}
$$
M'(U,\beta) \cap \Phi^{-1}(\Sk) \subset M(U,\beta) \cap \Phi^{-1}(\Sk) \subset M^{\smot}(U,\beta) \subset L $$
by \cref{lem:ISknodes}, and $L \cap \Phi^{-1}(\Sk) \subset M(U,\beta)$ by \cref{thm:ISk}.
So $M'(U,\beta) \cap \Phi^{-1}(\Sk) = L \cap M' \cap \Phi^{-1}(\Sk)$,
and so the Claim follows from \cref{cl:PhionG}.
\end{proof}

\begin{claim} \label{cl:secondcl}
$L \cap \Phi^{-1}(\Sk) \subset M'$.
\end{claim}

Assume we have proven the claim.
Then by \cref{thm:ISk} and \cref{lem:ISknodes},
$$
M(U,\beta) \cap \Phi^{-1}(\Sk) = L \cap M(U,\beta) \cap \Phi^{-1}(\Sk) = M'(U,\beta) \cap \Phi^{-1}(\Sk)
$$
and so by the above $\Phi^{-1}(\Sk) \cap M(U,\beta) \to \Sk_{\cV_M \times \cH}$ is topologically proper,
and quasi-finite in the no rats case, because then we are in the locus $M^{\sm}(U,\beta)$ where $\Phi$ is \'etale.
Now the result follows from simple point set topology, see \cite[9.9]{Keel_Yu_The_Frobenius}, and the proof of \cite[9.11]{Keel_Yu_The_Frobenius}. These imply $\tW \to W$ is proper.
By \cref{lem:dflat} that $\Phi^{-1}(\Sk)$ is contained in the
flat locus (and flat, and no-boundary, implies open).  Finally, in the \'etale case, we check the degree
of $\tW \to W$ is independent of $p$. But now if $J \subset W$ is a connected open subset, then
$\tW_J \to J$ is proper and open, so each connected component of the domain surjects, and so if $J$
meets $\Sk$, each connected component of $\tW_J$ meets $\Phi^{-1}(\Sk_{\cV_M \times \cH})$. Thus if we apply
the construction to $q \in J$, we obtain $\tW_J \to J$. As $V_M \times H$ is connected this implies
the degree is independent of $p$.

So it is enough to establish \cref{cl:secondcl}.

\begin{proof}[Proof of \cref{cl:secondcl}]
Let $L$ be a fibre of $L \to J$.
It is enough to prove \begin{claim} \label{cl:secondclL}
$L \cap \Phi^{-1}(\Sk) \subset M'$.
\end{claim}

The argument follows the proof of \cite[11.5]{Keel_Yu_The_Frobenius}.
By \cref{thm:ISk} $\Phi^{-1}(\Sk) \subset M^{\smot}(U,\beta) \subset M^{\sot}(U,\beta)$.
By the last inclusion the semi-stable end disk, $E$ (notation as in \cref{sec:bms}, we are leaving off the subscript as there is only one end) is an honest disk.
By definition the end disk $f:E \to Z$, has image in the toric end,
and the punctured disk, $E^* \coloneqq E \setminus \{e\}$, has image in $T_M$.
It follows (by balancing)
that the spine of the punctured disk,
whose domain is described by \cref{def:VMdef} and \cref{def:VMpdef},
factors through $[\os,e]$ and $\Sp(f)\colon [\os,e] \to \oM_{\bbR}$ is affine, with derivative $Q$, over all of $M^{\sm}$.
We consider $f\colon C \to Z$ in $M^{\sm}(U,\beta)\cap \Phi^{-1}(Sk)$ and show it lies in $M'$.
We note, as in \cite[11.5]{Keel_Yu_The_Frobenius}
that this is equivalent to the spine having no twigs attached along $[\oa,\ob]$, notation as in \cref{def:VMpdef}. We use the same ideas as in the proof of \cite[Claim 11.6]{Keel_Yu_The_Frobenius}.
Note the only twig that could attach at $\os$ is associated to a wall
$\fD \in \Wall_{\beta}$ with linear span containing the vector $Q$, by \cref{ass:wallass}.
$\Phi$ is flat  in a neighborhood of $\Phi^{-1}(\Sk)$ by \cref{lem:dflat}, with
no boundary by \cref{def:bms}. 
It follows that $M^{\sm}(U,\beta) \cap \Phi^{-1}(\Sk) \to \Sk$ is open,
see \cite{Ducros_Families_of_Berkovich_spaces}.
So we can choose,
by \cref{ass:wallass},
a net $f_{\lambda} \to f$ such that $(\tau \circ f_{\lambda})(s) \not \in \Wall_k$.
Near $\os$ the spine $\Sp(f_{\lambda})$ (note these are all skeletal curves by \cref{lem:source_of_skeletal_curve}) runs
{\it parallel} to any possible wall $\fD$ as above. 
Now it follows (just as in the proof of \cite[Claim 11.6]{Keel_Yu_The_Frobenius})
from \cref{lem:epsilon} that $f_{\lambda}([\oa,\ob]) \cap \Wall_k = \emptyset$,
see \cite[Figure 10]{Keel_Yu_The_Frobenius}.
Thus there can be no twigs attached along this interval.
That implies the same for the limit, by continuity of the tropical curve,
\cite[8.1]{Yu_Tropicalization_of_the_moduli_space_of_stable_maps}.
This completes the proof.
\end{proof}

\subsection{Definition of structure constants}

We take non-zero $P_1,\dots,P_n,Q \in \Sk(U,\bbZ)$, and $\beta \in \NE(Y,\bbZ)$.
Our goal is to define the structure constant
$\chi(P_1,\dots,P_n,Q,\beta)$ as in \cref{const:structure_constants}.

We use the basic moduli space from \cref{def:MUk},
with $\Bnd = \{1,2,\dots,n\}$
$F = \{q\}$, $I = \emptyset$.

We can pick $\cG \subset(\Sk(U))$ near $Q$, as in \cref{ass:G_f}, sufficiently small so that:

\begin{assumption} \label{ass:wallassG} If $(\bbY,\bbD)$ is almost minimal (see \cref{def:am}), we take $\cG$ in the interior of $\star(Q)$ (the union of open cones whose closure contains the ray spanned by $Q$ in $\Sigma$).
We require the linear span of any wall of $\Wall_\beta$ that intersects $\cG$ contains the line spanned by $Q$.

In the general case, we pick $\cG$ in the interior of one of the maximal cones of $\Sigma$ that contain $Q$, sufficiently general so it is disjoint from $\Wall_{\beta}$ from \cref{def:walldef}.
\end{assumption}

For the basic moduli space we take $v_q \coloneqq -Q$, noting that by our choice of $G$, $-Q \in T_G(\Sk(U))(\bbZ)$ makes sense.

We consider $\Phi\colon \tW \to W$ as in \cref{thm:scthm}.

\begin{definition-lemma} \label{prop:thedegree} 
  Notation and assumptions as in \cref{thm:scthm}.  By the properness in (4) of \cref{thm:scthm}, the
  degree $\deg(\tW/W)$ is well defined just as in \cref{def:spinecount}. The same argument used
  to prove the independence of the degree in the no rational curves case  (in the proof of \cref{thm:scthm}) shows the
  virtual degree is independent of $p$ in general. 
  When $U$ has no complete rational curves,
  this is the degree of this finite \'etale map, the naive length of any of the fibres.
    \end{definition-lemma}

The point of this section is to prove:

\begin{proposition} \label{prop:scwd} Notation as immediately above.
The degree of $\Phi\colon \tW \to W$ depends only on $P_1,\dots,P_n,\beta$.
\end{proposition}
\begin{proof} A priori this degree could depend on the choice of $\cG=\cG_f$ from \cref{const:data1} and for given choice of $\cG$,
the gluing isomorphism we use in \cref{const:Z}.
First the degree is independent of varying the gluing: Note that the degree decomposes into a sum of counts of
transverse spines,
see \cref{def:spinecount}, by \cref{prop:transversality}, and these counts are independent of the gluing
by \cref{thm:ncwd}. 

It is easy to see the degree is independent of shrinking $\cG$, so all that remains is to show it is independent of the choice of maximal cone (in whose interior $\cG$ lies).
This is clear when $(\bbY,\bbD)$
is almost minimal, \cref{def:am} , since we can move from one maximal cone to the adjacent one across the codimension one cone $\rho$ (see (2) of \cref{ass:G_f}).
By \cref{prop:tbeta} to check the independence we can replace $\bbY$ by a toric blowup.
Then we can assume $\bbY$ satisfies \cref{prop:modelprop} (i.e.\ $\bbY$ is $\tbbY$ of that proposition).
But now the result follows from the almost minimal case.
This completes the proof.
\end{proof}

In view of \cref{prop:scwd} we define:

\begin{definition} \label{def:structureconstants}
Notation as in \cref{prop:scwd}.
We define $\chi(P_1,\dots,P_n,Q,\beta)$ to be the degree of $\Phi\colon \tW \to W$.
\end{definition}

\begin{remark} \label{rem:formulate}
  In the no rats case, we can formulate the definition without reference to the open set $W$:
By \cref{thm:scthm} $\Phi^{-1}(\Sk_{\cV_M \times \cH}) \subset M(U,\beta)$
is contained in the \'etale locus of $\Phi$, and so the analytic degree makes sense at any point -- namely the length of the fibre of $\Phi$ at this point.
$\chi(P_1,\dots,P_n,Q,\beta)$ is equivalently defined as the sum of the analytic degrees at every point in the fibre over any $p \in \Sk_{\cV_M \times \cH}$, i.e.\ the length of this fibre.
As the map is \'etale, after base extension the fibre will be a discrete set and this length is its cardinality -- thus our structure constants are naive counts.
\end{remark}

\begin{definition}[Structure Disks] \label{def:structuredisk}
Notation as in \cref{prop:scwd}.
By a {\it structure disk}
contributing to $\chi(P_1,\dots,P_n,Q,\beta)$,
we mean $[f:(B,p_1,\dots,p_n,s) \to Y^{\an}]$ for $f \in \Phi^{-1}(\Sk_{\cV_M \times \cH})$
as in \cref{thm:scthm}.
We note by \cref{lem:restrict_to_skeleton} that these (or more precisely the extension $[f:C \to Z]$,
but see also \cref{rem:lskel}) are skeletal, and moreover the main component, see \cref{def:sot},
is free by \cref{lem:sksm}. 
We note that, by \cref{prop:constclass}, if we choose a general such $f$ (specifically, if $f(s) = (\tau \circ f)(s) \not \in \Sigma^{d-1}$) then $[f\colon B \to Y^{\an}] = \beta$. 

By the {\it monoid of disk classes}, $\DC(\bbY,\bbD)$, we mean the submonoid of $\NE(\bbY_s,\bbZ)$ generated by classes of structure disks, or equivalently $\beta \in \NE(\bbY_s,\bbZ)$ such that $\chi(P_1,\dots,P_n,Q,\beta) \neq 0$ for some choice of $P_i,Q$.
By \cref{prop:scwd}, $\DC$ depends only on $(\bbY,\bbD)$.
By construction the mirror algebra is defined over $Q[\DC(\bbY,\bbD)]$ (for $Q$ either $\bbQ$ or $\bbZ$
as in \cref{not:coeffs}). 
\end{definition}

\subsection{disks with cycle supported on $1$-strata}

Our principal interest is in structure disks, see \cref{def:structuredisk}, but most of the analysis does not require the completion of the disk to a rational curve, so we work in the context of \cref{def:iostrop}, for $U \subset Y$ defined as in \cref{ass:basicsetup}.

We will say that a spine $h\colon \Gamma \to \oSk(U)$ is transverse to $\Sigma_{(\bbY,\bbD)}$ if it satisfies \cref{def:transverse},
but with walls being just the codim one cones of this dual fan.

\begin{definition} \label{def:1scycle}
Let $h\colon \Gamma \to \oSk(U)$ be a balanced spine, with $h(\Gamma \setminus \Bnd) \subset \Sk^{\sm}(U)$ in the affine structure of \cref{sec:integeraffine}), transverse to $\Sigma_{(\bbY,\bbD)}$.
Then there is an associated cycle supported on the $1$-skeleton of $D$:

\[Z(h)\coloneqq\sum_{x\in h\inv(\Sigma^1_{(\bbY,\bbD)})} \abs{e_{h(x)}(d_x h)} Z_{h(x)} \in Z_1(Y),\]
where $\Sigma^{1}_{(\bbY,\bbD)}\subset\Sigma_{(\bbY,\bbD)}$ denotes the union of codimension 1 cones, $d_x h$ denotes the derivative at $x$, $e_{h(x)}$ denotes the primitive dual class vanishing on the codimension one cone $h(x) \in \sigma \in \Sigma^{d-1}_{(\bbY,\bbD)}$ and $Z_{h(x)}$ denotes the 1-stratum of $D$ corresponding to the codim one cone containing $h(x)$.
\end{definition}

\begin{proposition} \label{prop:clb} Notation as in \cref{def:1scycle}.
The $1$-cycle $[h\colon B \to Y] - Z(h) \in Z_1(\bbY_s,\bbZ)$ is effective.

If $[h\colon B \to Y] = Z(h)$ in $A_1(\bbY_s,\bbZ)$ then the two cycles are equal in $Z_1(Y,\bbZ)$.
\end{proposition}
\begin{proof} The difference of cycles is effective by \cref{prop:dccoeff},
which gives the first statement.
The second follows (if an effective cycle on a projective variety is numerically trivial, then it is zero as a cycle).
\end{proof}

\begin{proposition} \label{prop:notwigs1skel} Let $f:(B,p_1,\dots,p_k) \to Y$ be a semi-stable disk as in \cref{def:iostrop}.

Assume that the associated spine (see \cref{def:iostrop}) is transverse to $\Sigma_{(\bbY,\bbD)}$.
Then the disk class $[f] \in Z_1(\bbY_s,\bbZ)$ is supported on the union of $1$-strata of $\bbD$ if and only if the ios tropicalisation has no twigs.
In this case the spine is balanced,
and $[f\colon B \to Y] = Z(\Sp(f)) \in Z_1(Y,\bbZ)$.
\end{proposition}
\begin{proof} Since the spine is by assumption transverse to $\Sigma_{(\bbY,\bbD)}$, it follows from the definition of the ios tropicalisation $h\colon \Gamma \to \obbR^{\bbD^{\ess}}$, see \cref{def:iostrop}, that the contribution of the spine to the cycle $[f\colon B \to Y]$ is supported on $1$-strata.
On the other hand, by balancing, valence one vertices of twigs, other than their root, map to $\Sing(\Sk(U)) \subset |\Sigma^2_{(\bbY,\bbD)}|$, thus the corresponding irreducible component in the central fibre of a model for $f$
DOES NOT map into a $1$-stratum.
Thus the cycle is supported on the $1$-skeleton if and only if there are no twigs.
In this case the spine is balanced.
Finally $[f] = Z(\Sp(f))$ by \cref{prop:dccoeff}.
\end{proof}

\begin{definition} \label{rem:pracsmooth} By the {\it practical smooth locus} of the $(\bbY,\bbD^{\ess})$ SYZ fibration we mean the complement in $\Sk(U) = |\Sigma_{(\bbY,\bbD^{\ess})}|$
of the union of all closed cones of codimension at least two,
and all codimension one cones which {\it DO NOT} correspond to almost minimal (essential) $1$-strata, \cref{def:am}.
\end{definition}

\begin{proposition} \label{prop:1stratacount}
Let $h\colon \Gamma \to \oSk(U)$ be a spine transverse to $\Sigma_{(\bbY,\bbD)}$, with $N(h,Z(h)) \neq 0$.
The following hold:
\begin{enumerate}[wide]
\item The ios tropicalisation of any contributing disk has no twigs.
\item $h$ is balanced.
\item If $h(\Gamma)$ is contained in the {\it practical smooth locus}
of the $(\bbY,\bbD^{\ess})$ SYZ fibration, see \cref{prop:TonySYZ} and \cref{rem:pracsmooth}
(for example if the SYZ fibration is smooth outside the codimension two skeleton of $\Sigma_{(\bbY,\bbD^{\ess})}$, as holds if $U \subset Y$ is almost minimal, \cref{def:am})
then $N(h,Z(h)) = 1$.
\end{enumerate}

\end{proposition}
\begin{proof}
Let $f\colon B \to Y$ contribute to $N(h,Z(h))$.
By \cref{prop:clb},
$[f] = Z(h) \in Z_1(Y,\bbZ)$, thus the ios tropicalisation of $f$ has no twigs, and in particular $h$ is balanced.

Now suppose $h(\Gamma)$ is in the practical smooth locus for the SYZ fibration.

By cutting up $h$ and using the gluing formula we reduced to the case when either $h$ has image in the interior of a maximal cone of $\Sigma_{(\bbY,\bbD^{\ess})}$, or $h$ is an affine segment crossing a single codimension one wall (corresponding to an almost minimal $1$-stratum).
In either case we can find convex $\cG \subset \Sk^{\sm}(U)$, as in \cref{const:data1}.
Now consider a disk with this spine and class.
By \cref{prop:notwigs1skel} the ios tropicalisation has no twigs, i.e.\ is equal to $h$.
It follows that $f\colon B \to Y$ has image in $\tau^{-1}(\cG) \subset U$, and in particular lives in the analytification of a torus.
We thus reduce to the torus case.
Now (3) follows by \cite[Lemma 6.2]{Keel_Yu_The_Frobenius}.
\end{proof}

\section{Basic convexity, and the mirror algebra}
We take $(\bbY,\bbD)$ as in \cref{ass:basicsetup}. Let $\bbF = \bbZ - \bbP$ be a Cartier divisor decomposed into
its zeros and poles and $F = Z - P$ the generic fibre. See \cref{sec:BTS}.  
There is a natural $\val F: Y \setminus (|Z| \cap |P|) \to [-\infty,\infty]$ 
(finite valued off $|F|$). 
 Let $F^{\trop}: |\Sigma_{(\bbY,\bbD^{\ess})}| \to \bbR$ be
 the restriction (note $|\Sigma| \subset U$ consists of Abhyankar points, and so is disjoint from
 $|Z| \cup |P|$). 

We have the following, which we call {\it basic convexity}:

\begin{proposition} \label{prop:basicconvexity} Notation as immediately above.
  Let $E$ be a compact $k$-analytic curve, $f: E \to U$ an analytic map and $\tau: U \to |\Sigma_{(\bbY,\bbD^{\ess})}|$
  the Berkovich retraction. Assume 
  $f(E) \cap (|Z| \cap |P|) = \emptyset$ and that $f$ maps no irreducible component of $E$ into $|F|$.
  Consider the following conditions:

  \begin{enumerate}
  \item Each point of
    $\tau(\partial(E))$ is contained in the interior of a maximal cone of $\Sigma_{(\bbY,\bbD^{\ess})}$.
  \item $\val F$ factors through $\tau$ in a neighborhood of $\tau(\partial E)$.
  \item $\tau \circ f$ is constant on the closure of every connected component of $E \setminus \Sk(E)$.
    \end{enumerate}

    There exists a closed piecewise affine $W \subset \Sk(U)$ of dimension at most $n-1$
such that (1-3) holds so long as $(\tau \circ f)(\partial E) \subset W^c$.

If (1-3) hold then

$$
\sum_{p \in \partial(E)} D_p(F_E^{\trop}|_{\Sk(E)}) = \deg(f^*(F)) - c_1(F) \cdot [f:E \to Y].
$$
where $F_E^{\trop} \coloneqq F^{\trop} \circ \tau \circ f : E \to \bbR$, and $D_p$ indicates the derivative along the unique
edge of $\Sk(E)$ incident to $p$ (or zero if there is no such edge).
\end{proposition}

\begin{proof} 
  That (1-3) imply the displayed formula follows from \cite[15.6]{Keel_Yu_The_Frobenius} and its (very
  simple) proof.

  We check we can find $W$ as in the final statement.   (1) holds as soon as $W$ contains $|\Sigma_{(\bbY,\bbD^{\ess})}^{n-1}|$. 
(2) holds so long as $W$ contains
$(\Sigma^F)^{n-1}$ for $\Sigma^F$ as in \cite[15.5]{Keel_Yu_The_Frobenius}. (3) holds as long as there
are no twigs (as in \cref{def:twigdef}) attached to $f(\partial E)$. This is guaranteed as long as $W$ contains
$\Wall_k$ for any  $k \geq [f:E \to Y] \cdot A$ (for $A$ ample on $\bbY_s$).  This completes the proof.
\end{proof}

We can apply basic convexity to the main component of a structure disk contributing to a structure
constant, \cref{sec:sc}, or a count of a spine, \cref{sec:cospoa}.  We have for example:

\begin{corollary} \label{cor:filtration} Let $f$ be a non-zero rational function on $\bbY$, whose polar locus
  is contained in $|\bbD|$.  Suppose
  $\chi(P_1,\dots,P_n,Q,\gamma) \neq 0$. Then
  $$
  \sum_i f^{\trop}(P_i) \leq f^{\trop}(Q).
  $$
  Equivalently
  $$
  |f(Q)| \leq \prod_i |f(P_i)|.
  $$
 
  Moreover, if we have equality, the structure constant can be computed with disks whose
  image lies in the non-vanishing locus of $f$.
\end{corollary}

\begin{proof}  We apply \cref{prop:basicconvexity} to the body, $B_m \coloneqq B \cap C_m$
  of the main component $g_m: C_m \to Y$
  for $g: C \to Y$ contributing to the structure constant --  more precisely, to the complement, $E \subset B$
  of the union of sufficiently small open disks around
  the marked points $p_i$ (to obtain a compact curve mapping into $U$)
  with $\bbF$ the Cartier divisor of zeros and poles of $f$, which is of course linearly equivalent
  to zero.

  $g_m$ is free, so no component maps into $|F|$ (and $P \cap U = \emptyset$). 
  We can assume the spine is transverse (by \cref{prop:transversality})
  and so (1-3) hold.
  
  We have equality in the displayed formula of \cref{prop:basicconvexity}
  iff $g_m(E)$ is disjoint from the zeros of $f$.  Note $|f|$ is constant
  on the twigs (i.e. the union of components other than those of $C_m$)
  of the structure disk, because these map to complete curves in $U$. So we have equality iff the full $g: B \to Y$
  is disjoint from the zero locus.
\end{proof} 

\cref{cor:filtration} will give a filtration on the mirror algebra, with associated graded a mirror
algebra for the non-zero locus of $f$. See \cref{rem:associatedgraded}.

The following is the key for finite generation of the mirror algebra (under the right affiness assumptions): 
\begin{corollary} \label{cor:bound} Let $(\bbY,\bbD)$ be as in \cref{ass:basicsetup}.
  Fix $P_1,\dots,P_n \in \Sk(U)(\bbZ)$
  The following hold:
  \begin{enumerate}
    \item 
  There are at most finitely many $Q \in \Sk(U)(\bbZ)$ such that
  $\chi(P_1,\dots,P_n,Q,\gamma)$ is non-zero for some $\gamma \in \NE(Y,\bbZ)$.
\item If there is a Cartier divisor $\bbF$, ample (relative to $\bbY \to \obbY$), such that
  for the induced Cartier divisor $F$ on $Y$, $-F|_U$ is effective, and moreover there are no
  complete rational curves in $U$, then there are at most finitely
  many pairs $(Q,\gamma)$ such that $\chi(P_1,\dots,P_n,Q,\gamma) \neq 0$.
 \end{enumerate}
\end{corollary}
\begin{proof}
   
  For (1) we use the rational functions $f_1,\dots,f_k$ from \cref{ass:basicsetup}. The final
  inequality implies $Q$ is one of the finitely many integer points in a bounded subset of $\Sk(U)$
  (defined by the inequalities $|f_j| \leq \prod_i |f_j(P_i)|$).
  For (2) we apply \cref{prop:basicconvexity} to the body, $f: E \to Y$, of the main component of
  a structure disk, as in the proof of \cref{cor:filtration}. 
  
  We obtain
  $$
  c_1(\bbF) \cdot [f: E \to Y] = \deg(f^*(F)) + \sum_i F^{\trop}(P_i) - F^{\trop}(Q) \leq  \sum_i F^{\trop}(P_i) - F^{\trop}(Q).
  $$
  So this bounds the degree of the main component of the body disk, and this is the class of the disk if
  there are no complete curves. So there are only finitely many possibilities for $\gamma$ given
  $P_1,\dots,P_n,Q$. Now (2) follows from (1). This completes the proof.
  \end{proof}

  \begin{remark} \label{rem:noway}
    The above are key to all our applications. We do not see how to give a Gross-Siebert style
    formulation (and as a result do not know how to deduce our applications from Gross-Siebert).
    Consider e.g.  \cref{prop:basicconvexity}: the
first term on the RHS of the displayed formula is the degree of a Cartier divisor on a disk, there is no disk in
the Gross-Siebert theory: the associated GS object would be
the punctured log curve associated to the central fibre of a formal model for the disk. This log curve does 
not know about this term. 
\end{remark}

\section{Associativity, Torus action and finite generation} \label{sec:torus_action}

\begin{definition}[Mirror Algebra] \label{const:algebras}
  We assume we have $(\bbY,\bbD)$ as in \cref{ass:basicsetup}. We take $A$ as in \cref{eq:ma}.
Let $m \subset R_{\bbY_s} \eqqcolon R$ be the maximal monomial ideal (generated by all $z^{\gamma}$ over all  non-zero
effective curve classes). By (1) of \cref{cor:bound}, for each $n > 0$ the structure constants
\cref{def:structureconstants} determine
a commutative $R/m^n$ algebra structure on $A \otimes_R R/m^k$, and on $A$ itself if the conditions
of (2) of \cref{cor:bound} hold.
\end{definition} 

\begin{proposition} The mirror algebra is associative.
\end{proposition}
\begin{proof} The argument in \cite[\S 14]{Keel_Yu_The_Frobenius} carries through without change.
\end{proof}

\begin{corollary} \label{prop:filtration} Let $f$ be as in \cref{cor:filtration}.
  For each integer $n$ let 
  let $F^n_f \subset A$ be the free $R$ submodule with basis those $\theta_P$ with
  $f^{trop}(p) \geq -n$. Use the same notation for the induced free $R/m^k$ submodule of $A/m^k A$.

  This makes $A/m^k A$ into a filtered algebra 
  $F^k_f \cdot F^j_f \subset F^{k + j}_f$.

  When $f \in \cO^{\an}(U)$ is invertible, the filtration gives a $\bbZ$-grading (with degree $n$ graded part
  the submodule with $f^{trop}(p) =  -n$).

  The analogous statements hold for $A$ itself under the conditions (2) of \cref{cor:bound}. 
\end{corollary}
\begin{proof}  The first statement is immediate from \cref{cor:bound}, and implies the second
  (by applying the first to $f$ and $1/f$).
\end{proof}

\begin{remark} \label{rem:associatedgraded}
  Notation as in \cref{prop:filtration}. Specialize to the case of \cref{ex:POA} (a version will hold
  in general, but is easier to formulate in this case). 
  Consider the associated graded algebra
  $\oplus F_k/F_{k-1}$.  By \cref{cor:filtration}, the multiplication rule is computed with  structure disks
  that map into the non-vanishing locus $\bbU^{\an}_f \coloneqq \{x \in \bbU^{\an}| f(x) \neq 0\}$. Thus the associated
  graded is {\it the mirror algebra} for $\bbU_f \subset \bbY$. Italix here because $\bbU_f$ need not be log CY
  (it might e.g. have log general type), so a priori our construction does not apply.
  But $\bbU_f$ has volume form (the restriction of the volume form of
  $U$), with skeleton $\Sk(\bbU)$, there is a proper exhaustion map as in \cref{ass:basicsetup} (by adding $f$ to
  the restriction of the analytic functions on $U$),
  and the associated graded has structure constants
  exactly like \cref{def:structureconstants}  applied to $\bbU_f \subset \bbY$.  The equality shows that these structure
  constants define an associative algebra.

  Now by the Rees construction (which applies to any filtered ring),
  this gives a degeneration of the mirror algebra for $\bbU$ to the mirror algebra
  for $\bbU_f$. E.g. when $\bbU_f$ is an algebraic torus this gives a degeneration to the ring of Laurent polynomials.
  This is how the non-degeneracy in the Frobenius Structure Conjecture is proven in \cite{Keel_Yu_The_Frobenius}. 
  \end{remark}

\subsection{The torus action} \label{sec:ta} 

Here we explain the natural torus action on the mirror algebra, extending the results
of \cite[\S 16]{Keel_Yu_The_Frobenius}, whose notation we follow.

We take $(\bbY,\bbD)$ as in \cref{ass:basicsetup}. We take $I_D$ to be the set of irreducible
components of $\bbD$. 

We have an embedding
\[\Sk(U,\bbZ)\simeq\Sigma^\ess_{(\bbY,\bbD^{\ess})}(\bbZ)\subset\Sigma_{(\bbY,\bbD)}(\bbZ)\subset\bbZ^{I_D}=
  \Hom(I_D,\bbZ)\]
which we denote by $w$.
We also denote by $w$ the map \[w\colon N_1(\bbY_s)\longrightarrow\bbZ^{I_D}, \quad
  \gamma\longmapsto(\gamma\cdot \bbD_i)_{i\in I_D}.\]
Let $\bbT_D\coloneqq\Spec(Q[\bbZ^{I_D}])$ be the split torus with character group $\bbZ^{I_D}$
(with $Q$ as in \cref{not:coeffs}). 
Then $w\colon N_1(\bbY_s)\longrightarrow\bbZ^{I_D}$ induces a canonical homomorphism
$T_D \to T^{N_1(\bbY_s)}$, and thus an action of $\bbT_D$ on $\Spec(R)=\Spec(Q[\NE(\bbY_s)])$
($Q$ as in \cref{not:coeffs}).

The following generalizes \cite[\S 16]{Keel_Yu_The_Frobenius}, and follows by the same proof
(using \cref{prop:basicconvexity}).

\begin{lemma} \label{lem:weight}
Assume $\chi(P_1,\dots,P_n,Q,\gamma) \neq 0$.
Then \[w(Q) + w(\gamma) = \sum_{j =1}^n w(P_j).\]
\end{lemma}

Note the mirror algebra $A$, as an Abelian group, is free with basis
$z^{\gamma} \theta_P$, for $(\gamma,P) \in \NE(Y) \times \Sk(U,\bbZ)$.
\cref{lem:weight} implies the following theorem:

\begin{theorem} \label{thm:torus_action}
Let $\bbT_D$ act diagonally on the mirror algebra $A$ (viewed as a free Abelian group) with weight $w(P) + w(\gamma)$ on the basis vector $z^\gamma\theta_P$.
This gives an equivariant action of $\bbT_D$ on $\Spec(A/m^n) \to \Spec(R/m^n)$ for all $n$,
and on $\Spec(A) \to \Spec(R)$ when the full mirror algebra $A$ is defined (see \cref{const:algebras}). 
\end{theorem}

Let $ W \coloneqq \sum a_i \bbD_i$ be a Weil divisor supported on $\bbD$.
This gives a one parameter subgroup $\bbG_W \subset \bbT_D$.
$W$ is nef if and only if the action $$
\bbG_{W}\times \Spec(Q[\NE(\bbY_s,\bbZ)]) \to \Spec(Q[\NE(\bbY_s,\bbZ)])$$
extends to regular $$
\bbA^1_W \times \Spec(Q[\NE(\bbY_s,\bbZ)]) \to \Spec(Q[\NE(\bbY_s,\bbZ)])
$$
(where the extension of $\bbG_W \to \bbT^{N_1(\bbY_s)}$ to $\bbA^1_W \dasharrow \Spec(Q[\NE(\bbY_s,\bbZ)])$
is given by $\cdot W\colon A_1(\bbY_s,\bbZ) \to \bbN$).
\begin{proposition} \label{prop:neflim}
Notation immediately above.
If $W$
is effective, then the action of $\bbG_W$ on $\Spec(A/m^nA)$ extends to regular $$
\bbA^1_W \times \Spec(A/m^nA) \to \Spec(A/m^nA).
$$
The same holds for $A$ in place of $A/m^n$ when the full mirror algebra is defined (see \cref{const:algebras}). 
\end{proposition}
\begin{proof} $W$ gives a $\Sigma$-piecewise linear function on $W:\Sk(U,\bbZ) \to \bbZ$ using $w$ above
  (and dot product).
Since $W$ is effective, this is non-negative.
The pullback of $z^{\gamma} \cdot \theta_P$
under the action is $z^{W \cdot \gamma + W(q)} \cdot \theta_P$,
which is regular on $\bbA^1 \times \Spec(A/m^nA)$ as $W$ is nef.
This completes the proof.
\end{proof}

\begin{proposition} \label{prop:nsc} Assumptions as in \cref{ass:basicsetup}. Assume $U$ contains no
  complete rational curves. Let $\bbW$ be an effective divisor on $\bbY$ whose support
  contains no essential strata. Then for the general member $f:B \to Y$  of the moduli space for
  any structure disk (see \cref{def:structuredisk}),  
  no irreducible component of the domain maps
  into the support of $\bbW$, so in particular
  the pullback $f^{-1}(\bbW)$ is defined as an effective
  Cartier divisor) and 
  $$
  [f:B \to Y] \cdot \bbW = \deg f^{-1}(\bbW) \geq 0
  $$
  with equality iff $f$ extends (as in the definition of structure constants, see
  \cref{const:structure_constants}) to
  a stable curve $f: C \to Y$ with image disjoint from $\bbW$.
  \end{proposition} 
  
  \begin{proof} The argument of \cite[7.5]{Keel_Yu_The_Frobenius} applies. \end{proof} 

For the next Lemma, for a Weil divisor $\bbF$ on $\bbY$, let $\bbF_U$ be the sum of the components
(with the same coefficient as in $\bbF$) whose support is not contained in $|\bbD|$.

\begin{lemma} \label{lem:finiteness_fixed_weight}
  Assume there is an ample Weil divisor $\bbF$ on $\bbY$ such that $-\bbF_U$ (notation as immediately
  above) is effective, and
  contains no essential boundary strata, and $U$ contains no complete rational curves, then 
  given $z\in\bbZ^{I_D}$, there are only finitely many $\gamma\in\DC(Y)$ such that $w(\gamma)=z$.
\end{lemma}
\begin{proof} This follows from \cref{prop:basicconvexity} and \cref{prop:nsc} just as in the proof of
  \cite[16.4]{Keel_Yu_The_Frobenius}. 
  \end{proof} 
  
  \begin{lemma} \label{lem:finite_generation}
    For $n > 0$, if a set of $\theta_P$ generates $A/\fm A$, it generates $A/\fm^n A$ as $R/\fm^n$ algebra.

    If there exists an ample divisor $\bbF$ on $\bbY$ as in \cref{lem:finiteness_fixed_weight} then
    they generate $A$ as an $R$-algebra.
    \end{lemma}
    \begin{proof} The first statement is immediate from the nilpotent Nakayama lemma. For the second
      the argument for  \cite[16.6]{Keel_Yu_The_Frobenius} applies.
    \end{proof}

\begin{theorem} \label{thm:finite_generation} 
The mirror algebra $A$ is a finitely generated $R$-algebra in cases \cref{ex:affine} and \cref{ex:pipe}. 
\end{theorem}
\begin{proof} As in the proof of \cite[16.8]{Keel_Yu_The_Frobenius}, we can replace $\bbY$ by a blowup
  $\bbY' \to \bbY$ whose exceptional locus is contained in the support of $\bbD$. We show after such a blowup
  that there is an ample divisor $\bbF$ as in \cref{lem:finite_generation}. The proof of
  \cite[16.8]{Keel_Yu_The_Frobenius} gives this in the affine case, and \cref{prop:mfua} the
case \cref{ex:pipe}. This completes the proof. \end{proof}

\begin{proposition} \label{prop:mfua} Let $(\bbY,\bbD)$ be a formal snc pair with a projective birational
     morphism
     $b: \bbY \to \bbA$ to an affine formal scheme finitely presented over $k$.
     Then there is a composition of blowups of closed strata of $\bbD$
    (and its strict transform) $\bbY' \to \bbY$ and an ample Cartier divisor $\bbF$ on $\bbY'$ such
    that $-\bbF_U$ (notation as just above \cref{lem:finiteness_fixed_weight})
    is effective and contains no boundary strata.
  \end{proposition}
  \begin{proof} $b$ itself is a blowup, $\Proj$ of some ideal sheaf. It follows there is an effective
    Cartier divisor $\bbE$ with $\cO(-\bbE)$ relatively ample and thus (as $\bbA$ is affine) ample.
    Now we can blowup strata (and twist the inverse image by a negative amount of the exceptional divisor to
    preserve ampleness) so the final condition holds.
    \end{proof}

    We give a name to the conditions in \cref{lem:finite_generation}, as they imply the
    mirror algebra is algebraic:
    \begin{definition} \label{def:pb} Assumptions as in \cref{ass:basicsetup}. We say we have {\it positive boundary}
      if $U$ contains no complete curves, and there is an ample Cartier divisor $\bbF$ on $Y$ such
      that $-\bbF_U$ is effective and contains no strata.
    \end{definition}

    \begin{notation} \label{nota:algebraic} For $(\bbY,\bbD)$ as in \cref{ass:basicsetup} we will say the mirror algebra
      {\it is algebraic} if the multiplication rule is polynomial (i.e. given $P_1,\dots,P_n \in \Sk(U)(\bbZ)$,
      there are only finitely many pairs $(\gamma,Q) \in \NE(\bbY_s,\bbZ) \times \Sk(U)(\bbZ)$ with
      $\chi(P_1,\dots,P_n,Q,\gamma) \neq 0$) and finitely generated as $R_{\bbY_s}$ algebra.
      \end{notation}
        \begin{theorem} \label{thm:algebraic}
      Assumptions as in \cref{ass:basicsetup}. If there is an snc pair $(\bbY',\bbD')$ with
      $b: \bbY' \to \bbY$ projective and birational, with $|b^{-1}(\bbD)| = |\bbD'|$, and exceptional
      locus contained in $|\bbD'|$, and $(\bbY',\bbD')$ has positive boundary, then
      the multiplications rules on $A_{(\bbY,\bbD)}$ and $A_{(\bbY',\bbD')}$ are polynomial, and the algebras
      finitely generated. Moreover $A_{(\bbY,\bbD)} = A_{(\bbY',\bbD')} \otimes_{R_{\bbY'_s}} R_{\bbY_s}$,
      where $R_{\bbY'_s} \twoheadrightarrow R_{\bbY_s}$ is induced by
      $b: \NE(\bbY'_s,\bbZ) \twoheadrightarrow \NE(\bbY_s,\bbZ)$.

      There exist such blowups in the \cref{ex:affine} and \cref{ex:pipe} cases.
    \end{theorem}
    \begin{proof} The existence of such blowups is shown in the proof of \cref{prop:mfua}.
      The final statements follow from the proof of (the very simple) \cite[17.3.1]{Keel_Yu_The_Frobenius},
      using \cref{sec:itb}. 
    \end{proof}

    \begin{proposition} Assumptions as in \cref{ass:basicsetup}. Suppose there is a regular
      function $f$ on $\bbY$ whose zero locus is the same as $|\bbD|$, and furthermore that
      $g: \bbY \to \obbY$ is birational with exceptional locus contained in $|\bbD|$. 
      Then there an effective ample Cartier divisor with support $|\bbD|$. 
    \end{proposition}
    \begin{proof}
      
      By assumption $g:\bbY \to \obbY$ is projective, and the target is affine. It
      follows that this is the formal completion of a scheme, projective and birational over the spectrum of
      a complete local ring, $g:\cY \to \Spec(R) \eqqcolon  \ocY$, with $f$ induced by a function (which we
      indicate by the same symbol) $f \in m_R$. As $g$ is birational, it is $\Proj$ of an
      ideal sheaf, $I$, and as, by assumption the exceptional locus of $g: \cY \to \ocY$ is contained
      in $Z(f)$, $I \otimes A_f$ is principal, say $I \otimes A_f = h A_f$. $g^{-1}(I) = \cO(\cE)$ for
      an effective Cartier $\cE$. $\cO(1) = \cO(-\cE)$ is $g$ ample. Now $-\cE + g^{-1}(h)$ is relatively
      ample, with support contained in $|\cD|$, and so $\cE + g^{-1}(h) + Z(h^m)$ is relatively ample,
      and effective, with support $|\cD|$, for sufficiently large $m$. 
      Note $g$-relatively ample, and ample, are the same thing as the base is affine. This induces an
      analogous ample divisor on $\bbY$. This completes the proof. 
    \end{proof}
    
     \section{Change of snc compactification}
 \subsection{Change of snc compactification} \label{sec:change_of_snc_compactification}
 Most of the results from \cite[\S 17]{Keel_Yu_The_Frobenius} hold here, with essentially
 the same proof. Here we include two which we will use: 

 \begin{definition-lemma} \label{def:gma} Let $\bbY'$ be a normal formal scheme
   topologically of finite type over $k$, 
   with $\bbY' \to \obbY$ a projective morphism to
   an affine formal scheme as in \cref{ass:basicsetup}. Let $\bbD' \subset \bbY'$ be an effective
   reduced Cartier divisor.

   Assume there exists $(\bbY,\bbD)$ as in \cref{ass:basicsetup} such that $\bbY \to \obbY$ factors
   through projective birational $p: \bbY \to \bbY'$, with $|\bbD| = |p^{-1}(\bbD')|$, and
   exceptional locus of $p$ contained in $|\bbD|$. Finally assume $\NE(\bbY_s,\bbZ) \to \NE(\bbY'_s,\bbZ)$
   is surjective (in any case the image has finite index). Assume $A_{(\bbY,\bbD)}$ is algebraic, see
   \cref{nota:algebraic}. 
   Then 
   
   $A_{(\bbY',\bbD')} \coloneqq A_{(\bbY,\bbD)} \otimes_{R_{\bbY_s}} R_{\bbY'_s}$ depends only 
   on $\bbY' \to \obbY$, and $\bbD' \subset \bbY'$. I.e. it is independent of the choice
   of $(\bbY,\bbD)$ with the above conditions.
   \end{definition-lemma} 
   \begin{proof}
     We compare the rings given by two resolutions $(\bbY_i,\bbD_i)$. By the weak factorisation
     theorem \cite{Abramovich_Torification_and_factorization} we can assume $\bbY_1 \to \bbY_2$ is a permissible blowup as in \cref{def:perm}.
     Now the result follows easily from \cref{sec:itb}, see
     \cite[Remark 17.7]{Keel_Yu_The_Frobenius}. \end{proof}

   For example, in case \cref{ex:affine}, \cref{def:gma}  defines a mirror algebra for any normal projective
   compactification of a smooth affine log CY with maximal boundary.

   In the affine case, \cref{def:gma}, the mirror family is algebraic. For the next proposition we write

   \begin{proposition} \label{prop:act} Let $(\bbY,\bbD)$ be as in \cref{ass:basicsetup}, and
     assume it satisfies the conditions of \cref{thm:algebraic}. Let $b:\bbY' \to \bbY$ be permissible blowup
     (see \cref{def:perm}). Both mirror algebras are algebraic (in the sense of \cref{nota:algebraic}). The following hold: 
     \begin{enumerate}
     \item There is a canonical isomorphism
       $ A_{(\bbY,\bbD)}= A_{(\bbY',\bbD')} \otimes_{R_{\bbY'_s}} R_{\bbY_s} \to  A_{(\bbY,\bbD)}$,
       sending $\theta_u$ to $\theta_u \otimes 1$. 
     \item There is a canonical isomorphism
       $A_{(\bbY',\bbD')}\otimes_{R_{\bbY'_s}} Q[\NE(\bbY_s) + \bbZ \cdot F] \to
       A_{(\bbY,\bbD)} \otimes_Q Q[\bbZ]$
       sending $z^{\gamma} \cdot \theta_u$ to $z^{b_*{\gamma}}\cdot \theta_u \otimes x^{w(u) + w(\gamma)}$
       where $w(u)$ is the weight for the torus action \cref{sec:ta} restricted to the one parameter
       subgroup corresponding to the exceptional divisor $E \subset \bbY'_s$, $w(\gamma) = E \cdot \gamma$,
       and $F$ is the class of a line on the general fibre of (the projective space bundle) $E \to b(E)$. 
       \end{enumerate}
   \end{proposition}
   \begin{proof} These follow from the argument for (the nearly identical)
     \cite[17.3]{Keel_Yu_The_Frobenius}.
        \end{proof} 

\begin{corollary} \label{cor:fmf} Notation as in \cref{prop:act}. The set of isomorphism classes of fibres of the mirror family over the structure torus is independent
  of the compactification (in the sense of \cref{def:abuse}). \end{corollary}
\begin{proof} Follows from \cref{prop:act} and the weak factorisation theorem \cite{Abramovich_Torification_and_factorization}. 
\end{proof}

\section{Geometry of the mirror family}

\subsection{SLC singularities and dual complexes of log Calabi-Yau pairs} \label{sec:slc}

We learned the following lemma and its proof from Paul Hacking.

\begin{lemma} \label{lem:slcdclem}
Let $\bbk$ be a field.
Let $\Gamma$ be a simplicial complex of pure dimension $n$ together with an orientation of each $n$-dimensional simplex.
We consider the following conditions: 
\begin{enumerate}[wide]
\item For each $n-1$ dimensional simplex $\tau$ one of the following holds:
\begin{enumerate}[wide]
\item[(a)] $\tau$ is contained in exactly two $n$-dimensional simplices $\sigma_1,\sigma_2$, and the orientations of $\tau$ induced by $\tau \subset \sigma_1$ and $\tau \subset \sigma_2$ are opposite.
\item[(b)] $\tau$ is contained in a unique $n$-dimensional simplex $\sigma$.
\end{enumerate}
\item For each simplex $\tau \in \Gamma$, the link $\link_{\tau} \Gamma$ of $\tau$ in $\Gamma$ has reduced homology
  $\tilde{H}_i(\link_\tau \Gamma,\bbk)=0$ for $i< n - \dim \tau-1$.
  \item $\tilde{H}_i(\Gamma) =0 $ for $i < n$.
\end{enumerate}
Let $L \subset \Gamma$ denote the subcomplex of pure dimension $n-1$ with maximal simplices those $\tau$ of type (b) in (1) above.

Let $S$ and $T$ be the Stanley--Reisner rings of $\Gamma$ and $L$ respectively.
Let $\bbX=\Proj(S)$ and $\bbE=\Proj(T)$, $X = \Spec(S)$ and $E = \Spec(T)$. 
The natural surjection $S \rightarrow T$ determine closed embeddings $E \subset X$, $\bbE \subset \bbX$. 

If (1-2) above hold,
then the following conditions hold:
\begin{enumerate}[wide]
\item The scheme $\bbX$ is reduced and Cohen-Macaulay of pure dimension $n$ and has normal crossing singularities in codimension $1$.
\item The subscheme $\bbE \subset \bbX$ is a reduced Weil divisor which does not contain any irreducible component of the singular locus of $\bbX$.
\item The sheaf $\omega_{\bbX}(\bbE)$ is isomorphic to $\cO_{\bbX}$.
\item The pair $(\bbX,\bbE)$ has semi log canonical singularities.
\end{enumerate}
\end{lemma}

If (1-3) hold, the above hold with $(\bbX,\bbE)$ replaced by $(X,E)$.
\begin{proof}
We prove the $(\bbX,\bbE)$ case, the argument for $(X,E)$ is the same. 
Assumption (1)(a-b) clearly implies that $\bbX$ has normal crossing singularities in codimension $1$ and no component of $\bbE$ is contained in the singular locus.
The assumption (2) on the homologies of the links of the simplices implies that the scheme $\bbX$ is Cohen-Macaulay by \cite[Theorem 1]{Reisner_Cohen-Macaulay_quotients} (for $X$ we also need (3), this is the only difference between the two arguments). 

The normalization $\nu \colon \bbX^{\nu} \rightarrow \bbX$ together with the divisor $\Delta^{\nu}+\bbE^{\nu}$
(where $\Delta^{\nu}$ denotes the inverse image of the double locus $\Delta \subset \bbX$ and $\bbE^{\nu}$
denotes the inverse image of $\bbE$) is a disjoint union of toric pairs.
In particular $(\bbX^{\nu},\Delta^{\nu}+\bbE^{\nu})$ is log canonical.

We use the orientation of the maximal cells of $\Gamma$ to describe an isomorphism $\omega_{\bbX}({\bbE}) \simeq \cO_{\bbX}$ :
we glue the torus invariant forms $\frac{dz_1}{z_1} \wedge \cdots \wedge \frac{dz_n}{z_n}$ on the irreducible components of the normalization (with signs determined by the orientations of the associated simplices) to obtain a nowhere vanishing global section of $\omega_{\bbX}({\bbE})$.
Here we use the fact that $\omega_{\bbX}({\bbE})$ is (by definition) a double dual, and the dual of any quasi-coherent sheaf on an $S_2$ scheme is $S_2$.
So we can work in codimension one, where we have normal crossings,
and the usual description of the dualizing sheaf on a nodal curve together.

Now it follows from the definition of slc that the pair $({\bbX},{\bbE})$ is slc: $\omega_{\bbX}({\bbE})$ is a $\bbQ$-line bundle (in fact trivial), ${\bbX}$ has normal crossing singularities in codimension $1$
and satisfies Serre's condition $S_2$ (in fact is Cohen-Macaulay), and the normalization $({\bbX}^{\nu},\Delta^{\nu}+{\bbE}^{\nu})$ is log canonical.
(See \cite{Kollar_Singularities_of_the_minimal_model_program}, Definition-Lemma 5.10, p.
193,
for the definition of slc.)

\end{proof}

\begin{proposition} \label{rem:PaulSingrem} Assumptions as in \cref{ass:basicsetup}. We
  assume we are in one of the cases \cref{ex:affine}, the Fano case of \cref{ex:pipe}, or
  \cref{ex:CYD}. 

  In the first and last cases let $\Gamma$ be the essential dual complex to $(\bbY,\bbD)$, and in the
  second case the cone (in the sense of simplicial complexes, so the cone over a $d$ simplex is a $d+1$ simplex)
  over dual complex to $(X,E)$.

  In the first two cases $\Gamma$ satisfies hypothesis (1-3) of \cref{lem:slcdclem}, and in the
  last it satisfies (1-2).
\end{proposition}
\begin{proof}
  We treat here the \cref{ex:affine} case, the arguments for the others are essentially the same.
  Recall in this case $\bbY$ is a projective variety (rather than projective over a formal affine scheme). To
  stress this we write $(Y,D)$, $U \coloneqq Y \setminus D$, instead of $(\bbY,\bbD)$. Also we allow greater
  generality, namely $(Y,D)$ dlt. Note the dual complex of a dlt pair $(Y,D)$ is defined in \cite{KdFX}, Section~2. 

First the statements about orientation: 
Let $\Omega$ be a holomorphic volume form on $U$ such that $(\Omega)=W-D$.
If $p \in D \subset Y$ is an essential $0$-stratum then $Y$ is smooth and
$D$ is normal crossing at $p$ by \cite{Kollar_Singularities_of_the_minimal_model_program},
Theorem~4.16(1) and the definition of dlt \cite{Kollar_Singularities_of_the_minimal_model_program}, Definition~2.8.
We normalize $\Omega$ by requiring that the iterated Poincar\'e residue at $p$ equals $\pm 1$,
and choose the orientation of the associated simplex $\sigma$ so that the iterated residue for a compatible ordering of the boundary divisors equals $+1$.
By connectedness of the dual complex for $n>1$, the $0$-strata are connected by $1$-strata,
which are copies of $\bP^1$ with boundary $\{0,\infty\}$, by \cref{prop:dualcomplex}. 
It follows that the iterated residue at each $0$-stratum is $\pm 1$, and
if we fix an orientation of each maximal simplex as above,
then for $\tau$ the $(n-1)$-dimensional simplex corresponding to a $1$-stratum and $\sigma_1,\sigma_2$ the two $n$-dimensional simplices corresponding to the two $0$-strata contained in this $1$-stratum, the orientations induced by $\tau \subset \sigma_1$ and $\tau \subset \sigma_2$ are opposite. 

Now we turn to the vanishings, again allowing dlt. When $K + D$ is trivial, 
this follows from \cite{KX}, Theorem 2(1) and Proposition~31.
In general we reduce to this case by running the $K + D$ MMP. By
\cite[Prop. 11]{KdFX} (and its proof which extends the Proposition to the subboundary
case) the essential dual complex is independent of the steps in the MMP (up to PL-isomorphism).
This completes the proof. 
\end{proof}

%

\subsection{Geometry of the mirror family} \label{sec:geometry_of_the_mirror_family}

Here we give statements on singularities of the mirror family.

\begin{theorem} \label{thm:lcsings} Assumptions as in \cref{ass:basicsetup}, and \cref{ex:affine}.
  Recall from \cref{thm:finite_generation}
  that in this case the multiplication rule is algebraic, gives a canonical flat family
  $$
  \Spec(A_{(\bbY,\bbD)}) \to \Spec(Q[\NE(\bbY_s,\bbZ)])= \TV(\Nef(\bbY_s)).
  $$
  Consider the base change the mirror family from $\bbZ$ to $\bbQ$. 
Geometric fibres are semi-log-canonical, Gorenstein, $K$-trivial affine varieties of the same dimension as $\bbU$.
Fibers over the structure torus $\bbT_{\Pic(\bbY_s)}$ of the base are log canonical (in particular, normal).
\end{theorem}

\begin{theorem} \label{thm:lcsingsccy} Assumptions as in \cref{ass:basicsetup}, and \cref{ex:CYD}.
  Recall in this case $\Proj$ of the mirror algebra induces a polarized projective formal family
  $$ (\cX,\cO(1)) \to \Sp(\hR_{\bbY_s}).$$
  The special fibre is $K$-trivial Gorenstein SLC. The generic fibre is $K$-trivial Gorenstein SLC,
  and if $(\bbY,\bbD)$ is almost minimal (see \cref{def:am}) log canonical
  (in particular, normal).
  \end{theorem}

  \subsection{Singularities of the mirror family in case \cref{ex:pipe} }
  In this subsection we assume \cref{ass:basicsetup} and \cref{ex:pipe}. In this
  case $\bbY \to \obbY$ is a relative Mori dream space. Let $\bbY \dasharrow \bbY'$ be a small $\bbQ$-factorial
  modification (over $\obbY$), and $\bbD' \subset \bbY'$ the induced boundary. Note since $\bbD$ contains the $b$-exceptional
  locus, $U' \coloneqq Y' \setminus D'$ (defined exactly like $U \coloneqq Y \setminus D$ in \cref{ass:basicsetup}) is
  equal to $U$. 
  
  \begin{theorem} \label{thm:lcsingspipe} Notation as immediately above.
    Assume that the essential dual complex to $(\bbY,\bbD)$ satisfies assumptions (1-3) of \cref{lem:slcdclem}.
    By \cref{def:gma} and \cite[6.17]{HKY20} the mirror algebra $A_{(\bbY',\bbD')}$ induces a 
    a projective polarized family of tuples
    $$
    (\cX,\cE,\Theta,\cO(1)) \to \TV(\Nef(\bbY'_s)).
    $$
   
    Each fibre $(X,E)$ is slc and $K_X + E$ is trivial. Fibres over the structure torus of the base
    are log canonical (in particular such $X$ are normal).
  \end{theorem}
  
  \begin{remark} \label{rem:Fano} Note that in the Fano case, the conditions (1-3) on the dual complex hold by \cref{rem:PaulSingrem},
    so we can drop that assumption for that case of \cref{thm:lcsingspipe}. We expect these conditions to hold as well
    when $\obbY \to \Sp(k[[s]])$ is the GS partial smoothing of the vertex given by $(X,E,L)$, as in \cref{ex:pipe}:
    The GS mirror is a formal
    deformation of the vertex $\bbV_{\Sigma}$, $\Sigma$ the essential dual fan to $(X,E)$.
    $\tau \not \in \partial \Gamma$, for $\Gamma$ the essential dual complex to $(\bbY,\bbD)$,
    correspond to (essential) strata of $\bbD$ contracted to
    $0 \in \bbV_{\Sigma}$. Such a stratum (with its induced boundary) is a compact log CY with maximal boundary and
    $\link_{\tau} \Gamma$ is the dual complex of this pair, and so the vanishing follows as in the proof of
    the \cref{ex:affine} case of \cref{rem:PaulSingrem}. 

    Strata of $\bbV$ correspond
    to cones of the dual fan, which in turn correspond to strata of $(X,E)$, each of which is itself a log CY pair with
    maximal boundary. Gross has conjectured that the restriction of the GS deformation to the interior of a stratum of $\bbV$
    is isomorphic
    to (a natural base extension of) the mirror family for the corresponding log CY pair (i.e. the corresponding stratum of $(X,E)$
    with its induced boundary). If this {\it induction} conjecture holds, then we can get the required vanishing
    in case $\tau \in \partial \Gamma$ by the same argument applied with $(X,E)$ replaced by a boundary stratum. 
    \end{remark}

  Note once we have proven the rest, the final statement of each theorem
  (that the fibres are log canonical) is equivalent to
normality of these fibres.
This we will prove at the end of \cref{sec:scattering}.
Here we prove the rest:
\begin{proof}[Proof of everything in the above three theorems but normality]
 Consider first the affine case.  
To prove the theorem, by \cref{def:gma}, we can replace $\bbY$ by a blowup (outside of $\bbU$),
and so by \cite[19.2]{Keel_Yu_The_Frobenius} can assume $D$ supports an ample effective divisor.
The conditions are open, see the proof of \cite[8.42]{Gross_Canonical_bases}.
So using the $\bbT_D$ action, more precisely \cref{prop:neflim}, it is
enough to prove the statements for the central fibre. This now follows from
\cref{rem:PaulSingrem} and \cref{prop:vertex} below.
This completes the proof.

Now consider the compact CY degeneration case: Because the singularity statements are open, and
the family is proper, it's enough to prove the statement for the central fibre. This follows
from \cref{rem:PaulSingrem} and \cref{prop:vertex}.

Finally, consider the \cref{ex:pipe} case.
The statement is only about $(\cX,\cE)$ so it's enough to work on $\MovSec$ (the bogus cones
    do not add any new fibres $(X,E)$, the change is only in $\Theta$), and indeed on $\Mov$ (as the statement is
    about fibres, so we can pullback the famly under $\TV(\Mov) \to \TV(\MovSec)$. For
    $(\cX,\cE) \to \TV(\Nef(\bbY/\obbY))$ the argument is the same as for \cref{thm:lcsings}. It remains to
    consider flops of $\bbY$. But by \cite[Thm. 13]{KX} the dual complex for the flop is the same (up to PL isomorphism),
    so the same argument applies. This completes the proof.
  \end{proof}

\subsection{Central fiber as a vertex} \label{sec:vertex}

In this section we describe the central fiber of the affine mirror family
as the spectrum of a Stanley-Reisner ring, which, following \cite{Gross_Mirror_symmetry_for_log_Calabi-Yau_surfaces_I_v1}
we call the \emph{vertex}. Our results on singularities of fibres of the mirror family all
come by degenerating to this central fibre, whose singularities are described in 
\cref{sec:slc}. 

\begin{proposition} \label{prop:vertex} Let $(\bbY,\bbD)$ be as in \cref{ass:basicsetup}.
  Let $S$ be the Stanley-Reisner ring of the dual complex to $(\bbY,\bbD)$.

  We have $A_{(\bbY,\bbD)}\otimes_{R_{\bbY_s}} R_{\bbY_s}/\fm_{\bbY_s}\simeq S$,
  identifying the basis elements (note each of the algebras has canonical
  vector space basis identified with $|\Sigma_{(\bbY,\bbD^\ess)}|(\bbZ)$).
\end{proposition}

\begin{proof}
Since we work modulo $\fm_{\bbY_s}$, only disks with class zero contribute.
Now the result follows from \cref{prop:zeroclass}.
\end{proof}

To complete the proof of \cref{thm:lcsings} we need a different argument than that used in \cite{Keel_Yu_The_Frobenius}.
In \cite{Keel_Yu_The_Frobenius} we assumed $U$ contains an open algebraic torus $\bbT_M$ and so by \cite[18.1]{Keel_Yu_The_Frobenius} there is a degeneration of the mirror algebra to the toric case, for which the mirror is a Laurent polynomial ring for the lattice $\Sk(U,\bbZ) =M$.
In general we use a scattering diagram, which we turn to next.

\section{The Canonical Scattering diagram} \label{sec:scattering}
Next we complete the proof of the main singularity statement \cref{thm:lcsings} by proving integrality of
the total space of the mirror over $\Sp(\hR_{\bbY_s})$. 
For this we use the canonical scattering diagram as in \cite{Keel_Yu_The_Frobenius},
\cite{HDTV}, \cite{GSCWS}, \cite{Gross_Theta_Functions} and \cite{Gross_Mirror_symmetry_for_log_Calabi-Yau_surfaces_I_v1}.
Following the original idea of the Gross-Siebert program we use the scattering diagram to produce a deformation of the {\it vertex}.
\cite{Gross_Theta_Functions} (generalizing \cite{Gross_Mirror_symmetry_for_log_Calabi-Yau_surfaces_I_v1}) proves
this deformation is the same as (the spectrum of) the mirror algebra.
One advantage of the scattering approach (and the reason we use it here)
is that outside codim 2 it comes with simple and explicit Mumford/Toric charts, 
\cite[2.8]{Gross_Theta_Functions}, 
from which the desired integrality is obvious. 

We define the scattering diagram by counts of short spines $[-\epsilon,\epsilon] \to \Sk(U)$ with at most one bend -- which by our definition is a count of (skeletal)
analytic cylinders with this segment as their intrinsic spine, just as in \cite[\S 20]{Keel_Yu_The_Frobenius}.

As throughout the paper we assume \cref{ass:basicsetup}, the notation of which we follow.
In this section we additionally assume $(\bbY,\bbD)$ is almost minimal, \cref{def:am}.

By \cref{sec:integeraffine}  $\Sk(U) = |\Sigma|$, for $\Sigma$ the essential dual fan, has 
an integral affine structure, non-singular outside $\Sigma^{n-2}$.
This is the same integral affine structure given in \cite[1.3]{GSCWS},
$(\Sk(U),\Sigma)$ is a polyhedral manifold in the sense of \cite[Const.
1.1]{Gross_Theta_Functions}. Following GS, we sometimes write $\Lambda$ for the integer tangent space to $\Sk(U)$.

\begin{definition} \label{def:generic}
We say a point $x \in \Sk^{\sm}(U)(\bbR)$ is {\it generic} if it lies on at most one primitive $n^{\perp}$ for $n$ the germ of a primitive (non-zero) integer linear function near $x$.
\end{definition}

\begin{construction} \label{def:infintesimal_spine}
We consider piecewise affine $h\colon [-\epsilon,\epsilon] \to \Sk^{\sm}(U)$ with $x = h(0)$.
Here $\epsilon > 0$ will be chosen sufficiently small.
Suppose $N(h,\alpha) \neq 0$ for some $\alpha \in \NE(Y,\bbZ)$, and the image of $h$ is not contained in $\Wall_{\alpha}$.
Here we include all the cones of $\Sigma^1$ in $\Wall_{\alpha}$ (we are free to add any finite number of cones).
Then $h$ is affine near zero unless $x \in \Wall_{\alpha}$.
Now assume $x$ is generic and $N(h,\alpha) \neq 0$.
Then (if it exists at all)
$n^{\perp}$ is the linear span of the unique wall containing $x$,
and $h$ is affine on $[-\epsilon,0],[0,\epsilon]$ (note this determines primitive $n$ up to sign).
Say with derivatives $-v,-w$.
Note necessarily $v - w \in n^{\perp}$
(or in case $x$ is not on a wall, $v =w$ and $h$ is affine), since the bend in given by a realizable twig and these lie in walls, see \cref{prop:tree_bound}.
Then modulo shrinking $\epsilon$, $h$ depends only on $v,w,x$,
so we write $N(x,v,w,\alpha) \coloneqq N(h,\alpha)$.
If $x$ is not on a wall, then $N(x,v,w,\alpha) = 0$ unless $\alpha = 0$
and $v = w$.
In the other direction: \end{construction}

\begin{lemma} For generic $x$ not on any wall, $N(x,v,v,0) = 1$.
\end{lemma}
\begin{proof} Since $h$ has image in the complement of the walls,
the count is the same as the toric case, so the result follows from \cite[6.2,12.3]{Keel_Yu_The_Frobenius}.
See \cref{lem:straightspinecount}.
\end{proof}

\subsection{The canonical $A_1(Y,\bbR)$ principal bundle $p: \Sk(G) \to \Sk(U)$.} \label{sec:P+}
To connect with the formalism of \cite{Gross_Mirror_symmetry_for_log_Calabi-Yau_surfaces_I_v1},
\cite{Gross_Theta_Functions}, \cite{HDTV}, and in particular to
get from the scattering diagram a deformation of the vertex as in \cite{GSCWS}, we choose a strictly convex
toric monoid $\NE$, with ${\NE^{\gp}}$ a lattice, which we indicate by $N_1$ and a linear map
$s:\NE(\bbY_s,\bbZ) \to \NE$, with $s^{-1}(0) = 0$. E.g. we could take $(\NE,A_1) = (\NE(\bbY_s,\bbZ),A_1(\bbY_s,\bbZ))$
if the RHS satisfies the conditions. Or we could take $\NE(\bbY/\obbY,\bbZ),A_1(\bbY/\obbY,\bbZ)$ (again if the RHS
satisfies the assumptions). To conform with \cite{GSCWS}, we write $P \coloneqq \NE$. 
There is a canonical locally trivial $A_1(\bbR)$ principal bundle $p:G \to \Sk(U)$, with integer linear structure on
the total space, non-singular outside the (inverse image of the) union of codimension two cones of $\Sigma$ (see
\cref{sec:integeraffine}), with
$p$-linear (where $\Sk(U) = |\Sigma_{(\bbY,\bbD^{\ess})}|$ inherits linear structure from $(\bbY,\bbD)$), together with
a canonical continuous $\Sigma$-piecewise linear $P$-convex section $\varphi: \Sk(U) \to G$, with kink along a codimension
one $\rho \in \Sigma$ the (image in $\NE$ of) the class of the associated one stratum. See \cite[\S 3.1]{GSCWS}.

We note that for $y$ in the smooth locus of $\Sk(U)$, parallel transport identifies $T_{x,\Sk(G)}(\bbZ)$ for $x \in p^{-1}(y)$,
we denote this lattice as $P_y$.
Following \cite[2.11]{HDTV},
\cite[1.16]{Gross_Theta_Functions},
we denote by $$
P_y \supset P_y^+ \coloneqq \{d\varphi(v) + \alpha| \alpha \in \NE \}
\subset T_{y,\Sk(G)}(\bbZ).$$

For any tangent vector $v \in T_{x,\Sk(G)}$, $x \in \Sk^{\sm}(G)$,
by its {\it height with respect to } $\varphi$ we mean $$
v - d \varphi (dp(v)) \in A_1
$$
(we view it as {\it above} if the height is in the monoid $P = \NE$).

Let $x \in \Sk^{\sm}(G)$ with $p(x)$ generic.
We map $Q[\NE][T_{p(x)}(\bbZ)] \to Q[T_x(\bbZ)]$ by sending $t^{\alpha} z^v \to z^{d\varphi(v) + \alpha}$ (we note though $\varphi$
is only piecewise linear, its derivative is none the less well defined).

\begin{definition-lemma} \label{def:wall-crossing_map}
Notation as immediately above.
We assume $x$ is generic,
$x \in n^{\perp}$.
$v \in \Lambda_x$, $v \not \in n^{\perp}$.

We consider the formal sum:
\begin{equation} \label{eq:wall-crossing}
\Psi_{x}(z^{\varphi(v)} ) \coloneqq \sum_{\substack{w \in \Lambda_x \\
\alpha \in \NE }} N(x,v,w,\alpha) z^{\varphi(w) + \alpha}.
\end{equation}

If $x$ is generic but lies on no $n^{\perp}$, we define $\Psi_x(z^{\varphi(v)}) \coloneqq z^{\varphi(v)}$.
\end{definition-lemma}

As in \cite[\S 20]{Keel_Yu_The_Frobenius} this converges in a natural adic topology, which we now recall:
We note by definition that each of the monomials
$z^{\varphi(w) + \alpha}$ in the sum lies in $P_x^{+}$, which is invariant under the $\NE$ action
on $P_y$.  Let $I_{k} \subset Q[\NE]$ be the ideal generated by monomials $z^{\alpha}$
with $\alpha \in k \cdot (\NE \setminus \{0\})$ (meaning $\alpha$ can be written as a sum of $k$ non-zero elements in $\NE$)
and $I_{k} \subset Q[P_y]$ the induced ideal.  

\begin{lemma} \label{lem:conv} $\Psi_x(z^{\varphi(v)}) \in \hR \coloneqq \varprojlim_{ k \geq 1} Q[P_x^+]/I_{k}$.
\end{lemma}
\begin{proof} It is enough to show that the formal sum is finite mod each $I_{k}$.  
Note there are only
finitely many monomials in  $Q[\NE] \setminus I_{k}$ for each $k > 0$.  Now for each $\alpha$ (and $x,v$ fixed) there are
only finitely many possible $w$ with $N(x,v,w,\alpha) \neq 0$ by (2) of \cref{prop:weightbound}. This completes
the proof.
\end{proof} 


We note: As holds for any linear manifold,
for any convex subset $\sigma \subset \Sk^{\sm}(U)$ the integer tangent spaces $\Lambda_x$, for $x \in \sigma$ are canonically identified, we denote this by $\Lambda_{\sigma}$ (or just $\Lambda$ when $\sigma$ is understood).
We note (as for any linear manifold) there is a canonical inclusion $\sigma \subset \Lambda_{\sigma,\bbR}$.

We fix a closed, convex cone $\beta$, contained in a maximal cone of $\Sigma_{(\bbY,\bbD)}$, whose interior is open and
contained in $\Sk^{\sm}(U)$.
We write $N$ for the dual of $\Lambda \coloneqq \Lambda_{\beta}$.

\begin{definition} \label{def:wall-crossing_transformation}
Given $n\in N\setminus 0$, $x \in n^\perp\subset \Sk^{\sm}(U)$ generic, and $v\in \Lambda$ with $\braket{n,v}\ge 0$.
If $\braket{n,v}>0$, let $\Psi_{x,n}(z^{\varphi(v)})$ be as in \eqref{eq:wall-crossing}.
If $\braket{n,v}=0$, let $\Psi_{x,n}(z^{\varphi(v)})=z^{\varphi(v)}$.  Let $P^+_{\braket{n,\cdot}\geq 0} \subset P^+_x$ be the
inverse image of the half space $\braket{n,\cdot}\ge 0 \subset M$. 
We further extend to $\Psi_{x,n}\colon Q[P^+_{\braket{n,\cdot} \geq 0}] \to\hR$ by linearity over $Q[\NE]$.
\end{definition}

Now we have a series of results which are minor restatements of results from \cite[\S 20]{Keel_Yu_The_Frobenius}.
In this reference the bundle $p: \Sk(G) \to \Sk(U)$, and $P^+ \subset P_x$ is not used, instead everything is
phrased in terms of $\Lambda \oplus A_1(Y,\bbZ)$.  We note the natural identification
of $(\Lambda \oplus \NE)$ with $P^+_x$ induced by (the derivative of) $\varphi$ IS NOT linear
if $n^{\perp}$ contains a condimension one cone from $\Sigma_{(\bbY,\bbD)}$. However, if we restrict from $\Lambda$
to the half space $\braket{n,\cdot} \geq 0  \subset \Lambda$, then it is. With this the arguments  from \cite{Keel_Yu_The_Frobenius} apply
after minor (and obvious) purely notational changes. 

\begin{theorem}\label{thm:wall-crossing_homomorphism}
The map $\Psi_{x,n}$ is a ring homomorphism.
\end{theorem}
\begin{proof} The same as the proof of \cite[20.9]{Keel_Yu_The_Frobenius}. \end{proof} 

\begin{proposition} \label{prop:wall-crossing_function}
  There exists $f_x \in \hR$ such that 
for any $v \in M$ with $\braket{n,v} \geq 0$, we have \begin{equation} \label{eq:wall-crossing_function}
\Psi_{x,n}(z^{\varphi(v)}) = z^{\varphi(v)}  \cdot f_x^{\braket{n,v}}. 
\end{equation}

The counts $N(V_{x,v,w,\alpha})$ in \cref{eq:wall-crossing} depend only on $\braket{n,v}$, $w-v$ and $\alpha$. 
There exist $N(b,\alpha) \in Q$, for $b \in n^{\perp} \subset \Lambda_x$ and $\alpha \in \NE$
such that 
$$
f_x = \sum_{\substack{b \in n^{\perp} \subset \Lambda_x \\ \alpha \in \NE}} N(b,\alpha) z^{\varphi(b) + \alpha} \in \hR.
$$
\end{proposition}
(Recall that throughout the paper
$Q \coloneqq \bbZ$ if $U$ contains no complete rational curves, $\bbQ$ otherwise.)
\begin{proof} This follows from \cref{thm:wall-crossing_homomorphism}, by the (very short)
  proof of \cite[20.12,20.13]{Keel_Yu_The_Frobenius}.   \end{proof} 

\begin{remark} \label{rem:scmio} We view the $N(b,\alpha)$ as counts of Maslov index zero disks with ios tropicalisation
  contained in  $n^{\perp}$, as in \cref{def:twigdef}.
  From the proof it is clear that $N(b,\alpha) \neq 0$ implies
  there is a Maslov index $0$ disk with root $x$ and outward derivative (sum of weights) $b$
  as in \cref{lem:tbal}. We note that the Gross-Siebert scattering diagram is defined in terms of counts
  of (their punctured log versions of) such objects. We do not know how to directly define such counts
  in our theory. Note the counts can be negative (so obviously virtual)
  even in the simplest non-toric cases.  \end{remark}

\begin{proposition} \label{prop:wall_decomposition}
For any $k>0$, there is a finite set of pairs $\fD_k=\{(\fd,f_\fd)\}$, where $\fd$ is a closed codimension-one rational convex cone in $\Sk(U)$
contained in a maximal cone of $\Sigma_{(\bbY,\bbD)}$
and $f_\fd\in Q[P^+_{\fd}]/I_{k}$, such that the union of all $\fd$ is the closure of the set of generic $x$ with $f_x\neq 1$ modulo $I_{k}$, and that for any $x\in\fd$ generic, we have $f_\fd=f_x$ modulo $I_{k}$.
\end{proposition}
\begin{proof} Same as proof of \cite[20.7]{Keel_Yu_The_Frobenius}. \end{proof} 

\begin{definition} \label{def:prescat}
We denote \[\fD\coloneqq\Set{(x,f_x) | x\in n^\perp\subset \Sk(U)
\text{ generic } } \]
and call it the \emph{scattering diagram} associated to $\bbU\subset \bbY$.
In view of \cref{prop:wall_decomposition}, we call $\fD_k$ the $k$-th order approximation of $\fD$.
\end{definition}

\begin{remark} \label{rem:fD} We note the scattering diagram $\fD$ is canonically determined by the
  analytic cylinders
  tropicalising to $h\colon [-\epsilon,\epsilon] \to \Sk(U)$
as in \cref{def:infintesimal_spine}.
This collection of cylinders depends only on $U$, not on the formal model $(\bbY,\bbD)$.
In this way $\fD$ is canonically associated to $U$.
Different models can give different affine structures on $\Sk(U)$.
If we view the scattering diagram in the different affine structures, it can look different.
But they are equivalent; they encode exactly the same information.
This gives a simple explanation for the \it{mutation invariance} of the cluster scattering diagram in \cite[1.3]{Gross_Canonical_bases} (the original proof of which is quite involved).
In particular, the support $|\fD| = \cup_k |\fD_k| \subset \Sk(U)$ (the union of the walls) depends only on $U$.
\end{remark}
\subsection{Consistency} \label{sec:scattering:consistency}

Our scattering diagram $\fD$ is \emph{theta function consistent}, see \cref{prop:theta_function_consistent}. 

\begin{definition-lemma} \label{def:local_theta_function}
Given $x \in \Sk(U)$ generic and $0 \neq m \in \Sk(U,\bbZ) $, $e \in \Lambda_x \eqqcolon M$, let $\SP_{x,m,e}$ be the set of spines in $\Sk(U)$ with domain $[-\infty,0]$ such that $-\infty$ maps to $\partial\oSk(U)$ with derivative $-m$ and $0$ maps to $x$ with derivative $-e$.
We define the \emph{local theta function} $\theta_{x,m}$ to be the formal sum \[\theta_{x,m} \coloneqq \sum_{\substack{e\in \Lambda_x,\; S\in\SP_{x,m,e}\\ \alpha\in\NE(Y)}} N(S,\alpha) z^{\varphi(e) + \alpha}.\]

The formal sum converges in $\hR$ as in \cref{lem:conv}.
\end{definition-lemma}
\begin{proof} The same argument used to prove \cref{lem:conv} applies. \end{proof}

\begin{lemma} \label{lem:localthetabl} The local theta function, \cref{def:local_theta_function} is the same as the analogous
  expression using broken lines for the scattering diagram, as in \cite[4.2,4.4]{GSCWS}.
\end{lemma}
\begin{proof} This is immediate from the gluing formula, \cref{thm:gluing_concatenate}
  and the definition of the wall
  crossing function \cref{eq:wall-crossing}.
  \end{proof} 

\begin{proposition} \label{prop:theta_function_consistent}
The scattering diagram $\fD$ is theta function consistent in the following sense:
Given any $k>0$ and $(\fd,f_\fd)\in\fD_k$, choose $n\in N$ with $\fd\subset n^\perp$, and let $a,b\in \Sk(U)$ be two general points near a general point $x\in\fd$ with $\braket{n,a}>0$, $\braket{n,b}<0$.
Write $\theta_{a,m} = \theta_{a,m,+}+\theta_{a,m,0}+\theta_{a,m,-}$ and $\theta_{b,m} = \theta_{b,m,+}+\theta_{b,m,0}+\theta_{b,m,-}$, where we gather monomials $z^e$ according to the sign of the pairing $\braket{n,e}$, e.g.,
\[\theta_{a,m,+} \coloneqq \sum_{\substack{e\in \Lambda_x,\; \braket{n,e}>0\\S\in\SP_{a,m,e},\; \alpha\in\NE(Y)}} N(S,\alpha) z^{\varphi(e) + \alpha}.\]
The following hold modulo $I^k$:
\begin{align}
\Psi_{x,n}(\theta_{a,m,+})=\theta_{b,m,+},\label{eq:theta_function_consistent_1}\\
\Psi_{x,-n}(\theta_{b,m,-})=\theta_{a,m,-},\label{eq:theta_function_consistent_2}\\
\Psi_{x,n}(\theta_{a,m,0})=\theta_{b,m,0},\label{eq:theta_function_consistent_3}\\
\Psi_{x,-n}(\theta_{b,m,0})=\theta_{a,m,0}.\label{eq:theta_function_consistent_4}
\end{align}
\end{proposition}
\begin{proof} This follows by the same argument used to prove \cite[20.22]{Keel_Yu_The_Frobenius}. 
  \end{proof} 

For $x \in \Sk(G)$ with $p(x)$ generic we let $f_x$ be the scattering function from \cref{prop:wall-crossing_function}

\begin{proposition} \label{prop:can_scat} Notation as above.
$f_x =1$ mod $I \cdot \hR$. \end{proposition} 
\begin{proof}
  We follow the notation of \cref{prop:wall-crossing_function}.
  Let $w$ appear (with non-zero coefficient)
  on the RHS of \cref{eq:wall-crossing} (for some $v$ with $\braket{n,v} \geq 0$.  Note $w$ is in
  the same half space   $w \in \braket{n,\cdot } \geq 0$, and $\varphi$ restricts linearly to this half space
  (as $x$ is generic).  Thus
  $$
  f_x = \sum_{\substack{w \in \braket{n,\cdot} \geq 0  \\
      \alpha \in \NE}} N(x,v,w,\alpha) z^{\varphi(w-v) + \alpha} \in \hR
  $$
  (where we abuse notation and write $\varphi$ for the derivative of its
  restriction to $\braket{n,\cdot } \geq 0$). 
  Mod $I$, the class of a twig disk (of a cylinder contributing to $f_x$) cannot be trivial by \cref{prop:zeroclass}. So
  if a cylinder has any twig disks the corresponding monomial of $f_x$ lies in
  $I Q[P_x^{+}]$. 
  So  a cylinder which contributes mod $I$ has no twigs. Then the spine is the full ios
  tropicalisation and by \cref{lem:tbal} it is affine. Moreover, the entire cylinder lies
  in the affinoid domain
  $G \coloneqq r^{-1}(\cG)$  for small convex $x \in \cG \subset \Sk^{\sm}(U)$, contained (since $x$ is
  general, and we assume $(\bbY,\bbD)$ is almost minimal)  in the smooth locus of the SYZ fibration
  (by \cref{prop:GluingEnds}  ). 
  This is an affinoid domain of an algebraic torus (realized in the same SYZ way). 
  So mod $I$ the scattering diagram reduces
  to the toric case, where all scattering functions are trivial. This completes the proof.
  \end{proof}

\begin{construction}
  By the same argument used to prove \cref{prop:wall_decomposition} we can find, for each $k > 0$, a finite collection of walls, as in \cite{HDTV}, $\fd$ with $f_x$ modulo $I_{k}$ constant in
  $x$ for generic $x$ in $\fd$, containing any generic point with $f_x$ non-trivial (mod $I_{k}$).
We can (subdividing) assume each $\fd$ lies in a maximal cone of $\Sigma$.
\end{construction}

Now we check this collection of walls gives a scattering diagram in the sense of
\cite[2.11]{HDTV}, whose terminology we follow.

\begin{lemma} Let $j$ be a joint of $\fD$ of codim zero or one (see \cite[2.15]{HDTV}).
Let $x \in j$ be an interior point.
Let $\fd$ be a wall containing $x$.
Then any monomial $z^{q}$ that appears with non-zero coefficient in the scattering function $f_{\fd}$ has $q \in \cP_{x}^{+}$.

The wall crossing automorphism for $\fd$
is an automorphism of $Q[\cP_x^{+}]/I_{k}$.
\end{lemma}
\begin{proof}
  When the joint has codimension zero this follows from \cref{prop:wall-crossing_function},
  see \cref{rem:scmio}, because $\cP_x^{+}$ is constant (with
  respect to $x$) near $x$.
So we can assume the codimension is one, $j \subset \rho \in \Sigma^{n-1}$.
If $\fd \subset \rho$ then it again follows from \cref{prop:wall-crossing_function}, so we assume not.
Let $y \in \fd$ be a general point.
Now we consider a twig-disk for a cylinder contributing to $\fd$, and consider its monomial.
More precisely, consider $q' = \varphi(v) + \alpha$, where $v$ is the monomial, and $\alpha$ the class of the twig.
$q$ as in the statement is a sum of such monomials, $q'$, so we can prove the analogous statement for $q'$,
which we now call $q$. 
Clearly $q  \in \cP_y^{+}$.
So if $-v$ points towards the wall $\rho$, i.e.\ if (the primitive dual class, $N$, vanishing on $\rho$, positive on $y$, is non-negative on $v$) this follows from the convexity of $\varphi$, just as in the proof of \cite[2.39]{HDTV}.
So we can assume $v$ points towards the wall.
The statement is equivalent to showing that the coefficient of the cycle for the twig-disk is at least $N(v) [C_{\rho}]$.
This follows from \cref{prop:dccoeff}.
Now the final statement follows, see the proof of \cite[2.40]{HDTV}.
This completes the proof. 
\end{proof}

By the Proposition, $\fD$ is a scattering diagram in the sense of \cite{GSCWS}.
We refer to $\fD$ as the {\it canonical scattering diagram}.
We conjecture it is the same as the canonical scattering diagram of \cite{GSCWS}.
This holds in cluster cases, by \cite[1.19]{Keel_Yu_The_Frobenius}.

\begin{remark} We note that this equality is not obvious from
the definitions, e.g. our scattering diagram is defined in terms of counts of analytic cylinders, while the
scattering diagram of \cite{GSCWS} is more like a count of twigs
(the central fibre of a formal model for a twig gives rise to a log curve of the sort the Gross-Siebert wall crossings
count).  Counting cylinders has important advantages over counts of {\it twigs}, e.g. in the no rats case, these
(or rather the associated complete rational curves) are free (as in
\cref{def:MsmX}),
and the counts are naive (as opposed to virtual), which for example gives positivity of the coefficients of the
scattering functions.
\end{remark} 

Now we prove: 

\begin{theorem} \label{th:big_consist} The canonical scattering diagram is consistent, in the sense of
  \cite[2.19]{HDTV}.
\end{theorem}
\begin{proof} Following the definition, we check consistency in codimension zero,
one, and two (this is the definition of consistent).
For codimensions zero and one, the interior of the joint is in the smooth locus, and the definition of consistency is of the Kontsevich-Soibelman sort,
that a composition of wall crossing automorphisms is the identity (in the local ring associated to the joint).
For codimension two, the joint could be contained in the singular locus, and then the Kontsevich-Soibelman consistency does not make sense (there is no toric local ring), instead the definition is a version of theta function consistency, but in terms of {\it local theta functions}
for the local affine manifold $(B_j,\cP_j)$ as in \cite[2.2.3]{HDTV}.
This holds by the same argument as in the proof of \cref{prop:theta_function_consistent} above,
using \cref{lem:localthetabl}. 
For codimensions zero and one the (ordinary) theta functions satisfy \cref{prop:theta_function_consistent}.
This implies the desired Kontsevich-Soibelman consistency (Definitions \cite[2.17]{HDTV} and \cite[2.16]{HDTV}) since the local theta functions generate the local ring in question.
Indeed only the {\it straight broken line}, more precisely, one with no twigs, is non-trivial mod $I$, from this generation is easy.
This completes the proof.
\end{proof}

\begin{proof}[Completing the proof of \cref{thm:lcsings}]
  All that remains is to show that the fibres are integral. First assume $(\bbY,\bbD)$ is almost minimal,
  \cref{def:am}, and so the scattering diagram results above apply.
  
Using the $\bbT_D$ action as in the proof of (the other statements of)
\cref{thm:lcsings}, it's enough to show that the total space is integral along the central fibre, and for this it is enough to show the total space is normal generically along each component of the double (i.e. non-normal) locus of the vertex central fibre.
Now we have a main result of \cite{Gross_Theta_Functions} (generalizing the main construction of \cite{Gross_Mirror_symmetry_for_log_Calabi-Yau_surfaces_I_v1}):
The scattering diagram determines a flat deformation over $R_{\NE}/m^k$, where $R_{\NE} \coloneqq Q[\NE]$
(and $\NE$ is
the monoid introduced at the start of \cref{sec:P+}), 
whose global functions are
canonically identified with the base extension of $A_{(\bbY,\bbD)}/m^k$ to $R_{\NE}$.
Moreover,
outside of codimension two the deformation has simple toric charts, the purely toric Mumford deformation given by
$Q[P_y^{+}]/m^k$, see \cite[2.8]{Gross_Theta_Functions}. 
Now it is enough to show the inverse limit (over $k$) is integral and this holds since the ring associated to a toric monoid (e.g.\ $Q[P_y^{+}]$) is normal (which implies the same holds for the inverse limit).
This completes the proof in case $(\bbY,\bbD)$ is almost-minimal. We are assuming almost minimal
in case \cref{ex:CYD}, so without the almost-minimal assumption we are in case \cref{ex:affine} or \cref{ex:pipe}.
In these cases, by \cref{cor:fmf},
the set of fibres over the structure torus is independent of compactification, and so we reduce to the
almost-minimal case using \cref{prop:am}. This completes the proof. 
\end{proof}
\subsection{Construction of the canonical \texorpdfstring{$\varphi$}{φ}} \label{sec:varphi}
When $s: \NE(\bbY_s,\bbZ) \to \NE$ in \cref{sec:P+}  factors through $\NE(\bbY_s,\bbZ) \to \NE(\bbY/\obbY,\bbZ)$ there
is a natural geometric construction of the $A_1$ torsor $p: P \to \Sk(U)$, and then the section $\varphi$ becomes
an instance of a very general Berkovich construction:

\begin{construction} \label{const:varphi}
Let $\fX, \fY$ be formal schemes locally of finite presentation over the ring of integers $\kc$.
Let $f\colon\fX\to\fY$ be a flat morphism with geometrically integral fibers.
We can construct a canonical set-theoretic section $\varphi\colon\fY_\eta\longto\fX_\eta$ as follows:

For every $y\in\fY_\eta$, let $\cH(y)$ denote its complete residue field.
Let $\fZ$ be the pullback of $\fX$ along $\Spf\cH(y)^\circ\to\fY$.
By assumption, it is an admissible formal scheme over $\cH(y)^\circ$
with integral special fiber $\fZ_s$, and generic fiber $\fZ_\eta\simeq(\fX_\eta)_y$.
Let $\pi\colon\fZ_\eta\to\fZ_s$ be the reduction map (see \cite[\S 1]{Berkovich_Vanishing_cycles_for_formal_schemes}).
By \cite[2.4.4(ii)]{Berkovich_Spectral_theory}, the preimage of the generic point of $\fZ_s$ by $\pi$ gives a unique point in $\fZ_\eta\simeq(\fX_\eta)_y$, which we take to be the image $\varphi(y)\in(\fX_\eta)_y\subset\fX_\eta$.

\end{construction}

\begin{construction} \label{const:torsor}

  Now we specialize to the setting of the paper, \cref{ass:basicsetup}.
  We assume we have a finite rank lattice $A^1$ together with a linear $s: A^1 \to \Pic(\bbY)$ whose image
  contains an ample line bundle. 
We consider the corresponding torsor
\[p \colon \cT \coloneqq\underline{\Spec}\bigg(\bigoplus_{L\in A^1} s(L)\bigg)\longto \bbY,\]
which is a principal \[\bbT^{A^1}\coloneqq \bbT_{A_1} \coloneqq A_1 \otimes_\bbZ \bbG_{m,k} = \Spec(k[A^1])\]
bundle, and the action of $\bbT^{A^1}$ is given by the $A^1$-grading. This induces in a natural
way a $T_{A_1} \coloneqq \bbT_{A_1}^{\an}$ torsor over the generic fibre $Y$, which we indicate as $p:T \to Y$. 
Denote the restriction by $p\colon G \coloneqq T|_U \to U$. 
The torus bundle has a canonical relative volume form, wedging with $\omega_U$
gives a volume form on $G$, 
and we have the skeleton $\Sk(G)$.
$p\colon \Sk(G) \to \Sk(U)$ is a principal $A_1$-bundle 
Applying \cref{const:varphi}, we obtain a canonical section $\varphi\colon \Sk(U) \to \Sk(G)$.
\end{construction}

\begin{remark} \label{rem:varphiBT_explicit}
Let us give a more explicit description of $\varphi\colon\Sk(U) \to \Sk(G)$.
Let $S_1,\dots,S_m \in A^1$ be a basis so that $\cT \simeq L_1^\times \times \dots L_m^\times$, where $L_i$ denotes the dual of
$s(S_i)$, and $L_i^\times \subset L_i$ denotes the complement of the zero section.

Let $p\colon \ocT\coloneqq \bbP(\oplus_i L_i \oplus \cO)\to \bbY$.
Let $\Sigma \coloneqq \Sigma_{(\bbY,\bbD^{\ess})}$ and $\tSigma \coloneqq \Sigma_{(\ocT,(\ocT \setminus \cT)^{\ess})}$,
so that $\Sk(U) = |\Sigma|$, $\Sk(G) = |\tSigma|$.
Note we have a natural inclusion of cone complexes $\Sigma \subset \tSigma$, the subcomplex generated by the rays corresponding to the irreducible components of $p^{-1}(\bbD^{\ess})$;
moreover, we have a natural projection $\tSigma \to\Sigma$.
It is easy to check that these coincide with $\varphi,p$ of \cref{sec:varphi}. 
\end{remark}

\section{$\varphi$ and disk classes} \label{sec:varphidisk}
We take $(\bbY,\bbD)$ as in \cref{ass:basicsetup}. For a homomorphism $h: M \to \Pic(\bbY)$, with $M$ a finite
rank lattice, we have an associated torsor $\bbT^M \to \bbY$, a principal $\bbT_N \coloneqq M \otimes \bbG_m$
bundle, where $N = \Hom_{\bbZ}(M,\bbZ)$,
The construction is covariant: if we have
$M_1 \to M_2$, then there is a natural projection $\bbT^{M_2} \to \bbT^{M_1}$.
If we choose a basis $\bbZ^m \to M$,
this induces an snc compactification,
$$
\bbP \coloneqq \bbP(O \oplus L_1 \oplus \dots L_r) \supset L_1^{\times} \times \dots L_r^{\times} = \bbT^M
$$
where $L_i \coloneqq h(e_i)$
(the image of
the standard basis vector) and $L_i^{\times} \to \bbY$ is the corresponding $\bbG_m$ bundle.
Let $\bbE \subset \bbP$ be the boundary. 
Now passing to generic fibres this determines $(P,E) \to Y$,
$T \coloneqq P \setminus E \to Y$, is a $\bbT_N^{\an}$ principal bundle. Let
$G \coloneqq T|_U \to U$. $G \to U$ agrees with the $\bbT_N^{\an}$,
bundle over $U$ in \cref{const:torsor}, 
in particular $T \to Y$ and $G \to U$ are independent of the chosen basis of $M$. 

Now let $f: B \to Y$ be a structure disk, \cref{def:structuredisk}.

We make:

\begin{assumption} \label{ass:trivlin} We assume the pullback (under $f$) of every line bundle on $\bbY$
  is trivial.  We note by \cref{lem:pictriv} this is no condition if $B$ is an actual disk, as we can
  assume is the case for the {\it main component} of a structure disk, see \cref{def:structuredisk}.
    \end{assumption} 
So we can find a
lift $\tf: B \to T$. Assuming $f$ is skeletal (see \cref{def:skeletal}, and recall from \cref{def:structuredisk}
we can compute structure constants
with skeletal structure disks), we compose $\tf$ with the canonical fibrewise retraction, see
\cref{prop:canr} and \cref{prop:canequaltrop} to obtain
$\tf^{\trop}: \Gamma^{\circ} \to \Sk(G)$, where $\Gamma \subset B$ is the convex hull of the marked points, and
$\Gamma^{\circ}$ means we remove the vertices associated to $\Bnd$ (recall these are the points of $E$ that
map to the boundary $D \subset Y$). $p \circ \tf^{\trop} = f^{\trop}$, the canonical spine of the skeletal curve, see
\cref{sec:atc}. 
We can also pullback $\varphi: \Sk(U) \to \Sk(G)$, to obtain a second
$\varphi: \Gamma^{\circ} \to \Sk(G)$ (lifting $f^{\trop}$). The difference is an $N_{\bbR} \coloneqq \Hom(M,\bbR)$
valued function $\tf^{\trop} - \varphi: \Gamma^{\circ} \to \bbR$ (because $p: \Sk(G) \to \Sk(U)$ is an $N_{\bbR}$ torsor).

\begin{proposition} \label{prop:varphidc} Notation as immediately above. 
  $d_b(\tf^{\trop} - \varphi) = [B \to Y] \in N_{\bbR}$.
  where $b \in B$ is the Berkovich boundary point, and $d_b$ means the sum of the derivatives over all
  edges of $\Gamma$ incident to $b$.
\end{proposition}
\begin{remark} When the spine $f^{\trop}$ is transverse (which we can always assume for computing structure
  constants) there will be no twigs of the ios tropicalisation incident to $b$, and so only single edge
  in the definition of $d_b$. \end{remark}
\begin{proof} The constructions commute with projections (as at the start of \cref{sec:varphidisk}) so 
  it's enough to prove that for the torsor given by a single 
  $L \in \Pic(\bbY)$, the analogous sum has height $[p \circ f\colon C \to Y] \cdot c_1(L)$.
  $f^*(L)$ is trivial, we choose a nowhere vanishing section. 
  After base field extension we can choose a formal model $f: \bbB \to \bbY$ for $f$, and view
  the section as a rational section of $L|_{\bbB}$, regular and nowhere vanishing on the generic
  fibre, $\bbB_{\eta} = B$. Let $F$ be the associated Cartier divisor (which is supported on
  the central fibre of the model $\bbB$). We have the associated $F^{\trop}: \Gamma^{\circ} \to \bbR$,
  see \cite[15.1]{Keel_Yu_The_Frobenius}. 
  By \cref{prop:canequaltrop}, we
  can replace
  the canonical retraction by the retraction given by the compactification $\bbP$. Now tracing through the
  definitions one sees $F^{\trop} = \tf^{\trop} - \varphi$. 
  Since $F$ is supported on the central fibre, $F^{\trop}$ is constant near the infinite vertices
  associated to $\Bnd$ (the contact points with $D$), by \cite[15.2]{Keel_Yu_The_Frobenius}.
  Now the result is immediate from basic convexity, \cref{prop:basicconvexity} and \cite[15.3]{Keel_Yu_The_Frobenius}.
  This completes the proof. 
\end{proof}

\begin{corollary} \label{cor:bi} Let $(\bbY_i,\bbD_i)$, $i =1,2$ satisfy \cref{ass:basicsetup},
  compactifications of the same $\bbU$ in the sense of \cref{def:abuse}.
  We assume the birational map $\bbY_i \dasharrow \bbY_j$ is small.
  Note $f$ induces an isomorphism of the Picard groups and thus an
  identification $A_1(\bbY_i/\obbY) = A_1(\bbY_j/\obbY) \eqqcolon A_1$. We assume $\Pic(\bbY_i)$ is
  a finite rank lattice, and so we can consider the associated torsors. Note $G_1 \to U$ and $G_2 \to U$
  are canonically identified, and so we have two sections $\varphi_i$ of $G \to U$.

  We consider the mirror algebras and their 
  localisations $M_i \coloneqq A_{(\bbY_i,\bbD_i)} \otimes_{Q[\NE(\bbY_i/\obbY,\bbZ)]} Q[A_1]$.
  Note as $Q[A_1]$ modules these are free with the same basis, $\Sk(U)(\bbZ)$.
  
  The isomorphism of free $Q[A_1]$ modules $M_i \to M_j$ sending
  $\theta_u$ to $z^{\varphi_j(u) - \varphi_i(u)} \theta_u$ is an isomorphism of $Q[A_1]$ algebras.
  \end{corollary}
  \begin{proof} A structure disk $[f_i: B \to Y_i]$ canonically induces $[f_i: B \to Y_j]$ (note
    $f: B^{\circ} \to U$, where this means the complement of the $\Bnd$-type boundary points), the
    associated moduli spaces, and in particular the counts, are the same (by the proof of 
    \cref{lem:ISknodes}). 
    So it's enough
    to show $[f_i: B \to Y_i] + \varphi_i(u) = [f_j: B \to Y_j] + \varphi_j(u)$ (where
    the structure constant is for the coefficient of $\theta_u$ in some product of $\theta_{u_t}$).
    Decompose $f_i$ into the main component, and twigs, as in \cref{def:sot}. $[f_i:B \to Y_i]$ is
    the sum of the classes.     The twigs are complete
    rational curves mapping to $U$ (where the birational map is an isomorphism), it follows
    their classes are equal, and so it is enough to prove the equality with $f_i$ replaces by
    the main component. Now \cref{ass:trivlin} applies, and the equality 
    is immediate from \cref{prop:varphidc}: either side of the proposed equality is
    $d \tf^{trop}(u)$. This completes the proof. 
    \end{proof} 

    \subsection{The canonical retraction of an analytic torus fibration}
    Let $N$ be the dual of $M$, notation as at the start of \cref{sec:varphidisk}, and let
    $A \coloneq \bbT_N^{\an}$, and $p:T \to Y$ a principal $k$-analytic $A$-bundle. Let $S \subset T$
    be the fibrewise intrinsic skeleton, a principal $N_{\bbR}$ bundle.
    \begin{proposition} \label{prop:canr}  There is a canonical fibrewise deformation retraction $r: T \to S$.
    \end{proposition}
    Fibrewise, in the statement, means commuting with $p$. 
    \begin{proof} In \cite[\S 52]{Berkovich_Spectral_theory} there are the notions of
      $*$ multiplication and peaked points. By \cite[5.2.12]{Berkovich_Spectral_theory} the
      points $n \in N_{\bbR} \subset A$ are all peaked, and $n * m = n + m$. Then 
      \cite[5.2.8(v)]{Berkovich_Spectral_theory} gives a continuous map $T \to T$ defined by
      $t \to 0 * t$. By (ii) of the same proposition, the restriction of this map to the fibre
      $T_y \coloneqq p^{-1}(y)$ is identical to the $*$-multiplication $t \to 0 * t$ defined via
      the action $A_{\scrH(y)} \times T_y \to T_y$ (note this is the self-action of $\bbT_N^{\an}$ base
      extended to $\scrH(y)$). So for the proposition it is enough to prove $*$-multiplication by
      $0$ on $\bbT_N^{\an}$ defines the canonical retraction $r_N: \bbT_N^{\an} \to N_{\bbR} = \Sk(T_N)$.
      By \cite[5.2.8(i)]{Berkovich_Spectral_theory} we can work over any base field extension to
      check this, so we consider $T_N^{\an}$ defined over an algebraically closed nontrivially valued $K$. Fixing a $K$-point $t \in \bbT_N^{\an}$, $0 *t$ is the image of $0$ under the multiplication by $t$
      automorphism (see the paragraph following \cite[Corollary 5.2.7]{Berkovich_Spectral_theory}),
      which is precisely $r_N(t)$. Now the claim follows from continuity, since $K$-points are dense. 
    \end{proof}

    We follow the notation from the start of \cref{sec:varphidisk}. Let $\Sigma \coloneqq \Sigma_{(\bbY,\bbD^{\ess})}$
    and $\tSigma \coloneqq \Sigma_{(\bbP,\bbW^{\ess})}$ (the essential dual fans). Note we have canonical
    identifications $|\Sigma| = \Sk(U)$, $|\tSigma| = \Sk(G)$. Let $r_{\bbP}: T \to \Sk(G)$,
    $r_{\bbY}: U \to \Sk(U)$ be the associated retractions. 

    \begin{proposition} \label{prop:canequaltrop} The composition
      $$
      w: S|_{\Sk(U)} \subset T|_{\Sk(U)} \overset{r_{\bbP}}{ \to} \Sk(G)$$
      is an isomorphism. The induced retraction $w: T|_{\Sk(U)} \to S$ is equal to the canonical
      retraction $r$ of \cref{prop:canr}.
    \end{proposition}
    \begin{proof} $w$ is an isomorphism by \cite[8.8]{Keel_Yu_The_Frobenius}. 
      Either retraction commutes with the projection to $\Sk(U)$ so it's enough to check the equality
      on a single fibre of $T_{\Sk(U)} \to \Sk(U)$. This reduces to the statement that the retraction
      induced by any toric compactification of $\bbT_N$ is the canonical retraction.
      \end{proof}

    \subsection{The mirror family over $\TV(\Sec)$} \label{sec:mfs}

    Now we can obtain one of our main applications.

    \begin{lemma} Assumptions as in \cref{ass:basicsetup}.
      The natural restriction map $\Pic(\obbY) \to \Pic(\obbY_s)$ is an isomorphism. If $\bbY_s$ is the Spectrum of
      the Stanley-Reisner ring of a simplicial complex satisfying (1-3) of \cref{lem:slcdclem} then these Picard groups
      are both trivial. 
    \end{lemma}
    \begin{proof} Let $I$ be the ideal sheaf of $\obbY_s \subset \obbY$. We have the exact sequence
      $$
      0 \to (1 + I^k/I^{k+1}) \to (\cO/I^{k+1})^* \to (\cO/I^{k+1})^* \to 0
      $$
      Note as sheaves of Abelian groups $(1 + I^k/I^{k+1})$ is isomorphic to the coherent sheaf $I^k/I^{k+1}$ so
      $H^i(1 + I^k/I^{k+1}) = 0$ for $i > 0$ (as $\obbY$ is affine). The isomorphism follows.

      Now suppose the complex satisfies (1-3). Then by \cref{lem:slcdclem} $\bbY_s$ is $S_2$ and so
      the restriction from $\Pic(\bbY_s)$ to $\CH_{d-1}(\bbY_s)$ is injective, and the latter vanishes as the irreducible
      components of $\bbY_s$ are affine spaces. This completes the proof. 
    \end{proof}

    \begin{assumptions} \label{ass:sec} We assume we have a one parameter affine formal family
      $v:\obbY \to \Sp(k[[s]])$ as in \cref{ex:pipe}. We assume that the singular locus of $v$ is contained in the central fibre. 
      We assume we have projective birational dlt model: ie projective 
      $f: \bbY \to \obbY$, with $\bbY$ normal and $\bbQ$-factorial,
      such that the fibre $\bbD \subset \bbY$ is reduced, the exceptional
      locus of $f$ is contained in $\bbD$. Morever we assume $K_{\bbY} + \bbD$ is dlt and $f$-trivial,
      $\Pic(\bbY)$ is a finite rank lattice, and $\Pic(\obbY)$ is trivial.
      Finally assume the essential dual complex for $(\bbY,\bbY_s)$ satisfies
      (1-3) of \cref{lem:slcdclem}.
    \end{assumptions}

    Let $\bbT \to \bbY$ be the universal torsor, and we use the construction of \cref{sec:varphi}.
    In this case by \cite[Cor. 1.3.1]{BCHM}
    $f$ is a relative Mori dream space. In particular $\MoriFan(\bbY/\obbY)$ is a complete
    rational polyhedral fan (with support $\Pic(\bbY)_{\bbR}$). 
    As we will always be working over $\obbY$, we mostly leave that out of
    the notation. By \cite[\S 2]{HKY20}, there is a canonical coarsening $\Sec(\bbY)$ (with
    support $\Pic(\bbY)_{\bbR}$) with a subfan $\MovSec(\bbY)$ (with support $\Mov(\bbY)_{\bbR}$).
    We recall that maximal cones of $\MovSec(\bbY)$ correspond to small $\bbQ$-factorial modifications (SQMs)
    $\bbY \dasharrow \bbY_{\alpha}$. Each $\bbY_{\alpha}$ determines a canonical section $\varphi_{\alpha}: \Sk(U) \to \Sk(G)$.
    There is one maximal cone of $\MovSec(\bbY)$ for each such section, $\lambda$, it is the union of the maximal cones
    $\Nef(\bbY_{\alpha}) \in \MovFan(\bbY)$ with $\varphi_{\alpha} = \lambda$.

    We consider one such SQM $\bbY_{\alpha}$. The mirror algebra, $A_{\alpha} \coloneqq A_{(\bbY_{\alpha},\bbD)}$
    is naturally graded (by \cref{prop:filtration} applied to the pullback
    of $s$ on $\Sp(k[[s]])$). Recall that as $R_{\alpha} \coloneqq Q[\NE(\bbY_{\alpha}/\obbY,\bbZ)]$ module,
    $A_{\alpha}$ is free with basis $\Sk(U)(\bbZ) = |\Sigma_{(\tbbY_{\alpha},\bbD^{\ess})}|$ for any log resolution
    $(\tbbY_{\alpha},\bbD) \to \bbY_{\alpha}$ (when $(\bbY_{\alpha},\bbD)$ is dlt its dual complex is defined,
    see \cite[\S 2]{KdFX}, and we do not need the resolution). Let $I \subset A_{\alpha}$ be the ideal generated by $\theta_u$ with
    $u$ in the interior of $\Sk(U)(\bbR)$, or equivalently, corresponding to valuations with center the vertex
    of $\bbV$. We note each irreducible component $\bbW \subset \bbD$ gives $\theta_{\bbW} \in A_\alpha$ of degree
    one. Let
    $$
    \Theta \coloneqq Z(\sum_{\bbW \subset \bbD} \theta_{\bbW}) \subset \cX \eqqcolon \Proj(A_{\alpha}).
    $$
    Let $\cE \coloneqq Z(I) \subset \cX$. So we have a mirror family of tuples
    $(\cX,\cE,\Theta,\cO(1)) \to \TV(\Nef(\bbY_{\alpha}))$.

    \begin{proposition} Let $P_{\lambda}$ be dual to a maximal cone of $\MovSec(\bbY)$
      containing the maximal cone $\Nef(\bbY_{\alpha}) \in \MovFan(\bbY)$. Then the structure
      constants in $A_{\alpha}$ can be computed with disks with class in $P_{\lambda}$. Thus
      $A_{\alpha}$ is base extended from an algebra (with the same basis) over $Q[P_{\lambda}]$.
    \end{proposition}
    \begin{proof} It's enough to show the disk class lies in
      $\NE(\bbY_{\alpha}) \subset A_1(\bbY,\bbZ)= A_1(\bbY_{\alpha},\bbZ)$. This follows from \cref{prop:varphidc}: The
      structure disk gives a disk in each $Y_{\alpha}$, and by \cref{prop:varphidc} they all have the same class.
      In particular
      this class lies in each $\NE(\bbY_{\alpha},\bbZ)$. 
    \end{proof}

    It follows the mirror family is pulled back from a canonical family
    $(\cX,\cE,\Theta,\cO(1)) \to \TV(P_{\alpha}^{\vee}) \coloneqq \Spec(\bbZ[P_{\alpha}])$.

    \begin{proposition} \label{prop:mfms} Assumptions as in \cref{ass:sec}. 
      The mirror families canonically glue to give a family of tuples
      $$(\cX,\cE,\Theta,\cO(1)) \to \TV(\MovSec(\bbY/\obbY)).$$

        For each fibre, $K_X + E$ is trivial and SLC. For fibres over the structure torus $\bbT_{\Pic(\bbY)}$,
      $K_X + E$ is log canonical, in particular $X$ is normal.
    \end{proposition}
    \begin{proof} 
      The singularity statements of the fibres follow from \cref{thm:lcsingspipe}.
      The gluing holds by (the arguments in) \cref{sec:itb},
      see \cite[6.14]{HKY20}.
    \end{proof}

    Now the torus action, \cref{sec:ta}, determines a canonical extension of the mirror family over a Deligne-Mumford
    stack with coarse moduli space 
    the toric variety for the full (complete) secondary fan, exactly as in \cite[\S 6.4]{HKY20} (whose notation we
    follow). No new fibres $(X,E)$ are introduced, but some coefficients of $\Theta$ shrink to zero. We obtain: 

    \begin{theorem} \label{thm:mfs} Assumptions as in \cref{ass:sec}. The family of tuples
      extends canonically over the full $\cTV(\Sec(\bbY/\obbY))$. 
      with fibres satisfying the statements
      in \cref{prop:mfms}.
    \end{theorem}

      \begin{conjecture}\label{conj:mfs}  Assumptions as in \cref{thm:mfs}. Then for sufficiently small $\epsilon > 0$, each fibre
        $(X,E + \epsilon \Theta)$ of the mirror family is stable, and the induced map
        $\TV(\Sec(\bbY/\obbY)) \to \oM_{\SP}$ (the target is the moduli space of stable pairs)
        is finite.
      \end{conjecture}

\begin{remark} Note by \cref{thm:mfs}, each fibre
    $(X,E)$ is slc, so the conjectured stability is equivalent to the statement that $\Theta$ contains no
    log canonical center of this pair. This is closely related to \cref{conj:ind}, or more precisely, its
    analog for sections of an ample line bundle. 
    \end{remark}

    \begin{remark} Here we began with a dlt relative minimal model $\bbY \to \obbY$. But there is no reason
      we need this to be minimal, we could take any blowup (with exceptional locus contained in the central fibre)
      so that $(\bbY,\bbD)$ is dlt, the same procedure produces a family of tuples over $\TV(\Sec(\bbY))$. For any
      such blowup, the bases of the algebra is the same, $\Sk(U)(\bbZ)$, which is realized as the integer points in the
      cone over the essential dual complex for $(\bbY,\bbD)$. When the model is not minimal, there will be non-essential
      components of $\bbD$. These do not add moduli -- the associated subtorus of $\bbT_D$ (see \cref{thm:torus_action}  )
      acts equivariantly on the
      family {\it including $\Theta$} (Note:
      this is not the case for essential components, the corresponding subtorus acts equivariantly
      on $(\cX,\cE,\cO(1))$ but scales the coefficients of $\Theta$, indeed over an orbit of $\bbT_{D^{\ess}}$, $(X,E,\cO(1))$ is
      constant, but $\Theta$ varies over the {\it theta function structure torus} of $|\cO(1)|$ of members with all non-zero
      coefficients in the canonical basis). 
      So if $\bbY \to \obbY$ is not minimal, the family
      we obtain is obviously not modular: the conjectured map to moduli of stable pairs has positive dimensional
      fibres. That is why we focused here on the minimal case. \end{remark}

\begin{remark} \label{rem:roc} \cref{thm:mfs} depends on the assumption that the singular locus of
      $v: \obbY \to \Sp(k[[s]])$ is contained in the central fibre. And moreover that the dual
      complex $\cD_{(\bbY,\bbD)}$ satisfies the vanishing conditions (1-3) of \cref{lem:slcdclem}.  Both
      of these hold in the Fano case of \cref{ex:pipe}, see \cref{rem:Fano}. 
      So in the Fano case we have the full construction, in all dimensions, unconditionally.
      We note that in this case the fibres of the
      mirror family are (in general singular) Fanos: $E \in |\cO(1)|$, and $-K_X > 0$. See \cite[7.1]{HKY20}.

      We believe that for the pipe dream in general (running the mirror program twice as described in
      \cref{sec:rpt}) the vanishing conditions (1-3) for
      the essential dual complex to $(\bbY,\bbY_s)$ 
      are implied by a conjecture of Gross, see  \cref{rem:Fano}.
      \end{remark}

    \section{connection with Conjectures \cite[1.1-1.2]{HKY20} } \label{sec:rpt} 
Let $(X,E)$ be an anti-canonical snc divisor on a smooth projective variety, and assume $E$ has
a zero stratum (so the complement $X \setminus E$ is a smooth log CY with maximal boundary).
Let $\bbV$ be the vertex associated with the dual complex of $(X,E)$ (i.e. spectrum of the associated
Stanley-Reisner ring). The Gross-Siebert mirror construction produces a canonical affine formal family
$\cY \to \Sp(\hR_X)$ (where $\hR_X$ is the completion of $\bbZ[\NE(X,\bbZ)]$ at the maximal monomial ideal).
\begin{remark} If $X \setminus E$ is proper over an affine variety we can use the mirror construction of
  this paper in place of Gross-Siebert. 
\end{remark}
Let $L \in \Ample(X)$ and 
let $\Sp(k[[s]]) \to \Sp(\hR_X)$ be induced by the closure of a general orbit for the one parameter subgroup
of $\bbT_{\Pic(X)}$ associated to $L$, acting on 
  on $\TV(\Nef(X)) = \Spec(\bbZ[\NE(X,\bbZ)])$. Let $\obbY \to \Sp(k[[s]])$ be the associated deformation of $\bbV$.

  \begin{conjecture} \label{conj:dmc} Notation as immediately above. Assume the singular locus of $\obbY \to \Sp(k[[s]])$ is
    contained in the central fibre. 
    Apply the construction of \cref{sec:mfs} to $\obbY \to \Sp(k[[s]])$, to obtain the canonical family
    of tuples $(\cX,\cE,\Theta,\cO(1)) \to \TV(\Sec(\bbY/\obbY))$.

    For sufficiently small $\epsilon > 0$, for each fibre $(X,E + \epsilon \Theta)$ is stable. The induced
    map $\TV(\Sec(\bbY/\obbY)) \to \oM_{\SP}$ (the target is the moduli space of stable pairs) is finite, with
    image the irreducible component containing $(X,E + \epsilon \Theta)$ for a general $\Theta \in |L|$.
  \end{conjecture}

  The two dimensional case of \cref{conj:dmc} is the main result in \cite{HPV24}.

  \begin{remark} \label{rem:MCCYs}
    Here we considered the mirror construction in case \cref{ex:pipe}. Note we can consider an analogous
    construction in case \cref{ex:CYD}. One special feature of the \cref{ex:pipe} case is the birational
    map $\bbY \to \obbY$, with exceptional locus contained in the central fibre $\bbD$. This implies that for SQMs
    $\bbY \dasharrow \bbY'$, $U = U'$, and in particular is smooth so our construction applies. In relative
    dimension two this is automatic, and we can recover the construction of the universal family in \cite{GHKSK3} by
    a considerably simpler argument (\cite{GHKSK3} follows the GS formalism, in which the birational invariance needed
    to glue together the mirror families for the different SQMs is delicate, deduced from the scattering diagram,
    with our formulation it is automatic).
    We can do the same thing in any dimension if we add the assumption that the generic fibre of
    $\bbY \to \obbY = \Sp(k[[s]])$ has no divisor swept out by rational curves, e.g. in relative dimension three if
    a version of Clemens conjecture holds: The absence of such divisors implies that our counts are naive
    (the free part of the structure disk misses any codimension two analytic subset, in particular all the complete
    rational curves, so there are no {\it twigs}), and the structure disks will avoid the exceptional locus of any
    SQM giving the necessary birational invariance to glue the mirror families just as in the \cref{ex:pipe} case above.
  \end{remark}
  \begin{remark} \label{rem:wct} 
    We note why it is important to have the {\it correct} base for the family: As noted in
    \cite[8.7]{HKY20} a map from a complete toric variety to another variety is finite iff no boundary
    stratum is contracted to a point. So if show the stability in \cref{conj:dmc}, then to show the map
    to $\oM_{\SP}$ is finite, it's enough to show the restriction to boundary $1$-strata is non-constant. While
    the mirror family is very complicated, over $0$ strata associated to maximal cones of $\Mov(\bbY/\obbY)$    it is a purely combinatorial object, the vertex, see \cref{sec:vertex}, associated
    to the dual complex of $(\bbY,\bbD)$. It's enough to show for each $1$-strata the
    fibres over its two $0$ strata are different, e.g. when the two $0$ strata correspond to maximal
    cones of $\Mov$, that the associated dual complexes are different (note in the toric case, this is
    how the secondary fan is defined). The point here: if we view \cref{conj:dmc} as having two
    parts, stability and modularity, we expect modularity is relatively easy. Stability we think is deeper. Note it is closely related to \cref{conj:ind}. 
      \end{remark}

\section{Conjectures} \label{sec:conj}
Now we give a number of conjectures indicating how we expect basic constructions from toric geometry to generalize to affine log Calabi-Yau varieties with maximal boundary. We
continue to follow the notation of the paper, \cref{ass:basicsetup}, but exclusively in the case \cref{ex:affine}, so
there are no formal schemes, and the main object is a smooth affine log CY $\bbU$ with maximal boundary. 

\begin{definition} By a {\it partial minimal model} for a log Calabi-Yau variety $\bbU$ we mean a Zariski open embedding
  $\bbU \subset \bbY$ with $\bbY$ normal having divisorial boundary and such that the volume form has a pole on each
  irreducible component of the boundary.
\end{definition}
\begin{example} \label{ex:tv} The elementary theory of toric varieties implies that an affine (resp.
  complete) partial minimal model of an algebraic torus $\bbU = \bbT_M \coloneqq M \otimes_Z \bbG_m$ (where $M$ is a finite rank free
  Abelian group, the co-character lattice of $\bbT_M$) is the same thing as an affine (resp.
complete) $\bbT_M$-toric variety.
More generally, any $\bbT_M$-toric variety is a partial minimal model, and any partial minimal model is equal,
outside codimension two, to a $\bbT_M$-toric variety.
\end{example}

Let $\bbU$ be an affine log Calabi-Yau variety and let $\bbV$ be a general fibre of the mirror family.
\begin{remark} \label{rem:coeffs} By \cref{prop:act} the restriction of the mirror family to the structure torus of
  the base is independent of the compactification, up to base extension, and quotient by the $\bbT_D$ action.
  In particular the set of isomorphism classes of fibres is independent of compactification, see \cref{cor:fmf}. Here we
  abuse notation slightly and write e.g. $\Spec(A_{\Sk(\bbU)}) \to \bbT$ for this restriction -- which really
  means, pick a compactification $\bbU \subset \bbY$
  and then take the restriction of that mirror family to the structure
  torus $\bbT_{\Pic(\bbY)}$. Here the subscript $\Sk(\bbU)$ indicates the basis, but we deliberately omit the coefficients.
  We will use similar notation to indicate subalgebras given by a subset of the basis. 
  \end{remark}

In general $\bbV$
will be singular, we conjecture at worst canonical. If so, we can construct a mirror to $\bbV$ using \cref{ex:POA}.
We conjecture this mirror is algebraic and moreover independent of the resolution.  
Here in the discussion we assume all this (or the reader can just specialize to the case when $\bbV$ is smooth).

\begin{conjecture}[double mirror] \label{conj:dm}
  $\bbU$ is a fibre of the mirror family for $\bbV$.
  In particular, $\Sk(\bbV)$ parameterizes a basis of $\cO(\bbU)$, canonical up to individual scaling.
\end{conjecture} 

\begin{remark}
  This is the ultimate goal of the project, to prove the conjecture from
  \cite{Gross_Mirror_symmetry_for_log_Calabi-Yau_surfaces_I_v1} that an affine log Calabi-Yau variety with maximal boundary has a canonical basis of regular functions. This has been proven in dimension two, see \cite{LZ23}, and in many cluster examples, see \cite{GHKK}. 
\end{remark}

Thus there are (conjecturally)
two natural pairings on $\Sk(\bbU) \times \Sk(\bbV)$,
we can take $<u,v> \to \theta_v^{\trop}(u)$, or to $\theta_u^{\trop}(v)$.

\begin{conjecture}[symmetry] \label{conj:sym} The two pairings are equal.
\end{conjecture}

\begin{conjecture}[valuative independence] \label{conj:ind} Let $v_1,\dots,v_k, w \in \Sk(\bbV)$,
and $\alpha_1,\dots,\alpha_k \in k$ non-zero constants.
Then $\ord_w(\sum \alpha_i \theta_{v_i}) = \min_i(\ord_w(\alpha_i))$.
\end{conjecture}

Both conjectures are proven in dimension two in \cite[4.13,4.15]{Mandel_Tropical_theta_functions},
some higher dimensional cases (with important representation theoretic implications) of both conjectures are proven
in \cite[9.7,9.10.3]{Gross_Canonical_bases}.
Both conjectures have been proven in a wide range of cluster examples, including all for which the
{\it Full Fock-Goncharov Conjecture} has been established,  in \cite{Travis25}.
In \cite{BL} the authors prove that $\cO(\bbU)$ admits a {\it valuatively independent} basis, i.e. a vector
space basis that satisfies \cref{conj:ind} (such a basis is not, in general, unique, but they show that the
tropicalisations, i.e. the induced set of functions on $\Sk(U)$, is independent of the particular basis).

Let $\bbU \subset \bbY$ be a partial minimal model, with $u_1,\dots,u_k \in \Sk(\bbU)$
corresponding to the irreducible components of the boundary.
Let
$$
C_{\bbY} \coloneqq \{v \in \Sk(\bbV)| \theta_{u_i}^{\trop}(v) \geq 0 \text{ for all } i\}
$$

Conjecturally $\cO(\bbU)$ has a basis $$
\cO(\bbU) = \oplus_{v \in \Sk(\bbV)} \bbZ \cdot \theta_v $$

Let $\cO(\bbU)_{C_{\bbY}} \subset \cO(\bbU)$ be the subgroup with basis $C_{\bbY} \subset \Sk(\bbV)$.

\begin{corollary} \label{conj:bas}
The above conjectures imply $\cO(\bbU)_{C_{\bbY}} \subset \cO(\bbU)$
is $\cO(\bbY) \subset \cO(\bbU)$.
In particular $C_{\bbY}$ parameterizes a theta function basis for $\cO(\bbY)$.
\end{corollary}

\begin{proof}
By the independence conjecture,
$\cO(\bbY) \subset \cO(\bbU)$ has basis those $\theta_v$ which are
regular on each of the boundary divisors to $\bbU \subset \bbY$.
And now by the symmetry conjecture, those are exactly the $v \in C_{\bbY}$.
\end{proof}

In many applications, partial minimal models are more important than the log Calabi-Yau varieties themselves.
For example: a semi-simple algebraic group is a partial minimal model of its (log Calabi-Yau) open Bruhat cell. In
\cite{Gross_Canonical_bases} special cases of the above conjectures are proven, sufficient to obtain \cref{conj:bas},
endowing the vector space of regular functions on the group, and thus every finite dimensional irreducible representation
of the group, with a canonical bases.  See \cite[\S 0.4]{Gross_Canonical_bases}.

We apply the above conjectures in special case of particular Mori theoretic importance:

Let $\bbU' \subset \bbY'$ be an snc minimal model of a smooth affine log Calabi-Yau variety with maximal boundary.
Let $\bbY \to \bbY'$ be the universal torsor (see \cref{const:torsor}), 
and $\bbU \coloneqq \bbY|_{\bbU'} \to \bbU'$ the restriction.
$\bbU$ is itself smooth affine log Calabi-Yau with maximal boundary, and $\bbU \subset \bbY$ is
an snc partial minimal model.
We note $$
\cO(\bbY) = \Cox(\bbY') \coloneqq \oplus_{L \in \Pic(\bbY')} H^0(\bbY',L).
$$

We apply the above conjectures to $\bbU \subset \bbY$.
We note that $\bbY$ (and so in particular $\bbU$) has trivial Picard group, so
all fibres $\bbV$ of the mirror family are isomorphic.
The universal torsor is a $\bbT_{A_1(\bbY',\bbZ)}$-principal bundle, this induces
$p\colon \bbV \to \bbT_{\Pic(\bbY')}$, and $p\colon \Sk(\bbV) \to \Sk(T_{\Pic(Y')}) = \Pic(Y')_{\bbR}$.
We note inverse images of the irreducible components $E' \subset \bbD' \coloneqq \bbY' \setminus \bbU'$
give the irreducible components $E \subset \bbD \coloneqq \bbY \setminus \bbU$.
The corresponding theta functions $\theta_E \in \cO(\bbV)$ generate.
Let $\theta := \sum_{E \subset \bbD} \theta_E$.
Then $C_{\bbY} \subset \Sk(V)$
is given by $\theta^{\trop} \geq 0$.

\begin{conjecture}[Cox ring conjecture] \label{conj:cox}
  $\Spec(\Cox(\bbY'))$ is one fibre of the mirror family $\Spec(A_{C_{\bbY}}) \to T$ (see \cref{rem:coeffs}). 
  
In particular, $C_{\bbY}$ parameterizes a canonical theta function basis for $\Cox(\bbY')$.
Let $P_L \coloneqq p^{-1}(L)$ for $L \in \Pic(\bbY')$.

The integer points $P_L(\bbZ)$ parameterize a canonical theta function basis for $H^0(\bbY',L)$.
\end{conjecture}

We note by \cite{Hu_Mori_dream_spaces} that $\Cox(\bbY')$ controls the Mori theory of $\bbY'$.
The conjecture gives a canonical basis, and the structure constants (so the ring structure) as counts of disks.
\cite[0.20]{Gross_Canonical_bases} gives a proof of the full conjecture in an important representation theoretic example, and shows that the {\it lattice polytopes} $P_L$ vastly generalize the Tao-Knudsen hive polytopes,
see \cite{Magee}. \cite{KW24} proves the conjecture for $\bbU'$ of dimension two. 

\subsection{The polytope construction}
Here we give a simple construction we conjecture recovers from a theta function the associated element of the skeleton
of the mirror -- see \cref{conj:bd}. 

We continue considering $\bbU \subset \bbY$ an snc compactification of smooth affine log CY with maximal boundary.
We consider a non-zero rational function $f$ on $\bbY$, regular
on $\bbU$. 

\begin{definition-lemma} \label{def:F*}  Notation as immediately above. 
  Let $F_k \subset A_{\Sk(\bbU)}$ be the subalgebra with basis $\theta_u$, $f^{\trop}(u) + k \geq 0$.
  $F_*$ makes $A_{\Sk(\bbU)}$ a filtered ring (i.e. $F_i \cdot F_j \subset F_{i+j}$).
\end{definition-lemma}
\begin{proof} $F_i \cdot F_j \subset F_{i+j}$ by basic convexity, \cref{prop:basicconvexity}.
\end{proof}

Let $\tA \subset A[T,T^{-1}]$ be the associated graded subring -- i.e. generated by $a \cdot T^k$ with $a \in F_k$.

\begin{proposition}  \label{prop:polyval} Notation as immediately above. 
  Let $P \subset \tA$ be the submodule with basis $T^k \theta_u$ with $k + f^{\trop}(u) > 0$.

  The following hold:
  \begin{enumerate} 
 \item $P \subset \tA$ is an ideal. 
 We have a canonical identification $\tA/P = \gr_{F_*}(A)$ (the associated graded of the filtered ring).
 This is the mirror algebra to the open set $f \neq 0 \subset \bbU$ (see \cref{rem:associatedgraded}).
 \item 
    $P$ is prime iff for any $u_1,u_2 \in \Sk(U)(\bbZ)$, the minimum of $f^{\trop}(u)$ over
  all $u$ such that $\theta_u$ appears in $\theta_{u_1} \cdot \theta_{u_2}$ with non-zero coefficient,
  is $f^{\trop}(u_1) + f^{\trop}(u_2)$.  In this case
  $$
  v(\sum a_u \theta_u) \coloneqq \min f^{\trop}(u)
  $$
  gives a discrete valuation of $A_{\Sk(\bbU)}$. This is an element of $\Sk(\bbV)(\bbZ)$, for $\bbV$ a general
  member of the mirror family to $\bbU$.
  \item   Consider $\Proj(S)$.
    The open set $T \neq 0 \subset \Proj(S)$ is canonically identified with $\Spec(A_{\Sk(\bbU)})$.
  If $P$ is prime, then the (reduced) boundary is irreducible, the associated divisorial discrete valuation is $v$.
\item $P$ is prime if the open set $f \neq 0 \subset \bbU$ is log CY.
  \end{enumerate}
\end{proposition}

We note that we automatically have tropical independence for theta functions evaluated at such $v$, see
\cref{conj:ind}.  

\begin{proof}
  $P$ is an ideal because $F_*$ is a filtration, and $\tA/P = \gr_{F_*}(A)$ is obvious (and holds for any filtered ring).
  By \cref{prop:basicconvexity}, structure disks contributing to the multiplication rule mod $P$ lie in $f \neq 0$. This
  gives (1). 
  
  That primeness of $P$ is equivalent to the given minimum statement is immediate from
  the definition, and it's clear that this implies $v$ is a discrete valuation. This gives (2).
  It's easy to check that $(\tA_T)_0 = A \subset A[T,T^{-1}]$, this gives the open embedding, and (3).

  Now suppose that the boundary $(T = 0) \subset \Proj(S)$ is irreducible -- note this holds when
  $P \subset \tA$ is prime. We check that $v$ is the associated divisorial valuation. But it's clear that
  $P = (T)$, and that $T^s \theta_u \in (T^k)$ iff $s + f^{\trop}(u) \geq k$. From this it follows
  easily that $v$ is the divisorial valuation associated to the prime divisor $(T = 0) \subset \Proj(S)$.
 
  Finally we check that in the prime case, $v \in \Sk(\bbV)(\bbZ)$. 
    We consider the mirror family over $\TV(\Nef(X))$ for some snc compactification $\bbU \subset X$
  (following our convention, to this point we have not mentioned the coefficients, all the statements
  follow for any compactification, and any fibre of the mirror family).  All the constructions make
  sense over $\bbZ[\NE(X,\bbZ)]$.  $f^{\trop} \geq -1$ intersects each cone of the essential dual
  complex $\Sigma$ is a convex polytope (containing $0$). The central fibre of $\Proj(S)$ is
  corresponding union of polarized toric varieties, a partial compactification of the vertex. We
  check $K + Z$ is trivial and SLC on the central fibre, this is a purely combinatorial question, follows
  as in the proof of \cref{thm:lcsings}, see \cref{prop:vertex}. 
  Now these are open conditions, so hold at least generically along the boundary divisor on
  the generic fibre.
  \end{proof}

  We expect this gives a characterisation of {\it tropical theta functions}:

\begin{conjecture} \label{conj:bd} Notation as in \cref{prop:polyval}.
  Let $f \in \cO(\bbU)$ be non-zero. Then $f^{\trop} = \theta_w^{\trop}$
  $\bbU$ mirror to $\bbW$, for some $w \in \Sk(W)(\bbZ)$ iff 
  $P$ is prime. In that case $\bbW$ and $\bbV$ are each fibres of the mirror family (for $\bbU$),
  and $w = v$.
  \end{conjecture} 
  
  \appendix

  Here we include foundational material on the Berkovich generic fibre of a special formal scheme,
  taken from \cite{LoganThesis}, generalizing well known results in the more restrictive case of
  finitely presented formal schemes.

\section{Special Formal Schemes} \label{sec:BTS} 

\noindent 
\begin{definition} \label{def:SpfScheme} (\textit{Special Formal Scheme}) Here fix $k$ a complete, nonarchimedean valued field. A formal scheme $\bb{Y}$ over $k^{\circ}$ is called \textit{topologically formally of finite type (tfft)}, or \textit{special}, if it can be covered by finitely many affine opens $\bb{U} \subset \bb{Y}$ of the form
$$ \bb{U} = \Spf(k^{\circ}[[s_1,...,s_l]]\{t_1,...,t_m\} / (f_k, k \in K)) \rightarrow \Spf(k^{\circ}), $$
where $K$ is some possibly infinite index set. For any special formal $k^{\circ}$-scheme $\bb{Y}$, the symbol $\bb{Y}_s$ will refer to the underlying reduced scheme. Every Zariski open subset of a special formal scheme takes the above form \cite{achinger2021specializationproetalefundamentalgroup}.
\end{definition}

\begin{remark}\label{remark: GenFiber} (\textit{Generic Fiber}) Here, let $k$ be trivially or discretely valued. In \cite{Berkovich_Vanishing_cycles_for_formal_schemes_II}, Berkovich extends the definition of generic fiber of a formal scheme from finitely presented to special over $k^{\circ}$. For an affine special formal scheme over $k^{\circ}$ such as $\bb{U}$ in \cref{def:SpfScheme}, the generic fiber is defined by 
$$ \bb{U}_{\eta} := \bigcup_{0 < \delta < 1} \mathscr{M}(k\{\delta^{-1} s_1,...,\delta^{-1} s_{\ell}, t_1,...,t_m\}/(f_1,...,f_n)).$$
As in the case of finitely presented formal schemes, there is an anti-continuous reduction map $\bb{U}_{\eta} \rightarrow \mathbb{U}_s$ given as follows. Let $\mathcal{A}$ be the ring of global functions on $\mathbb{U}$. Then given a point $x \in \mathbb{U}_{\eta}$, there is a natural map $\mathcal{A} \rightarrow \mathscr{H}(x)^{\circ} \rightarrow \widetilde{\mathscr{H}(x)}$ which factors through $\mathcal{A}/\sqrt{(s_1, ...,s_l, \pi)} \rightarrow \widetilde{\mathscr{H}(x)}$, with $\pi \in k^{\circ \circ}$ a uniformizer. This defines a point of $\mathbb{U}_s$.

The above construction is functorial and takes open immersions to closed subdomain embeddings, and so gluings of affine special formal schemes along distinguished affine opens lead to gluings of their generic fibers along closed affinoid subdomains. A special formal scheme $\mathbb{Y}$ over $k^{\circ}$ thus has an analytic generic fiber $\mathbb{Y}_{\eta}$. The reduction maps on the affine opens glue to give a global reduction map $\text{red}_{\mathbb{Y}} : \mathbb{Y}_{\eta} \rightarrow \mathbb{Y}_s$ [loc. cit]. 
\end{remark}

\begin{example}\label{ex:DefInOneVar} (\textit{Deformations in One Variable}) The motivating example for our formal mirror construction is a single-parameter deformation $\bb{Y} \rightarrow \Spf(k[[s]])$ obtained, for instance, by applying the Gross-Siebert mirror construction to an algebraic log Calabi-Yau pair $(X, E)$, then restricting the resulting formal family over a 1-dimensional formal open disk in the toric base, centered at the unique $0$-stratum, and finally resolving singularities. Such a $\bb{Y}$ is in particular a special formal scheme over $k[[s]]$.

Any special formal scheme over $k[[s]]$ can be viewed dually as a special formal scheme over $k$. Then one observes two relevant generic fiber constructions when $k$ is given the trivial valuation. On the one hand, after fixing some $0 < \epsilon < 1$ and endowing $k[[s]]$ with the discrete valuation $|s| = \epsilon$, the generic fiber of $\bb{Y}$ the special formal $k[[s]]$-scheme is a $k((s))$-analytic space $Y_{\epsilon}$. On the other hand the generic fiber of $\bb{Y}$ the special formal $k$-scheme is a $k$-analytic space $Y$ with a flat map to the open unit disk $\bb{D}^{\circ}$, which we note is canonically isomorphic to $[0,1)$ as a topological space. The two generic fibers are related in the obvious way\textbf{---}the fiber of $Y \rightarrow \bb{D}^{\circ}$ over $\epsilon$ is nothing but $Y_{\epsilon}$.  

\end{example}

\begin{remark}\label{remark:Coh} (\textit{Coherent Sheaves}) Now fix $k$ trivially valued, and $\bb{Y}$ a special formal scheme over $k$. We recall from \cite{Ben_Bassat_2013} the generic fiber induces a natural equivalence $\_\__{\eta} : \text{Coh}_{\mathscr{O}_{\bb{Y}}} \rightarrow \text{Coh}_{\mathscr{O}_{\bb{Y}_{\eta}}}$ on coherent sheaves. The equivalence comes from the following observation: Fix an open affine $\bb{U} \subset \bb{Y}$ as in \cref{def:SpfScheme}, with $\mathcal{A}$ its defining ring of functions. Denote by $\mathcal{A}_{\delta}$ the affinoid algebra corresponding to $\delta$ in the union defining $\bb{U}_{\eta}$ in \cref{remark: GenFiber} Then for all $\delta < \delta'$, the restriction map $\mathcal{A}_{\delta'} \rightarrow \mathcal{A}_{\delta}$ is an isomorphism, as is the inclusion $\mathcal{A} \rightarrow \mathcal{A}_{\delta}$. Now coherent modules on $\bb{U}$ are just finitely generated $\mathcal{A}$-modules, while coherent modules on any $\mathscr{M}(\mathcal{A}_{\delta})$ are finitely generated $\mathcal{A}_{\delta}$-modules, so the same.

We note in particular that $\_\__{\eta}$ induces an equivalence of ideal sheaves on $\bb{Y}$ and $\bb{Y}_{\eta}$, hence also of Zariski closed subspaces. 
\end{remark}

\begin{definition-proposition}\label{def:IrrComponents} (\textit{Irreducible Components}) Consider $\bb{Y}$ a special formal scheme over a trivially valued field $k$, and $\bb{Z} = V(\mathcal{I})$ a closed subscheme of $\bb{Y}$ defined by a radical ideal sheaf $\mathcal{I} \subset \mathscr{O}_{\bb{Y}}$. Then there exist finitely many $\mathcal{I}_1,...,\mathcal{I}_s \subset \mathscr{O}_{\bb{Y}}$ radical ideal sheaves such that $\mathcal{I}_i \not \subset \mathcal{I}_j$ for any $i \neq j$, and $\mathcal{I}_1 \cap ... \cap \mathcal{I}_s = \mathcal{I}$, and so that any radical ideal sheaf containing $\mathcal{I}$ also contains one of the $\mathcal{I}_i$. The closed formal subschemes $V(\mathcal{I}_i) \subset \bb{Z}$ are called the \textit{irreducible components} of $\bb{Z}$. They are naturally in bijection with the irreducible components of $\bb{Z}_{\eta}$.
\begin{proof}
    It will suffice to prove the claim for $\bb{Z} = \bb{Y}$, and $\mathcal{I} = \mathcal{N}$ the nilpotent sheaf.

    We first observe that the Zariski topology of $\bb{Y}_{\eta}$ is noetherian. Indeed, the Zariski topology of $\bb{U}_{\eta}$ is noetherian for each affine $\bb{U} \subset \bb{Y}$, and is finer than the topology on $\bb{U}_{\eta}$ induced by the restriction of the Zariski topology of $\bb{Y}_{\eta}$. Then the claim follows from quasi-compactness of $\bb{Y}$.

    Fix $Y_1,...,Y_s$ the irreducible components of $\bb{Y}_{\eta}$, and define G-locally on $\bb{Y}_{\eta}$ the ideal sheaf $I_i$ of functions vanishing on $Y_i$ (see \cite{Ducros_families}, 1.3.19). Then the $I_i$ are radical ideal sheaves on $\bb{Y}_{\eta}$ such that $I_1 \cap ... \cap I_s = \mathcal{N}_{\eta}$, which is the sheaf of nilpotent functions on $\bb{Y}_{\eta}$. Moreover, for any $i \neq j$, $I_i \not \subset I_j$, and every non-zero radical ideal on $\bb{Y}_{\eta}$ contains one of the $I_i$. By \cref{remark:Coh}, for each $i = 1,...,s$, there is a unique ideal sheaf $\mathcal{I}_i$ on $\bb{Y}$ with generic fiber $I_i$, and for these $\mathcal{I}_i$ the statement holds.    
\end{proof}
\end{definition-proposition}

\begin{remark}\label{remark:LocProp} (\textit{Local Properties}) For $k$ trivially or discretely valued, the rings defining special formal schemes on open affine subsets are known to be excellent (for trivially valued case, see \cite{ducros2009lesespacesberkovichsont}, for discretely valued, see \cite{Conrad1999}). Thus for a special formal scheme $\bb{Y}$, a local property $P$, such as reducedness, normality, regularity, holds for all coordinate rings of open affine subsets if and only if it holds for all completed local rings, in which case $\bb{Y}$ is said to have property $P$.

When $k$ is trivially valued, from \cref{remark:Coh} one immediately sees that a local property $P$ holds for $\bb{Y}$ if and only if the same property holds for $\bb{Y}_{\eta}$.

\end{remark}

\begin{definition-lemma}\label{def:An-Ord} Fix $k$ a nonarchimedean complete valued field, $Y$ a normal $k$-analytic space, $Z \subset Y$ an irreducible Weil divisor, and $D$ Cartier on $Y$. For any affinoid domain $V \subset Y$ with $V \cap Z \neq \emptyset$ and any irreducible component $Z'$ of $V \cap Z$, we can compute $\text{ord}_{Z'}(D|_V) \in \mathbb{Z}$. This number is independent of the choices of $V$ and $Z'$. We call it $\ord_{Z}(D)$.
\begin{proof}
    It is enough to prove for $Y = \mathscr{M}(A)$ an affinoid space, and $\mathscr{M}(B) = V \subset Y$ an affinoid domain, that $\text{ord}_{Z}(D) = \text{ord}_{Z'}(D|_V)$.

    Let $Y^{\al}, V^{\al}$ denote the spectra of global functions on $Y, V$ respectively. We have the canonical commuting square 
    \begin{center}
        \begin{tikzcd}
            V \arrow[r, "\pi_V"] \arrow[d, "\iota"] & V^{\al} \arrow[d, "\iota^{\al}"] \\ 
            Y\arrow[r, "\pi_Y"] & X^{\al}
        \end{tikzcd}
    \end{center}
    Fix also the irreducible $Z^{\al} \subset Y^{\al}$ and $(Z')^{\al} \subset V^{\al}$ whose preimages are $Z$ and $Z'$ respectively. Now for some $\mathcal{W} \subset Y^{\al}$ Zariski open neighborhood of $Z^{\al}$, the reduced structure of $Z^{\al} \cap \mathcal{W} \subset \mathcal{W}$ is cut out by a single function $t \in \mathscr{O}_{\mathcal{W}}$. Fix $W = \pi_X^{-1}(\mathcal{W})$. Then $t$ likewise cuts out the reduced structure of $Z \cap W \subset W$. Reducedness is a G-local property \cite{Ducros_families}, so the restriction of $t$ cuts out the reduced structure of $Z \cap W \cap V \subset W \cap V$. Finally, we see that $t$ is a uniformizer for $(Z')^{\al} \subset V^{\al}$, since it now cuts out the reduced structure of $(Z')^{\al}$ on the intersection with $(\iota^{\al})^{-1}(\mathcal{W})$ of the complement of all irreducible components of $(\iota^{\al})^{-1}(Z^{\al})$ not equal to $(Z')^{\al}$.

\end{proof}
\end{definition-lemma}

\begin{definition}\label{def:Ord} (\textit{Order of Vanishing}) Fix $\bb{Y}$ a normal special formal scheme over $k$ a trivially valued field, $\bb{Z} \subset \bb{Y}$ a prime Weil divisor, and $\bb{D} \subset \bb{Y}$ Cartier. Then $\bb{D}$ has a well-defined order of vanishing $\ord_{\mathbb{Z}}(\mathbb{D}) := \ord_{\mathbb{Z}_{\eta}}(\mathbb{D}_{\eta})$ along $\bb{Z}$. Hence, we can speak of the \textit{Weil divisor} associated to $\bb{D}$, 
$$\text{cyc}(\mathbb{D}) := \sum_{\bb{Z} \subset \bb{Y} \text{ irred. Weil div.}} \ord_{\bb{Z}} (\bb{D}) \cdot \bb{Z}.$$ 
\end{definition}

\begin{definition}\label{def:CartierFnct} (\textit{The Function $|\bb{F}| : \bb{Y}_{\eta} \rightarrow [0, \infty)$}) For $\bb{Y}$ a special formal scheme over triv. valued $k$, and $\bb{F}$ an effective Cartier divisor on $\bb{Y}$, there is associated to $\bb{F}$ a canonical continuous function $|\bb{F}| : \bb{Y}_{\eta} \rightarrow [0,\infty)$, where $|\bb{F}|(y)$ (we will usually instead write $|\bb{F}(y)|$) takes the value $|f(y)|$ for $f$ a local representative of $\bb{F}$ near $\text{red}_{\bb{Y}}(y)$.

Suppose now $\bb{F}$ is only assumed Cartier. Write $\bb{F} = \bb{Z} - \bb{P}$, here $\bb{Z}, \bb{P}$ effective divisors. In this case we can still define $|\bb{F}| : \bb{Y}_{\eta} \setminus (\bb{Z}_{\eta} \cap \bb{P}_{\eta}) \rightarrow [0, \infty]$ by $|\bb{F}| = |\bb{Z}|/|\bb{P}|$. The following properties are clear: 
\begin{enumerate}[label=(\alph*)]
    \item We have $\red_{\bb{Y}}(y) \in \mathbb{Z}_s$ (resp. $\mathbb{P}_s$) if and only if $|\bb{F}(y)| < 1$ (resp. $|\bb{F}(y)| > 1$).
    \item We have $|\bb{F}(y)| = 0$ (resp. $|\bb{F}(y)| = \infty$) if and only if $y \in \bb{Z}_{\eta}$ (resp. $y \in \bb{P}_{\eta}$). 
    \item For $y \in \bb{Y}_{\eta} \setminus (\bb{Z}_{\eta} \cup \bb{P}_{\eta}),$ we have $|\bb{F}(y)| \in (0, \infty)$.    
\end{enumerate}
We define as well $\bb{F}^{\trop} : \bb{Y}_{\eta} \setminus (\bb{Z}_{\eta} \cap \bb{P}_{\eta}) \rightarrow [-\infty, \infty]$ by $\bb{F}^{\trop} := -\log |\bb{F}|$. Note that $\bb{F} \mapsto \bb{F}^{\trop}$ defines an additive group homomorphism from Cartier divisors on $\bb{Y}$ to continuous functions $\bb{Y}_{\eta} \setminus (\bb{Z}_{\eta} \cup \bb{P}_{\eta}) \rightarrow \bb{R}$.   
    
\end{definition}

\begin{definition-proposition}\label{def:Trunc} (\textit{Truncation by an ideal sheaf}) Fix $K$ a complete valued field, and $\bb{Z}$ a special formal scheme over $K^{\circ}$. Set $Z = \bb{Z}_{\eta}$, and let $\mathcal{I}$ be an ideal sheaf on $\mathscr{O}_{\bb{Z}}$. Then there is a closed analytic domain $Z\{r^{-1}\mathcal{I}\} \subset Z$ defined by the Weierstrass condition $|f| \leq r, f \in \mathcal{I}$ on the class of analytic domains of $Z$ which arise as generic fibers of affine Zariski opens of $\bb{Z}$.

Let $r^n = |\pi|$ for some $\pi \in K^{\circ \circ}$. Fix $\hat{b} : \fBl_{(\pi) + \mathcal{I}^n}(\bb{X}) \rightarrow \bb{X}$ the formal blowup, and define $\bb{Z}\{r^{-1}\mathcal{I}\} \subset \fBl_{(\pi) + \mathcal{I}^n}(\bb{Z})$ the Zariski open subset on which $\pi$ generates the inverse image ideal sheaf $\hat{b}^{-1}\big( (\pi) + \mathcal{I}^n \big)$. Then $\bb{Z}\{r^{-1} \mathcal{I}\}_{\eta} \cong Z\{r^{-1}\mathcal{I}\}$. If $(\pi) + \mathcal{I}$ is a defining ideal sheaf of $\bb{Z}$, then $\bb{Z}\{r^{-1}\mathcal{I}\}$ is an admissible formal model of $Z\{r^{-1}\mathcal{I}\}$. 
\begin{proof}
    Take first $\bb{Z} = \Spf(\mathcal{A})$ an affine open subset, with $Z = \cup_{0 < \delta < 1} \mathscr{M}(A_{\delta})$ as in \cref{remark: GenFiber}. Fixing a set of generators $f_1,...,f_d$ for $\mathcal{I} \leq \mathcal{A}$, we have $\mathscr{M}(A_{\delta}\{r^{-1}f_1,...,r^{-1}f_d\}) \subset \mathscr{M}(A_{\delta})$ a Weierstrass domain of $\mathscr{M}(A_{\delta})$ \cite{Berkovich_Spectral_theory}, whose underlying topological space is the set of all points of $\mathscr{M}(A_{\delta})$ where $|f_i| \leq r$ for each $i$. For another choice of generators $g_1,...,g_{e}$ of $\mathcal{I}$ and any point $p \in Z$, we have $\max_i |f_i(p)| \leq r \iff \max_j |g_j(p)| \leq r$, so $Z\{r^{-1}I\}$ is well-defined. Now suppose $\bb{V} = \Spf(\mathcal{A}_{\bb{V}})$ affine open contained in $\bb{Z}$, and $V = \bb{V}_{\eta}$. The $f_i$ restrict to a set of generators of $\mathcal{I}|_{\bb{V}}$, which implies $V\{r^{-1}\mathcal{I}|_{\bb{V}}\} = Z\{r^{-1} \mathcal{I}\} \cap V$. The general construction now follows by gluing for a cover of $Z$ by analytic domains of the form $\Spf(\mathcal{A})_{\eta}$ for $\Spf(\mathcal{A}) \subset \bb{Z}$ Zariski open.

    The claims of the second paragraph can likewise be checked Zariski locally, so we suppose again $\bb{Z} = \Spf(\mathcal{A})$ is affine, say with defining ideal $\mathcal{J}$, and also $Z = \cup_{0 < \delta < 1} \mathscr{M}(A_{\delta})$. We have $b : \text{Bl}_{(\pi) + \mathcal{I}^n}(\Spec(\mathcal{A})) \rightarrow \Spec(\mathcal{A})$, and $\hat{b} : \fBl_{(\pi) + \mathcal{I}^n}(\bb{Z}) \rightarrow \bb{Z}$ the $b^{-1}\mathcal{J}$-adic completion. By definition $\bb{Z}\{r^{-1}\mathcal{I}\} \subset \fBl_{(\pi) + \mathcal{I}^n}(\bb{Z})$ is the $b^{-1}\mathcal{J}$-adic completion of the distinguished open subscheme $D_+(\pi) \subset \text{Bl}_{(\pi) + \mathcal{I}^n}(\Spec(\mathcal{A}))$. The ring of global sections of $\bb{Z}\{r^{-1}\mathcal{I}\}$ is then described as follows: Let $\frak{a} \subset \mathcal{A}\{u_{m_1...m_r} : m_1 + ... + m_r = n\}$ denote the ideal generated by the elements $\pi u_{m_1...m_r} - f_1^{m_1}...f_r^{m_r}$, and $\frak{a}^{\pi\text{-sat}}$ its $\pi$-saturation; then 
    $$ \Gamma(\bb{Z}\{r^{-1}\mathcal{I}\}, \mathscr{O}_{\bb{Z}\{r^{-1}\mathcal{I}\}}) \cong \mathcal{A}\{u_{m_1...m_r} : m_1 + ... + m_r = n\}/ \frak{a}^{\pi\text{-sat}}.$$
    Now by the definitions of $\pi$-saturation and of the generic fiber construction (wherein $\pi$ is inverted), the generic fiber of $\bb{Z}\{r^{-1}\mathcal{I}\}$ will agree with that of $\Spf(\mathcal{A}\{u_{m_1...m_r}\} / \frak{a})$, from which it follows that $\bb{Z}\{r^{-1}\mathcal{I}\}_{\eta} \cap \mathscr{M}(A_{\delta})$ is precisely $\mathscr{M}(A_{\delta}\{f_1^{m_1}...f_r^{m_r} / \pi :  m_1 + ... + m_r = n \})$, ie. the affinoid subdomain of $\mathscr{M}(A_{\delta})$ where $|f_1^{m_1}...f_r^{m_r}| \leq |\pi|$ for all partitions $m_1 + ... + m_r = n$. This is the same as requiring $|f_i^n| \leq |\pi|$ for each $i$, or $|f_i| \leq r$. We deduce $\mathbb{Z}\{r^{-1} \mathcal{I}\}_{\eta} = Z\{r^{-1}\mathcal{I}\}$, as desired.

    The final claim follows from Section II Proposition 1.1.4(2) of \cite{fujiwara2017foundationsrigidgeometryi}, since in this case the blowup along $(\pi) + \mathcal{I}$ is an admissible blowup in the sense of [loc. cit., II 1.1]. 
    \end{proof}
\end{definition-proposition}

\begin{proposition}\label{prop:TruncFunc} (\textit{Functoriality of Truncation}) Fix $L/K$ an extension of complete (trivially or discretely) valued fields. Fix also $\bb{X}, \bb{Y}$ special formal schemes over $L^{\circ}, K^{\circ}$ respectively, and suppose there is a map $\frak{f} : \bb{X} \rightarrow \bb{Y}$. Take $\mathcal{I}$ a sheaf of ideals on $\bb{Y}$. Then there is a natural identification $X\{r^{-1}(\frak{f}^{-1}\mathcal{I} \cdot \mathscr{O}_{\bb{X}})\} \overset{\cong}{\longrightarrow } X \times_Y Y\{r^{-1}\mathcal{I}\}$ of closed analytic domains of $X$.

Suppose now $K^{\circ}$ (hence also $L^{\circ}$) is discretely valued, and fix some $r^n \in \sqrt{|K^*|}$. Then there is a natural closed immersion $\bb{X}\{r^{-1}(\frak{f}^{-1} \mathcal{I} \cdot \mathscr{O}_{\bb{X}})\} \hookrightarrow \bb{X} \times_{\bb{Y}} \bb{Y}\{r^{-1} \mathcal{I}\}$, an isomorphism when $\frak{f}$ is flat. Its generic fiber is the isomorphism of the first paragraph. 
\begin{proof}

For the first paragraph, we simply observe that as topological subspaces of $X$, 
$$X \times_Y Y \{r^{-1} \mathcal{I}\} = f^{-1}(Y \{r^{-1}\mathcal{I}\}) = X\{r^{-1}(\frak{f}^{-1} \mathcal{I} \cdot \mathscr{O}_{\bb{X}})\}.$$
But these are both closed analytic domains in $X$, and so are determined up to isomorphism by their underlying subspaces.

For the second paragraph, we note there is a surjective map of graded $\mathscr{O}_{\bb{X}}$-algebras 
$$ \oplus_{j = 0}^{\infty} \frak{f}^* ((\pi) + \mathcal{I}^n)^j \rightarrow \oplus_{j = 0}^{\infty} ((\pi) + \frak{f}^{-1}\mathcal{I}^n \cdot \mathscr{O}_{\bb{X}})^j, $$
an isomorphism when $\frak{f}$ is flat. Taking Proj on both sides and completing the distinguished opens $D_+(\pi)$ with respect to a defining ideal sheaf for $\bb{X}$, we obtain precisely the closed immersion of the statement. 
\end{proof}
\end{proposition}

\noindent

\begin{corollary}\label{cor:PropSepPreserved} (\textit{Preservation of Proper / Separated}) Suppose $\frak{f}: \bb{X} \rightarrow \bb{Y}$ is a proper (resp. separated) map of special formal schemes over $K^{\circ}$. Then $\frak{f}_{\eta} : \bb{X}_{\eta} \rightarrow \bb{Y}_{\eta}$ is a proper (resp. separated) map of $K$-analytic spaces. 
\begin{proof}
    Let $X = \bb{X}_{\eta}, Y = \bb{Y}_{\eta}, f = \frak{f}_{\eta}$. Fix $\mathcal{I} \leq \mathscr{O}_{\bb{Y}}$ a defining ideal sheaf, in which case we can write $X = \cup_{0 < r < 1} X\{r^{-1} (f^{-1}\mathcal{I} \cdot \mathscr{O}_{\bb{X}})\}$, $Y = \cup_{0 < r < 1} Y\{r^{-1} \mathcal{I}\}$ as unions of compact analytic domains, with $f$ restricting for each $r$ to a map of the $r$-indexed domains by \cref{prop:TruncFunc}. Now fix some $0 < r < 1$, and $L/K$ a discrete valuation field with an element $\pi \in L^{\circ}$ such that $r^n = |\pi|$ for some $n$. Since $\frak{f}$ is proper (resp. separated), so is the base change $\bb{X}_{L^{\circ}} \times_{\bb{Y}_{L^{\circ}}} \bb{Y}_{L^{\circ}}\{r^{-1} \mathcal{I}\} \rightarrow \bb{Y}_{L^{\circ}}\{r^{-1} \mathcal{I}\}$, and then, since $\bb{X}_{L^{\circ}}\{r^{-1} (f^{-1}\mathcal{I} \cdot \mathscr{O}_{\bb{X}_{L^{\circ}}})\} \rightarrow X_{L^{\circ}} \times_{\bb{Y}_{L^{\circ}}} \bb{Y}_{L^{\circ}}\{r^{-1} \mathcal{I}\}$ is a closed immersion, we have
    $$ \bb{X}_{L^{\circ}}\{r^{-1} (f^{-1}\mathcal{I} \cdot \mathscr{O}_{\bb{X}_{L^{\circ}}})\} \rightarrow \bb{Y}_{L^{\circ}}\{r^{-1} \mathcal{I}\} $$
    proper (resp. separated). By the final statements of \cref{def:Trunc} and \cref{prop:TruncFunc}, this is a map of admissible formal schemes over $L^{\circ}$ whose generic fiber is the base extension 
    $$ X\{r^{-1}(f^{-1} \mathcal{I} \cdot \mathscr{O}_{\bb{X}})\}_L \rightarrow Y\{r^{-1}\mathcal{I}\}_L, $$
    which is therefore proper (resp. separated). The result now follows from descent for valuation extensions (Theorem 9.2 of \cite{Conrad_Descent}). 
\end{proof}
\end{corollary}

\begin{definition-corollary}\label{cor:FormalModels} (\textit{Generic Fiber of a Map to $\bb{Y}$}) Fix $K$ a complete, discrete valuation field, and $L / K$ an arbitrary valuation extension. Suppose $\bb{Y}$ is a special formal scheme over $K^{\circ}$ with $\mathcal{I} \leq \mathscr{O}_{\bb{Y}}$ a defining ideal sheaf, and $\bb{X}$ a finite-type formal scheme over $L^{\circ}$ together with a $K^{\circ}$-map $\frak{f} : \bb{X} \rightarrow \bb{Y}$ of adic formal schemes. Then there exists some $r \in \sqrt{|K^*|}$ such that $\frak{f}$ factors uniquely as the composition of a map $\bb{X} \rightarrow \bb{Y}\{r^{-1} \mathcal{I}\}$ of admissible formal schemes with $\bb{Y}\{r^{-1} \mathcal{I}\} \rightarrow \bb{Y}$. We call the composition $\frak{f}_{\eta} : \bb{X}_{\eta} \rightarrow \bb{Y}_{\eta}$ of generic fibers the \textit{generic fiber of $\frak{f}$}. The generic fiber is compatible with the reduction maps, ie. $\red_{\bb{Y}} \circ \frak{f}_{\eta} = \frak{f}_s \circ \red_{\bb{X}}$.    
\begin{proof}
    We have $\frak{f}^{-1}\mathcal{I} \cdot \mathscr{O}_{\bb{X}} \subset (\alpha) \cdot \mathscr{O}_{\bb{X}}$ for some $\alpha \in L^{\circ \circ}$, since $\mathcal{I}$ is generated by finitely many topologically nilpotent elements (note here we are using quasi-compactness of $\bb{X}$ to pick an $\alpha$ that works globally). Fix $\pi \in K^{\circ \circ}$, and $n$ so that $|\alpha|^n < |\pi|$. Then $f^{-1} \mathcal{I}^n \cdot \mathscr{O}_{\bb{X}} \subset (\pi) \cdot \mathscr{O}_{\bb{X}}$. In particular $f^{-1}((\pi) + \mathcal{I}^n) \cdot \mathscr{O}_{\bb{X}}$ is principal, generated by $\pi$. Now by the universal property of the admissible blowing up (see Prop. 1.1.4(3) of \cite{fujiwara2017foundationsrigidgeometryi}), $\frak{f}$ factors as the composition of a uniquely determined map $\tilde{\frak{f}} : \bb{X} \rightarrow \Bl_{(\pi) + \mathcal{I}^n}(\bb{Y})$ with the projection to $\bb{Y}$. Our observations also show the image of $\tilde{\frak{f}}$ is contained in the Zariski-open of $\Bl_{(\pi) + \mathcal{I}^n}(\bb{Y})$ where $\pi$ generates the inverse image sheaf of $(\pi) + \mathcal{I}^n$, so we obtain the map $\bb{X} \rightarrow \bb{Y}\{r^{-1}\mathcal{I}\}$ of the claim. For any two choices of $r$, say $r_1 < r_2$, the maps factor as $\bb{X} \rightarrow \bb{Y}\{r_1^{-1} \mathcal{I}\} \rightarrow \bb{Y}\{r_2^{-1} \mathcal{I}\}$, giving independence of $\frak{f}_{\eta}$ from the particular choice of $r$.

    For the compatibility of $\frak{f}_{\eta}$ with the reduction maps of $\bb{X}$ and $\bb{Y}$, we apply successively the analogous statements for the two maps $\bb{X} \rightarrow \bb{Y}\{r^{-1} \mathcal{I}\}$ and $\bb{Y}\{r^{-1} \mathcal{I}\} \rightarrow \bb{Y}$. 
\end{proof}
\end{definition-corollary}

\subsection{Polyhedral Complexes and Subdivisions}
In this brief section, we prove an elementary result, \cref{prop:CofinalStarSubdiv}, concerning $*$-subdivisions of geometric complexes of rational, strictly convex polyhedra. Specifically, given such a complex $\mathscr{P}$, and another $\mathscr{G}$, together with a saturated embedding $\mathscr{G} \rightarrow \mathscr{P}$, we establish the existence of a composition of $*$-subdivisions $\mathscr{P}' \rightarrow \mathscr{P}$ such that the image of each element of $\mathscr{G}$ in $\mathscr{P}$ is a union of elements of $\mathscr{P}'$, and so that any $*$-subdivision featured in the composition has center a point (or ray) contained in an element of $\mathscr{G}$.

A secondary goal of the section is to establish notation concerning various constructions arising from toric geometry of geometric objects arising from polyhedral complexes.

\begin{definition}\label{def: *-subdivision} (\textit{$*$-subdivisions}) Fix $P \subset M_{\bb{R}}$ a rational, strictly convex polyhedron, and $p \in P$ (resp. $r \in \rec(P)$) a rational point (resp. a direction in the recession cone). We let $*(P, p)$ denote the geometric complex obtained by performing the $*$-subdivision of $P$ with center $p$ (resp. $r$). Fix $\mathscr{P}$ a geometric complex of rational, strictly convex polyhedra, $p \in |\mathscr{P}|$ (resp. $r \in \rec(\mathscr{P})$) a rational point (resp. a direction in the recession fan). We let $*(\mathscr{P}, p)$ (resp. $*(\mathscr{P}, r)$) denote the corresponding $*$-subdivision.

\end{definition}

\begin{lemma}\label{lemma:StarSubdivLemma} \begin{enumerate}[label=(\alph*)]
    \item Take $\mathscr{P}$ a geometric complex of rational strictly convex polyhedra. Then there is a sequence of star subdivisions resulting in a simplicial complex $\mathscr{P}_s$. 
    \item Suppose $P$ is a polyhedron in $M_{\mathbb{R}}$. Then for any $k$-dimensional rational cross-section $Q \subset P$, $k \leq \dim P$, and point $p \in Q$, there is a natural surjective morphism $*(P,p) \rightarrow *(Q,p)$ of polyhedral complexes. 
\end{enumerate} 
\begin{proof}
    (a) Suppose inductively that all polyhedra of $\mathscr{P}$ of dimension at most $k$ are simplices. That is to say that every polyhedron of dimension $m \leq  k$ is of the form 
    $$ \text{conv}(v_0,...,v_{\ell}) + \text{cone}(r_{\ell + 1},...,r_m),$$
    for some choice of vertices $v_0,...,v_{\ell}$ and rays $r_{\ell + 1},...,r_m$.
    Take $P$ a dimension $k + 1$ polyhedron of $\mathscr{P}$. Write $P = \Delta + C$. If the extremal rays of $C$ are not linearly independent, choose a ray $r$ in the interior of $C$ and subdivide to $\mathscr{P}' := *(\mathscr{P}, r)$. Otherwise, choose a point $p_0$ in the interior of $P$ and subdivide to $\mathscr{P}' := *(\mathscr{P},p_0)$. In either case, we have a polyhedral complex $\mathscr{P}'$ whose dimension $\leq k$ polyhedra remain simplicial and for whom the number of non-simplicial dimension $k + 1$ polyhedra has decreased by at least 1. 
    \\

    \noindent 
    (b) It obviously suffices to prove the claim for $k = \dim P - 1$, since for any $k$ we can take a sequence of hyperplanes cutting out $Q$ and apply the result successively to each further hyperplane intersection.

    Let $\mathscr{P}$ be the complex associated to $P$ and $\mathscr{Q}$ the complex associated to $Q$. Take a defining set of half-spaces $\{(H_i,n_i)\}_{i \in I}$ for $P$. Since $\dim Q = \dim P - 1$, there is a hyperplane $H$ such that $Q = P \cap H$. Let $n \in M$ be a normal vector to $H$. Then $\{(H, n), (H, -n)\} \cup \{(H_i, n_i)\}_{i \in I}$ is a defining set of half-spaces for $Q$. For any subset $S \subset I$, $Q \cap \big( \bigcap_{i \in S} H_i\big)$ is a face of $Q$, and every face of $Q$ can be written as such, so there is a natural surjective map $\mathscr{P} \rightarrow \mathscr{Q}$ given by intersection with $Q$.

    For the claim, we note that for any $G \in \mathscr{P}$ proper, $F := Q \cap G$ is a proper polyhedron in $\mathscr{Q}$ and conv$(G, p) \cap Q = \text{conv}(F,p)$. Thus we obtain the map $*(P,p) \rightarrow *(Q,p)$, which is surjective, since every proper $F \in \mathscr{Q}$ is of the form $Q \cap G$ for some $G \in \mathscr{P}$ proper. 
\end{proof}
\end{lemma}

\begin{proposition}\label{prop:CofinalStarSubdiv} Let $\mathscr{G} \rightarrow \mathscr{P}$ a saturated embedding of geometric complexes of rational, strictly convex polyhedra. Then there is a polyhedral complex $\mathscr{P}'$ refining $\mathscr{P} \cup \mathscr{G}$, obtained from $\mathscr{P}$ by a sequence of star subdivisions. If, moreover, we assume $\mathscr{P}$ simplicial, then we only need to use star subdivisions centered at points (or rays) contained in $|\mathscr{G}| \subset |\mathscr{P}|$. 
\begin{proof}
    First perform a sequence of star subdivisions of $\mathscr{P}$ to obtain a simplicial complex containing all vertices and extremal rays of polyhedra in $\mathscr{G}$. Then each polyhedron of $\mathscr{G}$ is a union of cross-sections of polyhedra of $\mathscr{P}$. By \cref{lemma:StarSubdivLemma}, we can perform a further sequence of star subdivisions to get a simplicial complex $\mathscr{P}_s$ whose polyhedra intersect $|\mathscr{G}|$ simplicially. Let $\mathscr{G}_s$ be the induced simplicial refinement of $\mathscr{G}$. Whichever vertices of $\mathscr{G}_s$ (resp. extremal rays $r$) are not already vertices (resp. extremal rays) of $\mathscr{P}_s$, we obtain $\mathscr{P}'$ by performing any sequence of star subdivisions of $\mathscr{P}_s$ centered at these points (resp. extremal rays). Indeed, since $\mathscr{G}_s$ is simplicial, these subdivisions of $\mathscr{P}_s$ do not induce any further subdivision of $\mathscr{G}_s$. Hence each simplex $S$ of $\mathscr{G}_s$ is the intersection of a simplex $P$ of $\mathscr{P}'$ with $|\mathscr{G}|$, and the vertices and extremal rays of $S$ are all vertices and extremal rays of $P$ so that $S \subset P$ a face. 
\end{proof}
\end{proposition}

For the remainder of this section, we fix $K$ a complete, discretely valued field. In our applications this will often just be $k[[\pi]]$ for some $k$ trivially valued. Set $\bb{T}_M := \Spec(K^{\circ}[N])$. 
 
\begin{remark}\label{remark:ToricNotation} \textit{(From Polyhedra to Analytic Spaces)} We set notation for several constructions of analytic spaces associated to polyhedral complexes in $M_{\bb{R}}$. Fix $\mathscr{P}$, then, a complex of rational, strictly-convex polyhedra in $M_{\bb{R}}$. Recall from \cite{Kempf_Toroidal_embeddings} the $\bb{T}_M$-toric $K^{\circ}$-scheme associated to $\mathscr{P}$, which we will call $\bb{T}_M[\mathscr{P}^{\vee}]$. This is a degeneration over $K^{\circ}$ whose generic fiber is the $T_M$-toric variety $\TV(\rec(\mathscr{P}))$ over $K$, and whose special fiber is a reducible union of toric varieties over $\tilde{K}$, with combinatorics determined by the finite part of $\mathscr{P}$. We note that $\bb{T}_M[\mathscr{P}^{\vee}]$ is affine if and only if $\mathscr{P}$ is the complex associated to a single polyhedron in $M_{\bb{R}}$, and $\bb{T}_M[0^{\vee}] = \bb{T}_M[N] = \bb{T}_M$. The action of $T_M$ on the generic fiber extends to a $K^{\circ}$-scheme action of $\bb{T}_M$ on $K^{\circ}[\mathscr{P}^{\vee}]$.

We define also $\bb{T}_M\{\mathscr{P}^{\vee}\}$ the $K^{\circ \circ}$-adic completion of $\bb{T}_M[\mathscr{P}^{\vee}]$, and $T_M\{\mathscr{P}^{\vee}\}$ its $K$-analytic generic fiber. Then $T_M\{\mathscr{P}^{\vee}\} \subset \TV(\rec(\mathscr{P}))^{\an}$ is the analytic domain with underlying subspace $r^{-1}(\overline{|\mathscr{P}|}),$ where here $r : \TV(\rec(\mathscr{P}))^{\an} \rightarrow \overline{M}_{\bb{R}}$ the canonical retraction. We note $\bb{T}_M\{0^{\vee}\} = \bb{T}_M\{N\}$ is a group in the category of admissible formal $K^{\circ}$-schemes, which acts on each $\bb{T}_M\{\mathscr{P}^{\vee}\}$, and whose generic fiber, $T_M\{N\}$, is the maximal compact subgroup of $T_M^{\an}$, which in turn acts on each $T_M\{\mathscr{P}\}$ via the generic fiber of the $\bb{T}_M\{N\}$-action on $\bb{T}_M\{\mathscr{P}^{\vee}\}$.

\end{remark}

\begin{remark}\label{rem:ProjSubdiv} (\textit{Projective  Subdivisions}) Fix a geometric complex $\mathscr{P} = \{P_{\lambda}\}_{\lambda \in \Lambda}$ of rational, strictly convex polyhedra and a piecewise integral-affine function $\varphi : |\mathscr{P}| \rightarrow \mathbb{R}$ that is convex on restriction to any polyhedron of $\mathscr{P}$. Let $\mathscr{P}_{\varphi} \rightarrow \mathscr{P}$ denote the refinement whose maximal polyhedra are the domains of affineness of $\varphi$. We say $\mathscr{P}_{\varphi}$ is a \textit{projective subdivision} of $\mathscr{P}$.

There are three properties of projective subdivisions of which we make special note. First, the composition of projective subdivisions is projective; that is, if $\mathscr{P}' \rightarrow \mathscr{P}$ and $\mathscr{P} \rightarrow \mathscr{P}''$ are projective subdivisions, then so is $\mathscr{P}' \rightarrow \mathscr{P}''$ \cite{Kempf_Toroidal_embeddings}. Second, if $\mathscr{P}' \rightarrow \mathscr{P} \rightarrow \mathscr{P}''$ is a projective subdivision, then so is $\mathscr{P}' \rightarrow \mathscr{P}$. Third, $*$-subdivisions are projective, and by \cref{prop:CofinalStarSubdiv} arbitrary finite compositions of star subdivisions are cofinal among all subdivisions of $\mathscr{P}$, so in particular we find that projective subdivisions are cofinal.

Finally we recall the significance of projective subdivisions in toric and toroidal geometries. For $\mathscr{P}$ a polyhedral complex in $M_{\bb{R}}$ and $\mathscr{P}_{\varphi} \rightarrow \mathscr{P}$ the projective subdivision associated to a piecewise integral-affine, convex on each cone $\varphi : |\mathscr{P}| \rightarrow \mathbb{R}$, there is a fractional ideal sheaf $\mathscr{I}_{\varphi}$ on $\bb{T}_M[\mathscr{P}^{\vee}]$ whose normalized blowing up is the canonical map $\bb{T}_M[\mathscr{P}_{\varphi}^{\vee}] \rightarrow \bb{T}_M[\mathscr{P}^{\vee}]$ [loc. cit.]. Similarly, if $\Sigma$ is the fan of a toroidal variety $Y$ over a field $k$, and $\varphi : |\Sigma| \rightarrow \mathbb{R}$ piecewise integral-linear, piecewise convex, then there exists a fractional ideal sheaf $\mathscr{I}_{\varphi}$ on $Y$ such that the normalized blowup of $Y$ along $\mathscr{I}_{\varphi}$ is the unique toroidal modification corresponding to the subdivision $\Sigma_{\varphi} \rightarrow \Sigma$ [loc. cit.].
\end{remark}

\begin{proposition}\label{prop:SYZAuts} Fix $K$ a discretely valued field and $T_M$ the $K$-analytification of an algebraic torus over $K$. Take $\mathcal{G} \subset M_{\bb{R}}$ a rational polytope. An SYZ-preserving automorphism (ie. an automorphism preserving the Berkovich retraction to $\mathcal{G}$) of $T_M\{\mathcal{G}^{\vee}\}$ is equivalent to an automorphism of $\bb{T}_M\{\mathcal{G}^{\vee}\}$ induced by a map $\bb{T}_M\{\mathcal{G}^{\vee}\} \rightarrow \bb{T}_M\{N\}$. 
\begin{proof}
    Fix a basis of characters $x_1,...,x_n$ for $T_M$. Then an automorphism $\phi$ of $T_M\{\mathcal{G}^{\vee}\}$ such that $r_M(\phi(p)) = r_M(p)$ for all $p \in T_M\{\mathcal{G}^{\vee}\}$ satisfies $\phi^* x_i = u_i x_i$, with $|u_i| = 1$ on $T_M\{\mathcal{G}^{\vee}\}$. Now fix $A_{\mathcal{G}}$ the affinoid algebra defining $T_M\{\mathcal{G}^{\vee}\}$, so that $\bb{T}_M\{\mathcal{G}^{\vee}\} = \Spf(A_{\mathcal{G}}^{\circ})$ is the canonical $K^{\circ}$-model of $T_M\{\mathcal{G}^{\vee}\}$. We may therefore define $u : \bb{T}_M\{\mathcal{G}^{\vee}\} \rightarrow \bb{T}_M\{N\}$ by $u^* x_i = u_i$, and then $\phi$ arises as the generic fiber of the automorphism of $\bb{T}_M\{\mathcal{G}^{\vee}\}$ obtained by composing $u$ with multiplication $\bb{T}_M\{N\} \times \bb{T}_M\{\mathcal{G}^{\vee}\} \rightarrow \bb{T}_M\{\mathcal{G}^{\vee}\}$. Conversely any morphism such as $u$ determines an SYZ-automorphism of $T_M\{\mathcal{G}^{\vee}\}$. 
\end{proof}
\end{proposition}

\subsection{SYZ Fibration} \label{sec:lSYZ}

Throughout this section, $k$ is a perfect field equipped with its trivial valuation. Our aim is to generalize results of \cite{Kempf_Toroidal_embeddings} and \cite{Thuillier_Geometrie_toroidale} from the algebraic to the formal case. We recall that in \cite{Thuillier_Geometrie_toroidale} Thuillier shows that, for $(\bb{Y}, \bb{D})$ a pair of varieties over $k$ with toroidal singularities, the skeleton of the pair is canonically embedded in $\bb{Y}_{\eta} \setminus \bb{D}_{\eta}$, and there is a retraction mapping, in fact a deformation retract, of $\bb{Y}_{\eta} \setminus \bb{D}_{\eta}$ onto the image of the embedding.

From \cref{def:SNCPairs}-\cref{prop:SNCCharts}, we define and study the snc condition for pairs of special formal schemes over $k$, giving in particular several equivalent characterizations of snc pairs, e.g. \cref{prop:SNCCharts} shows a pair $(\bb{Y}, \bb{D})$ is snc if and only if Zariski-locally at any point $\bb{Y}$ admits an \'etale chart to the formal completion of an algebraic, toric snc pair along a subvariety such that $\bb{D}$ is the preimage of the toric boundary.

In \cref{def:DualFan}, we define the dual fan $\Sigma_{(\bb{Y}, \bb{D})}$ of an snc pair $(\bb{Y}, \bb{D})$ of special formal schemes over $k$ and the SYZ fibration $U := \bb{Y}_{\eta} \setminus \bb{D}_{\eta} \rightarrow |\Sigma_{(\bb{Y}, \bb{D})}|$. We also define the $(\bb{Y}, \bb{D})$-skeleton, $\Sk(\bb{Y}, \bb{D})$, the union of relative interiors of cones of $\Sigma_{(\bb{Y}, \bb{D})}$ corresponding to algebraic strata of $\bb{D}$ and, when $\bb{Y}_s$ is a union of strata of $\bb{D}$, the image of the SYZ fibration. In \cref{prop:FunctTrop}, we demonstrate tropicalizations of maps of snc pairs and their compatibility with the generic fiber functor.

In \cref{def:BasicCharts}-\cref{prop:BlSubdiv}, the combinatorial theory of toric blowups first introduced in \cite{Kempf_Toroidal_embeddings} is generalized to snc pairs of special formal schemes, and in \cref{theorem:SYZMain} we utilize the resulting theory of snc blowups to give a birational interpretation of $\Sigma_{(\bb{Y}, \bb{D})}$, along with a canonical embedding $\Sk(\bb{Y}, \bb{D}) \hookrightarrow U$. There is a proper retraction onto $\Sk(\bb{Y}, \bb{D})$ of the open analytic domain of $U$ consisting of points with center in the union of algebraic strata of $\bb{D}$, generalizing Thuillier's retraction to the formal case.

 In \cref{def:IntLinStructure}-\cref{prop:SkelOfForm}, we study log Calabi-Yau type conditions on $(\bb{Y}, \bb{D})$. In \cref{def:IntLinStructure}, we discuss when $\Sk(\bb{Y}, \bb{D})$ is locally canonically an integral-linear manifold in codimension 1 and describe the smooth structure explicitly. Then in \cref{prop:GluingEnds} we prove an important technical result which gives good formal models for the gluing of toric ends of [Moduli Spaces]. In \cref{prop:SkelOfForm}, the skeleton in $U$ of a $(\bb{Y}, \bb{D})$-log form $\omega$ is observed to be the union of cones of $\Sk(\bb{Y}, \bb{D})$ corresponding to the $\omega$-essential algebraic strata of $\bb{D}$.

We conclude the section by examining the interaction of the SYZ fibration with base extension. In \cref{def:SYZ}, when $k$ is algebraically closed we observe a canonical SYZ fibration on any scalar extension of $U$, compatible with the SYZ fibration of $U$. In \cref{prop:trop}, for a complete valuation extension $K/k$, we examine maps of strictly semistable formal models over $K^{\circ}$ (in the sense of [GRW]) into $\bb{Y}$, proving in particular an analogous result to \cref{prop:FunctTrop} in this setting. In \cref{prop:balancing} we show that tropical balancing holds for tropicalizations of strictly semistable curves in $\bb{Y}$.

\begin{definition}\label{def:SNCPairs} (\textit{SNC Pairs}) Fix a pair $(\bb{Y}, \bb{D})$ of special formal schemes over $k$, with $\bb{D} \subset \bb{Y}$ a reduced divisor having irreducible components $\bb{D}_1,...,\bb{D}_N$, and $\bb{Y}$ irreducible of dimension $n$. The pair is called \textit{simple normal crossing}, or \textit{snc}, if for every $p \in \bb{Y}_s$, the analytic local ring $\hat{\mathscr{O}}_{\bb{Y}, p}$ is isomorphic to a formal power series ring $k(p)[[y_1,...,y_r]]$, where $r = n - \text{tr}_k(k(p))$, and each $\bb{D}_i$ containing $p$ is locally cut out by exactly one of the $y_j$.

We observe immediately the definition implies $\mathscr{O}_{\bb{Y}, p}$ is regular, since $\whscrO_{\bb{Y}, p}$ is. Thus we can fix $y_i \in \mathscr{O}_{\bb{Y}, p}$ vanishing along $\bb{D}_i$, whenever $p \in \bb{D}_i$. Now $\mathscr{O}_{\bb{Y}, p} \rightarrow \whscrO_{\bb{Y}, p}$ is faithfully flat, and the $y_i$ form part of a regular system of parameters in $\whscrO_{\bb{Y}, p}$, by the snc condition, so they can likewise be completed to a regular system of parameters for $\mathscr{O}_{\bb{Y}, p}$. Therefore $(\bb{Y}, \bb{D})$ is snc if and only if each local ring $\mathscr{O}_{\bb{Y}, p}$ is regular, with a regular system of parameters $y_1,...,y_r$ such that each $\bb{D}_i$ containing $p$ is locally cut out by exactly one of the $y_i$. 

\end{definition}

\begin{proposition}\label{prop:SNCclosed} It suffices to check the snc condition at closed points.
\begin{proof}
    For some $p \in \bb{Y}_s$, fix $q$ a specialization of $p$, and $\bb{U} \subset \bb{Y}$ an affine neighborhood of $q$. Then set $\bb{U}^{\al}$ the affine algebraic spectrum of $\Gamma(\bb{U}, \mathscr{O}_{\bb{U}}).$ All closed points of $\bb{U}^{\al}$ belong to the closed subscheme $\bb{U}_s \subset \bb{U}^{\al}$, implying regularity of $\bb{U}^{\al}$ by \cref{remark:LocProp}. Since $\whscrO_{\bb{U}^{\al}, q} = \whscrO_{\bb{U}, q}$, by an identical argument as in the second paragraph of \cref{def:SNCPairs}, we may fix a regular system of parameters $y_1,...,y_r \in \mathscr{O}_{\bb{U}^{\al},q}$ so that each $\bb{D}_i$ containing $q$ is cut out in the local ring by a single $y_i$. WLOG suppose for some $r' < r$ that $\bb{D}_1,...,\bb{D}_{r'}$ are the irreducible components containing $p$, and $y_i$ cuts out $\bb{D}_i$ at $q$ for each $i \leq r'$. Then $y_1,...,y_{r'}$ form part of a regular system of parameters at $p$, and still $y_i = 0$ defines $\bb{D}_i$ there. We conclude the snc condition holds there. 
\end{proof}
\end{proposition}

\begin{lemma}\label{lemma:VanishDiff} Suppose $\psi : \bb{Y} \rightarrow \bb{X}$ a map of special formal schemes over $k$. Fix $y \in \bb{Y}_s$ and $x = \psi(y)$. Then there is a natural isomorphism $\widehat{\Omega}^1_{\bb{Y}/\bb{X}, y} \overset{\cong}{\longrightarrow} \widehat{\Omega}^1_{\whscrO_{\bb{Y}, y} / \whscrO_{\bb{X}, x}}$ of completed $\whscrO_{\bb{Y},y}$-modules. 
\begin{proof}
    The existence of a continuous homomorphism $\widehat{\Omega}^1_{\bb{Y}/\bb{X}, y} \rightarrow \widehat{\Omega}^1_{\whscrO_{\bb{Y}, y} / \whscrO_{\bb{X}, x}}$ is clear. To obtain the inverse, we give a canonical $\frak{m}_y$-adically continuous $\whscrO_{\bb{X}, x}$-derivation $d : \whscrO_{\bb{Y}, y} \rightarrow \widehat{\Omega}^1_{\bb{Y}/\bb{X}, y}$.

    A general element of $\widehat{\mathscr{O}}_{\bb{Y}, y}$ takes the form $f = \sum_{j = 0}^{\infty} f_j,$ where $f_j \in \frak{m}_y^j \leq \mathscr{O}_{\bb{Y}, y}$. Say $f_j$ is defined on affine open $\bb{U}_j \subset \bb{Y}$, with descending inclusions $\bb{U}_0 \supset \bb{U}_1 \supset ...$\textbf{.} Fixing $d_j : \mathscr{O}_{\bb{Y}, y}(\bb{U}_j) \rightarrow \Omega^1_{\bb{Y} / \bb{X}}(\bb{U}_j)$ the universal derivation for each $j$, we observe that $d_j f_j \in \frak{m}_y^{j - 1} \Omega^1_{\bb{Y} / \bb{X}}(\bb{U}_j)$, so the germ of $d_j f_j$ lies in $\frak{m}_y^{j - 1} \Omega^1_{\bb{Y}/\bb{X},y}$. Therefore we define $df = \sum_{j = 0}^{\infty} d_j f_j$ as an element of $\widehat{\Omega}^1_{\bb{Y}/\bb{X}, y}$.

    Suppose that $f = \sum_{j = 0}^{\infty} g_j$, for $g_j \in \frak{m}_y^j$ possibly distinct from $f_j$. We can insist, by shrinking all our opens $\bb{U}_j$ if necessary, that both $f_j, g_j$ are defined on $\bb{U}_j$. For each $\lambda \in \bb{N}$, there exists $j_{\lambda}$ so that $\sum_{j = 0}^{j_{\lambda}} f_j - g_j \in \frak{m}^{\lambda}_y $. Therefore 
    $$ \sum_{j = 0}^{j_{\lambda}} d_j f_j = d_{j_{\lambda}} \big(\sum_{j = 0}^{j_{\lambda}} f_j \big) \in \frak{m}_y^{\lambda - 1} \Omega_{\bb{Y}/\bb{X}}^1(\bb{U}_{j_{\lambda}}), $$
    it follows that $\sum_{j = 0}^{\infty} d_j (f_j - g_j) = 0$, so $df$ is well-defined.

    Suppose $f$ lay in the image of $\whscrO_{\bb{X}, x} \rightarrow \whscrO_{\bb{Y}, y}$. Then we can write $f = \sum_{j = 0}^{\infty} f_j$ with $f_j$ in the image of $\frak{m}_x^j \rightarrow \frak{m}_y^j$. Then $d_j f_j = 0$, so $df = 0$. For $f,g \in \whscrO_{\bb{Y}, y}$ general (with $\frak{m}_y$-convergent series representations $f = \sum_j f_j, g = \sum_j g_j)$, we have 
    $$fg = \sum_{j = 0}^{\infty} \big( \sum_{i = 0}^{j} f_j g_{j - i} \big)$$ 
    $$ \Rightarrow d(fg) = \sum_{j = 0}^{\infty} \big( \sum_{i = 0}^j d_j (f_i g_{j - i})\big) $$
    $$= \sum_{j = 0}^{\infty} \big(\sum_{i = 0}^j f_i d_j g_{j - i}+ g_{j - i} d_j f_i) = f dg + g df,$$
    where for the last equality, we have used that the $d_j$ are compatible on restriction, so taking germs at $p$ gives $d_j f_i = d_i f_i$ and $d_j g_{j - i} = d_{j - i} g_{j - i}$. We conclude $d$ is a $\whscrO_{\bb{X}, x}$-derivation.

\end{proof}
\end{lemma}

\begin{definition}\label{def:Stratum} Fix a pair $(\bb{Y}, \bb{D})$ of special formal schemes over $k$, with $\bb{D} \subset \bb{Y}$ a Weil divisor having irreducible components $\bb{D}_1,...,\bb{D}_N$. A \textit{closed stratum} of $\bb{D}$ is either an irreducible component of an intersection of finitely many $\bb{D}_i$ or else $\bb{Y}$ itself. A closed stratum $\bb{S}$ of $\bb{D}$ is \textit{algebraic} if $\bb{S} \subset \bb{Y}_s$. 
\end{definition}

\begin{proposition}\label{prop:SNCCharts} Fix a pair $(\bb{Y}, \bb{D})$ as in \cref{def:Stratum}. Then $(\bb{Y}, \bb{D})$ is an snc pair if and only if, for every $p \in \bb{Y}_s$, there exists a Zariski open formal subscheme $\bb{U} \subset \bb{Y}$ with $p \in \bb{U}_s$, a lattice $L$ of rank $\dim(\bb{U})$, a regular toric monoid $\tau \subset L_{\bb{R}}$, and an \'etale map $\psi : \bb{U} \rightarrow \whTV(\tau)$, the completion taken with respect to an ideal of $k[\tau^{\vee}_{\bb{Z}}]$, such that each irreducible component of $\bb{D}$ is the preimage of a single irreducible component of the toric boundary.

There exists moreover a \textit{minimal choice} of $L, \tau$ such that, for any other choice of $L', \tau'$ and \'etale $\psi' : \bb{U}' \rightarrow \whTV(\tau' \subset L')$ with $p \in \bb{U}'_s$, there is a natural lattice isomorphism $L \rightarrow L'$ taking $\tau$ onto a face of $\tau'$, so by pulling back the open immersion $\whTV(\tau \subset L) \hookrightarrow \whTV(\tau' \subset L')$ there is a Zariski open $\bb{U}'' \subset \bb{U}'$, and the restriction of $\psi'$ to $\bb{U}''$ is an \'etale map $\psi'' : \bb{U}'' \rightarrow \whTV(\tau \subset L)$. The minimal choice is determined by the minimal stratum $\bb{S}$ of $\bb{D}$ containing $p$. In particular there is an isomorphism $M_{\bb{S}} \oplus \bb{Z}^{\oplus \codim_{\bb{U}}(\bb{S})} \rightarrow L$ identifying $\sigma_{\bb{S}} \oplus \{0\}$ with $\tau$.

For $L, \tau$ the minimal choice, we call an \'etale map $\psi : \bb{U} \rightarrow \whTV(\tau \subset L)$ a \textit{minimal chart} for $\bb{S}$. The formal completion of the target along the union of algebraic strata of the toric boundary likewise depends only on $\bb{S}$. 
\begin{proof}
    For the reverse direction, we use \cref{prop:SNCclosed} to reduce to the case $p$ closed. Then, make a base field extension $k' / k$ so that $p$ splits into finitely many $k'$-points. Fix $p'$ one such point. \'Etale is preserved by base extension, and \'etale maps induce isomorphisms of completed local rings at rational points, so $(\bb{Y}_{k'}, \bb{D}_{k'})$ is snc at $p'$. By Serre's criterion, $\mathscr{O}_{\bb{Y}, p}$ is regular. Furthermore, letting $y_i \in \mathscr{O}_{\bb{Y}, p}$ cut out $\bb{D}_i$ locally, we also have that the image of $y_i$ in $\mathscr{O}_{\bb{Y}_{k'}, p'}$ cuts out the unique irreducible component of $(\bb{D}_i)_{k'}$ which contains $p'$ (note here we use $k$ perfect, for geometric regularity of $\bb{D}_i$). The result follows, since the images of the $y_i$ are linearly independent elements of $\frak{m}_{p'} / \frak{m}_{p'}^2$, which implies the $y_i$ have linearly independent image in $\frak{m}_p / \frak{m}_p^2$, which in turn means the $y_i$ form part of a regular system of parameters for $\mathscr{O}_{\bb{Y}, p}$.

    For the forward direction, it suffices to prove the statement for $p$ closed, by density. Fix an affine neighborhood $\bb{U} = \Spf(B) \subset \bb{Y}$ of $p$ intersecting only the strata of $\bb{D}$ which contain $p$, and furthermore suppose each irreducible component of $\bb{D}|_{\bb{U}}$ (we enumerate these components $\bb{D}_1,...,\bb{D}_r$) is cut out by a single regular function $f_i$ on $\bb{U}$. Complete the $f_i$ to a sequence $f_1,...,f_n$ of functions on $\bb{U}$ forming a regular sequence of parameters for the local ring $\mathscr{O}_{\bb{Y}, p}$. We set $L = \oplus_{i = 1}^n\bb{Z} \cdot e_i$, $\tau = \big(\oplus_{i \leq r} \bb{R}_{\geq 0} \cdot e_i\big) \oplus \{0\}$, and $F$ the union of faces of $\tau$ corresponding to subsets $J \subset [r]$ such that $\cap_{i \in J} \bb{D}_i$ is non-algebraic.

    We note there is a natural injective homomorphism $k[\tau^{\vee}_{\bb{Z}}] \rightarrow \mathscr{O}_{\bb{U}}$ sending $t^{e_i^*} \mapsto f_i$ for $i \leq r$ and $t^{e_i^*} \mapsto f_i + 1$ for $i > r$, defining a map $\psi : \bb{U} \rightarrow \TV(\tau)$. Denote by $N \subset k[\tau_{\bb{Z}}^{\vee}]$ the ideal of elements which become topologically nilpotent in $\mathscr{O}_U$. Then $N$ defines an ideal of $\TV(\tau)$, whose intersection with $\tau_{\bb{Z}}^{\vee}$ generates the equivariant ideal associated to $F \subset \tau$. Define $\whTV(\tau)$ the completion along $N$. We note now that $\psi$ factors through an adic morphism $\bb{U} \rightarrow \whTV(\tau \subset L_{\bb{R}})$, which by abuse of notation we also call $\psi$. We set $\bb{T} = \whTV(\tau \subset L_{\bb{R}})$ for the remainder of the argument.

     Letting $A = \Gamma(\bb{T}, \mathscr{O}_{\bb{T}})$ and $t_i := t^{e_i^*} \in A$, we consider the adic homomorphism $\psi^* : A \rightarrow B$ of topological rings.  By construction $\{t_i\}_{i = 1}^r \cup \{t_i - 1\}_{i = r + 1}^n$ and $\{f_i\}_{i = 1}^n$ form regular sequences of parameters at $\psi(p) = \{0\}^r \times \{1\}^{n - r}$ and $p$, respectively. Now $\widehat{\mathscr{O}}_{\bb{T}, \psi(p)} \rightarrow  \widehat{\mathscr{O}}_{\bb{U}, p}$ induced by $\psi$ is just $k[[x_1,...,x_n]] \hookrightarrow  k(p)[[x_1,...,x_n]]$, a constant base change of the finite, separable extension $k \hookrightarrow k(p)$ (here also we use that $k$ is perfect). Then $\Omega^1_{\whscrO_{\bb{U}, p} / \whscrO_{\bb{T}, \psi(p)}} = 0$, and \cref{lemma:VanishDiff} gives $\Omega^1_{\bb{U}/\bb{T}, p} = 0$, implying $\Omega^1_{\bb{U} / \bb{T}} = 0$ in a formal Zariski open neighborhood of $p$. Similarly, noting the natural isomorphisms of completed local rings $\whscrO_{\Spec(A), \psi(p)} \cong \whscrO_{\bb{T}, \psi(p)}$ and $\whscrO_{\Spec(B), p} \cong \whscrO_{\bb{U}, p}$, we see $\mathscr{O}_{\Spec(A), \psi(p)} \rightarrow \mathscr{O}_{\Spec(B), p}$ is flat, so up to shrinking $\bb{U}$ we can assume $\psi$ is flat and unramified, hence \'etale (see I 5.3(b) of \cite{fujiwara2017foundationsrigidgeometryi}).

    We note the $L, \tau$ constructed above \textit{are} the minimal choice of the second paragraph. Indeed, consider another \'etale morphism $\psi' : \bb{U}' \rightarrow \whTV(\tau' \subset L')$ with $p \in \bb{U}'_s$. Fix a basis for $L' = \oplus_{i = 1}^n \bb{Z} \cdot e_i$ so that $\tau' = \big(\oplus_{i = 1}^{r'} \bb{R}_{\geq 0} \cdot e_i\big) \oplus \{0\}$. We can assume the first $r$ generators $e_i$ correspond to the components which pull back to the $\bb{D}_i$. Now we have an isomorphism $L \rightarrow L'$ of lattices taking $\tau$ to a face of $\tau'$. We also see $\psi'$ determines the homomorphisms $k[(\tau')^{\vee}_{\bb{T}}] \rightarrow k[\tau^{\vee}_{\bb{T} }]\rightarrow \mathscr{O}_{\bb{U}'}$. With $N' \leq k[(\tau')_{\bb{T}}^{\vee}]$ the ideal of topologically nilpotent elements, we have $N'$ the preimage of $N$, which in particular implies that $F$ is the intersection with $\tau$ of the preimage of $F'$ in $L$. The lattice isomorphism gives rise to an open immersion $\whTV(\tau \subset L_{\bb{R}}) \hookrightarrow \whTV(\tau' \subset L'_{\bb{R}})$, the completion along the ideal defined by $N'$ of the toric $\TV(\tau \subset L) \hookrightarrow \TV(\tau' \subset L')$. The $\psi'' : \bb{U}'' \rightarrow \whTV(\tau)$ of the claim is now just defined by pullback. 
\end{proof}
\end{proposition}

\begin{definition}\label{def:DualFan} (\textit{Dual Fan}) We define here a geometric complex $\Sigma_{(\bb{Y}, \bb{D})}$ of finite rational polyhedral cones. For each closed stratum $\bb{S}$ of $\bb{D}$, consider the integral lattice $M_{\bb{S}} \cong \oplus_{\bb{D}_i \supset \bb{S}}\mathbb{Z} \cdot [\bb{D}_i]^*$ and the cone $\sigma_{\bb{S}} \hookrightarrow M_{\bb{S}, \mathbb{R}}$ given by the monoid sum $\oplus_{\bb{D}_i \supset \bb{S}} \mathbb{R}_{\geq 0} \cdot [\bb{D}_i]^*$. The elements of $\Sigma_{(\bb{Y}, \bb{D})}$ are the cones $\sigma_{\bb{S}}$. For every inclusion $\bb{S}' \subset \bb{S}$ of strata, there is a canonical inclusion of integral lattices $M_{\bb{S}} \hookrightarrow M_{\bb{S}'}$, which realizes $\sigma_{\bb{S}}$ as a face of $\sigma_{\bb{S}'}$. The resulting inclusions $\sigma_{\bb{S}} \hookrightarrow \sigma_{\bb{S}'}$ are the morphisms of $\Sigma_{(\bb{Y},\bb{D})}$. The geometric realization $|\Sigma_{(\bb{Y}, \bb{D})}|$ is the piecewise integral-linear space obtained by gluing the cones of $\Sigma_{(\bb{Y},\bb{D})}$ along the inclusion morphisms. In case the intersection of any finite collection of $\bb{D}_i$ is irreducible, then $\Sigma_{(\bb{Y}, \bb{D})}$ can be realized as a cone complex in $\oplus_{i \leq N} \mathbb{R} \cdot [\bb{D}_i]^*$.

We define also a complex $\overline{\Sigma}_{(\bb{Y}, \bb{D})}$ of extended cones, where for each $\bb{S}$ we have the extended lattice $\overline{M}_{\bb{S}, \mathbb{R}} := \oplus_{\bb{D}_i \supset \bb{S}} [-\infty, \infty] \cdot [\bb{D}_i]^*$, and $\overline{\sigma}_{\bb{D}_i} \hookrightarrow \overline{M}_{\bb{S}, \mathbb{R}}$ the closure of $\sigma_{\bb{D}_i}$. The morphisms $\sigma_{\bb{S}} \hookrightarrow \sigma_{\bb{S}'}$ extend naturally to $\overline{\sigma}_{\bb{S}} \hookrightarrow \overline{\sigma}_{\bb{S}'}$, and the gluing of the $\overline{\sigma}_{\bb{S}}$ along these gives a topological space $|\overline{\Sigma}_{(\bb{Y}, \bb{D})}|$ containing $|\Sigma_{(\bb{Y}, \bb{D})}|$ as an open, dense subset. Furthermore, for any closed stratum $\bb{S}$ of $\bb{D}$, let $\bb{D}_{\bb{S}}$ denote the restriction to $\bb{S}$ of the irreducible components of $\bb{D}$ that do not contain $\bb{S}$. For any closed stratum $\bb{T}$ of $\bb{D}_{\bb{S}}$, there is a natural inclusion of $\overline{\sigma}_{\bb{T}} \in \overline{\Sigma}_{(\bb{S}, \bb{D}_{\bb{S}})}$ in $\overline{\sigma}_{\bb{T}} \in \overline{\Sigma}_{(\bb{Y}, \bb{D})}$, 
$$ \sum_{\bb{D}_i \not \supset \bb{S}} a_i e_{\bb{D}_i \cap \bb{S}} \mapsto \sum_{\bb{D}_i \supset \bb{S}} +\infty \cdot [\bb{D}_i]^* + \sum_{\bb{D}_i \not \supset \bb{S}} a_i [\bb{D}_i]^*. $$
These inclusions, ranging over the closed strata $\bb{T}$ of $\bb{D}_{\bb{S}}$, are compatible with the morphisms of $\overline{\Sigma}_{(\bb{S}, \bb{D}_{\bb{S}})}$ and also with those of $\overline{\Sigma}_{(\bb{Y}, \bb{D})}$, so they define a canonical $|\overline{\Sigma}_{(\bb{S}, \bb{D}_{\bb{S}})}| \hookrightarrow |\overline{\Sigma}_{(\bb{Y}, \bb{D})}|$. The images, ranging over closed strata $\bb{S}$ of $\bb{D}$, of the open dense subsets $|\Sigma_{(\bb{S}, \bb{D}_{\bb{S}})}|$ under these inclusions stratify  $|\overline{\Sigma}_{(\bb{Y}, \bb{D})}|$.

For each stratum $\bb{S}$ of $\bb{D}$, there is a natural map $r_{\bb{S}} : Y \rightarrow \overline{\sigma}_{\bb{S}}$ given by 
$$ y \mapsto \sum_{\bb{D}_i \supset \bb{S}} \bb{D}_i^{\trop}(y) \cdot [\bb{D}_i]^*.$$
We then define a canonical $r_{(\bb{Y}, \bb{D})} : Y \rightarrow |\overline{\Sigma}_{(\bb{Y}, \bb{D})}|$ by $r_{(\bb{Y}, \bb{D})}(y) = r_{\bb{S}}(y)$ for all $y$ such that $\bb{S}$ is the minimal stratum containing $\red_{\bb{Y}}(y)$. 
\\

\noindent 
\textbf{Claim:} The map $r_{(\bb{Y}, \bb{D})}$ is continuous.
\begin{proof}
    Fix a point $y \in Y$, and a sequence $\{y_i\}_{i \in \mathbb{N}}$ with $\lim_{i \to \infty} y_i = y$. Fixing $\bb{S}$ the minimal stratum containing $\red_{\bb{Y}}(y)$, up to taking a subsequence we can assume all $\red_{\bb{Y}}(y_i) \in \bb{S}_s$, by anticontinuity of the reduction map. And up to taking a further subsequence, we can assume there exists a stratum $\bb{S}' \subset \bb{S}$ so that, for all $i$, $\bb{S}'$ is the minimal stratum containing $\red_{\bb{Y}}(y_i)$. Then as $i \to \infty$, 
    $$r_{(\bb{Y}, \bb{D})}(y_i) = r_{\bb{S}'}(y_i) \to r_{\bb{S}'}(y) = r_{(\bb{Y}, \bb{D})}(y),$$
    because $r_{\bb{S}'} = r_{\bb{S}}$ on the locus of points $\red_{\bb{Y}}^{-1}(\bb{S}) \setminus \red_{\bb{Y}}^{-1}(\bb{S}')$. 
\end{proof}
\hspace{1mm}

For every closed stratum $\bb{S}$ of $\bb{D}$, the restriction of $r_{(\bb{Y}, \bb{D})}$ to $\bb{S}_{\eta} \hookrightarrow \bb{Y}_{\eta}$ maps into the image of the embedding $|\overline{\Sigma}_{(\bb{S}, \bb{D}_{\bb{S}})}| \hookrightarrow |\overline{\Sigma}_{(\bb{Y}, \bb{D})}|$ and is, under these identifications, precisely $r_{(\bb{S}, \bb{D}_{\bb{S}})}$.

For each stratum $\bb{S}$ of $\bb{D}$, define $\sigma_{\bb{S}}^{\circ} \subset \sigma_{\bb{S}}$ the complement of the proper faces, and similarly $\overline{\sigma}^{\circ}_{\bb{S}} \subset \overline{\sigma}_{\bb{S}}$. We define the $(\bb{Y}, \bb{D})$-skeleton, or $\Sk(\bb{Y}, \bb{D})$, as the union in $|\Sigma_{(\bb{Y}, \bb{D})}|$ of the $\sigma_{\bb{S}}^{\circ}$ ranging over all algebraic strata $\bb{S}$, and $\oSk(\bb{Y}, \bb{D}) \subset |\overline{\Sigma}_{(\bb{Y}, \bb{D})}|$ the union of the $\overline{\sigma}_{\bb{S}}^{\circ}$. For later convenience, we define also $\hsigma_{\bb{S}} = \sigma_{\bb{S}} \cap \Sk(\bb{Y}, \bb{D})$, $\hosigma_{\bb{S}} = \osigma_{\bb{S}} \cap \oSk(\bb{Y}, \bb{D})$, and $*\sigma_{\bb{S}} \subset |\Sigma_{(\bb{Y}, \bb{D})}|$ the union of $\sigma_{\bb{S}'}^{\circ}$ over all strata $\bb{S}' \subset \bb{S}$.

We note finally the basic compatibility of the reduction map with the SYZ fibration, which follows immediately from the properties (a)-(c) listed in \cref{def:CartierFnct}. For an algebraic stratum $\bb{S}$ of $\bb{D}$, we have $\text{red}_{\bb{Y}}^{-1}(\bb{S} \setminus \bb{D}_{\bb{S}}) \cap \oSk(\bb{Y}, \bb{D}) = r_{(\bb{Y}, \bb{D})}^{-1}(\osigma_{\bb{S}}^{\circ})$. 
\end{definition}

\begin{remark}\label{rem:RealizableFan} (\textit{Realizable Dual Fan}) Note that we have elected for generality in the above definition, as we do not insist the complex $\Sigma_{(\bb{Y}, \bb{D})}$ is realizable (ie. living in some $\bb{R}^N$, as it does in our main application). This is because, for a general SNC pair as we have defined it, intersections of arbitrary finite collections of irreducible components of $\bb{D}$ are not necessarily irreducible. However, in the case that all such intersections are irreducible, then $\Sigma_{(\bb{Y}, \bb{D})}$ is realizable as a fan in $\bb{R}^{\bb{D}}$ (one factor for each irreducible component of $\bb{D}$) in the obvious way: For each stratum $\bb{S}$, $\sigma_{\bb{S}} \subset \bb{R}^{\bb{D}}$ is the $\bb{R}_{\geq 0}$-span of the standard basis vectors corresponding to components of $\bb{D}$ which contain $\bb{S}$. 

\end{remark}

\begin{proposition}\label{prop:FunctTrop} (\textit{Tropicalization of a Map}) Fix $k'/k$ a trivially valued extension, $(\bb{Y}, \bb{D})$ and $(\bb{Y}', \bb{D}')$ snc pairs of special formal schemes over $k$ and $k'$, respectively, and a map $f : \bb{Y}' \rightarrow \bb{Y}$ such that $f^{-1}(\bb{D}) \subset \bb{D}'$. Then $f$ induces a canonical continuous map $\trop(f) : |\oSigma_{(\bb{Y}', \bb{D}')}| \rightarrow |\oSigma_{(\bb{Y}, \bb{D})}|$, restricting to piecewise integral-linear $|\Sigma_{(\bb{Y}', \bb{D}')}| \rightarrow |\Sigma_{(\bb{Y}, \bb{D})}|$, such that 
$$r_{(\bb{Y}, \bb{D})} \circ f_{\eta} = \trop(f) \circ r_{(\bb{Y}', \bb{D}')} $$
as maps $Y' \rightarrow |\oSigma_{(\bb{Y}, \bb{D})}|$. 
\begin{proof}
    Fix a stratum $\bb{S}'$ of $\bb{D}'$, and $\bb{S}$ the minimal stratum of $\bb{Y}$ containing its image (note this could be the empty stratum, ie. $\bb{Y})$. We define $\sigma_{\bb{S}'} \rightarrow \sigma_{\bb{S}}$ by 
    $$ [\bb{D}_j']^* \mapsto \sum_{\bb{D}_i \supset \bb{S}} \text{ord}_{\bb{D}_j'}(f^* \bb{D}_i)[\bb{D}_i]^* $$
    for each $\bb{D}_j' \supset \bb{S}'$. Consider $\bb{S}'_1 \supset \bb{S}'_2$ an inclusion of strata of $\bb{D}'$, with $\bb{S}_1 \supset \bb{S}_2$ the minimal strata of $\bb{D}$ containing their images. We compute the image of a $[\bb{D}_j']^*$ such that $\bb{D}_j' \supset \bb{S}_1'$ under the map $\sigma_{\bb{S}_2'} \rightarrow \sigma_{\bb{S}_2}$ to be $\Sigma_{\bb{D}_i \supset \bb{S}_2} \ord_{\bb{D}_j'}(f^* \bb{D}_i) [\bb{D}_i]^*$. If $\bb{D}_i \supset \bb{S}_2$ but $\bb{D}_i \not \supset \bb{S}_1$, then $f^* \bb{D}_i$ intersects $\bb{D}_j'$ in positive codimension, so $\ord_{\bb{D}_j'} (f^* \bb{D}_i) = 0$. We see from this argument that the various maps $\sigma_{\bb{S}'} \rightarrow \sigma_{\bb{S}}$ are compatible, so glue to the piecewise integral-linear $\trop(f)$ of the statement.

    For the compatibility of $\trop$ with the SYZ fibration, fix a point $y' \in Y' \setminus D'$ with $y = f_{\eta}(y')$ and take $\bb{S}'$ the minimal stratum of $\bb{D}'$ containing $\text{red}_{\bb{Y}'}(y')$. Fix also $\bb{S}$ the minimal stratum of $\bb{D}$ containing $f(\bb{S}')$, which we observe is the same as the minimal stratum containing $f_s(\text{red}_{\bb{Y}'}(y')) = \text{red}_{\bb{Y}}(y)$. We then have 
    $$ \trop(r_{(\bb{Y}', \bb{D}')}(y')) = \trop\big( \sum_{\bb{D}_j \supset \bb{S}'} (\bb{D}_j')^{\trop}(y') \cdot [\bb{D}_j']^*\big) $$
    $$ = \sum_{\bb{D}_i \supset \bb{S}} \big(\sum_{\bb{D}_j \supset \bb{S}'}  \ord_{\bb{D}_j'}(f^*\bb{D}_i) (\bb{D}_j')^{\trop}(y') \big) \cdot [\bb{D}_i]^* $$
    $$ = \sum_{ \bb{D}_i \supset \bb{S}} \bb{D}_i^{\trop}(y) \cdot [\bb{D}_i]^* = r_{(\bb{Y}, \bb{D})}(y). $$
    For the extension to $|\oSigma_{(\bb{Y}', \bb{D}')}|$, we apply the same argument to $f|_{\bb{S}'} : \bb{S}' \rightarrow \bb{Y}$ and the SNC pair $(\bb{S}', \bb{D}'_{\bb{S}'})$ for each stratum $\bb{S}'$ of $\bb{D}'$. 
\end{proof}
\end{proposition}

\begin{definition}\label{def:BasicCharts}(\textit{Basic Charts}) Fix $(\bb{X}, \bb{E})$ snc, $L$ a lattice, and $\Sigma_t \subset L_{\bb{R}}$ a regular fan. Set $\bb{Z}_{t} = \whTV(\Sigma_t)$, completed with respect to some ideal sheaf on $\TV(\Sigma_t)$ whose vanishing intersects every boundary stratum in a geometrically irreducible subvariety, and let $\partial \bb{Z}_{t}$ be the toric boundary. We note the assumption on the defining ideal sheaf guarantees that $\Sigma_{(\bb{Z}_t, \partial \bb{Z}_t)} = \Sigma_t$.

An \'etale map $\psi : \bb{X} \rightarrow \bb{Z}_{t}$ satisfying $\bb{E} = \psi^{-1}(\partial\bb{Z}_{t})$ will be called $\textit{basic}$ if it induces an isomorphism of the stratification of $\bb{E}$ with the toric boundary of the codomain, ie. $\psi$ induces a bijection of closed strata which preserves both algebraicity as well as the partial order defined by inclusion of strata. In this case we see there is a canonical isomorphism $\Sigma_{(\bb{X}, \bb{E})} \cong \Sigma_t$ identifying $\Sk(\bb{X}, \bb{E})$ with $\Sk(\bb{Z}_t, \partial \bb{Z}_{t})$. The associated piecewise integral-linear homeomorphism $|\Sigma_{(\bb{X}, \bb{E})}| \rightarrow |\Sigma_t|$ is $\trop(\psi)$.

If $(\bb{Y}, \bb{D})$ is snc and $\psi : \bb{U} \rightarrow \whTV(\tau \subset L_{\bb{R}})$ is a minimal chart for the stratum $\bb{S}$ of $\bb{D}$, then $\psi$ can be made basic by subtracting off all strata not containing $\bb{S}$. 
\end{definition}

\begin{definition}\label{def:ToricBlowups} (\textit{Toric Blowups}) Recall from \cref{rem:ProjSubdiv} the notion of projective subdivisions of a polyhedral complex. We take $\Sigma$ the fan $\Sigma_{(\bb{Y}, \bb{D})}$, and fix $\varphi : |\Sigma| \rightarrow \bb{R}$ a piecewise integral-linear function, convex on restriction to each cone, inducing a projective subdivision $\Sigma_{\varphi}$ of $\Sigma$. Let $\mathscr{K}_{\bb{Y}}$ denote the total sheaf of fractions. We define the fractional ideal sheaf $\mathscr{I}_{\varphi} \leq \mathscr{K}_{\bb{Y}}$ as follows: Fix a basic chart $\psi : \bb{U} \rightarrow \whTV(\sigma_{\bb{S}} \oplus \{0\})$ for a stratum $\bb{S}$. The ideal sheaf $\mathscr{I}_{\varphi}|_{\bb{U}}$ is defined as the $\psi$-inverse image of the fractional ideal sheaf on $\TV(\sigma_{\bb{S}} \oplus \{0\})$ determined by $\varphi|_{\sigma_{\bb{S}}}$. To see this is well-defined, suppose $\bb{S}' \subset \bb{S}$, and $\psi' : \bb{U}' \rightarrow \whTV(\sigma_{\bb{S}'} \oplus \{0\})$ basic for $\bb{S}'$ such that $p \in \bb{U}_s'$. By the second paragraph of \cref{prop:SNCCharts}, we have a Cartesian diagram 
\begin{center}
    \begin{tikzcd}
        \bb{U}'' \arrow[r,"\psi''"] \arrow[d] & \whTV(\sigma_{\bb{S}} \oplus \{0\}) \arrow[d] \\
        \bb{U}' \arrow[r, "\psi'"] & \whTV(\sigma_{\bb{S}'} \oplus \{0\})
    \end{tikzcd}
\end{center}
so the inverse image in $\bb{U}'$ of the fractional ideal sheaf on $\TV(\sigma_{\bb{S}'} \oplus \{0\})$ determined by $\varphi|_{\sigma_{\bb{S}'}}$ restricts on $\bb{U}''$ to the inverse image of the fractional ideal sheaf on $\TV(\sigma_{\bb{S}} \oplus \{0\})$ determined by $\varphi|_{\sigma_{\bb{S}}}$. Now $\psi$ and $\psi''$ both induce basic, minimal charts $\bb{U} \cap \bb{U}'' \rightarrow \whTV(\sigma_{\bb{S}} \oplus \{0\})$. The two pullbacks to $\mathscr{K}_{\bb{Y}}(\bb{U} \cap \bb{U}'')$ of an element of $(M_{\bb{S}} \oplus \bb{Z}^{\codim_{\bb{Y}}(\bb{S})})^{\vee}$ under these charts will agree up to multiplication by an element of $\mathscr{O}_{\bb{U} \cap \bb{U}''}^*$, since they are both invertible away from $\bb{D}$ and have identical orders of vanishing on each irreducible component of $\bb{D}|_{\bb{U} \cap \bb{U}''}$. Thus the inverse images in $\bb{U} \cap \bb{U}''$ of the fractional ideal sheaf determined by $\varphi|_{\sigma_{\bb{S}}}$ agree also, and $\mathscr{I}_{\varphi}$ is well-defined.

Let $ \Bl_{\varphi}(\bb{Y})$ be the normalization of the formal blowup of $\bb{Y}$ along $\mathscr{I}_{\varphi}$, with associated projective map $b_{\varphi} : \Bl_{\varphi}(\bb{Y}) \rightarrow \bb{Y}$. We will see shortly a canonical identification of cone complexes $\Sigma_{(\Bl_{\varphi}(\bb{Y}), b_{\varphi}^{-1}(\bb{D}))} \equiv \Sigma_{\varphi}$, under which $\trop(b_{\varphi})$ is identified with the subdivision map $|\Sigma_{\varphi}| \rightarrow |\Sigma|$.  
\end{definition}

\begin{lemma}\label{lemma:BlPreservesBasic} Fix notation as in the first paragraph of \cref{def:BasicCharts}.
Also fix $\varphi : \Sigma_t \rightarrow \bb{R}$ convex and piecewise integral-linear, defining the complex $\Sigma_{t, \varphi}$ of polyhedra in $L$ with support $|\Sigma_t|$. Set $\bb{Z}_{t, \varphi} = \whTV(\Sigma_{t, \varphi}) = \Bl_{\varphi}(\bb{Z}_{t})$. Then the square diagram 
\begin{center}
    \begin{tikzcd}
        \Bl_{\varphi}(\bb{X}) \arrow[r,"b_{\varphi}"] \arrow[d, "\tilde{\psi}"] & \bb{X} \arrow[d, "\psi"] \\ 
         \bb{Z}_{t, \varphi}\arrow[r, "b_{t,\varphi}"] & \bb{Z}_t
    \end{tikzcd}
\end{center}
is Cartesian (here $\tilde{\psi}$ is obtained from the universal property of blowing up), and $\tilde{\psi}$ is basic, inducing an identification of cone complexes $\Sigma_{(\Bl_{\varphi}(\bb{X}), b_{\varphi}^{-1}(\bb{E}))} \equiv \Sigma_{t,\varphi}$ such that $\trop(b_{\varphi})$ is the subdivision $\Sigma_{t,\varphi} \rightarrow \Sigma_t$. 
\begin{proof}
     To avoid ambiguity, we use $\mathscr{I}_{\varphi}$ to represent the ideal sheaf on $\bb{Z}_t$ induced by $\varphi$. We recall from the definition that the corresponding ideal sheaf on $\bb{X}$ is just $\psi^{-1} \mathscr{I}_{\varphi} \cdot \mathscr{O}_{\bb{X}}$.

     For the first claim, we decompose the square diagram of the statement to two Cartesian squares: 
     \begin{center}
         \begin{tikzcd}
             \Bl_{\varphi}(\bb{X}) \arrow[r] \arrow[d, "\tilde{\psi}"] & \fBl_{\psi^{-1}\mathscr{I}_{\varphi} \cdot \mathscr{O}_{\bb{X}}}(\bb{X}) \arrow[d] \arrow[r] & \bb{X} \arrow[d, "\psi"] \\ 
         \bb{Z}_{t, \varphi}\arrow[r] & \fBl_{\mathscr{I}_{\varphi}}(\bb{Z}_t) \arrow[r]  & \bb{Z}_t
         \end{tikzcd}
     \end{center}
     The horizontal maps of the rightmost (resp. leftmost) square are formal blowup (resp. normalization). The rightmost square is Cartesian because $\psi$ is flat, and so the natural surjective map $\psi^* \mathscr{I}_{\varphi}^n \rightarrow \psi^{-1} \mathscr{I}_{\varphi}^n \cdot \mathscr{O}_{\bb{X}}$ is an isomorphism. Then we get an isomorphism of graded $\mathscr{O}_{\bb{X}}$-algebras 
     $$ \oplus_{n = 0}^{\infty} \psi^* \mathscr{I}_{\varphi}^n \rightarrow \oplus_{n = 0}^{\infty} \psi^{-1} \mathscr{I}_{\varphi}^n \cdot \mathscr{O}_{\bb{X}}, $$
     the (completed) Proj of the left (resp. right) ring defining $\bb{X} \times_{\bb{Z}_{t}} \bb{Z}_{t, \varphi}$ (resp. $\fBl_{\psi^{-1} \mathscr{I}_{\varphi} \cdot \mathscr{O}_{\bb{X}}}(\bb{X}))$. The leftmost square is Cartesian because its fiber product is already normal, having an \'etale map to the normal $\bb{Z}_{t, \varphi}$.

     Fix a stratum $\bb{S}_{\varphi}$ of $b_{\varphi}^{-1}(\bb{E})$. Since $\tilde{\psi}$ is \'etale, $\tilde{\psi}(\bb{S}_{\varphi})$ is a dense Zariski open of a stratum $\bb{S}_{t, \varphi}$ of $\partial \bb{Z}_{t, \varphi}$. In turn, $b_{t, \varphi}(\bb{S}_{t, \varphi})$ is a stratum $\bb{S}_t$ of $\partial \bb{Z}_t$. Now because $\psi$ is basic, the preimage of $\bb{S}_t$ in $\bb{X}$ is a stratum $\bb{S}$ of $\bb{E}$, so we have also $b_{\varphi}(\bb{S}_{\varphi}) = \bb{S}$. In particular note that $\bb{S}_{t, \varphi}$ determines the image of $\bb{S}_{\varphi}$ under $b_{\varphi}$.

     To show $\tilde{\psi}$ is basic, we must show that, given any stratum $\bb{S}_{t, \varphi}$ of $\partial \bb{Z}_{t, \varphi}$, then $\tilde{\psi}^{-1}(\bb{S}_{t, \varphi})$ is an irreducible stratum $\bb{S}_{\varphi}$ of $b_{\varphi}^{-1}(\bb{E})$. Since $\tilde{\psi}$ is basic, we know already that $\tilde{\psi}^{-1}(\bb{S}_{t, \varphi})$ is a (nonempty) union of strata of $b_{\varphi}^{-1}(\bb{E})$, and by the previous paragraph all these strata have the same image stratum $\bb{S}$ of $\bb{E}$. We then obtain an isomorphism $\tilde{\psi}^{-1}(\bb{S}_{t, \varphi}) \cong \bb{S}_{t,\varphi} \times_{\bb{S}_t} \bb{S}$, so it will suffice to show this fiber product is irreducible.

     We will argue by passing to the generic fiber. Therefore we set $S = \bb{S}_{\eta}, S_t = (\bb{S}_t)_{\eta}$, and $S_{t, \varphi} = (\bb{S}_{t, \varphi})_{\eta}$. Note $b_{\varphi}$ is a proper map, so also $b_{\varphi}^{-1}(\bb{S}) \rightarrow \bb{S}$ is, and $\tilde{\psi}^{-1}(\bb{S}_{t, \varphi}) \subset b_{\varphi}^{-1}(\bb{S})$ a closed formal subscheme, making $\tilde{\psi}^{-1}(\bb{S}_{t, \varphi}) \rightarrow \bb{S}$ proper. We've seen each irreducible component of $\tilde{\psi}^{-1}(\bb{S}_{t, \varphi})$ maps surjectively onto $\bb{S}$, so there must be a point $s \in S$ (fix $s_t$ its image in $S_t$) such that the fiber $\tilde{\psi}_{\eta}^{-1}(S_{t, \varphi})_s \cong (S_{t, \varphi})_{s_t} \times_{\mathscr{H}(s_t)} \mathscr{H}(s)$ is reducible.

     Now observe that $\bb{S}_{t, \varphi} \rightarrow \bb{S}_t$ is obtained as the formal completion of a proper, equivariant map of toric varieties, say $\mathcal{S}_{t, \varphi} \rightarrow \mathcal{S}_t$, the torus for $\mathcal{S}_{t}$ a quotient of the torus for $\mathcal{S}_{t, \varphi}$, and the kernel of the quotient homomorphism a torus acting on each fiber with a big torus orbit. But, by \cref{cor:PropSepPreserved} properness implies the fibers of $S_{t, \varphi} \rightarrow S_t$ are the same as fibers of $\mathcal{S}_{t, \varphi}^{\an} \rightarrow \mathcal{S}_t^{\an}$ over points of $S_t \subset \mathcal{S}_t^{\an}$, so they will likewise contain big torus orbits, making them geometrically irreducible. Hence $\tilde{\psi}^{-1}(\bb{S}_{t, \varphi})$ must be irreducible.

\end{proof}
\end{lemma}

\begin{proposition}\label{prop:BlSubdiv} Fix $(\bb{Y}, \bb{D})$ snc and $\varphi : |\Sigma_{(\bb{Y}, \bb{D})}| \rightarrow \bb{R}$ piecewise integral-linear and convex on restriction to each cone. Let $\bb{Y}' = \Bl_{\varphi}(\bb{Y})$ and $\bb{D}' = b_{\varphi}^{-1} \bb{D}$. Then there is a natural identification $\Sigma_{(\bb{Y}', \bb{D}')} \equiv \Sigma_{\varphi}$, under which $\trop(b_{\varphi})$ is identified with the subdivision map $|\Sigma_{\varphi}| \rightarrow |\Sigma_{(\bb{Y}, \bb{D})}|$. 
\begin{proof}
    Suppose $\bb{U} \subset \bb{Y}$, $\psi : \bb{U} \rightarrow \bb{Z}$ a basic affine, minimal for a stratum $\bb{S}$ of $\bb{D}$. Let $\bb{U}' \rightarrow \bb{U}$ be the normalized blowup induced by $\varphi$, similarly $\bb{Z}' \rightarrow \bb{Z}$ for the toric model. Then by \cref{lemma:BlPreservesBasic}, $\bb{U}' \subset \bb{Y}'$ has a basic map $\bb{U}' \rightarrow \bb{Z}'$, so $\Sigma_{(\bb{U}', \bb{D}'|_{\bb{U}'})}$ is the subdivision of $\sigma_{\bb{S}}$ induced by $\varphi$, and $|\Sigma_{(\bb{U}', \bb{D}'|_{\bb{U}'})}| \hookrightarrow |\Sigma_{(\bb{Y}', \bb{D}')}|$ is the preimage under $\trop(b_{\varphi})$ of $\sigma_{\bb{S}}$, since $\bb{U}'$ will see all strata of $\bb{D}'$ which lie over strata of $\bb{D}$ containing $\bb{S}$. The result follows. 
\end{proof}
\end{proposition}

\noindent 

\begin{definition}\label{def:Prim} Fix $(\bb{Y}, \bb{D})$ snc. Suppose $\bb{Z}$ is a prime Weil divisor living on a formal blowup $\bb{Y}' \rightarrow \bb{Y}$. The \textit{center of $\bb{Z}$}, or center$_{\bb{Y}}(\bb{Z})$, is the reduced-induced formal subscheme structure on the image of $\bb{Z}$ in $\bb{Y}$. Denote by $\text{Prim}^*(\bb{Y}, \bb{D})$ the union of $\{0\}$ and the set of prime Weil divisors $\bb{Z}$ living on an snc toric blowup of $\bb{Y}$ with center$_{\bb{Y}}(\bb{Z})$ a stratum of $\bb{D}$, quotiented by the minimal equivalence relation which identifies $\bb{Z}'$ with $\bb{Z}''$ whenever $\bb{Y}'' \rightarrow \bb{Y}' \rightarrow \bb{Y}$ are snc blowups, $\bb{Z}'$ (resp. $\bb{Z}''$) lives on $\bb{Y}'$ (resp. $\bb{Y}'')$, and $\bb{Z}''$ is the strict transform of $\bb{Z}'$. An element of $\text{Prim}^*(\bb{Y}, \bb{D})$ is called an \textit{irreducible $(\bb{Y}, \bb{D})$-divisor}.

We define $\underset{\to}{\text{Div}}(\bb{Y}, \bb{D})$ the free abelian group over $\text{Prim}^*(\bb{Y}, \bb{D}) \setminus \{0\}$, $\underset{\to}{\text{Div}}^*(\bb{Y}, \bb{D})$ its dual, which contains in particular the dual elements $\text{Prim}(\bb{Y}, \bb{D})$ to the $\text{Prim}^*(\bb{Y}, \bb{D}) \subset \underset{\to}{\text{Div}}(\bb{Y}, \bb{D})$. We also set Prim$_{\alg}(\bb{Y}, \bb{D}) \subset \Prim(\bb{Y}, \bb{D})$ those elements dual to algebraic $(\bb{Y}, \bb{D})$-divisors, and $0$ if $\bb{Y}$ itself is algebraic. 
\end{definition}

\begin{theorem}\label{theorem:SYZMain} There is a canonical identification of Prim$(\bb{Y}, \bb{D})$ with the primitive integer points of $|\Sigma_{(\bb{Y}, \bb{D})}|$, therefore also an identification $\bb{N} \cdot \text{Prim}(\bb{Y}, \bb{D}) \equiv |\Sigma_{(\bb{Y}, \bb{D})}|(\bb{Z})$, which restricts to $\bb{N} \cdot \Prim_{\alg}(\bb{Y}, \bb{D}) \equiv \Sk(\bb{Y}, \bb{D})(\bb{Z})$.

The above identifications induce a canonical embedding $\oSk(\bb{Y}, \bb{D}) \hookrightarrow Y$, whose composition with the SYZ fibration $r_{(\bb{Y}, \bb{D})}$ is the identity, after restricting the codomain. Let $\Delta \subset \bb{Y}_s$ denote the union of algebraic strata of $\bb{D}$. The restriction of $r_{(\bb{Y}, \bb{D})}$ to the open analytic domain $\red_{\bb{Y}}^{-1}(\Delta)$ is a proper retraction onto $\oSk(\bb{Y}, \bb{D})$. For each stratum $\bb{S}$ of $\bb{D}$, the composition $\overline{\Sk}(\bb{S}, \bb{D}_{\bb{S}}) \hookrightarrow \oSk(\bb{Y}, \bb{D}) \hookrightarrow Y$ agrees with $\oSk(\bb{S}, \bb{D}_{\bb{S}}) \hookrightarrow S := \bb{S}_{\eta}$.
\begin{proof}
     We first construct a natural map $u : \Prim_{\alg}(\bb{Y}, \bb{D}) \rightarrow U$. Fix an irreducible algebraic $(\bb{Y}, \bb{D})$-divisor $[\bb{Z}]$, here $\bb{Z}$ a prime divisor defined on an snc blowup $b : \bb{Y}' \rightarrow \bb{Y}$. Since $\bb{Z}$ is algebraic of codimension 1, its local ring $\mathscr{O}_{\bb{Y}', \eta_{\bb{Z}}}$ is a DVR, so defines a valuation with center the generic point $\eta_{\bb{Z}} \in \bb{Y}'_s$, which then gives rise to a point of $U = \bb{Y}'_{\eta} \setminus (b^{-1}_* \bb{D})_{\eta}$ we call $u([\bb{Z}]^*)$. This is independent of the representative $\bb{Z}$ of the equivalence class $[\bb{Z}]$, since the exceptional locus of an snc blowup intersects the strict transform of $\bb{Z}$ in positive codimension.

    The composition of $u$ with $r_{(\bb{Y}, \bb{D})}$ extends naturally to a map $\Prim(\bb{Y}, \bb{D}) \rightarrow |\Sigma_{(\bb{Y}, \bb{D})}|(\bb{Z})$, defined in the following manner: Fix an irreducible $(\bb{Y}, \bb{D})$-divisor $[\bb{Z}]$, $\bb{Z}$ defined on an snc blowup $b : \bb{Y}' \rightarrow \bb{Y}$, and set $\bb{S} = \text{center}_{\bb{Y}}(\bb{Z})$. The image of $[\bb{Z}]^* \in \Prim(\bb{Y}, \bb{D})$ in $|\Sigma_{(\bb{Y}, \bb{D})}|(\bb{Z})$ is the point $ \Sigma_{\bb{D}_i \supset \bb{S}} \ord_{\bb{Z}}(b^* \bb{D}_i) [\bb{D}_i]^* \in \sigma_{\bb{S}}$. Note this is precisely the primitive point of $|\Sigma_{(\bb{Y}, \bb{D})}|(\bb{Z})$ which generates the ray corresponding to $\bb{Z}$ in the subdivision of $\Sigma_{(\bb{Y}, \bb{D})}$ that induces the blowup $b$. We also observe every primitive point is in the image, since an arbitrary rational ray $r$ in some cone of $\Sigma_{(\bb{Y}, \bb{D})}$ can be found in some further regular subdivision of $*(\Sigma_{(\bb{Y}, \bb{D})}, r)$, and the divisor corresponding to $r$ in the corresponding snc blowup of $\bb{Y}$ will map to the primitive integer point sitting on $r$.

    Fix $\bb{Z}'$ (resp. $\bb{Z}''$) defined on an snc blowup $b_{\varphi'} : \bb{Y}' \rightarrow \bb{Y}$ (resp. $b_{\varphi''} : \bb{Y}'' \rightarrow \bb{Y}$) associated to piecewise integral-linear, piecewise convex functions $\varphi', \varphi'' : |\Sigma_{(\bb{Y}, \bb{D})}| \rightarrow \bb{R}$. Let $\Sigma'$ (resp. $\Sigma''$) be the subdivision of $\Sigma_{(\bb{Y}, \bb{D})}$ induced by $\varphi'$ (resp. $\varphi''$). Suppose that $[\bb{Z}'], [\bb{Z}'']$ both map to the same point of $|\Sigma_{(\bb{Y}, \bb{D})}|$, spanning a ray $r \in \sigma_{\bb{S}}^{\circ}$. Then $r$ is common to both subdivisions $\Sigma'$ and $\Sigma''$. Let $\Sigma_r$ be a snc subdivision of $*(\Sigma_{(\bb{Y}, \bb{D})}, r)$, and $\bb{Y}_r \rightarrow \bb{Y}$ the corresponding blowup. Up to further snc blowup of $\bb{Y}', \bb{Y}''$ we can assume both $\Sigma'$ and $\Sigma''$ refine $\Sigma_r$. In that case $\varphi', \varphi'' : |\Sigma_r| \rightarrow \bb{R}$ are both piecewise integral-linear, piecewise convex, so $\bb{Y}',\bb{Y}'' \rightarrow \bb{Y}$ factor through snc blowups $\bb{Y}', \bb{Y}'' \rightarrow \bb{Y}_r$, and $\bb{Z}', \bb{Z}''$ are both strict transforms of the boundary divisor of $\bb{Y}_r$ corresponding to the ray $r$. Therefore $[\bb{Z}'] = [\bb{Z}'']$ and $\Prim(\bb{Y}, \bb{D}) \rightarrow |\Sigma_{(\bb{Y}, \bb{D})}|(\bb{Z})$ is an injective map with image the primitive integer points.

    Now take $\bb{S}$ an algebraic stratum of $\bb{D}$. We have a canonical map $\bb{R}_{\geq 0} \cdot (\sigma_{\bb{S}}^{\circ})_{\bb{Z}} \rightarrow U$, which we do not yet know extends to continuous $\sigma_{\bb{S}} \rightarrow U$ (and further, to continuous $\osigma_{\bb{S}} \rightarrow Y$). To get such an extension, we consider the completed local ring of $\bb{Y}$ at the generic point $\xi_{\bb{S}}$ of $\bb{S}$, which by definition of SNC is isomorphic to $k(\xi_{\bb{S}})[[y_1,...,y_r]]$, here $r = \codim_{\bb{Y}}(\bb{S})$, and each $\bb{D}_i$ containing $\bb{S}$ the vanishing of exactly one $y_i$. Equipping $k(\xi_{\bb{S}})$ with its trivial valuation, $\whbbY_{\xi_{\bb{S}}} := \Spf(\whscrO_{\bb{Y}, \xi_{\bb{S}}})$ becomes a special formal scheme over $k(\xi_{\bb{S}})$ with a natural map $h_{\bb{S}} : \whbbY_{\xi_{\bb{S}}} \rightarrow \bb{Y}$ whose generic fiber is the inclusion $B_{\xi_{\bb{S}}} \hookrightarrow Y$ of a product of $r$ Berkovich open unit disks over $k(\xi_{\bb{S}})$. Let $\whbbD_{i, \xi_{\bb{S}}}$ denote the pullback of $\bb{D}_{i}$ to $\whbbY_{\xi_{\bb{S}}}$, and $\whbbD_{\xi_{\bb{S}}}$ the sum of all the $\whbbD_{i, \xi_{\bb{S}}}$, so in particular $h_{\bb{S}}^{-1}(\bb{D}) = \whbbD_{ \xi_{\bb{S}}}$.

    The tropicalization from \cref{prop:FunctTrop} is just the natural cone inclusion $\trop(h_{\bb{S}}) : \osigma_{\bb{S}} \rightarrow | \oSigma_{(\bb{Y}, \bb{D})} |$, which we observe restricts on $\bb{R}_{ \geq 0} \cdot (\sigma_{\bb{S}})_{\bb{Z}}$ to the canonical inclusion of divisorial valuations via the identification $\bb{N} \cdot \Prim_{(\bb{Y}, \bb{D})} \equiv |\Sigma_{(\bb{Y}, \bb{D})}(\bb{Z})|$. Letting $r_{\bb{S}} : B_{\xi_{\bb{S}}} \rightarrow \osigma_{\bb{S}}$ the SYZ fibration associated to the SNC pair $(\whbbY_{\xi_{\bb{S}}}, \whbbD_{\xi_{\bb{S}}})$, we have $\trop(h_{\bb{S}}) \circ r_{\bb{S}} = r_{(\bb{Y}, \bb{D})} \circ (h_{\bb{S}})_{\eta}$, so the restriction of $(h_{\bb{S}})_{\eta}$ over $\osigma^{\circ}_{\bb{S}} \subset B_{\xi_{\bb{S}}}$ gives a continuous and injective $\iota_{\bb{S}} : \osigma^{\circ}_{\bb{S}} \rightarrow Y$ agreeing on $\bb{R}_{\geq 0} \cdot (\sigma_{\bb{S}}^{\circ})_{\bb{Z}}$ with $\trop(h_{\bb{S}})$. But if $\bb{S}' \supset \bb{S}$, then $\trop(h_{\bb{S}})|_{\osigma_{\bb{S}'}} = \trop(h_{\bb{S}'})$. Therefore we may define a continuous extension $\iota_{\bb{S}} : \osigma_{\bb{S}}^{\circ} \cup \osigma_{\bb{S}'}^{\circ} \rightarrow Y$ over points $p \in \osigma_{\bb{S}'}^{\circ}$ by $\iota_{\bb{S}}$ on $\osigma_{\bb{S}}^{\circ}$ and $\iota_{\bb{S}'}$ on $\osigma_{\bb{S}'}^{\circ}$. Indeed, if $\{p_j\}_{j = 1}^{\infty} \subset \sigma_{\bb{S}}^{\circ}$ is a sequence of divisorial points converging to $p \in \osigma_{\bb{S}'}^{\circ}$, then 
    $$ \lim_j \iota_{\bb{S}}(p_j) = \lim_j \trop(h_{\bb{S}})(p_j) = \trop(h_{\bb{S}})(p) = \trop(h_{\bb{S}'})(p) = \iota_{\bb{S}'}(p). $$
    We conclude by extending $\iota_{\bb{S}}$, for each fixed $\bb{S}$, over all $\sigma_{\bb{S}'}^{\circ}$ for all algebraic $\bb{S}' \supset \bb{S}$, and then gluing all of the resulting $\iota_{\bb{S}}$.

 \end{proof}
\end{theorem}

\begin{definition}\label{def:TropOfSubvariety} (\textit{Tropicalization of a Subvariety}) Fix $(\bb{Y}, \bb{D})$ a snc pair, and $\bb{Z} \subset \bb{Y}$ a closed formal subscheme, none of whose irreducible components are contained in $\bb{D}$. The \textit{tropicalization of $\bb{Z}$}, or $\trop(\bb{Z}) \subset |\Sigma_{(\bb{Y}, \bb{D})}|$, is the subset $r_{(\bb{Y}, \bb{D})}(\bb{Z}_{\eta} \cap (\bb{Y}_{\eta} \setminus \bb{D}_{\eta}))$. 
\end{definition}

\begin{proposition}\label{prop:PolyhedralTrop} Notation as in \cref{def:TropOfSubvariety}$, \trop(\bb{Z})$ is contained in a finite union, say $\mathcal{R} \subset |\Sigma_{(\bb{Y}, \bb{D})}|$, of closed, rational polyhedral cones, each of dimension at most $\dim \bb{Z}$, such that $\mathcal{R} \cap \Sk(\bb{Y}, \bb{D}) = \trop(\bb{Z}) \cap \Sk(\bb{Y}, \bb{D})$. 
\begin{proof} 
Suppose $\bb{Z}$ is irreducible. Because $\bb{Z} \not \subset \bb{D}$, we have $\codim_{\bb{Z}}(\bb{Z} \cap \bb{D}) = 1$, in fact $\bb{E} := \bb{Z} \cap \bb{D}$ a Cartier divisor on $\bb{Z}$. By the main result of \cite{Temkin_2012}, there exists a desingularization $\bb{Z}' \rightarrow \bb{Z}$. Let $\bb{E}' \subset \bb{Z}'$ denote the preimage of $\bb{E}$, and apply \cite{temkin2017functorialdesingularizationqboundaries} to $(\bb{Z}', \bb{E}')$ to obtain $\bb{Z}'' \rightarrow \bb{Z}$, with $\bb{E}''$ the preimage of $\bb{E}'$ such that $(\bb{Z}'', \bb{E}'')$ is an snc pair. The resulting proper map $\frak{f} : \bb{Z}'' \rightarrow \bb{Y}$ satisfies $\frak{f}^{-1}(\bb{D}) = \bb{E}''$, and so we can apply \cref{prop:FunctTrop}.

We first show $\trop(\bb{Z})$ is contained in the image $\mathcal{R}$ of $\trop(\frak{f})$. Fix $z \in \bb{Z}_{\eta} \cap (\bb{Y}_{\eta} \setminus \bb{D}_{\eta})$. The map $\frak{f}_{\eta}$ is surjective, so we can choose a lift $z''$ to $\bb{Z}''_{\eta}$. Now we have $\trop(\frak{f})(r_{(\bb{Z}'', \bb{E}'')}(z'')) = r_{(\bb{Y}, \bb{D})}(z)$, so indeed $r_{(\bb{Y}, \bb{D})}(z) \in \mathcal{R}$.

Now consider a point $p \in |\Sigma_{(\bb{Z}'', \bb{E}'')}|$. Let $\sigma_{\bb{S}''}$ the smallest cone of $\Sigma_{(\bb{Z}'', \bb{E}'')}$ containing $p$. Viewing carefully the definition of $\trop$ in the first paragraph of \cref{prop:FunctTrop}, we see that if $\trop(\frak{f})(p) \in \Sk(\bb{Y}, \bb{D})$, then the smallest stratum of $\bb{D}$ which contains the image of $\bb{S}''$ must be algebraic. Therefore $\bb{S}''$ is itself algebraic, because $\frak{f}$ is a proper map. In particular we have $p \in \Sk(\bb{Z}'', \bb{E}'')$. By \cref{theorem:SYZMain} there is a point $z'' \in \bb{Z}''_{\eta}$ such that $r_{(\bb{Z}'', \bb{E}'')}(z'') = p$ (e.g. $z''$ the instance of the point $p$ living in $\bb{Z}''_{\eta}$). Now $z = \frak{f}_{\eta}(z'')$ is a point of $\bb{Z}_{\eta}$ such that $r_{(\bb{Y}, \bb{D})}(z) = \trop(\frak{f})(p)$, completing the proof.   
\end{proof}    
\end{proposition}

\begin{proposition}\label{prop:TropProp} Fix $(\bb{Y}, \bb{D})$ snc, and $\bb{F}$ a Cartier divisor on $\bb{Y}$. The domain of the function $\bb{F}^{\trop}$ (see \cref{def:CartierFnct}) contains $\Sk(\bb{Y}, \bb{D})$, and the restriction $\bb{F}^{\trop}|_{\Sk(\bb{Y}, \bb{D})}$ is piecewise integral-linear, with finitely many maximal domains of linearity. In particular $\bb{F}^{\trop}|_{\Sk(\bb{Y}, \bb{D})}$ extends canonically to a piecewise integral-linear map defined on the closure of $\Sk(\bb{Y}, \bb{D})$ in $|\Sigma_{(\bb{Y}, \bb{D})}|$. 
\begin{proof}
    Fix $\bb{S}$ an algebraic stratum of $\bb{D}$. If $\bb{F} \not \supset \bb{S}$, then $\bb{F}^{\trop}|_{\sigma_{\bb{S}}^{\circ}} = 0$, so we will assume $\bb{F} \supset \bb{S}$.

    Fix notation as in the fourth paragraph of the proof of $\cref{theorem:SYZMain}$. From the argument of the theorem $\sigma_{\bb{S}}^{\circ} \hookrightarrow \Sk(\bb{Y}, \bb{D})$ is precisely the image under $(h_{\xi_{\bb{S}}})_{\eta}$ of $\Sk(\whbbY_{\xi_{\bb{S}}}, \whbbD_{\xi_{\bb{S}}})$, and so $\bb{F}^{\trop}|_{\sigma^{\circ}_{\bb{S}}}$ can be computed by pulling back a defining equation of $\bb{F}$ at $\xi_{\bb{S}}$ to $\whscrO_{\bb{Y}, \xi_{\bb{S}}}$, and evaluating on $\Sk(\whbbY_{\xi_{\bb{S}}}, \whbbD_{\xi_{\bb{S}}})$. The general statement then follows from the special case when $\bb{Y}$ is $\Spf(k[[y_1,...,y_r]])$ and $\bb{D}$ is the vanishing of $y_1...y_r$. For a power series $f = \Sigma_{\underline{\nu} \in \bb{N}^r} c_{\underline{\nu}} \cdot y_1^{\nu_1}...y_r^{\nu_r}$, here $c_{\underline{\nu}} \in k$ for each $\underline{\nu} \in \bb{N}^r$, we have 
    $$ -\log|f|\hspace{0.5mm} |_{\sigma_{\bb{S}}^{\circ}}(a_1[\bb{D}_1]^* + ... + a_r [\bb{D}_r]^*) = \inf_{c_{\underline{\nu}} \neq 0} \nu_1 a_1 + ... + \nu_r a_r.  $$
    That is, the restriction of $-\log |f|$ to $\sigma_{\bb{S}}^{\circ}$ is an infimum of (usually infinitely many) elements of $(\sigma_{\bb{S}}^{\vee})_{\bb{Z}} \subset T^*_{\bb{Z}} \sigma_{\bb{S}}$. But from this infinite collection of elements, one can choose a finite subcollection with identical infimum (a consequence of noetherianity of the polynomial ring). 
\end{proof}
\end{proposition}

\begin{construction}\label{app:ToricBlowupsOverk[[s]]} (\textit{Special Formal $k[[\pi]]$-schemes}) Fix $K^{\circ} = k[[\pi]]$, and set $|\pi| = \epsilon$ for some $0 < \epsilon < 1$ with $-\log \epsilon \in \bb{Q}$. We recall from \cref{ex:DefInOneVar} the duality of special formal $K^{\circ}$-schemes with special formal $k$-schemes defined over $\Spf(K^{\circ})$. Suppose $(\bb{Y}, \bb{D})$ is an snc pair over $k$ with a distinguished map $\bb{Y} \rightarrow \Spf(K^{\circ})$, with special fiber contained in $\bb{D}$. Let $\bb{H} \subset \bb{D}$ consist of the components flat over $K^{\circ}$, and $\bb{V} \subset \bb{D}$ the irreducible components of the special fiber. We call components of $\bb{H}$ \textit{horizontal} and components of $\bb{V}$ \textit{vertical}. Call by $Y_{\epsilon} \subset Y$ the $K$-analytic generic fiber, and similarly define $H_{\epsilon} \subset H =: \bb{H}_{\eta}$. The trop of the map to $\Spf(K^{\circ})$ is a canonical continuous $|\oSigma_{(\bb{Y}, \bb{D})}| \rightarrow \bb{R}_{\geq 0}$, restricting on $\Sigma$ to piecewise integral-linear. The fiber over $-\log \epsilon$ determines a complex of rational, strictly convex polyhedra, say $\mathscr{P}_{(\bb{Y}, \bb{D})}$, with compactification $\overline{\mathscr{P}}_{(\bb{Y}, \bb{D})}$, and an \textit{SYZ fibration} $Y_{\epsilon}  \rightarrow |\overline{\mathscr{P}}_{(\bb{Y}, \bb{D})}|$, restricting to $Y_{\epsilon} \setminus H_{\epsilon} \rightarrow |\mathscr{P}_{(\bb{Y}, \bb{D})}|$. The infinite rays of $\mathscr{P}_{(\bb{Y}, \bb{D})}$ are in 1-1 bijection with horizontal components, while the vertices are in 1-1 bijection with vertical components. 
\end{construction}

\begin{definition-remark}\label{def:IntLinStructure} (\textit{Canonical Integral-Linear Structure}) Here $\bb{S}$ is an algebraic stratum of $\bb{D}$. We say that $\Sk(\bb{Y}, \bb{D})$ is \textit{smooth near $\bb{S}$} if there is a neighborhood of $\sigma_{\bb{S}}^{\circ}$ in $\Sk(\bb{Y}, \bb{D})$ homeomorphic to a manifold of dimension $\dim(\bb{Y})$. Note this is the same as requiring that $*\sigma_{\bb{S}}$ is a manifold. If $\bb{S}$ is a $0$-stratum, then $\Sk(\bb{Y}, \bb{D})$ is always smooth near $\bb{S}$, since $\sigma_{\bb{S}}^{\circ}$ itself is an open manifold in $\Sk(\bb{Y}, \bb{D})$ of dimension $\dim(\bb{Y})$.

If $\bb{S}$ is a complete 1-stratum, then smoothness near $\bb{S}$ is the same as $\bb{S}$ containing exactly two $0$-strata of $\bb{D}$. In this case, just as when $\bb{S}$ is a $0$-stratum, $*\sigma_{\bb{S}}$ has a canonical structure of integral-linear manifold. Fix $\bb{D}_1 \cap ... \cap \bb{D}_{n - 1} = \bb{S}$ and $\bb{D}_{n}, \bb{D}_{n + 1}$ the irreducible components intersecting $\bb{S}$ transversely. The linear functions on $*\sigma_{\bb{S}}$ are given by formal sums $a_1 [\bb{D}_1] + ... + a_{n + 1} [\bb{D}_{n + 1}]$, here $a_1,...,a_{n + 1} \in \bb{R}$, such that $\sum_{i = 1}^{n + 1} a_i \deg\big( \mathscr{O}(\bb{D}_i)|_{\bb{S}} \big) = 0$. The associated transition function $T\sigma_{\bb{S} \cap \bb{D}_n} \rightarrow T\sigma_{\bb{S} \cap \bb{D}_{n + 1}}$ is determined by the assignments $[\bb{D}_i]^* \mapsto [\bb{D}_i]^*$ for $i = 1,...,n - 1$, and
$$ [\bb{D}_n]^* \mapsto -[\bb{D}_{n + 1}]^* - \sum_{i = 1}^{n - 1} (\bb{D}_i \cdot \bb{S}) [\bb{D}_i]^*. $$
Let $M = M_{\bb{S} \cap \bb{D}_{n + 1}}$. The natural inclusion $\sigma_{\bb{S} \cap \bb{D}_{n + 1}} \hookrightarrow M_{\bb{R}}$ now induces one $\overline{* \sigma_{\bb{S}}} \hookrightarrow M_{\bb{R}}$, which we see is strictly convex precisely when $\bb{D}_i \cdot \bb{S} < 0$ for each $i < n$. Supposing $\bb{S}$ is a rational curve, then Proposition 5.4 of \cite{Nicaise_Xu_Yu_The_non-archimedean_SYZ_fibration} shows that this convexity implies the formal completion of $\bb{Y}$ along $\bb{S}$ is isomorphic to the completion of the projective over affine toric variety $\TV(\overline{* \sigma_{\bb{S}}} \subset M_{\bb{R}})$ along its distinguished boundary 1-stratum (note here $\overline{* \sigma_{\bb{S}}}$ is given the fan structure determined by the two cones $\sigma_{\bb{S} \cap \bb{D}_n} \cup \sigma_{\bb{S} \cap \bb{D}_{n + 1}}),$ which is isomorphic to $\bb{P}^1$ and contains exactly two $0$-strata.

Borrowing from [loc. cit.], we say $\bb{Y}$ is \textit{log Calabi-Yau} along $\bb{S}$ if $\Sk(\bb{Y}, \bb{D})$ is smooth near $\bb{S}$, and $\bb{S} \cong \bb{P}^1_k$. Assume the setting of the previous paragraph, with $\bb{Y}$ log Calabi-Yau along $\bb{S}$. Suppose further $\mathcal{G} \subset *\sigma_{\bb{S}}$ is a convex polytope in the integral-linear structure. For each $k \in \bb{N}$, let $\sigma_k \subset M_{\bb{R}}$ denote the closed subcone of $\sigma_{\bb{S} \cap \bb{D}_{n}}$ generated (as a monoid over $\bb{R}_{\geq 0}$) by the elements $[\bb{D}_1]^*,...,[\bb{D}_{n - 1}]^*,$ and $v_k := k \cdot [\bb{D}_{n}]^* + \sum_{i = 1}^{n - 1} [\bb{D}_i]^*$. We can write $* \sigma_{\bb{S}} \subset M_{\bb{R}}$ as the union $\sigma_{\bb{S} \cap \bb{D}_{n + 1}}^{\circ} \cup \big(\bigcup_{k = 1}^{\infty} \sigma_k \big)$. Now we have $\mathcal{G} \subset \sigma_{\bb{S} \cap \bb{D}_{n + 1}}^* \cup \sigma_k$ for some $k \in \bb{N}$.

Recall the $*$-subdivision $*(\Sigma_{(\bb{Y}, \bb{D})}, \bb{R}_{\geq 0} \cdot v_k)$ is the smallest simplicial subdivision of $\Sigma_{(\bb{Y}, \bb{D})}$ containing the ray $\bb{R}_{\geq 0} \cdot v_k$. The cone $\sigma_k$ appears as an element of the corresponding simplicial complex, and because $\sigma_k$ is already regular, we can take $\Sigma' \rightarrow  *(\Sigma_{(\bb{Y}, \bb{D})}, \bb{R}_{\geq 0} \cdot v_k)$ a further subdivision such that $\sigma_{\bb{S} \cap \bb{D}_{n + 1}} , \sigma_k \in \Sigma'$ remain maximal cones, and $\Sigma'$ a regular fan so that $\Sigma' \rightarrow \Sigma_{(\bb{Y}, \bb{D})}$ is a projective subdivision, isomorphic over the complement of $*\sigma_{\bb{S}}$. Letting $\bb{Y}' \rightarrow \bb{Y}$ the corresponding toric blowup, $\bb{D}_i'$ the strict transform of $\bb{D}_i$ for each $i$, and $\bb{S}'$ the strict transform of $\bb{S}$, we have $\bb{D}_i' \cdot \bb{S}' < 0$ for every $i < n$. This setup establishes \textit{Kontsevich-Soibelman smoothness} of the SYZ fibration in \cref{prop:GluingEnds}(a).
\end{definition-remark}

\begin{lemma}\label{lem:PicOfFCLogCY1Strat} Fix notation as in the last two paragraphs of \cref{def:IntLinStructure}, and set $\mathscr{Y}$ the formal completion of $\bb{Y}'$ along $\bb{S}'$, so that in particular we have $\scrY \cong \whTV(\overline{*\sigma_{\bb{S}'}} \subset M_{\bb{R}})$, with the formal completion of the latter taken along the distinguished boundary $\bb{P}^1$ (see second paragraph of [loc. cit.]). Then the restriction map $\Pic(\scrY) \rightarrow \Pic(\scrY_s) \cong \bb{Z}$ is an isomorphism. 
\begin{proof}
    Let $T = \TV(\overline{*\sigma_{\bb{S}'}})$, with toric boundary divisors $E_1,...,E_{n + 1}$. Call by $C$ the distinguished $\bb{P}^1$, the intersection of $E_1,...,E_{n - 1}$, and let $\mathscr{I}$ denote the ideal sheaf on $T$ cutting out $C$, so that in particular we have $\mathcal{N}_{C / T}^* \cong \mathscr{I} / \mathscr{I}^2 \cong \oplus_{i = 1}^{n - 1} \mathscr{O}(-E_i)$ a sum of positive line bundles. Since $C$ is a local complete intersection in $T$, we have moreover $S^m(\mathcal{N}^*_{C/T}) \cong \mathscr{I}^m/ \mathscr{I}^{m + 1}$ (see Theorem 8.21A(e) of \cite{Hartshorne_Algebraic_geometry}) for any $m \in \bb{N}$, which is likewise a sum of positive line bundles, using the basic identities of symmetric powers. Therefore $H^1(T, \mathscr{I}^m / \mathscr{I}^{m + 1}) = 0$ for every $m \in \bb{N}$.

    Now for each $m$, let $\mathscr{Y}_m$ denote the vanishing of $\mathscr{I}^m$ in $T$. We have $\Pic(\mathscr{Y}) \cong \underset{\leftarrow}{\lim} \hspace{0.5mm} \Pic(\mathscr{Y}_m)$ [loc. cit., Ex. II.9.6]. Assume inductively that $\Pic(\mathscr{Y}_m) \rightarrow \Pic(C)$ is an isomorphism for some fixed $m$. We show $\Pic(\mathscr{Y}_{m + 1}) \rightarrow \Pic(\mathscr{Y}_m)$ is an isomorphism, from which the statement follows directly. On $\mathscr{Y}_{m + 1}$, $\mathscr{Y}_m$ is cut out by the nilpotent ideal sheaf $\mathscr{I}^m / \mathscr{I}^{m + 1}$, so we conclude using the long exact sequence of [loc. cit., Ex. III.4.6].

\end{proof}
\end{lemma}

\begin{proposition}\label{prop:GluingEnds} (\textit{Gluing Toric Ends}) Fix $(\bb{Y}, \bb{D})$ snc and $\bb{S}$ a stratum of $\bb{D}$ such that either $\bb{S}$ is a $0$-stratum or else $\bb{Y}$ is log Calabi-Yau along $\bb{S}$. Then, fix $\mathscr{G} \hookrightarrow \Sigma_{(\bb{Y}, \bb{D})}$ an embedding of a realizable polyhedral complex such that $\mathcal{G} := |\mathscr{G}| \subset *\sigma_{\bb{S}}$ is a convex polytope. The following hold: 
\begin{enumerate}[label=(\alph*)]
    \item There exists an open, convex cone $\mathcal{C} \subset *\sigma_{\bb{S}}$ containing $\mathcal{G}$, a lattice $M$, a piecewise linear embedding $\iota : \overline{*\sigma_{\bb{S}}} \hookrightarrow M_{\bb{R}}$, and an analytic isomorphism $r_{(\bb{Y}, \bb{D})}^{-1}(\mathcal{C}) \overset{\cong}{\longrightarrow} r_M^{-1}(\iota(\mathcal{C}))$ that commutes with the retractions to $\mathcal{C}$. 
    \item Let $K^{\circ} = k[[\pi]]$. Then there is a further subdivision $\mathscr{G}_v \rightarrow \mathscr{G}$, an admissible formal blowup $\bb{Y}_v  \rightarrow \bb{Y}_{K^{\circ}}$, and a Zariski open subscheme $\bb{G} \subset \bb{Y}_v$ isomorphic to $\bb{T}_M\{\mathscr{G}_v^{\vee}\}$, such that $G := \bb{G}_{\eta} = r_{(\bb{Y}, \bb{D})}^{-1}(\mathcal{G})$. 
    \item Fix $\mathscr{I} \leq \mathscr{O}_{\bb{Y}}$ a defining ideal sheaf, $\mathscr{I}_v$ the inverse image sheaf in $\bb{Y}_v$. For all $0 < r < 1$ sufficiently close to $1$ such that $r^n = |\pi|$ for some $n \in \bb{N}$, we have $\bb{Y}_v\{r^{-1} \mathscr{I}_v\} \rightarrow \bb{Y}_v$ an isomorphism over $\bb{G}$. 
    \item Fix an analytic isomorphism $\phi : G \rightarrow T_M\{\iota(\mathcal{G})^{\vee}\}$ commuting with the retractions to $\mathcal{G}$, e.g. the restriction of the isomorphism from (a). Take $\bb{X}$ a smooth, projective $T_M$-toric variety over $k$, requiring that the restriction of the piecewise linear structure of the defining fan to $\iota(\overline{*\sigma_{\bb{S}}})$ is isomorphic to the restriction of the $\Sigma_{(\bb{Y}, \bb{D})}$ structure to $\overline{*\sigma_{\bb{S}}}$. Let $X = \bb{X}_{\eta}$ the analytification, fix $0 < r < 1$ close to 1, and define a gluing of $k$-analytic spaces $Z = Y\{r^{-1}\mathscr{I}_v\} \cup_{\phi} X$. There exists then an admissible formal model $\bb{X}_v$ over $K^{\circ}$ of $X_K$, with an open immersion $\bb{T}_M\{\mathscr{G}_v^{\vee}\} \hookrightarrow \bb{X}_v$ whose generic fiber is $T_M\{\iota(\mathcal{G})^{\vee}\} \subset T_M$, and an isomorphism $\Phi : \bb{G} \rightarrow \bb{T}_M\{\mathscr{G}_v^{\vee}\}$ such that $\Phi_{\eta} = \phi$, and so that $\bb{Z} = \bb{Y}_v\{r^{-1} \mathscr{I}_v\} \cup_{\Phi} \bb{X}_v$ satisfies $\bb{Z}_{\eta} = Z_K$. 
    \item Fix notation as in (d). Any line bundle on $\bb{Y}$ (resp. $\bb{Y}_s$) extends to a line bundle on $\bb{Z}$ (resp. $\bb{Z}_s$). 
    \item Suppose $\bb{Y}$ is quasi-projective over affine, and fix notation as in (d). Then there exists a line bundle on $\bb{Z}$ whose restriction to each of $\bb{Y}_v\{r^{-1}\mathscr{I}_v\}$ or $\bb{X}_{v}$ is relatively ample over $K^{\circ}$. 
    
\end{enumerate}
\begin{proof}
       The final paragraph of \cref{def:IntLinStructure} implies (a). For the rest of the argument, we apply frequently the duality of special formal schemes over $k[[s]]$ (see \cref{ex:DefInOneVar} and \cref{app:ToricBlowupsOverk[[s]]}), which allows us to consider our various $\Sigma$ as polyhedral complexes instead of fans, and to make finite subdivisions of the $\Sigma$ using toric blowups.

       For (b), we have $\bb{Y}' \rightarrow \bb{Y}$ as in \cref{def:IntLinStructure}. Let $\Sigma_h = \Sigma'$ and $\bb{Y}_h = \bb{Y}'_{K^{\circ}}$, $h$ standing for horizontal. Fix using \cref{prop:CofinalStarSubdiv} a regular projective subdivision $\Sigma_v \rightarrow \Sigma_{(\bb{Y}, \bb{D})}$ containing a subdivision $\mathscr{G}_v$ of $\mathscr{G}$, such that rec$(\Sigma_v) \equiv \rec(\Sigma)$. Then apply again \cref{prop:CofinalStarSubdiv} to pick $\whSigma \rightarrow \Sigma_v$ a regular projective subdivision which also refines $\Sigma_h$, and which restricts over $\mathscr{G}_v$ to an isomorphism of subcomplexes (ie. $\mathscr{G}_v$ sits unchanged in $\whSigma$). Set $\bb{Y}_v$ the admissible blowup of $\bb{Y}_{K^{\circ}}$ corresponding to $\Sigma_v \rightarrow \Sigma$, $v$ standing for vertical. Then set $\whbbY$ the formal blowup of $\bb{Y}_{K^{\circ}}$ corresponding to $\whSigma \rightarrow \Sigma$. By construction we have a commuting square of formal blowups: 
       \begin{center}
           \begin{tikzcd}
               \whbbY \arrow[r] \arrow[d] & \bb{Y}_v \arrow[d] \\ 
               \bb{Y}_h \arrow[r] & \bb{Y}_{K^{\circ}}
           \end{tikzcd}
       \end{center}
       Let $\bb{G} \subset \whbbY$ the Zariski open subscheme whose generic fiber is $r^{-1}_{(\bb{Y}, \bb{D})}(\mathcal{G})$. The top horizontal map embeds $\bb{G}$ also as a Zariski open subscheme of $\bb{Y}_v$, while the left vertical map shows $\bb{G}$ is isomorphic to $T_M\{\mathscr{G}_v^{\vee}\}$. Indeed, $\bb{Y}_h$ completed along $\bb{S}'$ in the central fiber (over $K^{\circ}$) $\bb{Y}'$ is isomorphic to the base extension to $K^{\circ}$ of $\whTV(\sigma_k \cup \sigma_{\bb{S} \cap \bb{D}_{n + 1}})$, the completion of the toric variety along its distinguished 1-stratum. But $\bb{G}$ is collapsed in $\bb{Y}_h$ to $\bb{S}'$, so is isomorphic to the Zariski open subset (of the toric blowup of $\whTV(\sigma_k \cup \sigma_{\bb{S} \cap \bb{D}_{n + 1}})$ induced by restricting $\whSigma \rightarrow \Sigma'$) corresponding to $\mathscr{G}_v \hookrightarrow \whSigma$. That is exactly $\bb{T}_M\{\mathscr{G}_v^{\vee}\}$.

       For (c), take $\mathscr{I}_h$ the inverse image sheaf of $\mathscr{I}$ in $\bb{Y}_h$ and $\whscrI$ in $\whbbY$. Now $\whscrI|_{\bb{G}}$ is an ideal of definition of $\bb{G}$, so also is $\mathscr{I}_v|_{\bb{G}} \cong \whscrI|_{\bb{G}}$, which shows $(\pi) + (\mathscr{I}_v|_{\bb{G}})^n$ is generated by $\pi$ (recall $r$ is chosen close to 1, or $n$ near infinity). Referring to \cref{def:Trunc}, we see that $\bb{Y}_v\{r^{-1} \mathscr{I}_v\} \rightarrow \bb{Y}_v$ restricts to an isomorphism over $\bb{G}$.

       From the local isomorphism of fans near $\sigma_{\bb{S}}$ we obtain $\bb{X}_v, \bb{X}_h,$ and $\whbbX$ from $\bb{X}_{K^{\circ}}$ via identical sequences of $*$-subdivisions as those used to obtain $\bb{Y}_v,\bb{Y}_h$ and $\whbbY$ from $\bb{Y}_{K^{\circ}}$. In particular we have the commuting square 
        \begin{center}
           \begin{tikzcd}
               \whbbX \arrow[r] \arrow[d] & \bb{X}_v \arrow[d] \\ 
               \bb{X}_h \arrow[r] & \bb{X}_{K^{\circ}}
           \end{tikzcd}
        \end{center}
        and $\bb{T}_M\{\mathscr{G}_v^{\vee}\} \hookrightarrow \bb{X}_v$ Zariski open, whose generic fiber gives an affinoid subdomain embedding $T_M\{\iota(\mathcal{G})^{\vee}\} \hookrightarrow X_K$. By \cref{prop:SYZAuts}, the given $\phi : G \rightarrow T_M\{\iota(\mathcal{G})^{\vee}\}$ is the generic fiber of a uniquely determined isomorphism of the canonical model of $G$ with $\bb{T}_M\{\iota(\mathcal{G})^{\vee}\}$. Performing the admissible snc blowup corresponding to the sequence of $*$-subdivisions $\mathscr{G}_v \rightarrow \mathcal{G}$, we obtain an isomorphism $\Phi : \bb{G} \rightarrow \bb{T}_M\{\mathscr{G}_v^{\vee}\}$ such that $\Phi_{\eta} = \phi$, proving (d).

        For (e), fix $\bb{L}$ on $\bb{Y}$. We observe $\bb{Y}_h$ (resp. $\bb{X}_h$) is a constant base extension of an snc pair over $k$, log Calabi-Yau along the distinguished 1-stratum corresponding to the codimension 1 cone $\sigma_{\bb{S}'}$ of $\Sigma_h$ (resp. $\iota(\sigma_{\bb{S}'})$). Pulling back $\bb{L}$ to get $\bb{L}_h$ on $\bb{Y}_{h}$, and then pulling back again to $\whbbL$ on $\whbbY$, by \cref{lem:PicOfFCLogCY1Strat} we find that $\whbbL|_{\bb{G}}$ is determined by the degree of the restriction of $\bb{L}_h$ to $\bb{S}'$. We can find a line bundle $\bb{L}_{\bb{X}}$ on $\bb{X}$ which, after pulling back to $\bb{X}_h$, has identical degree along the distinguished 1-stratum. Indeed, from standard toric geometry this is equivalent to choosing a piecewise integral-linear function on $M_{\bb{R}}$, linear on each cone of the defining fan of $\bb{X}$, whose kink along the distinguished codimension 1 cone is the desired degree. Let $\bb{L}_{\bb{X}, v}$ be the restriction of $\bb{L}_{\bb{X}}$ to $\bb{X}_v$. We have the isomorphisms
        $$ \bb{L}_{\bb{X}, v}|_{\bb{T}_M\{\mathscr{G}_v^{\vee}\}} \cong \whbbL_{\bb{X}}|_{\bb{T}_M\{\mathscr{G}_v^{\vee}\}} \cong \whbbL|_{\bb{G}} \cong \bb{L}_v|_{\bb{G}},$$
        the first and last induced by $\whbbX \rightarrow \bb{X}_v$ and $\whbbY \rightarrow \bb{Y}_v$, the middle because $\bb{G}$ and $\bb{T}_M\{\mathscr{G}_v^{\vee}\}$ lie over the distinguished 1-strata (recall these are isomorphic to $\bb{P}^1$) of $\bb{Y}_h$ and $\bb{X}_h$, and are therefore determined by the degrees of the restrictions to these 1-strata. This proves the part of the statement concerning extension of a line bundle on $\bb{Y}$. To extend a line bundle on $\bb{Y}_s$ to $\bb{Z}_s$, go through the argument, adding a subscript $s$ wherever relevant.

        Finally, for (f) we note that by definition $b : \bb{Y}_v\{r^{-1}\mathscr{I}_v\} \rightarrow \bb{Y}_{K^{\circ}}$ is quasi-projective, with relatively ample $\bb{E}$ a sum of exceptional divisors. Fix $\bb{L}$ on $\bb{Y}$ relatively ample over affine. Then $b^*\bb{L} \otimes \mathscr{O}(\bb{E})$ is relatively ample over $K^{\circ}$. Choose similarly on $\bb{X}_v$ a relatively ample sum $\bb{E}_{\bb{X}}$ of exceptional divisors of the toric blowup $b_{\bb{X}} : \bb{X}_v \rightarrow \bb{X}_{K^{\circ}}$, so that the restriction of $\bb{E}_{\bb{X}}$ to $\bb{T}_M\{\mathscr{G}_v^{\vee}\}$ is identified with $\bb{E}|_{\bb{G}}$ by $\Phi$, as well as $\bb{L}_{\bb{X}}$ ample on $\bb{X}$ so that the degrees (of the pullbacks to $\bb{Y}_h$ and $\bb{X}_h$, resp.) on restriction to the distinguished 1-strata are identical\textbf{---}for this, we can just pick any two amples, and take multiples so the degrees are the same. Then also $b_{\bb{X}}^* \bb{L}_{\bb{X}} \otimes \mathscr{O}(\bb{E}_{\bb{X}})$ is rel. ample over $K^{\circ}$. By (e) $b^* \bb{L}$ and $b_{\bb{X}}^* \bb{L}_{\bb{X}}$ glue to a line bundle on $\bb{Z}$, and by construction so do $\mathscr{O}(\bb{E})$ and $\mathscr{O}(\bb{E}_{\bb{X}})$. The gluing of $b^* \bb{L} \otimes \mathscr{O}(\bb{E})$ and $b_{\bb{X}}^* \bb{L}_{\bb{X}} \otimes \mathscr{O}(\bb{E}_{\bb{X}})$ now gives the desired line bundle on $\bb{Z}$. 
\end{proof}
\end{proposition}

For what follows, assume $k$ is an algebraically closed field of characteristic $0$, and $K$ a complete discrete valuation field extending $k$.

\begin{definition-remark}\label{def:SYZ} Consider $(\bb{Y}, \bb{D})$ snc pair of special formal schemes over $k$, and $L/k$ an arbitrary valuation extension. We define the \textit{SYZ fibration} $r_{(\bb{Y}, \bb{D})}^L : Y_L \rightarrow  |\oSigma_{(\bb{Y}, \bb{D})}|$ as the composition of the projection $\pi_L : Y_L \rightarrow Y$ with $r_{(\bb{Y}, \bb{D})}$. Since $k$ is algebraically closed, by [Poineau] there is a canonical section $Y \rightarrow Y_L$ of $\pi_L$, and therefore from \cref{theorem:SYZMain} a canonical embedding $\oSk(\bb{Y}, \bb{D}) \hookrightarrow Y_L$. Note by definition, for any polyhedral subset $\mathcal{P} \subset |\Sigma|$, $(r^L_{(\bb{Y}, \bb{D})})^{-1}(\mathcal{P})$ is canonically isomorphic to $r_{(\bb{Y}, \bb{D})}^{-1}(\mathcal{P})_L$. Therefore the exact statement of \cref{prop:GluingEnds}(a) holds as well with $r^L_{(\bb{Y}, \bb{D})}$ replacing $r_{(\bb{Y}, \bb{D})}$.   
\end{definition-remark}

\begin{proposition}\label{prop:SkelOfForm} Fix $\omega$ a volume form on $U$, with at worst simple poles along $\bb{D}$. Then $\Sk(\omega)$ is the subset of $\Sk(\bb{Y}, \bb{D})$ made up of all $\sigma_{\bb{S}}^{\circ}$ such that $\omega$ has a simple pole along each $\bb{D}_i$ that contains $\bb{S}$, provided such a stratum $\bb{S}$ exists. For $L/k$ an arbitrary valuation extension, $\Sk(\omega_L) \subset U_L$ is precisely the canonical lift of $\Sk(\omega)$.  
\begin{proof}
    Fix $p \in \bb{Y}_s$ and $\bb{S}$ the minimal stratum of $\bb{D}$ containing $p$. Fix also $\bb{W} \subset \bb{Y}$ a Zariski open neighborhood of $p$ with an \'etale chart $\psi : \bb{W} \rightarrow \bb{Z}$, here $\bb{Z}$ the formal completion of $k[z_1,...,z_n]$ along some ideal. Suppose WLOG that $\bb{D}_1,...,\bb{D}_r,$ $r \leq n$, are the irreducible components of $\bb{D}$ which contain $\bb{S}$, and that $\bb{D}_i$ is cut out analytically locally near $p$ by $\psi^*z_i$. Let $\bb{E} \subset \bb{Z}$ denote the Cartier divisor $z_1...z_r = 0$.

    Fix some $u \in U$ with center $p \in \bb{Y}_s$, and let $\| \cdot \|_y$, $\| \cdot \|_{\psi_{\eta}(y)}$ denote the Temkin K\"ahler seminorms on $\Omega^n_{U/k, y}$, $\Omega^n_{Z/k, \psi_{\eta}(y)}$ respectively (see \cite{Temkin_Metrization_of_differential_pluriforms} for the definition). Set $\zeta := \frac{dz_1}{z_1} \wedge ... \wedge \frac{dz_n}{z_n}$. We know that $\omega = h \psi^* \zeta$ for some regular function $h$ defined at $p$, which satisfies $|h(u)| \leq 1$ automatically, so that the inequalities $\| \omega \|_u \leq |h(u)| \| \psi^* \zeta\|_u \leq |h(u)| \|\zeta \|_{\psi_{\eta}(u)} \leq 1$ hold always.

    We go case by case. If $\bb{S}$ is non-algebraic, then the center of $\psi_{\eta}(u)$ is not a generic point of a stratum of $\bb{E}$, so $\| \zeta \|_{\psi_{\eta}(u)} < 1$ from the toric case, and hence $\|\omega \|_u < 1$. If $\bb{S}$ is algebraic, but, say, $\bb{D}_1$ is not a pole of $\omega$, then we can write $h = (\psi^* z_1) g$ for another regular function $g$ defined at $p$. But $|(\psi^* z_1)(u) | < 1$, because the center of $u$ is contained in $\bb{D}_1$, and $|g| \leq 1$, so we have also in this case $\| \omega \|_u < 1$. Finally, we consider the case $\bb{S}$ algebraic and $\omega$ having a pole along each $\bb{D}_i$, $i \in [r]$. In this case $h$ is invertible at $p$, so $|h(u)| = 1$, and the $\psi^*z_i$ give a family of monomial parameters at $u$, so that $\| \omega \|_u = |h(u)| \| \psi^* \zeta \|_u$ [loc. cit., Corollary 8.1.3].

    For the final claim, we apply the pullback inequality [loc. cit., 6.3.2] for base field extensions to first show $\Sk(\omega_L)$ must live over $\Sk(\omega)$. Then, using the $\psi$ above, we can find an affinoid domain $G \subset U$, together with quasi-\'etale $G \rightarrow \bb{G}^n_{m}$, which we also call $\psi$ (in fact it is the restriction of $\psi_{\eta}$), such that $\psi(u) \in \Sk(\bb{G}^n_{m})$. As above $\zeta$ is a volume form on $\bb{G}^n_m$. We base extend, giving quasi-\'etale $\psi_L : G_L \rightarrow \bb{G}_{m,L}^n$, and the commuting square
    \begin{center}
        \begin{tikzcd}
            G_L \arrow[d,"\pi_L"] \arrow[r,"\psi_L"] & \bb{G}_{m, L}^n \arrow[d, "\pi_L"] \\ 
            G \arrow[r, "\psi"] & \bb{G}_{m}^n
        \end{tikzcd}
    \end{center}
    The only point of $G_L$ living both over $u$ and also over the canonical lift of $\psi(u)$ in $\bb{G}_{m,L}^n$ is the canonical lift $u_L$ of $u$ to $G_L$, and the K\"ahler seminorm $\| \omega_L \|_{u_L}$ equals 1, since 
    $$  1 = \| \zeta_L \|_{\psi_L(u_L)}= \| \psi_L^* \zeta_L \|_{u_L} \leq |h_L^{-1}(u_L)| \|\omega_L\|_{u_L} = \| \omega_L \|_{u_L} \leq 1, $$
    the second equality by combining Remark 6.1.10(i) and Theorem 5.6.4(i) of [loc. cit.]. Meanwhile at any other point $u'$ of $G_L$ living over $u$, 
    $$\|\omega_L \|_{u'} \leq |h(u')| \|\psi_L^* \zeta_L \|_{u'} = \|\zeta_L\|_{\psi_L(u')} < 1,$$
    because $\psi_L(u') \not \in \Sk(\bb{G}^n_{m,L})$.
    
\end{proof}
\end{proposition}

\begin{proposition}\label{prop:trop} Let $L/K$ an arbitrary valuation extension with tropical value group $\Gamma := -\log|L^*| \subset \bb{R}$. We consider $(\frak{X}, \frak{H})$ a strictly semistable formal model over $L^{\circ}$, in the sense of \cite{Gubler_Skeletons_and_tropicalizations}. We also fix $(\bb{Y}, \bb{D})$ snc over $k$, along with a map $\frak{f} : \frak{X} \rightarrow \bb{Y}$ of formal schemes, which we assume satisfies $\frak{f}_{\eta}^{-1}(\bb{D}_{\eta}) \subset \frak{H}_{\eta}$ (see \cref{cor:FormalModels} for the definition of $\frak{f}_{\eta}$ in this context). Then there is a canonical continuous $\trop(\frak{f}) : \oSk(\frak{X}, \frak{H}) \rightarrow |\oSigma_{(\bb{Y}, \bb{D})}|$ such that  
$$ r_{(\bb{Y}, \bb{D})} \circ \frak{f}_{\eta} = \trop(\frak{f}) \circ r_{(\frak{X}, \frak{H})}. $$
Furthermore, $\trop(\frak{f})$ restricts to a piecewise integral, $\Gamma$-affine $\Sk(\frak{X}, \frak{H}) \rightarrow |\Sigma_{(\bb{Y}, \bb{D})} |$.
\begin{proof}
    Fix a stratum $\frak{S}$ of $\frak{X}_s + \frak{H}$ [loc. cit., 3.15], and $\frak{U} \subset \frak{X}$ a \textit{building block} for $\frak{S}$ (see [loc. cit., Prop. 4.1]; building block is the analog of basic chart in this setting). In particular we have an \'etale map 
    $$ \frak{U} \rightarrow \Spf(L^{\circ}\{x_0,...,x_d\}/(x_0...x_r - \pi)),$$
    here $\pi \in L^{\circ \circ}$ and $0 \leq r \leq d$, and $x_{r + 1},...,x_{r + s}$ cut out the components of $\frak{H}$ containing $\frak{S}$, implying in particular $\Sk(\frak{U}, \frak{H}|_{\frak{U}}) = \Delta(r, \pi) \times \bb{R}_{\geq 0}^s$ [loc. cit., Ex. 3.10]. and a canonical retraction $\frak{U}_{\eta} \setminus \frak{H}_{\eta} \rightarrow \Sk(\frak{U}, \frak{H}|_{\frak{U}})$ determined by evaluating the functions $-\log|x_i|$, $i = 0,..., r+s$.

    Let $\bb{S}$ denote the smallest stratum of $\bb{D}$ such that $\frak{S} \subset \frak{f}^{-1}(\bb{S})$. Up to shrinking $\frak{U}$, we can fix a Zariski open $\bb{U} \subset \bb{Y}$, the domain of a basic chart for $\bb{S}$, such that $\frak{U} \subset \frak{f}^{-1}(\bb{U})$. Fixing $\bb{D}_1,...,\bb{D}_{\ell}$ the irreducible components of $\bb{D}$ which contain $\bb{S}$, and $y_1,...,y_{\ell}$ a set of defining functions on $\bb{U}$, we can apply [loc. cit., Prop. 5.2] to each pullback $\frak{f}^* y_i$, giving $-\log |\frak{f}^* y_i| : \frak{U}_{\eta} \setminus \frak{H}_{\eta} \rightarrow \bb{R}_{\geq 0}$ an integral $\Gamma$-affine function which factors through the retraction to the skeleton. Then we obtain integral $\Gamma$-affine 
    $$(-\log|\frak{f}^* y_1|,...,-\log|\frak{f}^* y_{\ell}|) : \Sk(\frak{U}, \frak{H}|_{\frak{U}}) \rightarrow \sigma_{\bb{S}} \subset \Sk(\bb{Y}, \bb{D}), $$
    which is now clearly compatible with the retractions on either side. Arguing as in [loc. cit., Prop. 5.5], this map is independent of the choice of building block $\frak{U}$ for $\frak{S}$, and for $\frak{S}' \supset \frak{S}$ with building block $\frak{U}'$, the map determined by $\frak{S}$ restricts to the map determined by $\frak{S}'$, and so they all glue to give $\Sk(\frak{X}, \frak{H}) \rightarrow |\Sigma_{(\bb{Y}, \bb{D})}|$.

    Extending over $\oSk(\frak{X}, \frak{H})$ is accomplished by arguing the same over each stratum of the Cartier $\frak{H}$, regarded as strictly semistable in its own right, with horizontal divisor the restriction of the union of components of $\frak{H}$ not containing the stratum in question.  
\end{proof}

\end{proposition}

\begin{proposition}\label{prop:balancing} (\textit{Tropical Balancing}) Fix $L/K$ an arbitrary valuation extension, and consider $(\frak{B}, \frak{P})$ a strictly semistable pointed curve over $L^{\circ}$ with a map $\frak{f} : \frak{B} \rightarrow \bb{Y}$ such that $\frak{f}_{\eta}^{-1}(\bb{D}_{\eta}) \subset \frak{P}_{\eta}$. A finite vertex $v$ of $\Sk(\frak{B}, \frak{P})$ determines a proper component $C_v$ of $\frak{B}_s$. Fix $\bb{E}$ a Cartier divisor on $\bb{Y}$ such that $\frak{f}_{\eta}^{-1}(\bb{E}_{\eta})$ is a finite set of points of $\frak{B}_{\eta}$. Then
$$ \sum_{v \in e} d_e \trop(\frak{f})^* \bb{E}^{\trop} = \deg \big((\frak{f}_s|_{C_v})^* \mathscr{O}(\bb{E})\big),$$
here the sum is over all edges $e$ of $\Sk(\frak{B}, \frak{P})$ containing $v$, each oriented so the positive direction points away from $v$, and $\trop(\frak{f})^* \bb{E}^{\trop}$ the pullback to $\Sk(\frak{B}, \frak{P})$ of the function $\bb{E}^{\trop}$ (see \cref{def:CartierFnct}).

Suppose $\Sk(\bb{Y}, \bb{D})$ is smooth near $\bb{S}$ (see \cref{def:IntLinStructure}), an irreducible stratum of $\bb{D}$ of dimension at most $1$, and suppose $\trop(\frak{f})(v) \in *\sigma_{\bb{S}}$. Then $\trop(\frak{f})$ is \textit{balanced at $v$}, meaning that $\sum_{v \in e} d_e \trop(\frak{f}) = 0$ in the canonical integral-linear structure of $*\sigma_{\bb{S}}$. 
\begin{proof}
    We have chosen to state the proposition with reference to $(\bb{Y}, \bb{D})$ as it is convenient for our application. But the more general statement concerns only a line bundle on $\frak{B}$ and a rational section of the line bundle, in which case the first statement is classical (see e.g. \cite{thuillier2005theorie}).

    For the second statement, we regard $\sum_{v \in e} d_e\trop(\frak{f})$ as a tangent vector of $*\sigma_{\bb{S}}$ at $\trop(\frak{f})(v)$. The linear functionals are the sums $a_1[\bb{D}_1] + ... + a_n [\bb{D}_{n + 1}]$ such that $\sum_{i = 1}^{n + 1} a_i \deg(\mathscr{O}(\bb{D}_i)|_{\bb{S}}) = 0$ (see \cref{def:IntLinStructure}). Taking the dual pairing,
    $$ \langle \sum_{i = 1}^{n + 1} a_i [\bb{D}_i] , \sum_{v \in e} d_e \trop(\frak{f}) \rangle = \sum_{i = 1}^{n + 1}a_i  \big(\sum_{v \in e} d_e \trop(\frak{f})^* \bb{D}_i^{\trop} \big) $$
    $$= \sum_{i = 1}^{n + 1} a_i  \deg \big( (\frak{f}_s |_{C_v})^* \mathscr{O}(\bb{D}_i)\big)$$
    There are two possibilities. Either $\frak{f}_s$ contracts $C_v$ to a point, and then the sum vanishes, or else $\bb{S}$ is an algebraic 1-stratum of $\bb{D}$ and $(\frak{f}_s)_*([C_v]) = d[\bb{S}] \in \text{NE}(\bb{Y}_s)$ for some $d > 0$, in which case the sum becomes 
    $$ = d \sum_{i = 1}^{n + 1} a_i \deg\big( \mathscr{O}(\bb{D}_i)|_{\bb{S}} \big) = 0. $$
    So the pairing is always $0$, giving $\sum_{v \in e} d_e \trop(\frak{f}) = 0$. 
\end{proof}
\end{proposition}

\bibliographystyle{plain}
\bibliography{Ldahema}

@article {Reisner_Cohen-Macaulay_quotients,
	AUTHOR = {Reisner, Gerald Allen},
	TITLE = {Cohen-{M}acaulay quotients of polynomial rings},
	JOURNAL = {Advances in Math.},
	FJOURNAL = {Advances in Mathematics},
	VOLUME = {21},
	YEAR = {1976},
	NUMBER = {1},
	PAGES = {30--49},
	ISSN = {0001-8708},
	MRCLASS = {14M05 (13C10)},
	MRNUMBER = {407036},
	MRREVIEWER = {Ezio Stagnaro},
	DOI = {10.1016/0001-8708(76)90114-6},
	URL = {https://doi.org/10.1016/0001-8708(76)90114-6},
}

@article{HPV24,
        title={The KSBA moduli space of stable log CY surfaces},
        author={Alexeev, Valery and Arguz, Hulya and Pierrick Bousseau, Pierrick},
	journal= {arXiv:2402.15117 preprint},
	year ={2024},
	}

@article{GHKSK3,
	title={Theta functions for {K}3 surfaces},
	author={Gross, Mark and Hacking, Paul and Keel, Sean and Siebert, Bernd},
	journal={preprint},
	year={2018},
}

@article {HX,
    AUTHOR = {Hacon, Christopher D. and Xu, Chenyang},
     TITLE = {Existence of log canonical closures},
   JOURNAL = {Invent. Math.},
  FJOURNAL = {Inventiones Mathematicae},
    VOLUME = {192},
      YEAR = {2013},
    NUMBER = {1},
     PAGES = {161--195},
      ISSN = {0020-9910,1432-1297},
      }

@article {BCHM,
	AUTHOR = {Birkar, Caucher and Cascini, Paolo and Hacon, Christopher D.
	and McKernan, James},
	TITLE = {Existence of minimal models for varieties of log general type},
	JOURNAL = {J. Amer. Math. Soc.},
	FJOURNAL = {Journal of the American Mathematical Society},
	VOLUME = {23},
	YEAR = {2010},
	NUMBER = {2},
	PAGES = {405--468},
	ISSN = {0894-0347},
	MRCLASS = {14E30 (14E05)},
	MRNUMBER = {2601039},
	MRREVIEWER = {Mark Gross},
	DOI = {10.1090/S0894-0347-09-00649-3},
	URL = {https://doi.org/10.1090/S0894-0347-09-00649-3},
}

@incollection {KdFX,
    AUTHOR = {de Fernex, Tommaso and Koll\'{a}r, J\'{a}nos and Xu, Chenyang},
     TITLE = {The dual complex of singularities},
 BOOKTITLE = {Higher dimensional algebraic geometry---in honour of
              {P}rofessor {Y}ujiro {K}awamata's sixtieth birthday},
    SERIES = {Adv. Stud. Pure Math.},
    VOLUME = {74},
     PAGES = {103--129},
 PUBLISHER = {Math. Soc. Japan, Tokyo},
      YEAR = {2017},
   MRCLASS = {14B05 (14D06 14J17)},
  MRNUMBER = {3791210},
MRREVIEWER = {Chen Jiang},
       DOI = {10.2969/aspm/07410103},
       URL = {https://doi.org/10.2969/aspm/07410103},
}

@article {KX,
	AUTHOR = {Koll\'{a}r, J\'{a}nos and Xu, Chenyang},
	TITLE = {The dual complex of {C}alabi-{Y}au pairs},
	JOURNAL = {Invent. Math.},
	FJOURNAL = {Inventiones Mathematicae},
	VOLUME = {205},
	YEAR = {2016},
	NUMBER = {3},
	PAGES = {527--557},
	ISSN = {0020-9910},
	MRCLASS = {14J17 (14D06 14F35 14J32)},
	MRNUMBER = {3539921},
	MRREVIEWER = {Alan Matthew Thompson},
	DOI = {10.1007/s00222-015-0640-6},
	URL = {https://doi.org/10.1007/s00222-015-0640-6},
}

@incollection {Kollar_Moduli_of_varieties_of_general_type,
	AUTHOR = {Koll\'{a}r, J\'{a}nos},
	TITLE = {Moduli of varieties of general type},
	BOOKTITLE = {Handbook of moduli. {V}ol. {II}},
	SERIES = {Adv. Lect. Math. (ALM)},
	VOLUME = {25},
	PAGES = {131--157},
	PUBLISHER = {Int. Press, Somerville, MA},
	YEAR = {2013},
	MRCLASS = {14D20 (14D22)},
	MRNUMBER = {3184176},
	MRREVIEWER = {Nicolae Manolache},
}

@incollection {Alexeev_Moduli_spaces_MgnW,
	AUTHOR = {Alexeev, Valery},
	TITLE = {Moduli spaces {$M_{g,n}(W)$} for surfaces},
	BOOKTITLE = {Higher-dimensional complex varieties ({T}rento, 1994)},
	PAGES = {1--22},
	PUBLISHER = {de Gruyter, Berlin},
	YEAR = {1996},
	MRCLASS = {14D20 (14D22 14J10)},
	MRNUMBER = {1463171},
	MRREVIEWER = {Wolfgang K. Seiler},
	DOI = {10.1006/jcat.1996.0357},
	URL = {https://doi.org/10.1006/jcat.1996.0357},
}

@article {Kollar_Threefolds_and_deformations,
	AUTHOR = {Koll\'{a}r, J. and Shepherd-Barron, N. I.},
	TITLE = {Threefolds and deformations of surface singularities},
	JOURNAL = {Invent. Math.},
	FJOURNAL = {Inventiones Mathematicae},
	VOLUME = {91},
	YEAR = {1988},
	NUMBER = {2},
	PAGES = {299--338},
	ISSN = {0020-9910},
	MRCLASS = {14J10 (14D20 14J30 32G10 32G13)},
	MRNUMBER = {922803},
	MRREVIEWER = {Yujiro Kawamata},
	DOI = {10.1007/BF01389370},
	URL = {https://doi.org/10.1007/BF01389370},
}

@incollection {Hu_Mori_dream_spaces,
	AUTHOR = {Hu, Yi and Keel, Sean},
	TITLE = {Mori dream spaces and {GIT}},
	NOTE = {Dedicated to William Fulton on the occasion of his 60th
	birthday},
	JOURNAL = {Michigan Math. J.},
	FJOURNAL = {Michigan Mathematical Journal},
	VOLUME = {48},
	YEAR = {2000},
	PAGES = {331--348},
	ISSN = {0026-2285},
	MRCLASS = {14L24 (14E30)},
	MRNUMBER = {1786494},
	MRREVIEWER = {P. E. Newstead},
	DOI = {10.1307/mmj/1030132722},
	URL = {https://doi.org/10.1307/mmj/1030132722},
}

@article{LZ23,
	title= {Mirror Symmetry for log Calabi-Yau surfaces II},
	author={Lai, Jonathon and Zhou, Yan},
	journal={arXiv preprint arXiv:2201.12703},
	year={2022}
}

@article{Travis25,
        title ={VALUATIVE INDEPENDENCE AND CLUSTER THETA RECIPROCITY},
	author={Man-Wei Cheung and Travis Mandel and  Timothy Magee and Greg Muller},
	journal={preprint},
	year ={2024},
	}

@book {Kollar_Singularities_of_the_minimal_model_program,
	AUTHOR = {Koll\'{a}r, J\'{a}nos},
	TITLE = {Singularities of the minimal model program},
	SERIES = {Cambridge Tracts in Mathematics},
	VOLUME = {200},
	NOTE = {With a collaboration of S\'{a}ndor Kov\'{a}cs},
	PUBLISHER = {Cambridge University Press, Cambridge},
	YEAR = {2013},
	PAGES = {x+370},
	ISBN = {978-1-107-03534-8},
	MRCLASS = {14E30 (14B05)},
	MRNUMBER = {3057950},
	MRREVIEWER = {Tommaso De Fernex},
	DOI = {10.1017/CBO9781139547895},
	URL = {https://doi.org/10.1017/CBO9781139547895},
}

@article {Abramovich_Torification_and_factorization,
	AUTHOR = {Abramovich, Dan and Karu, Kalle and Matsuki, Kenji and
	W\l{}odarczyk, Jaros\l{}aw},
	TITLE = {Torification and factorization of birational maps},
	JOURNAL = {J. Amer. Math. Soc.},
	FJOURNAL = {Journal of the American Mathematical Society},
	VOLUME = {15},
	YEAR = {2002},
	NUMBER = {3},
	PAGES = {531--572},
	ISSN = {0894-0347},
	MRCLASS = {14E05 (14E15 14L24)},
	MRNUMBER = {1896232},
	MRREVIEWER = {Alexandr V. Pukhlikov},
	DOI = {10.1090/S0894-0347-02-00396-X},
	URL = {https://doi.org/10.1090/S0894-0347-02-00396-X},
}

@article {Mandel_Tropical_theta_functions,
	AUTHOR = {Mandel, Travis},
	TITLE = {Tropical theta functions and log {C}alabi-{Y}au surfaces},
	JOURNAL = {Selecta Math. (N.S.)},
	FJOURNAL = {Selecta Mathematica. New Series},
	VOLUME = {22},
	YEAR = {2016},
	NUMBER = {3},
	PAGES = {1289--1335},
	ISSN = {1022-1824},
	MRCLASS = {14J32 (13F60 14M25 14T05)},
	MRNUMBER = {3518552},
	MRREVIEWER = {Hoil Kim},
	DOI = {10.1007/s00029-015-0221-y},
	URL = {http://dx.doi.org/10.1007/s00029-015-0221-y},
}

@book {Kollar_Rational_curves_on_algebraic_varieties,
	AUTHOR = {Koll{\'a}r, J{\'a}nos},
	TITLE = {Rational curves on algebraic varieties},
	SERIES = {Ergebnisse der Mathematik und ihrer Grenzgebiete. 3. Folge. A
	Series of Modern Surveys in Mathematics [Results in
	Mathematics and Related Areas. 3rd Series. A Series of Modern
	Surveys in Mathematics]},
	VOLUME = {32},
	PUBLISHER = {Springer-Verlag, Berlin},
	YEAR = {1996},
	PAGES = {viii+320},
	ISBN = {3-540-60168-6},
	MRCLASS = {14-02 (14C05 14E05 14F17 14J45)},
	MRNUMBER = {1440180},
	MRREVIEWER = {Yuri G. Prokhorov},
	DOI = {10.1007/978-3-662-03276-3},
	URL = {http://dx.doi.org/10.1007/978-3-662-03276-3},
}

@article {Gross_Moduli_of_surfaces,
	AUTHOR = {Gross, Mark and Hacking, Paul and Keel, Sean},
	TITLE = {Moduli of surfaces with an anti-canonical cycle},
	JOURNAL = {Compos. Math.},
	FJOURNAL = {Compositio Mathematica},
	VOLUME = {151},
	YEAR = {2015},
	NUMBER = {2},
	PAGES = {265--291},
	ISSN = {0010-437X},
	MRCLASS = {14J10 (14J26)},
	MRNUMBER = {3314827},
	MRREVIEWER = {Makiko Mase},
	DOI = {10.1112/S0010437X14007611},
	URL = {http://dx.doi.org/10.1112/S0010437X14007611},
}

@book {Hartshorne_Algebraic_geometry,
	AUTHOR = {Hartshorne, Robin},
	TITLE = {Algebraic geometry},
	NOTE = {Graduate Texts in Mathematics, No. 52},
	PUBLISHER = {Springer-Verlag, New York-Heidelberg},
	YEAR = {1977},
	PAGES = {xvi+496},
	ISBN = {0-387-90244-9},
	MRCLASS = {14-01},
	MRNUMBER = {0463157},
	MRREVIEWER = {Robert Speiser},
}

@article{Temkin_Metrization_of_differential_pluriforms,
	title={Metrization of differential pluriforms on Berkovich analytic spaces},
	author={Temkin, Michael},
	journal={arXiv preprint arXiv:1410.3079},
	year={2014}
}

@book {FK,
    AUTHOR = {Fujiwara, Kazuhiro and Kato, Fumiharu},
     TITLE = {Foundations of rigid geometry. {I}},
    SERIES = {EMS Monographs in Mathematics},
 PUBLISHER = {European Mathematical Society (EMS), Z\"urich},
      YEAR = {2018},
     PAGES = {xxxiv+829},
      ISBN = {978-3-03719-135-4},
   MRCLASS = {14G22 (14F05)},
  MRNUMBER = {3752648},
  }

@preamble{
   "\def\cprime{$'$} "
}

@article {Gross_Theta_Functions,
    AUTHOR = {Gross, Mark and Hacking, Paul and Siebert, Bernd},
     TITLE = {Theta functions on varieties with effective anti-canonical
              class},
   JOURNAL = {Mem. Amer. Math. Soc.},
  FJOURNAL = {Memoirs of the American Mathematical Society},
    VOLUME = {278},
      YEAR = {2022},
    NUMBER = {1367},
     PAGES = {xii+103},
      ISSN = {0065-9266,1947-6221},
      ISBN = {978-1-4704-5297-1; 978-1-4704-7167-5},
   MRCLASS = {14J33 (14J32 14J45)},
  MRNUMBER = {4426709},
MRREVIEWER = {Anatoly\ Libgober},
       DOI = {10.1090/memo/1367},
       URL = {https://doi.org/10.1090/memo/1367},
}

@article {Gross_Canonical_bases,
	AUTHOR = {Gross, Mark and Hacking, Paul and Keel, Sean and Kontsevich,
	Maxim},
	TITLE = {Canonical bases for cluster algebras},
	JOURNAL = {J. Amer. Math. Soc.},
	FJOURNAL = {Journal of the American Mathematical Society},
	VOLUME = {31},
	YEAR = {2018},
	NUMBER = {2},
	PAGES = {497--608},
	ISSN = {0894-0347},
	MRCLASS = {13F60 (14J33)},
	MRNUMBER = {3758151},
	DOI = {10.1090/jams/890},
	URL = {https://doi.org/10.1090/jams/890},
}

@article {Strominger_Mirror_symmetry_is_T-duality,
    AUTHOR = {Strominger, Andrew and Yau, Shing-Tung and Zaslow, Eric},
     TITLE = {Mirror symmetry is {$T$}-duality},
   JOURNAL = {Nuclear Phys. B},
  FJOURNAL = {Nuclear Physics. B},
    VOLUME = {479},
      YEAR = {1996},
    NUMBER = {1-2},
     PAGES = {243--259},
      ISSN = {0550-3213},
     CODEN = {NUPBBO},
   MRCLASS = {32J17 (14J32 32J81 81T30)},
  MRNUMBER = {1429831 (97j:32022)},
MRREVIEWER = {Mark Gross},
       DOI = {10.1016/0550-3213(96)00434-8},
       URL = {http://dx.doi.org/10.1016/0550-3213(96)00434-8},
}

@book {Fresnel_Rigid_analytic_geometry_and_its_applications,
    AUTHOR = {Fresnel, Jean and van der Put, Marius},
     TITLE = {Rigid analytic geometry and its applications},
    SERIES = {Progress in Mathematics},
    VOLUME = {218},
 PUBLISHER = {Birkh\"auser Boston, Inc., Boston, MA},
      YEAR = {2004},
     PAGES = {xii+296},
      ISBN = {0-8176-4206-4},
   MRCLASS = {14G22 (30G06 30H05 32P05)},
  MRNUMBER = {2014891 (2004i:14023)},
MRREVIEWER = {Lorenzo Ramero},
       DOI = {10.1007/978-1-4612-0041-3},
       URL = {http://dx.doi.org/10.1007/978-1-4612-0041-3},
}

@article {Thuillier_Geometrie_toroidale,
    AUTHOR = {Thuillier, Amaury},
     TITLE = {G\'eom\'etrie toro\"\i dale et g\'eom\'etrie analytique non
              archim\'edienne. {A}pplication au type d'homotopie de certains
              sch\'emas formels},
   JOURNAL = {Manuscripta Math.},
  FJOURNAL = {Manuscripta Mathematica},
    VOLUME = {123},
      YEAR = {2007},
    NUMBER = {4},
     PAGES = {381--451},
      ISSN = {0025-2611},
     CODEN = {MSMHB2},
   MRCLASS = {14G22 (14E15 14M25)},
  MRNUMBER = {2320738 (2008g:14038)},
MRREVIEWER = {Alessandra Bertapelle},
       DOI = {10.1007/s00229-007-0094-2},
       URL = {http://dx.doi.org/10.1007/s00229-007-0094-2},
}

@article{Ducros_Families_of_Berkovich_spaces,
  title={Families of {B}erkovich spaces},
  author={Ducros, Antoine},
  journal={arXiv preprint arXiv:1107.4259},
  year={2011}
}

@book{Bosch_Non-Archimedean_analysis,
	Author = {Bosch, S. and G{\"u}ntzer, U. and Remmert, R.},
	Isbn = {3-540-12546-9},
	Mrclass = {32K10 (30G05 46P05)},
	Mrnumber = {746961 (86b:32031)},
	Mrreviewer = {W. Bartenwerfer},
	Note = {A systematic approach to rigid analytic geometry},
	Pages = {xii+436},
	Publisher = {Springer-Verlag, Berlin},
	Series = {Grundlehren der Mathematischen Wissenschaften [Fundamental Principles of Mathematical Sciences]},
	Title = {Non-{A}rchimedean analysis},
	Volume = {261},
	Year = {1984}}

@book{Berkovich_Spectral_theory,
	Author = {Berkovich, Vladimir G.},
	Isbn = {0-8218-1534-2},
	Mrclass = {32P05 (32C15 32C37 46S10 47S10)},
	Mrnumber = {1070709 (91k:32038)},
	Mrreviewer = {W. Bartenwerfer},
	Pages = {x+169},
	Publisher = {American Mathematical Society, Providence, RI},
	Series = {Mathematical Surveys and Monographs},
	Title = {Spectral theory and analytic geometry over non-{A}rchimedean fields},
	Volume = {33},
	Year = {1990}}

@article{Johnston_Comparison,
	Author = {Sam Johnson},
	Journal = {arXiv preprint arXiv:2204.00940},
	Title = {Comparison of non-archimedean and logarithmic mirror constructions via the Frobenius structure theorem},
	Year = {2022}}

@article{Berkovich_Vanishing_cycles_for_formal_schemes_II,
	Author = {Berkovich, Vladimir G.},
	Coden = {INVMBH},
	Fjournal = {Inventiones Mathematicae},
	Issn = {0020-9910},
	Journal = {Invent. Math.},
	Mrclass = {14G40 (14C25)},
	Mrnumber = {1395723 (98k:14031)},
	Number = {2},
	Pages = {367--390},
	Title = {Vanishing cycles for formal schemes. {II}},
	Volume = {125},
	Year = {1996}}

@article{Berkovich_Vanishing_cycles_for_formal_schemes,
	Author = {Berkovich, Vladimir G.},
	Coden = {INVMBH},
	Fjournal = {Inventiones Mathematicae},
	Issn = {0020-9910},
	Journal = {Invent. Math.},
	Mrclass = {14F20 (14D15 14G20 32C18 32P05)},
	Mrnumber = {1262943 (95f:14034)},
	Mrreviewer = {W. Bartenwerfer},
	Number = {3},
	Pages = {539--571},
	Title = {Vanishing cycles for formal schemes},
	Volume = {115},
	Year = {1994}}

@article{Gross_Mirror_symmetry_for_log_Calabi-Yau_surfaces_I_v1,
	Author = {Gross, Mark and Hacking, Paul and Keel, Sean},
	Journal = {arXiv preprint arXiv:1106.4977v1},
	Title = {Mirror symmetry for log {C}alabi-{Y}au surfaces {I}},
	Year = {2011}}

@article{GHKIv2,
	AUTHOR = {Gross, Mark and Hacking, Paul and Keel, Sean},
	TITLE = {Mirror symmetry for log {C}alabi-{Y}au surfaces {I}},
	JOURNAL = {Publ. Math. Inst. Hautes \'{E}tudes Sci.},
	FJOURNAL = {Publications Math\'{e}matiques. Institut de Hautes \'{E}tudes
	Scientifiques},
	VOLUME = {122},
	YEAR = {2015},
	PAGES = {65--168},
	ISSN = {0073-8301},
	MRCLASS = {14J33 (14J32 14N35)},
	MRNUMBER = {3415066},
	MRREVIEWER = {Amin Gholampour},
	DOI = {10.1007/s10240-015-0073-1},
	URL = {https://doi.org/10.1007/s10240-015-0073-1},
}

@article{Gross_Mirror_symmetry_for_log_Calabi-Yau_surfaces_I_published,
	AUTHOR = {Gross, Mark and Hacking, Paul and Keel, Sean},
	TITLE = {Mirror symmetry for log {C}alabi-{Y}au surfaces {I}},
	JOURNAL = {Publ. Math. Inst. Hautes \'{E}tudes Sci.},
	FJOURNAL = {Publications Math\'{e}matiques. Institut de Hautes \'{E}tudes
	Scientifiques},
	VOLUME = {122},
	YEAR = {2015},
	PAGES = {65--168},
	ISSN = {0073-8301},
	MRCLASS = {14J33 (14J32 14N35)},
	MRNUMBER = {3415066},
	MRREVIEWER = {Amin Gholampour},
	DOI = {10.1007/s10240-015-0073-1},
	URL = {https://doi.org/10.1007/s10240-015-0073-1},
}

@article{Conrad_Descent,
	Author = {Conrad, Brian and Temkin, Michael},
	Journal = {Preprint},
	Publisher = {Citeseer},
	Title = {Descent for non-archimedean analytic spaces},
	Year = {2010}}

@article {Magee,
    AUTHOR = {Magee, Timothy},
     TITLE = {Littlewood-{R}ichardson coefficients via mirror symmetry for
              cluster varieties},
   JOURNAL = {Proc. Lond. Math. Soc. (3)},
  FJOURNAL = {Proceedings of the London Mathematical Society. Third Series},
    VOLUME = {121},
      YEAR = {2020},
    NUMBER = {3},
     PAGES = {463--512},
      ISSN = {0024-6115},
   MRCLASS = {14J33 (05E10 05E14 13F60)},
  MRNUMBER = {4100115},
MRREVIEWER = {Li Li},
       DOI = {10.1112/plms.12329},
       URL = {https://doi.org/10.1112/plms.12329},
}

@article {Temkin_On_local_properties_II,
	AUTHOR = {Temkin, M.},
	TITLE = {On local properties of non-{A}rchimedean analytic spaces.
	{II}},
	JOURNAL = {Israel J. Math.},
	FJOURNAL = {Israel Journal of Mathematics},
	VOLUME = {140},
	YEAR = {2004},
	PAGES = {1--27},
	ISSN = {0021-2172},
	MRCLASS = {14G22 (32P05)},
	MRNUMBER = {2054837},
	MRREVIEWER = {Elmar Grosse-Kl\"{o}nne},
	DOI = {10.1007/BF02786625},
	URL = {https://doi.org/10.1007/BF02786625},
}

@book{Kempf_Toroidal_embeddings,
	Address = {Berlin},
	Author = {Kempf, G. and Knudsen, Finn Faye and Mumford, D. and Saint-Donat, B.},
	Mrclass = {14E15 (14D20 14E05 14M20 20G15)},
	Mrnumber = {0335518 (49 \#299)},
	Mrreviewer = {G. Harder},
	Pages = {viii+209},
	Publisher = {Springer-Verlag},
	Series = {Lecture Notes in Mathematics, Vol. 339},
	Title = {Toroidal embeddings. {I}},
	Year = {1973}}

@misc{Stacks_project,
	Author = {The {Stacks Project Authors}},
	Howpublished = {\url{http://stacks.math.columbia.edu}},
	Shorthand = {Stacks},
	Title = {{Stacks Project}},
	Year = {2013}}

@incollection{Kontsevich_Affine_structures,
	Address = {Boston, MA},
	Author = {Kontsevich, Maxim and Soibelman, Yan},
	Booktitle = {The unity of mathematics},
	Mrclass = {14J32 (14G22 32Q25)},
	Mrnumber = {2181810 (2006j:14054)},
	Mrreviewer = {Mark Gross},
	Pages = {321--385},
	Publisher = {Birkh\"auser Boston},
	Series = {Progr. Math.},
	Title = {Affine structures and non-{A}rchimedean analytic spaces},
	Volume = {244},
	Year = {2006}}

@article{Berkovich_Etale_cohomology,
	Author = {Berkovich, Vladimir G.},
	Coden = {PMIHA6},
	Fjournal = {Institut des Hautes \'Etudes Scientifiques. Publications Math\'ematiques},
	Issn = {0073-8301},
	Journal = {Inst. Hautes \'Etudes Sci. Publ. Math.},
	Mrclass = {14F20 (14F30 32C37 32P05)},
	Mrnumber = {1259429 (95c:14017)},
	Mrreviewer = {W. Bartenwerfer},
	Number = {78},
	Pages = {5--161 (1994)},
	Title = {\'{E}tale cohomology for non-{A}rchimedean analytic spaces},
	Year = {1993}}

@article {Yu_Enumeration_of_holomorphic_cylinders_I,
	AUTHOR = {Yu, Tony Yue},
	TITLE = {Enumeration of holomorphic cylinders in log {C}alabi--{Y}au
	surfaces. {I}},
	JOURNAL = {Math. Ann.},
	FJOURNAL = {Mathematische Annalen},
	VOLUME = {366},
	YEAR = {2016},
	NUMBER = {3-4},
	PAGES = {1649--1675},
	ISSN = {0025-5831},
	CODEN = {MAANA},
	MRCLASS = {14N35 (14G22 14J26 14J32 14T05)},
	MRNUMBER = {3563248},
	DOI = {10.1007/s00208-016-1376-3},
	URL = {http://dx.doi.org/10.1007/s00208-016-1376-3},
}

@article{Yu_Enumeration_of_holomorphic_cylinders_II,
	title={Enumeration of holomorphic cylinders in log {C}alabi-{Y}au surfaces. {II}. {P}ositivity, integrality and the gluing formula},
	author={Yu, Tony Yue},
	journal={arXiv preprint arXiv:1608.07651},
	year={2016}
}

@article{Yu_Tropicalization_of_the_moduli_space_of_stable_maps,
	Author = {Yu, Tony Yue},
	Journal = {Mathematische Zeitschrift},
	Title = {Tropicalization of the moduli space of stable maps},
	VOLUME = {281},
	Number = {3},
	pages = {1035-1059},
	DOI = {10.1007/s00209-015-1519-3},
	Year = {2015}}

@article{Yu_Gromov_compactness,
	AUTHOR = {Yu, Tony Yue},
	TITLE = {Gromov compactness in non-archimedean analytic geometry},
	JOURNAL = {J. Reine Angew. Math.},
	FJOURNAL = {Journal f\"{u}r die Reine und Angewandte Mathematik. [Crelle's
	Journal]},
	VOLUME = {741},
	YEAR = {2018},
	PAGES = {179--210},
	ISSN = {0075-4102},
	MRCLASS = {14G22 (53D30)},
	MRNUMBER = {3836147},
	MRREVIEWER = {Marco A. Garuti},
	DOI = {10.1515/crelle-2015-0077},
	URL = {https://doi.org/10.1515/crelle-2015-0077},
}

@article{Keel_Yu_The_Frobenius,
	title={The {F}robenius structure theorem for affine log
	{C}alabi-{Y}au varieties containing a torus},
	author={Keel, Sean and Yu, Tony Yue},
	journal={arXiv preprint arXiv:1908.09861},
	year={2019}
}

@article{Porta_Yu_Representability_theorem,
	Author = {Porta, Mauro and Yu, Tony Yue},
	Title = {Representability theorem in derived analytic geometry},
	Journal = {arXiv preprint arXiv:1704.01683},
	note = {To appear in \emph{Journal of the European Mathematical Society}},
	Year = {2017}}

@article{Porta_Yu_Derived_Hom_spaces,
	title={Derived Hom spaces in rigid analytic geometry},
	author={Porta, Mauro and Yu, Tony Yue},
	journal={arXiv preprint arXiv:1801.07730},
	note = {To appear in \emph{Publications of the Research Institute for Mathematical Sciences, Special issue dedicated to Professor Masaki Kashiwara on his seventieth birthday}},
	year={2018},
}

@article{Porta_Yu_Non-archimedean_quantum_K-invariants,
	title={Non-archimedean quantum K-invariants},
	author={Porta, Mauro and Yu, Tony Yue},
	note = {In preparation},
}

@article{Nicaise_Xu_Yu_The_non-archimedean_SYZ_fibration,
	AUTHOR = {Nicaise, Johannes and Xu, Chenyang and Yu, Tony Yue},
	TITLE = {The non-archimedean {SYZ} fibration},
	JOURNAL = {Compos. Math.},
	FJOURNAL = {Compositio Mathematica},
	VOLUME = {155},
	YEAR = {2019},
	NUMBER = {5},
	PAGES = {953--972},
	ISSN = {0010-437X},
	MRCLASS = {14J33 (14E30 14G22 32Q25)},
	MRNUMBER = {3946280},
	DOI = {10.1112/s0010437x19007152},
	URL = {https://doi.org/10.1112/s0010437x19007152},
}

@article {Gubler_Skeletons_and_tropicalizations,
	AUTHOR = {Gubler, Walter and Rabinoff, Joseph and Werner, Annette},
	TITLE = {Skeletons and tropicalizations},
	JOURNAL = {Adv. Math.},
	FJOURNAL = {Advances in Mathematics},
	VOLUME = {294},
	YEAR = {2016},
	PAGES = {150--215},
	ISSN = {0001-8708},
	MRCLASS = {14G22 (14T05)},
	MRNUMBER = {3479562},
	MRREVIEWER = {Siu-Cheong Lau},
	DOI = {10.1016/j.aim.2016.02.022},
	URL = {https://doi.org/10.1016/j.aim.2016.02.022},
}

@article {HKY20,
	AUTHOR = {Hacking, Paul and Keel, Sean and Yu, Tony},
	TITLE = {Secondary Fan, theta functions and
	moduli of Calabi-Yau pairs},
	JOURNAL = {arXiv preprint arXiv:2008.02299},
        year = {2020},
}

@article {HDTV,
	AUTHOR = {ARguz, Hulya and Gross, Mark},
	TITLE = {The Higher Dimensional Tropical Vertex},
	JOURNAL = {arXiv preprint arXiv:2007.08347v1},
        year = {2020},
}

@article {Keel_Rational_curves,
	AUTHOR = {Keel, Se{\'a}n and McKernan, James},
	TITLE = {Rational curves on quasi-projective surfaces},
	JOURNAL = {Mem. Amer. Math. Soc.},
	FJOURNAL = {Memoirs of the American Mathematical Society},
	VOLUME = {140},
	YEAR = {1999},
	NUMBER = {669},
	PAGES = {viii+153},
	ISSN = {0065-9266},
	CODEN = {MAMCAU},
	MRCLASS = {14J26 (14B10 14E20 14E30)},
	MRNUMBER = {1610249},
	MRREVIEWER = {Alessio Corti},
	DOI = {10.1090/memo/0669},
	URL = {http://dx.doi.org/10.1090/memo/0669},
}

@article {GHKK,
	AUTHOR = {Gross, Mark and Hacking, Paul and Keel, Sean and Kontsevich,
	Maxim},
	TITLE = {Canonical bases for cluster algebras},
	JOURNAL = {J. Amer. Math. Soc.},
	FJOURNAL = {Journal of the American Mathematical Society},
	VOLUME = {31},
	YEAR = {2018},
	NUMBER = {2},
	PAGES = {497--608},
	ISSN = {0894-0347},
	MRCLASS = {13F60 (14J33)},
	MRNUMBER = {3758151},
	DOI = {10.1090/jams/890},
	URL = {https://doi.org/10.1090/jams/890},
}

@article {KW24,
	AUTHOR = {Gross, Mark and Hacking, Paul and Keel, Sean},
	TITLE = {MTheta Function Basis of the Cox ring of Postive 2d Looijenga pairs},
	JOURNAL = {archiv preprint arXiv:2412.01774},
	YEAR = {2025},
}

@article{LoganThesis,
  Title={Mirror algebras for formal deformations},
  Author={Logan White},
  Journal={UT Austin Ph D Thesis},
  Year   = {2024}
  }

@article{LS,
Title={Skeletal curves with boundary},
Author={Logan White},
Journal ={preprint},
Year = {2026}
}

@unpublished{BL,
	Author = {Blum, Harold and Liu, Yuchen},
	note = {preprint},
	Title = {Valuative Independence for Calabi--Yau varieties},
	journal ={archiv preprint arXiv:2604.27890},
	Year = {2026}}

@inproceedings{Baker_On_the_structure,
	Author = {Baker, Matthew and Payne, Sam and Rabinoff, Joseph},
	Booktitle = {Tropical and Non-archimedean Geometry},
	Isbn = {9781470410216},
	Lccn = {2013027062},
	Pages = {93--121},
	Publisher = {Amer. Math. Soc., Providence, RI},
	Series = {Contemp. Math.},
	Title = {On the structure of nonarchimedean analytic curves},
	Volume = {605},
	Year = {2013}}

@article {GSIMS,
    AUTHOR = {Gross, Mark and Siebert, Bernd},
     TITLE = {Intrinsic mirror symmetry},
   JOURNAL = {J. Amer. Math. Soc.},
  FJOURNAL = {Journal of the American Mathematical Society},
    VOLUME = {39},
      YEAR = {2026},
    NUMBER = {2},
     PAGES = {313--451},
      ISSN = {0894-0347,1088-6834},
   MRCLASS = {14J33 (14A21 14N35)} 
}

@unpublished{GSCWS,
    AUTHOR = {Gross, Mark and Siebert, Bernd},
     TITLE = {The Canonical Wall Structure and Mirror Symmetry},
   JOURNAL = {preprint arXiv:2105.02502},
      YEAR = {2021} }

@article {Wlod,
	AUTHOR = {W\l odarczyk, Jaros\l aw},
	TITLE = {TOROIDAL VARIETIES AND THE WEAK FACTORIZATION THEOREM},
	JOURNAL = {archiv preprint arXiv:math/9904076v5},
	YEAR = {2002},
}

@article {Ducros_families,
    AUTHOR = {Ducros, Antoine},
     TITLE = {Families of {B}erkovich spaces},
   JOURNAL = {Ast\'{e}risque},
  FJOURNAL = {Ast\'{e}risque},
      YEAR = {2018},
    NUMBER = {400},
     PAGES = {vii+262},
      ISSN = {0303-1179,2492-5926},
      ISBN = {978-2-85629-885-5},
      MRCLASS = {14G22},
  MRNUMBER = {3826929},
MRREVIEWER = {Marco\ Maculan},
}

@article{PY,
         TITLE = {Non-Archimedean Gromov-Witten Invariants},
	author={Porta, Mauro and Yu, Tony Yue},
	journal={arXiv preprint 2209.13176 },
	year = {2022},
	}

@misc{achinger2021specializationproetalefundamentalgroup,
      title={Specialization for the pro-\'etale fundamental group}, 
      author={Piotr Achinger and Marcin Lara and Alex Youcis},
      year={2021},
      eprint={2107.06761},
      archivePrefix={arXiv},
      primaryClass={math.AG},
      url={https://arxiv.org/abs/2107.06761}, 
}

@article{Ben_Bassat_2013,
   title={Berkovich spaces and tubular descent},
   volume={234},
   ISSN={0001-8708},
   url={http://dx.doi.org/10.1016/j.aim.2012.10.016},
   DOI={10.1016/j.aim.2012.10.016},
   journal={Advances in Mathematics},
   publisher={Elsevier BV},
   author={Ben-Bassat, Oren and Temkin, Michael},
   year={2013},
   month=Feb, pages={217–238} }

@misc{ducros2009lesespacesberkovichsont,
      title={Les espaces de Berkovich sont excellents}, 
      author={Antoine Ducros},
      year={2009},
      eprint={0706.0666},
      archivePrefix={arXiv},
      primaryClass={math.AG},
      url={https://arxiv.org/abs/0706.0666}, 
}

@article{Conrad1999,
  author    = {Conrad, Brian},
  title     = {Irreducible components of rigid spaces},
  journal   = {Annales de l'Institut Fourier},
  volume    = {49},
  number    = {2},
  pages     = {473--541},
  year      = {1999},
  publisher = {Association des Annales de l'Institut Fourier},
  doi       = {10.5802/aif.1681},
  url       = {https://aif.centre-mersenne.org/articles/10.5802/aif.1681/}
}

@misc{fujiwara2017foundationsrigidgeometryi,
      title={Foundations of Rigid Geometry I}, 
      author={Kazuhiro Fujiwara and Fumiharu Kato},
      year={2017},
      eprint={1308.4734},
      archivePrefix={arXiv},
      primaryClass={math.AG},
      url={https://arxiv.org/abs/1308.4734}, 
}

@misc{temkin2017functorialdesingularizationqboundaries,
      title={Functorial desingularization over Q: boundaries and the embedded case}, 
      author={Michael Temkin},
      year={2017},
      eprint={0912.2570},
      archivePrefix={arXiv},
      primaryClass={math.AG},
      url={https://arxiv.org/abs/0912.2570}, 
}

@article{Temkin_2012,
   title={Functorial desingularization of quasi-excellent schemes in characteristic zero: the nonembedded case},
   volume={161},
   ISSN={0012-7094},
   url={http://dx.doi.org/10.1215/00127094-1699539},
   DOI={10.1215/00127094-1699539},
   number={11},
   journal={Duke Mathematical Journal},
   publisher={Duke University Press},
   author={Temkin, Michael},
   year={2012},
   month=Aug }

@phdthesis{thuillier2005theorie,
  author       = {Thuillier, Amaury},
  title        = {Th{\'e}orie du potentiel sur les courbes en g{\'e}om{\'e}trie analytique non archim{\'e}dienne. Applications {\`a} la th{\'e}orie d'Arakelov},
  school       = {Universit{\'e} Rennes 1},
  year         = {2005},
  type         = {Ph.D. Thesis},
  address      = {Rennes, France},
  note         = {viii + 184 p.},
  url          = {https://theses.hal.science/tel-00010990v1}
}
\end{document}